\documentclass[reqno]{amsart}
\usepackage{amssymb}
\usepackage{amsmath}
\usepackage{hyperref}
\hypersetup{colorlinks=true,linkcolor=blue,citecolor=red}

\usepackage[margin=1in]{geometry}
\usepackage{amsfonts,amsmath,amssymb}
\usepackage{amsthm,nccmath}
\usepackage{graphicx}
\usepackage{float}
\usepackage{lipsum}
\usepackage{comment}
\usepackage{systeme}
\usepackage{lineno}
\usepackage{mathtools}
\usepackage{esint}
\usepackage{mathrsfs}
\usepackage{theoremref}
\usepackage{empheq}

\theoremstyle{definition}

\numberwithin{equation}{section}

\newcommand{\R}{\mathbb{R}}
\newcommand{\norm}[1]{\left\lVert#1\right\rVert}
\newcommand{\inner}[2]{\left\langle #1, #2 \right\rangle}
\newcommand{\I}{\int\limits}
\newcommand{\FI}{\fint\limits}
\newcommand{\F}{\mathcal{F}}

\newcommand{\RN}[1]{%
  \textup{\uppercase\expandafter{\romannumeral#1}}%
}

\DeclareMathOperator{\Supp}{supp}
\DeclareMathOperator{\loc}{loc}
\DeclareMathOperator{\Div}{div}
\DeclareMathOperator{\dist}{dist}
\DeclareMathOperator{\data}{\textbf{data}}
\DeclareMathOperator*{\osc}{osc}

\def\u{\upsilon}
\def\O{\Omega}

\def\F{\mathcal{F}}
\def\P{\mathcal{P}}
\def\N{\mathcal{N}}

\theoremstyle{plain}
\newtheorem{thm}{Theorem}[section]
\newtheorem{lem}{Lemma}[section]

\theoremstyle{definition}
\newtheorem{defn}{Definition}[section]

\theoremstyle{remark}

\newtheorem{rmk}{Remark}[section]
\numberwithin{equation}{section}

\begin{document}

\title[Regularity for Orlicz phase problems]
{Regularity for Orlicz phase problems}


\author{Sumiya Baasandorj}
\address{Department of Mathematical Sciences, Seoul National University, Seoul 08826, Korea}
\email{summa2017@snu.ac.kr}

\author{Sun-Sig Byun}
\address{Department of Mathematical Sciences and Research Institute of Mathematics,
Seoul National University, Seoul 08826, Korea}
\email{byun@snu.ac.kr}

\thanks{This work was supported by NRF-2017R1A2B2003877.}


\date{\today.}



\begin{abstract}
We provide comprehensive regularity results and optimal conditions for a general class of functionals involving Orlicz multi-phase of the type
\begin{align}
	\label{abst:1}
	\u\mapsto \I_{\O} F(x,\u,D\u)\,dx,
\end{align}
exhibiting non-standard growth conditions and non-uniformly elliptic properties. 

The model functional under consideration is given by the Orlicz multi-phase integral
\begin{align}
	\label{abst:2}
	\u\mapsto \I_{\O} f(x,\u)\left[ G(|D\u|) + \sum\limits_{k=1}^{N}a_k(x)H_{k}(|D\u|) \right]\,dx,\quad N\geqslant 1,
\end{align}
where $G,H_{k}$ are $N$-functions and $ 0\leqslant a_{k}(\cdot)\in L^{\infty}(\O) $ with $0 < \nu \leqslant f(\cdot) \leqslant L$. Its ellipticity ratio varies according to the geometry of the level sets $\{a_{k}(x)=0\}$ of the modulating coefficient functions $a_{k}(\cdot)$ for every $k\in \{1,\ldots,N\}$. 

We give a unified treatment to show various regularity results for such multi-phase problems with the coefficient functions $\{a_{k}(\cdot)\}_{k=1}^{N}$ not necessarily H\"older continuous even for a lower level of the regularity. Moreover, assuming that minima of the functional in \eqref{abst:2} belong to better spaces such as $C^{0,\gamma}(\O)$ or $L^{\kappa}(\O)$ for some $\gamma\in (0,1)$ and $\kappa\in (1,\infty]$, we address optimal conditions on nonlinearity for each variant under which we build comprehensive regularity results. 

On the other hand, since there is a lack of homogeneity properties in the nonlinearity, we consider an appropriate scaling with keeping the structures of the problems under which we apply Harmonic type approximation in the setting varying on the a priori assumption on minima. We believe that the methods and proofs developed in this paper are suitable to build regularity theorems for a larger class of non-autonomous functionals.

\end{abstract}

\maketitle
\tableofcontents
\section{Introduction}

We aim to provide optimal and comprehensive regularity results for minimizers of  functionals featuring a non-standard  growth  and a non-uniform ellipticity. The primary model case in mind under investigation is given by an Orlicz multi-phase functional
\begin{align}
    \label{ifunct}
    W^{1,1}(\O)\ni \u\mapsto \mathcal{P}(\u,\O):= \I_{\O} \Psi(x,|D\u|)\,dx
\end{align}
for a bounded open domain $\O\subset\R^n$ with $n\geqslant 2$, where throughout the paper we shall denote by
\begin{align}
	\label{psi}
	\Psi(x,t):= G(t) +a(x)H_{a}(t)+ b(x)H_{b}(t)
	\quad (x\in\O,\, t\geqslant 0)
\end{align}
for functions $G,H_{a}, H_{b}\in \mathcal{N}$  in the sense of Definition \ref{N-func} and $0\leqslant a(\cdot),b(\cdot)\in L^{\infty}(\O)$.
The Orlicz multi-phase functional $\mathcal{P}$ in \eqref{ifunct} is naturally defined for  functions $\u\in W^{1,1}(\O)$, which is a considerable one including the following examples of functionals for the regularity theory:

\begin{enumerate}
	\item[1.] $p$-growth: $\Psi(x,t) := t^{p}$,\quad $p>1$, see for instance \cite{Evans1, GG1,LU1, Lewis1, Lind1, Man1, Uh1, Ur1}.
	\item[2.] Orlicz growth: $\Psi(x,t):= G(t)$, see for instance \cite{BBL1,BC1,DSV1, DSV2,L1}.
	\item[3.] $(p,q)$-double phase: $\Psi(x,t):= t^{p} + a(x)t^{q}$ for $1<p\leqslant q$, see for instance \cite{BCM1, BCM3, CM1,CM2,CM3,DM1}.
	\item[4.] Borderline case of double phase: $\Psi(x,t):= t^{p} + a(x)t^{p}\log(1+t)$ for $1<p$, see for instance \cite{BCM2,BO2}.
	\item[5.] Multi-phase: $\Psi(x,t):= t^{p} + a(x)t^{q} + b(x)t^{s}$ for $1<p\leqslant q,s$, see for instance \cite{BBO2,DO1}.
	\item[6.] Orlicz double phase: $\Psi(x,t):= G(t)+a(x)H_{a}(t)$, see for instance \cite{ BBO1, BO3}.
	\item[7.] Orlicz multi-phase: $\Psi(x,t):= G(t)+a(x)H_{a}(t) + b(x)H_{b}(t)$, see for instance \cite{BBO2}.
\end{enumerate}

Over last several years a systematic analysis of the functionals aforementioned has been an object of intensive studies for the regularity theory. Among them the $(p,q)$-double phase functional is a significant example given by 

\begin{align}
	\label{func:(p,q)}
	W^{1,1}(\O)\ni v\mapsto \P_{p,q}(v,\O):= \I_{\O}\left[ |Dv|^{p} + a(x)|Dv|^{q} \right]\,dx,\quad 1<p\leqslant q.
\end{align}
Another example is the so-called borderline case of double phase defined by 
\begin{align}
	\label{func:log}
	W^{1,1}(\O)\ni v\mapsto \P_{\log}(v,\O):= \I_{\O}\left[ |Dv|^{p} + a(x)|Dv|^{p}\log(1+|Dv|) \right]\,dx,\quad 1<p.
\end{align}
 The last functional we would like to single out is the so-called multi-phase functional introduced in \cite{DO1} is of type
 \begin{align}
 	\label{func:multi}
 	W^{1,1}(\O)\ni v\mapsto \P_{p,q,s}(v,\O):= \I_{\O}\left[ |Dv|^{p} + a(x)|Dv|^{q} + b(x)|Dv|^{s} \right]\,dx,\quad 1<p\leqslant q,s.
 \end{align}

 The $(p,q)$-double phase functional was initially introduced by Zhikov \cite{Zh1,Zh2,ZKO} in order to study the feature of strongly anisotropic materials in the context of homogenization and nonlinear elasticity. A main feature of the functionals $\P_{p,q}$, $\P_{p,q,s}$ and $\P_{\log}$ in \eqref{func:(p,q)}-\eqref{func:multi} is that their integrand changes their growth and ellipticity ration depending on the geometric behavior of the coefficient functions $a(\cdot)$ and $b(\cdot)$, which determines the geometry of the mixture of different materials. As shown in \cite{ELM3,FMM,Zh3,Zh4}, such functionals exhibit Lavrentiev phenomenon which means that minimizers are discontinuous. 
 
Each functional mentioned above belongs to a family of functionals satisfying nonstandard growth conditions of $(p,q)$-type. These are functionals of type 
\begin{align*}
	W^{1,1}(\O)\ni v\mapsto \I_{\O}F(x,Dv)\,dx,
\end{align*} 
 whose energy density $F(x,z)$ satisfies
 \begin{align*}
 	|z|^{p} \lesssim F(x,z) \lesssim |z|^{q} +1,\quad 1<p<q,
 \end{align*}
according to Marcellini's terminology \cite{Ma1,Ma2,Ma3}. Over the several decades, functionals with nonstandard growth have been extensively investigated, see for instance 
\cite{Br1, ELM1, ELM2, ELM3, FMM, FS1, Sch1, Sch2} and the references therein. Those functionals aforementioned give a relevant example of the energy overlying in the so-called Musielak-Orlicz space which will described in Section \ref{subsec:2.2} below.

For the regularity theory, the optimal conditions for the gradient of a local minimizer $v$ of the functional $\P_{p,q}$ to be H\"older continuous have been discovered in \cite{BCM3,CM1,CM2}. They are

\begin{subequations}
	 \begin{empheq}[left={\empheqlbrace}]{align}
        \quad & \frac{q}{p} \leqslant 1+\frac{\alpha}{n} 
	\quad\text{and}\quad
	0\leqslant a(\cdot)\in C^{0,\alpha}(\O)   &\text{if }& v\in W^{1,p}(\O)  \label{(p,q):con1},\\
        & q \leqslant p + \alpha
	\quad\text{and}\quad
	0\leqslant a(\cdot)\in C^{0,\alpha}(\O) &\text{if }& v\in  W^{1,p}(\O)\cap L^{\infty}(\O)  \label{(p,q):con2},\\
       & q < p + \frac{\alpha}{1-\gamma}
	\quad\text{and}\quad
	0\leqslant a(\cdot)\in C^{0,\alpha}(\O)
         &\text{if }& v\in W^{1,p}(\O)\cap C^{0,\gamma}(\O)\text{ for some }\gamma\in (0,1).\label{(p,q):con3}
      \end{empheq}
	\end{subequations}
Those conditions in \eqref{(p,q):con1}-\eqref{(p,q):con3} are essentially sharp in the sense of Lavrentiev gap for the functional $\P_{p,q}$. Here we refer to the very recent interesting paper \cite{BDS} and see also \cite{ELM3,FMM}. On the other hand, letting $a(\cdot)\in C^{\omega_a}(\O)$ with a concave function $\omega_{a}: [0,\infty)\rightarrow [0,\infty)$ vanishing at the origin, the conditions for a local minimizer $v$ of the functional $\P_{\log}$ to be regular have been discovered in \cite{BCM2}, which are

\begin{subequations}
	 \begin{empheq}[left={\empheqlbrace}]{align}
        \quad & v \text{ is H\"older continuous}  && \text{if } \limsup\limits_{t\to 0^{+}}\omega_{a}(t)\log\left(\frac{1}{t} \right) < \infty  \label{log:con1},\\
        & v\text{ is H\"older continuous with an arbitrary exponent}  && \text{if } \limsup\limits_{t\to 0^{+}}\omega_{a}(t)\log\left(\frac{1}{t} \right) = 0  \label{log:con2},\\
       & Dv \text{ is H\"older continuous} 
         && \text{if } \omega_{a}(t)\lesssim t^{\alpha}\text{ with some } \alpha\in (0,1).\label{log:con3}
      \end{empheq}
	\end{subequations}
	
Furthermore, the optimal condition for a local minimizer of the multi-phase functional $\P_{p,q,s}$ in \eqref{func:multi} to be gradient H\"older continuous has been obtained in \cite{DO1}, that is
\begin{align}
	\label{(p,q,s):con1}
	\frac{q}{p} \leqslant 1+\frac{\alpha}{n},\quad
	\frac{s}{p} \leqslant 1 + \frac{\beta}{n},\quad
	0\leqslant a(\cdot)\in C^{0,\alpha}(\O)\quad\text{and}\quad
	0\leqslant b(\cdot)\in C^{0,\beta}(\O)
\end{align}
for some $\alpha,\beta\in (0,1]$. This condition essentially is a natural outcome of the condition \eqref{(p,q):con1} and sharp. So looking at the conditions presented in \eqref{(p,q):con1}-\eqref{(p,q):con3}, \eqref{(p,q,s):con1} and the ones \eqref{log:con1}-\eqref{log:con3}, there is a natural question as to whether  the coefficient functions $a(\cdot), b(\cdot)$ in \eqref{func:(p,q)} and \eqref{func:multi} are necessarily H\"older continuous even for a lower level of the regularity, depending on the a priori assumptions on a local minimizer under consideration. Here in this paper we intend to answer such questions by treating much more general class of functionals with Orlicz multi-phase growth below. Apart from the papers mentioned above, the regularity theory for the double phase problems has been the object of an intensive investigation over the last years, see for instance \cite{BM, BO1, BY1, DM2, DM3, HHT, Ok1, Ok2} and the references therein.

 In this paper we shall deal with a class of  general functionals of type 
 \begin{align}
 	\label{functional}
 	W^{1,1}(\O)\ni \u\mapsto \F(\u,\O):= \I_{\O} F(x,\u,D\u)\,dx,
 \end{align}
 where $F:\O\times\R\times \R^{n}\rightarrow \R$ is a Carathe\'odory function satisfying the double-sided bound
 \begin{align}
 	\label{sa:1}
 	\nu\Psi(x,|z|) \leqslant F(x,y,z) \leqslant L\Psi(x,|z|)\quad (x\in\O,\, y\in\R,\,z\in\R^n),
 \end{align}
 where $\Psi$ is the same as in \eqref{psi}. Under the growth conditions \eqref{sa:1}, local minimizers ($Q$-minimizers) of the functional $\mathcal{F}$ in \eqref{functional} for some number $Q\geqslant 1$ can be defined as follows:
 \begin{defn}
    \label{def_min}
 	A function $u\in W^{1,1}_{\loc}(\O)$ is a local minimizer ($Q$-minimizer) of the functional $\F$ defined in \eqref{functional}  if $\Psi(x,|Du|)\in L^{1}(\O)$ and the minimality condition 
 	$\F(u,\Supp(u-\u))\leqslant \F(\u,\Supp(u-\u))$ $\left(\F(u,\Supp(u-\u))\leqslant Q\F(\u,\Supp(u-\u))\right)$ is satisfied, whenever $\u\in W^{1,1}_{\loc}(\O)$ with $\Supp(u-\u)\Subset\O$. 
 \end{defn}  
In the rest of the paper we always assume $a(\cdot)\in C^{\omega_{a}}(\O)$ and $b(\cdot)\in C^{\omega_{b}}(\O)$, where  $\omega_{a}, \omega_{b} : [0,\infty)\rightarrow [0,\infty) $ are concave functions such that $\omega_{a}(0)=0$ and $\omega_{b}(0)=0$ unless they are specified. Throughout the paper we define the function $\Lambda : (0,\infty)\times (0,\infty)\rightarrow (0,\infty)$ given by
\begin{align}
	\label{Lambda}
	\Lambda(s,t):= \frac{\omega_{a}(s)}{1+ \omega_{a}(s)} \frac{H_{a}(t)}{G(t)}+ \frac{\omega_{b}(s)}{1+ \omega_{b}(s)} \frac{H_{b}(t)}{G(t)}
	\quad\text{for any}\quad
	s,t>0.
\end{align}

 We shall consider  a local $Q$-minimizer $u$ of the functional $\P$ in \eqref{ifunct} or a local minimizer $u$ of the functional $\F$ in \eqref{functional} under one of the following main assumptions:
 \begin{fleqn}[\parindent]
\begin{align}
     \label{ma:1}
     \begin{cases}
         u\in W^{1,\Psi}(\O), \\
          \displaystyle \lambda_{1}:= \sup\limits_{t>0} \Lambda\left(t,G^{-1}(t^{-n})\right) < \infty,
     \end{cases}
 \end{align}
 \end{fleqn}

\begin{fleqn}[\parindent]
\begin{align}
     \label{ma:2}
     \begin{cases}
         u\in W^{1,\Psi}(\O)\cap L^{\infty}(\O), \\
          \displaystyle \lambda_{2}:=\sup\limits_{t>0} \Lambda\left(t,\frac{1}{t}\right) < \infty,
     \end{cases}
 \end{align}
 \end{fleqn}

\begin{fleqn}[\parindent]
\begin{align}
     \label{ma:3}
     \begin{cases}
         u\in W^{1,\Psi}(\O)\cap C^{0,\gamma}(\O)\quad\text{for some}\quad \gamma\in (0,1), \\
          \displaystyle \lambda_{3}:=\sup\limits_{t>0} \Lambda\left(t^{\frac{1}{1-\gamma}},\frac{1}{t}\right) < \infty.
     \end{cases}
 \end{align}
 \end{fleqn}
 Here $G^{-1}$ is the inverse function of $G$. We straightforwardly check that the conditions \eqref{ma:1}-\eqref{ma:3} read as \eqref{(p,q):con1}-\eqref{(p,q):con3}, respectively, when $G(t)=t^{p}$, $H_{a}(t)=t^{q}$, $w_a(t)=t^{\alpha}$ and $b(\cdot)\equiv 0$ for some $1<p\leqslant q$ and $\alpha\in (0,1]$. Also the condition \eqref{(p,q,s):con1} is the same as \eqref{ma:1} for relevant choices of the functions. Moreover, $\eqref{ma:1}_{2}$ and $\eqref{ma:2}_{2}$ are the same as the one in \eqref{log:con1} when $G(t)=t^{p}$ and $H_{a}(t)=t^{p}\log(1+t)$ for some $p>1$ and $b(\cdot)\equiv 0$. It has been shown that the assumption \eqref{sa:1} is not enough already in the special case of $G(t)=t^{p}$ for $p > 1$ together with $a(\cdot)\equiv 0$ and $b(\cdot) \equiv 0$ for obtaining higher regularity of minimizers of the functional $\F$ in \eqref{functional}. In this regard we consider the energy density $F$ in \eqref{functional} of type

\begin{align}
	\label{type}
	F(x,y,z):= F_{G}(x,y,z)+ a(x)F_{H_{a}}(x,y,z) + b(x)F_{H_{b}}(x,y,z)
\end{align}
for every $x\in\O$, $y\in\R$ and $z\in\R^{n}$, where $F_{G}(\cdot)$, $F_{H_{a}}(\cdot)$ and $F_{H_{b}}(\cdot)$ are continuous functions belonging to $C^{2}(\R^{n}\setminus \{0\})$ with respect to $z$-variable and satisfying the following structure assumptions with fixed constants $0<\nu\leqslant L$:
 
 \begin{align}
 	\label{sa:2}
 	\begin{cases}
 		|D_{z}F_{\Phi}(x,y,z)||z| + |D^2_{zz}F_{\Phi}(x,y,z)||z|^2 \leqslant L\Phi(|z|), \\
 		\nu\dfrac{\Phi(|z|)}{|z|^2}|\xi|^2 \leqslant \inner{D_{zz}^{2}F_{\Phi}(x,y,z)\xi}{\xi}, \\
 		|D_{z}F_{\Phi}(x_1,y,z)-D_{z}F_{\Phi}(x_2,y,z)||z| \leqslant  L\omega(|x_1-x_2|)\Phi(|z|), \\
 		|F_{\Phi}(x,y_1,z)-F_{\Phi}(x,y_2,z)| \leqslant L\omega(|y_1-y_2|)\Phi(|z|)
 	\end{cases}
 \end{align}
for every $\Phi\in\{G,H_{a},H_{b} \}$, whenever  $x,x_1,x_2\in\O$, $y,y_1,y_2\in \R$, $z\in\R^n\setminus\{0\}$, $\xi\in\R^n$, here either 
 \begin{align}
 	\label{omega1} 	
 	\omega(t) := \min\{t^{\mu},1\} \text{ with some } \mu\in (0,1) \text{ for all } t\geqslant 0
 	\end{align}
 or
  \begin{align}
     \label{omega2}
     \omega : [0,+\infty)\rightarrow [0,+\infty) \text{ is concave such that } \omega(0)=0 \text{ and } \omega(\cdot) \leqslant 1.
 \end{align}
 
 The structure conditions in \eqref{sa:2} are satisfied for instance by the model functional 
 \begin{align*}
 	W^{1,1}(\O)\ni\u\mapsto \I_{\O}f(x,\u)\Psi(x,|D\u|)\,dx,
 \end{align*}
 where $0<\nu_1 \leqslant f(x,y)\leqslant L_1$ for some constants $\nu_1, L_1$ and for some suitable  continuous function $f(\cdot)$ satisfying the following inequality
 \begin{align*}
 	|f(x_1,y_1)-f(x_2,y_2)|\leqslant L\omega(|x_1-x_2| +|y_1-y_2|)
\end{align*} 
whenever $x_1,x_2\in\R^n$ and $y_1,y_2\in\R$, where $\omega$ is the same as defined in \eqref{omega1} or \eqref{omega2}. We also remark that those general functionals mentioned above have not been considered in the present literature for the regularity theory as far as we are concerned, moreover the functionals in \eqref{functional} with structure assumptions \eqref{sa:1} and \eqref{sa:2} is not differentiable with respect to the second variable and so it can not be treated by its Euler-Lagrange equation. In order to shorten the notations in this paper, for a given local minimizer $u$ of the functional $\F$, we shall use a set of various basic parameters which is ``data of the problem" depending on which assumption of \eqref{ma:1}-\eqref{ma:3} is  considered as follows: 

\begin{align}
	\label{data}
	\data
	 \equiv \left\{\begin{array}{ll}
        &n,\lambda_{1},s(G),s(H_{a}),s(H_{b}),\nu,L, \norm{a}_{C^{\omega_{a}}(\O)},\norm{b}_{C^{\omega_{b}}(\O)}, \omega(\cdot), \norm{\Psi(x,|Du|)}_{L^1(\O)}, \\
        &  \norm{u}_{L^{1}(\O)}, \omega_a(1), \omega_{b}(1)\quad \text{if } \eqref{ma:1} \text{ is considered, }\\
        & n,\lambda_{2},s(G),s(H_{a}),s(H_{b}),\nu,L,  \norm{a}_{C^{\omega_{a}}(\O)},\norm{b}_{C^{\omega_{b}}(\O)},\omega(\cdot), \norm{u}_{L^{\infty}(\O)},\omega_a(1), \omega_{b}(1)\\
         & \text{if } \eqref{ma:2} \text{ is considered, } \\
       & n,\lambda_{3},s(G),s(H_{a}),s(H_{b}),\nu,L,  \norm{a}_{C^{\omega_{a}}(\O)},\norm{b}_{C^{\omega_{b}}(\O)},\omega(\cdot), [u]_{0,\gamma},\omega_a(1), \omega_{b}(1) \\
       & \text{ if } \eqref{ma:3} \text{ is considered, }
        \end{array}\right.
\end{align}
where $\lambda_1$, $\lambda_2$, $\lambda_3$ are the same numbers as defined in \eqref{ma:1}-\eqref{ma:3} and $s(G),s(H_a),s(H_{b})$ are indices of the functions $G,H_{a}, H_{b}$ in the sense of Definition \ref{N-func}, respectively. For a given local $Q$-minimizer $u$ of the functional $\P$, $\data$ is understood by the above set of parameters with the constants $L,\nu$ having been replaced by $Q$ in any case of \eqref{ma:1}-\eqref{ma:3} into the consideration. 
With $\O_0\Subset\O$ being a fixed open subset, we also denote by $\data(\O_0)$ the set of parameters in \eqref{data} together with $\dist(\O_{0},\partial\O)$ under one of the assumptions \eqref{ma:1}-\eqref{ma:3}:
\begin{align}
    \label{data0}
    \data(\O_0) \equiv \data, \dist(\O_0,\partial\O).
\end{align}

Now we are ready to state our main results in this paper.
 
 \begin{thm}[Maximal regularity]
 	\label{mth:mr}
 	Let $u\in W^{1,\Psi}(\O)$ be a local minimizer of the functional $\F$ defined in \eqref{functional}, under the assumptions \eqref{sa:1}, \eqref{sa:2} and \eqref{omega1}. Suppose that $\omega_{a}(t) = t^{\alpha}$ and $\omega_{b}(t) = t^{\beta}$ for some $\alpha,\beta\in (0,1]$. If one of the following assumptions
 \begin{subequations}
	 \begin{empheq}[left={\empheqlbrace}]{align}
        &\eqref{ma:1}   \label{mth:mr:1},\\
        &\eqref{ma:2}  \label{mth:mr:2}, \\
        &\eqref{ma:3} \quad\text{with}\quad \limsup\limits_{t\to 0^{+}} \Lambda\left(t^{\frac{1}{1-\gamma}},\frac{1}{t}\right)=0 \label{mth:mr:3}
      \end{empheq}
	\end{subequations}
	is satisfied,  then there exists $\theta\in (0,1)$ depending only on $n,s(G),s(H_{a}),s(H_{b}),\nu,L,\alpha,\beta$ and $\mu$ such that $Du\in C^{0,\theta}_{\loc}(\O)$. 	
 \end{thm}

\begin{thm}[Morrey decay]
	\label{mth:md}
	Let $u\in W^{1,\Psi}(\O)$ be a local minimizer of the functional $\F$ defined in \eqref{functional}, under the assumptions \eqref{sa:1}, \eqref{sa:2} and \eqref{omega2}.
 If one of the following assumptions
	\begin{subequations}
	 \begin{empheq}[left={\empheqlbrace}]{align}
          &\eqref{ma:1}  \quad\text{with}\quad \limsup\limits_{t\to 0^{+}} \Lambda\left(t,G^{-1}(t^{-n})\right)=0 \label{mth:md:1},\\
         &\eqref{ma:2}\quad\text{with}\quad \limsup\limits_{t\to 0^{+}} \Lambda\left(t,\frac{1}{t}\right)=0 \label{mth:md:2},\\
        &\eqref{ma:3} \quad\text{with}\quad \limsup\limits_{t\to 0^{+}} \Lambda\left(t^{\frac{1}{1-\gamma}},\frac{1}{t}\right)=0,\label{mth:md:3} \\
        &\eqref{ma:1}  \quad\text{with}\quad \omega_{a}(t)=t^{\alpha}
          \quad\text{and}\quad \omega_{b}(t)=t^{\beta}\quad \text{for some}\quad \alpha,\beta\in (0,1] \label{mth:md:4},\\
          &\eqref{ma:2}\quad\text{with}\quad \omega_{a}(t)=t^{\alpha}
          \quad\text{and}\quad \omega_{b}(t)=t^{\beta}\quad \text{for some}\quad \alpha,\beta\in (0,1]  \label{mth:md:5}
      \end{empheq}
	\end{subequations}
	is satisfied, then
	\begin{align}
	    \label{mth:md:6}
	    u\in C^{0,\theta}_{\loc}(\O)\quad \text{for every}\quad \theta\in (0,1).
	\end{align}
	Moreover, for every $\sigma\in (0,n)$, there exists a positive constant $c\equiv c(\data(\O_0),\sigma)$ such that the decay estimate
	\begin{align}
		\label{mth:md:7}
		\I_{B_{\rho}} \Psi(x,|Du|)\,dx \leqslant c\left(\frac{\rho}{R} \right)^{n-\sigma}\I_{B_{R}}\Psi(x,|Du|)\,dx
	\end{align}
holds for every concentric balls  $B_{\rho}\subset B_{R} \subset \O_{0}\Subset\O$ with $R\leqslant 1$.
\end{thm}

First of all we note that the gradient H\"older regularity of a local minimizer in Theorem \ref{mth:mr} is already optimal in the classical $p$-Laplace case that $G(t)=t^{p}$ and $a(\cdot)\equiv b(\cdot)\equiv 0$ \cite{Uh1,Ur1}. The assumptions in \eqref{mth:mr:1}-\eqref{mth:mr:3} are optimal by the counterexamples given in \cite{ELM2,FMM}, see also \cite{BDS}.
The regularity results reported here complement in a unified way the main results of \cite{BCM2,BCM3,CM1,CM2,DO1}, where the functions in \eqref{func:(p,q)}-\eqref{func:multi} are considered under the corresponding conditions we have discussed in \eqref{(p,q):con1}-\eqref{(p,q):con3},\eqref{log:con1}-\eqref{log:con3} and \eqref{(p,q,s):con1}, respectively, and the arguments used in those papers are strongly dependent of the number of phases along with the H\"older continuity of the coefficient function in the non-linearity. Our approaches for proving the above theorems are in fact independent of this weakness. The assumptions of the above theorems lead to showing new instances of Lavrantiev phenomenon \cite{Zh1,Zh2,Zh3,Zh4,ZKO}. According to the classical definition, the Lavrentiev gap for the functional $\F$ defined in \eqref{functional} under the growth assumption \eqref{sa:1} may appear if 
\begin{align}
	\label{lg:1}
	\inf\limits_{v\in v_0 + W_{0}^{1,G}(B)}\F(v,B)
	< 
	\inf\limits_{v\in v_0 + W_{0}^{1,G}(B)\cap W_{\loc}^{1,\Psi_{\O}^{+}}(B) }\F(v,B)
\end{align}
holds for a ball $B\Subset \O$ and a function $v_0\in W^{1,\infty}(B)$, where $\Psi_{\O}^{+}$ is defined in \eqref{ispsi} below. That is, local minimizers of $\F$ do not belong to $W^{1,\Psi_{\O}^{+}}_{\loc}(B)$ in general. To see this, let us look at the classical case that $G(t)=t^{p}$, $H_{a}(t)=t^{q}$, $\omega_{a}(t)=t^{\alpha}$ and $\omega_{b}(\cdot)\equiv 0$ for $1<p <q$ and $\alpha\in (0,1]$ such that 
\begin{align}
	\label{lg:2}
	1< p < n < n+\alpha <q.
\end{align}
Under classical double phase setting together with \eqref{lg:2}, the results of \cite[Theorem 4.1]{CM1} and \cite[Section 3]{ELM3} provide us the existence of a coefficient function $a(\cdot)\in C^{0,\alpha}(\O)$ and a boundary datum $u_0 \in W^{1,p}(B)\cap L^{\infty}(B)$ such that the Lavrentiev phenomenon \eqref{lg:1} is occurred. In this regard, we show that there is no Lavrentiev gap for the functional $\F$ in \eqref{functional} satisfying the basic structure assumption \eqref{sa:1} under the one of assumptions $\eqref{ma:1}_{2}$, $\eqref{ma:2}_{2}$ and $\eqref{ma:3}_{2}$, see Theorem \ref{lp:thm1} below. The approaches we present in this paper lead to avoiding the use of difference quotient methods employed in \cite{CM1,CM2} for obtaining various regularity properties of minimizers of the functional $\P$ representing the $(p,q)$-double phase growth. In fact, the difference quotient techniques can deal with the case that the coefficient functions in the nonlinearity are H\"older continuous. On the other hand, we are treating the case of not necessarily having H\"older continuous coefficient functions in the nonlinearity by applying a Harmonic type approximation (see Lemma \ref{hta: lemma_hta} below) for comparing a homogeneous equation with a limiting equation having the lipschitz regularity property (see Lemma \ref{lem:2ce} and \ref{lem:3ce}).
\begin{rmk}	
	\label{rmk:multi}
We would like to point out, in the same spirit as this paper, the results of Theorem \ref{mth:mr} and Theorem \ref{mth:md} can be restated and proved for the functional having a finite number of phases with replacing the function in \eqref{psi} by 
\begin{align}
	\label{multi:1}
	\Psi(x,t):= G(t) + \sum\limits_{i=1}^{N}a_i(x)H_{i}(t),\quad m\geqslant 1,
\end{align}
where $G,H_{i}\in \mathcal{N}$ in the sense of the Definition \ref{N-func} and $a_i(\cdot)\in C^{\omega_{i}}(\O)$ with $\omega_{i} : [0,\infty)\rightarrow [0,\infty)$ being a concave function vanishing at the origin for every $i\in \{1,\ldots,N\}$. Under this setting we replace the function in \eqref{Lambda} by 
\begin{align}
	\label{multi:2}
	\Lambda(s,t):= \sum\limits_{i=1}^{N} \frac{\omega_i(s)}{1+\omega_i(s)}\frac{H_i(t)}{G(t)}\quad\text{for every}\quad s,t>0.
\end{align}
The coefficient functions in Theorem \ref{mth:mr} along with \eqref{mth:md:4} and \eqref{mth:md:5} in  Theorem \ref{mth:md} are understood by letting $\omega_{i}(t)=t^{\alpha_i}$ with some $\alpha_{i}\in (0,1]$ for every $i\in \{1,\ldots,N\}$.
\end{rmk}
The contents in this paper could provide a guideline to deal with a very general class of non-autonomous functionals whose energy density behaves like
\begin{align}
	\label{gcf:1}
	F(x,y,z)\approx \Phi(x,|z|)
\end{align}
for $\Phi$ being a certain Young function as we shall introduce in Definition \ref{N-func} below. The investigation of such problems has been a field of interest for research activities over the decades. In fact, a main difficulty lies in discovering the optimal conditions to be placed on $\Phi(x,t)$ with respect to $(x,t)$-variables. Here we mention a very recent and interesting paper \cite{HO} in which the authors give a reasonable answer to such a question by considering a class of functionals of Uhlenbeck type without any a priori assumption on the minimizers involved. Essensially, the assumption \cite[(VA1)]{HO} is not comparable with the assumption $\eqref{ma:1}_{2}$. Moreover, the method used in \cite{HO} can not be applicable to treat the regularity of minimizers of the functional $\F$ in \eqref{functional} having the solution dependence. Besides the papers mentioned before, there is a rich literature, see for instance \cite{AM1,AM2,AM3,BF,BBr1,BR1,BY1,Cup1,CMM,EMT1,Ka1, KS1,RT1,RT2,TU} and reference therein. We also refer to a survey paper \cite{M1}.

Finally, we close the introduction part with outlining the organization of the paper. In the next section, we introduce our basic settings and some backgrounds. In Section \ref{sec:3} we shall deal with the absence of Lavrentiev phenomenon for the functional $\F$ under our main assumptions $\eqref{ma:1}_{2}$, $\eqref{ma:2}_{2}$ and $\eqref{ma:3}_{2}$. Section \ref{sec:4} is devoted to obtaining Sobolev-Poincar\'e type inequality. In Section \ref{sec:5} we start with showing basic regularity results of quasi-minimizers of the functional $\P$ introduced in \eqref{ifunct}, while in Section \ref{sec:6} we provide a certain harmonic type approximation fitting in our settings. From Section \ref{sec:7} we start with making the comparison estimates in order to reach a limiting functional and in Section \ref{sec:8} and Section \ref{sec:9} we prove Theorem \ref{mth:mr} and Theorem \ref{mth:md}. In Section \ref{sec:10} we shall consider a more general class of functionals having the Orlicz double phase growth. Finally, we shall dealt with our multi-phase problems under some additional integrability condition of Sobolev type and some relevant optimal condition in Section \ref{sec:11}.


\section{Notations and preliminaries}

\label{sec:2}

\subsection{Notations}
\label{subsec:2.1}
In what follows we shall always denote by $c$ to mean a generic positive constant, possibly varying from line to line, while special constants will be denoted by $c_1,\bar{c}, c_{*}, c_{\varepsilon}$, and so on. All such constants will be always not smaller than one; moreover relevant dependencies on parameters will be emphasized using parentheses, that is, for example $c\equiv c(n,s(G),\nu,L)$ means that $c$ depends only on $n,s(G),\nu,L$. We denote by
 $B_{R}(x_0) = \{ x\in\R^n : |x-x_0|<R \}$ the open ball in $\R^n$ centered at $x_0\in \R^n$ with a radius $R > 0$. If the center is clear in the context, we shall omit the center point by writing $B_{R}\equiv B_{R}(x_0)$.  We shall also denote $B_{1}\equiv B_{1}(0)\subset\R^n$ unless the center is specified. With $f:\mathcal{B}\rightarrow \R^{N}$ $(N\geqslant 1)$ being a measurable map for a measurable subset $\mathcal{B}\subset\R^n$ having finite and positive measure, we denote by
\begin{align*}
    (f)_{\mathcal{B}}\equiv \FI_{\mathcal{B}} f(x)\,dx = \frac{1}{|\mathcal{B}|}\I_{\mathcal{B}}f(x)\,dx
\end{align*}
its integral average over $\mathcal{B}$. For a measurable map $f:\O\rightarrow \R$ and an open subset $\mathcal{B}\subset\O$ with $\sigma : [0,\infty)\rightarrow [0,\infty)$ being a concave function such that $\sigma(0)=0$, we shall use the notation as 
\begin{align*}
    [f]_{\sigma;\mathcal{B}}:= \sup\limits_{x,y\in\mathcal{B},x\neq y} \frac{|f(x)-f(y)|}{\sigma(|x-y|)}
    \quad\text{and}\quad
    [f]_{\sigma}\equiv [f]_{\sigma;\O}.
\end{align*}
We denote by $C^{\sigma}(\O)$ the space of uniformly continuous functions on $\O$ whose modulus of continuity does not exceed $\sigma$. The space $C^{\sigma}(\O)$ is endowed with the norm defined for a function $f$ by 
\begin{align*}
	\norm{f}_{C^{\sigma}(\O)} = \norm{f}_{L^{\infty}(\O)} + [f]_{\sigma;\O}.
\end{align*}
In particular, if $\sigma(t)=t^{\alpha}$ for some $\alpha\in (0,1]$, then we denote 
\begin{align*}
    [f]_{0,\alpha;\mathcal{B}}:= \sup\limits_{x,y\in\mathcal{B},x\neq y} \frac{|f(x)-f(y)|}{|x-y|^{\alpha}}
    \quad\text{and}\quad
    [f]_{0,\alpha}\equiv [f]_{0,\alpha;\O}.
\end{align*}
For a given concave function $\sigma : [0,\infty)\rightarrow [0,\infty)$ vanishing at the origin, we shall use some elementary properties in the future as
\begin{align}
	\label{concave1}
	&\sigma(\lambda t) \leqslant \lambda \sigma(t)
	\quad\text{for every}\quad \lambda\geqslant 1 
	\quad\text{and}\quad t\geqslant 0
\end{align}
and
\begin{align}
	\label{concave2}
	\frac{1}{\sigma(\lambda t)} \leqslant \frac{1}{\sigma(t)} + \frac{1}{\lambda\sigma(t)} 
	\quad\text{for every}\quad \lambda, t> 0 
	\quad\text{unless}\quad \sigma\text{ is constant. }
\end{align}

Throughout the paper, for any given open subset $\mathcal{B}\subset\O$, we shall also use the notations by
\begin{align}
	\label{ispsi}
	\begin{split}
		& a^{-}(\mathcal{B}):= \inf\limits_{x\in \mathcal{B}}a(x),\quad
		a^{+}(\mathcal{B}):= \sup\limits_{x\in \mathcal{B}}a(x),
		\\&
		b^{-}(\mathcal{B}):= \inf\limits_{x\in \mathcal{B}}b(x),\quad
		b^{+}(\mathcal{B}):= \sup\limits_{x\in \mathcal{B}}b(x),
		\\&
		\Psi_{\mathcal{B}}^{-}(t):= G(t) + \inf\limits_{x\in \mathcal{B}}a(x)H_{a}(t) + 
		\inf\limits_{x\in \mathcal{B}}b(x)H_{b}(t),
		\\&
		\Psi_{\mathcal{B}}^{+}(t):= G(t) + \sup\limits_{x\in \mathcal{B}}a(x)H_{a}(t) + 
		\sup\limits_{x\in \mathcal{B}}b(x)H_{b}(t)
	\end{split}
\end{align}
for every $t\geqslant 0$.

\begin{defn}
\label{N-func}
A measurable function $\Phi : \O\times [0,\infty)\rightarrow [0,\infty)$ is called  an Young function if, for any fixed $x\in\O$, the function $\Phi(x,\cdot)$ increasing and convex such that 
\begin{align*}
    \Phi(x,0) = 0,\ \lim_{t\to\infty}\Phi(x,t) = +\infty,\ \lim_{t\to 0^{+}}\frac{\Phi(x,t)}{t} = 0 \ \text{and} \
    \lim_{t\to\infty}\frac{\Phi(x,t)}{t} = +\infty.
\end{align*}

We denote by $\mathcal{N}(\O)$ the set of Young functions $\Phi : \O\times [0,\infty) \rightarrow [0,\infty)$ such that, for any fixed $x\in\O$, $\Phi(x,\cdot)\in C^{1}([0,\infty))\cap C^{2}((0,\infty))$ and there exists a constant $s(\Phi) \geqslant 1$ with
\begin{align}
    \label{defs}
     \frac{1}{s(\Phi)} \leqslant \frac{\partial_{tt}^{2}\Phi(x,t)t}{\partial_{t}\Phi(x,t)} \leqslant s(\Phi)
\end{align}
uniformly for all $x\in\O$ and $t>0$, where in the future we shall call this number $s(\Phi)$ by an index of $\Phi$. Furthermore, we denote also $\mathcal{N}$ to mean the set of Young functions $\Phi\in \mathcal{N}(\O)$ such that $\Phi$ does not depend on the first variable $x$.

\end{defn}

As a direct consequence of the above definition, for any $\Phi\in\mathcal{N}(\O)$ with an index $s(\Phi)\geqslant 1$ and any fixed point $x\in\O$, we can observe
\begin{align}
	\label{approx}
		t^2\partial_{tt}^{2}\Phi(x,t) \approx t\partial_{t}\Phi(x,t) \approx \Phi(x,t)
\end{align}
	for uniformly all $t > 0$, where note that all implied constants only depend only on $s(\Phi)$.
Now we state some important properties of functions of $\mathcal{N}(\O)$, see \cite{BBO1,BBO2,BO3} for their proofs.
\begin{lem}
    \label{lem:nf1}
    Let $\Phi\in \mathcal{N}(\O)$ with an index $s(\Phi) \geqslant 1$. Then, for any fixed $x\in\O$, we have
    \begin{enumerate}
        \item[1.] $\Lambda_{0}^{1+\frac{1}{s(\Phi)}}\Phi(x,t)\leqslant \Phi(x,\Lambda_{0} t) \leqslant \Lambda_{0}^{s(\Phi)+1}\Phi(x,t)$ for any $\Lambda_{0}\geqslant 1$ and $t\geqslant 0$.
        \item[2.] $\lambda_{0}^{1+s(\Phi)}\Phi(x,t) \leqslant \Phi(x,\lambda_{0} t) \leqslant \lambda_{0}^{\frac{1}{s(\Phi)}+1}\Phi(x,t)$ for any $0 < \lambda_{0}\leqslant 1$ and $t\geqslant 0$.
        \item[3.] $ \Lambda_{0}^{\frac{1}{1+s(\Phi)}}\Phi^{-1}_{t}(x,t) \leqslant \Phi^{-1}_{t}(x,\Lambda_{0} t) \leqslant \Lambda_{0}^{\frac{s(\Phi)}{1+s(\Phi)}}\Phi^{-1}_{t}(x,t)$ for any $\Lambda_{0}\geqslant 1$ and $t\geqslant 0$. 
        \item[4.] $\lambda_{0}^{\frac{s(\Phi)}{1+s(\Phi)}}\Phi^{-1}_{t}(x,t) \leqslant \Phi^{-1}_{t}(x,\lambda_{0} t) \leqslant \lambda_{0}^{\frac{1}{1+s(\Phi)}}\Phi^{-1}_{t}(x,t)$ for any $0<        
        \lambda_{0}\leqslant 1$ and $t\geqslant 0$. 
    \end{enumerate}
\end{lem}
In the above lemma, for a fixed point $x\in\O$, $\Phi^{-1}_{t}(x,t)$ is understood by the inverse function of $\Phi(x,t)$ with respect to $t$-variable.

\begin{rmk}
    \label{rmk:nf1}
    For a given $\Phi\in \mathcal{N}(\O)$ with an index $s(\Phi)\geqslant 1$, we notice useful but direct consequences of Lemma \ref{lem:nf1} as
    \begin{align}
        \label{growth1}
        \Phi(x,t + s) \leqslant \Phi(x,2t) + \Phi(x,2s) \leqslant 2^{1+s(\Phi)}\left( \Phi(x,t) + \Phi(x,s) \right) 
    \end{align}
    for every $x\in\O$ and $t,s\geqslant 0$. Furthermore, for any fixed $x\in\O$, we have 
    \begin{align*}
        \partial_{t}\Phi(x,t) \leqslant (1+s(\Phi))\frac{\Phi(x,t)}{t} \leqslant (1+s(\Phi)) [\Phi(x,1)]^{\frac{s(\Phi)}{1+s(\Phi)}}[\Phi(x,t)]^{\frac{1}{1+s(\Phi)}} \text{  for every  }
        0\leqslant t\leqslant 1
    \end{align*}
    and 
     \begin{align*}
        \partial_{t}\Phi(x,t) \leqslant (1+s(\Phi))\frac{\Phi(x,t)}{t} \leqslant (1+s(\Phi)) [\Phi(x,1)]^{\frac{1}{1+s(\Phi)}}[\Phi(x,t)]^{\frac{s(\Phi)}{1+s(\Phi)}} \text{  for every  }
        t\geqslant 1.
    \end{align*}
    Putting together the last two inequalities, we have the following very useful inequality which will be applied in the future
    \begin{align}
        \label{growth2}
        \frac{\Phi(x,t)}{t} \approx
        \partial_{t}\Phi(x,t)\leqslant (1+s(\Phi))\left([\Phi(x,1)]^{\frac{s(\Phi)}{1+s(\Phi)}}[\Phi(x,t)]^{\frac{1}{1+s(\Phi)}}  + [\Phi(x,1)]^{\frac{1}{1+s(\Phi)}}[\Phi(x,t)]^{\frac{s(\Phi)}{1+s(\Phi)}}\right)
    \end{align}
     for every $x\in\O$ and $t\geqslant 0$.
\end{rmk}

\begin{lem}
    \label{lem:nf2}
    Let $\Phi, \tilde{\Phi} \in \mathcal{N}(\O)$ with indices $s(\Phi),s(\tilde{\Phi}) \geqslant 1$. Then,
    \begin{enumerate}
        \item[1.]  For any non-negative real numbers $a,b$ satisfying $a+b > 0$, $a\Phi + b\tilde{\Phi} \in\mathcal{N}(\O)$ with $s(a\Phi + b\tilde{\Phi}) = s(\Phi) + s(\tilde{\Phi})$ and $\Phi\tilde{\Phi}\in \mathcal{N}(\O)$ with $s(\Phi\tilde{\Phi})= 4s(\Phi)s(\tilde{\Phi})(s(\Phi)+s(\tilde{\Phi}))$.
        \item[2.] For any number $m\geqslant 1$, $\Phi^{m}\in \mathcal{N}(\O)$ with $s(\Phi^{m}) = s(\Phi) + (m-1)(s(\Phi)+1)$.
        \item[3.] For any number $\mu \geqslant 0$, $\Phi_{\mu}(x,t):=t^{\mu}\Phi(x,t)\in \mathcal{N}(\O)$ with  $s(\Phi_{\mu})= \mu + 3[s(\Phi)]^2$.
        \item[4.] There exists $\theta_{0}\in (0,1)$ depending only on $s(\Phi)$ such that 
        $\Phi^{\theta} \in \mathcal{N}(\O)$ for every $\theta\in (\theta_{0},1]$ with $s(\Phi^{\theta})$ depending only on $s(\Phi)$ and $\theta$.
    \end{enumerate}
\end{lem}

\begin{lem}
	\label{lem:nf2_1}
	Let $\Phi\in \mathcal{N}$ with an index $s(\Phi)\geqslant 1$. Then $t\mapsto \Phi\left(t^{\frac{1}{s(\Phi)+1}}\right)$ is a concave function.
\end{lem}

\begin{lem}
\label{lem:nf3}
 Let $\Phi\in\mathcal{N}(\O)$ with an index $s(\Phi) \geqslant 1$. Then there exists a positive constant $c\equiv c(s(\Phi))$ such that
 \begin{align*}
     s_1\partial_{t}\Phi(x,s_2) + s_2\partial_{t}\Phi(x,s_1) \leqslant \varepsilon \Phi(x,s_1) + \frac{c}{\varepsilon^{s(\Phi)}}\Phi(x,s_2)
 \end{align*}
 holds for all $s_1,s_2\geqslant 0$ and $0 < \varepsilon \leqslant 1$.
\end{lem}

\begin{rmk}
	\label{rmk:nf2} We note that $\Psi\in\mathcal{N}(\O)$ with an index $s(\Psi)= s(G) + s(H_{a}) + s(H_{b})$ by Lemma \ref{lem:nf2}. In particular, for every open subset $\mathcal{B}\subset\O$, it holds that $\Psi_{\mathcal{B}}^{+}, \Psi_{\mathcal{B}}^{-}\in \mathcal{N} $ with indices $s\left(\Psi_{\mathcal{B}}^{+}\right)= s(G)+s(H_{a}) + s(H_{b})$ and $s\left(\Psi_{\mathcal{B}}^{-}\right)= s(G)+s(H_{a}) + s(H_{b})$.
\end{rmk}

For a given Young function $\Phi\in\mathcal{N}(\O)$ with an index $s(\Phi)\geqslant 1$, we define the vector field $V_{\Phi}: \O\times\R^n\setminus \{0\}\rightarrow \R^n$ as follows 
\begin{align}
	\label{defV}
	V_{\Phi}(x,z):= \left[ \frac{\partial_{t}\Phi(x,|z|)}{|z|} \right]^{\frac{1}{2}} z.
\end{align}

Furthermore, we shall often use the following inequalities that
\begin{align}
	\label{defV1}
	\I_{0}^{1}\frac{\Phi(x,|\theta z_1 + (1-\theta)z_2|)}{|\theta z_1 + (1-\theta)z_2|^{2}}\,d\theta
	\approx
	\frac{\Phi(x,|z_1|+|z_2|)}{(|z_1|+|z_2|)^2},
\end{align}
\begin{align}
	\label{defV2}
	|V_{\Phi}(x,z_1)-V_{\Phi}(x,z_2)|^2 \approx \partial_{tt}^{2}\Phi(x,|z_1|+ |z_2|)|z_1-z_2|^2
	\approx \frac{\partial_{t}\Phi(x,|z_1| + |z_2|)}{|z_1|+|z_2|}|z_1-z_2|^{2},
\end{align}
\begin{align}
	\label{defV3}
	\Phi(x,|z_1-z_2|) \lesssim \Phi(x,|z_1|+|z_2|)\dfrac{|z_1-z_2|}{|z_1|+|z_2|}
\end{align}
and
\begin{align}
	\label{defV4}
	\inner{\partial_{t}\Phi(x,|z_1|)\frac{z_1}{|z_1|}-\partial_{t}\Phi(x,|z_2|)\frac{z_2}{|z_2|}}{z_1-z_2}\approx |V_{\Phi}(x,z_1)-V_{\Phi}(x,z_2)|^2
\end{align}
hold true, whenever $x\in\O$ and $z_1,z_2\in\R^n\setminus\{0\}$, where all implied constants in \eqref{defV1}-\eqref{defV4} depend on $n$ and $s(\Phi)$ (see \cite{DE1} for further discussions). Moreover, we have the following useful inequality
\begin{align}
    \label{defV5}
    |V_{\Phi}(x,z_2)-V_{\Phi}(x,z_1)|^2 \lesssim \I_{0}^{1}|V_{\Phi}(x,\theta z_2 + (1-\theta)z_1)-V_{\Phi}(x,z_1)|^2 \,\frac{d\theta}{\theta}\quad (\forall x\in \O),
\end{align}
which follows from the following estimates that
\begin{align*}
    \begin{split}
        \I_{0}^{1}&|V_{\Phi}(x,\theta z_2 + (1-\theta)z_1)-V_{\Phi}(x,z_1)|^2\,\frac{d\theta}{\theta}
        \\&
        \stackrel{\eqref{defV2}}{\gtrsim}
        \I_{0}^{1}  \frac{\Phi(x,|\theta z_2 + (1-\theta)z_1|+|z_1|)}{(|\theta z_2 + (1-\theta)z_1|+|z_1|)^2}\theta |z_2-z_1|^2\,d\theta
        \\&
        \gtrsim
        \frac{|z_2-z_1|^2}{(|z_2|+|z_1|)^2}\I_{0}^{1}  \Phi(x,|\theta z_2 + (1-\theta)z_1|+|z_1|) \theta\,d\theta
        \\&
        \gtrsim
        \frac{|z_2-z_1|^2}{(|z_2|+|z_1|)^2}\Phi\left(x, \I_{0}^{1}  (|\theta z_2 + (1-\theta)z_1|+|z_1|) \theta\,d\theta\right)
        \\&
        \gtrsim
        \frac{|z_2-z_1|^2}{(|z_2|+|z_1|)^2}\Phi\left(x,|z_2|+|z_1|\right)
        \stackrel{\eqref{defV2}}{\approx} |V_{\Phi}(x,z_2)-V_{\Phi}(x,z_1)|^2
    \end{split}
\end{align*}
hold with having all implied constants in the above display depending on $n$ and $s(\Phi)$, whenever $x\in\O$ and $z_1,z_2\in\R^n\setminus\{0\}$, where in the third inequality of the last display we have applied Jensen's inequality to the convex function $\Phi(x,\cdot)$ with respect to measure $\theta\,d\theta$.

Moreover, the maps introduced in \eqref{defV} are very convenient to formulate the monotonicity properties of the vector field $D_{z}F(x,y,z)$ with respect to the gradient variable $z$ and some growth properties of the integrand $F$ defined in \eqref{functional}.
\begin{lem}
    \label{lem:nf4}
    Let $F: \O\times\R\times\R^n\rightarrow \R$ be a function defined in \eqref{type} satisfying \eqref{sa:1} and \eqref{sa:2}.Then there exist positive constants $c_1,c_2\equiv c_1,c_2(n,s(G),s(H_{a}),s(H_{b}),\nu)$ and $c_{3}\equiv c_{3}(n,s(G),s(H_{a}),s(H_{b}),L)$ such that the following inequalities
     \begin{align}
	    \label{defV1_1}
	    \begin{split}
	    &|V_{G}(z_1)-V_{G}(z_2)|^2 + a(x)|V_{H_{a}}(z_1)-V_{H_{a}}(z_2)|^2
	    + b(x)|V_{H_{b}}(z_1)-V_{H_{b}}(z_2)|^2 
	    \\&
	    \leqslant c_1\inner{D_{z}F(x,y,z_1)-D_{z}F(x,y,z_2)}{z_1-z_2},
	    \end{split}
    \end{align}
    \begin{align}
        \label{defV1_2}
        \begin{split}
            |V_{G}(z_1)-V_{G}(z_2)|^2 +& a(x)|V_{H_{a}}(z_1)-V_{H_{a}}(z_2)|^2 + b(x)|V_{H_{b}}(z_1)-V_{H_{b}}(z_2)|^2
            \\&
            \quad
             + c_2\inner{D_{z}F(x,y,z_1)}{z_2-z_1}
            \leqslant
            c_2[F(x,y,z_2)-F(x,y,z_1)]
        \end{split}
    \end{align}
    and 
    \begin{align}
        \label{defV1_3}
        \begin{split}
        |F(x_1,y,z)&-F(x_2,y,z)| 
        \\&
        \leqslant
        c_3\omega(|x_1-x_2|)\left[ G(|z|) + \min\{a(x_1),a(x_2)\}H_{a}(|z|) + \min\{b(x_1),b(x_2)\}H_{b}(|z|) \right]
        \\&
        \quad
         + c_3|a(x_1)-a(x_2)|H_{a}(|z|)
        + c_3|b(x_1)-b(x_2)|H_{b}(|z|)
        \end{split}
    \end{align}
    hold true, whenever $z,z_1,z_2\in \R^n \setminus \{0\}$, $x,x_1,x_2\in\O$ and $y\in\R$.
\end{lem}

\begin{proof}
    It follows from $\eqref{sa:2}_{2}$ that
    \begin{align*}
        \begin{split}
            &\inner{D_{z}F(x,y,z_1)-D_{z}F(x,y,z_2)}{z_1-z_2} 
            =
            \I_{0}^{1}\inner{D_{zz}^{2}F(x,y,\theta z_1 + (1-\theta)z_2)[z_1-z_2]}{z_1-z_2}\,d\theta
            \\&
            \geqslant
            \nu \I_{0}^{1}\frac{\Psi\left(x, \theta z_1 + (1-\theta)z_2 \right)}{|\theta z_1 + (1-\theta)z_2|^2}|z_1-z_2|^2\,d\theta 
            \\&
            \geqslant
            c\left(|V_{G}(z_1)-V_{G}(z_2)|^2 + a(x)|V_{H_{a}}(z_1)-V_{H_{a}}(z_2)|^2 + b(x)|V_{H_{b}}(z_1)-V_{H_{b}}(z_2)|^2\right),
        \end{split}
    \end{align*}
    where in the last inequality of the last display we have used \eqref{defV1} and \eqref{defV2}. Then \eqref{defV1_1} follows. The inequality \eqref{defV1_2} follows from the following observation that
    \begin{align*}
        \begin{split}
            &[F(x,y,z_2)-F(x,y,z_1)]-\inner{D_{z}F(x,y,z_1)}{z_2-z_1} 
            \\&
             =
             \I_{0}^{1}\inner{D_{z}F(x,y, \theta z_2+ (1-\theta)z_1)-D_{z}F(x,y,z_1)}{z_2-z_1}\,d\theta
             \\&
             \stackrel{\eqref{defV1_1}}{\geqslant}
             c\I_{0}^{1} \frac{1}{\theta} \left(|V_{G}(\theta z_2 + (1-\theta)z_1)-V_{G}(z_1)|^2 + a(x)|V_{H_{a}}(\theta z_2 + (1-\theta)z_1)-V_{H_{a}}(z_1)|^2\right)\,d\theta
             \\&
             \qquad
             +
             \I_{0}^{1}\frac{b(x)}{\theta}|V_{H_{b}}(\theta z_2 + (1-\theta)z_1)-V_{H_{b}}(z_1)|^2\,d\theta
             \\&
             \stackrel{\eqref{defV5}}{\geqslant}
             c|V_{G}(z_1)-V_{G}(z_2)|^2 + a(x)|V_{H_{a}}(z_1)-V_{H_{a}}(z_2)|^2 +
             b(x)|V_{H_{b}}(z_1)-V_{H_{b}}(z_2)|^2.
        \end{split}
    \end{align*}
    Since $F(x,y,0)=0$ for every $x\in\O$ and $y\in\R$, using \eqref{type}, we have
    \begin{align*}
    \begin{split}
        &|F(x_1,y,z)-F(x_2,y,z)| = |(F(x_1,y,z)-F(x_1,y,0))-(F(x_2,y,z)-F(x_2,y,0))|
        \\&
        =
        \left|\I_{0}^{1}\inner{D_{z}F(x_1,y,\theta z)}{z}\,d\theta - \I_{0}^{1}\inner{D_{z}F(x_2,y,\theta z)}{z}\,d\theta\right|
        \\&
        \leqslant
        \I_{0}^{1}|D_{z}F(x_1,y,\theta z)-D_{z}F(x_2,y,\theta z)|\,|z|d\theta
        \\&
        \leqslant
        \I_{0}^{1}\left|D_{z}F_{G}(x_1,y,\theta z)-D_{z}F_{G}(x_2,y,\theta z)\right||z|\,d\theta
        \\&
        \quad
        + \I_{0}^{1}\left|a(x_1)D_{z}F_{H_{a}}(x_1,y,\theta z)-a(x_2)D_{z}F_{H_{a}}(x_2,y,\theta z)\right||z|\,d\theta
        \\&
        \quad 
        + \I_{0}^{1}\left|b(x_1)D_{z}F_{H_{b}}(x_1,y,\theta z)-b(x_2)D_{z}F_{H_{b}}(x_2,y,\theta z)\right||z|\,d\theta.
    \end{split}
    \end{align*}
Without loss of generality, we can assume $a(x_2)\leqslant a(x_1)$ and $b(x_2)\leqslant b(x_1)$. Then using the structure assumption \eqref{sa:2}, we find 
\begin{align*}
	\begin{split}
	\I_{0}^{1}&\left|a(x_1)D_{z}F_{H_{a}}(x_1,y,\theta z)-a(x_2)D_{z}F_{H_{a}}(x_2,y,\theta z)\right||z|\,d\theta
	\\&
	\leqslant
	L|a(x_1)-a(x_2)|\I_{0}^{1}\frac{H_{a}(\theta|z|)}{\theta}\,d\theta
	+
	a(x_2)\omega(|x_1-x_2|)\I_{0}^{1}\frac{H_{a}(\theta|z|)}{\theta}\,d\theta
	\\&
	\leqslant	
	c a(x_2)H_{a}(|z|) + c\omega(|x_1-x_2|)H_{a}(|z|)
	\end{split}
\end{align*}
for some constant $c\equiv c(s(H_{a}),L)$. Similarly, we get 
\begin{align*}
	\begin{split}
	\I_{0}^{1}&\left|b(x_1)D_{z}F_{H_{b}}(x_1,y,\theta z)-b(x_2)D_{z}F_{H_{b}}(x_2,y,\theta z)\right||z|\,d\theta
	\\&
	\leqslant
	c b(x_2)H_{b}(|z|) + c\omega(|x_1-x_2|)H_{b}(|z|),
	\end{split}
\end{align*}
    where the validity of the last display is ensured by $\eqref{sa:2}_{3}$. Combining the last three displays, \eqref{defV1_3} follows.
\end{proof}


\subsection{Musielak-Orlicz and Musielak-Orlicz-Sobolev spaces}
\label{subsec:2.2}
We now introduce the Musielak-Orlicz spaces (generalized Orlicz spaces), which generalize the Orlicz spaces. Let $\Phi : \Omega \times [0,\infty) \rightarrow [0,\infty)$ be an Young function. Here we present some definitions and properties associated to Young functions. 
\begin{defn}
Let $\Phi$ be a Young function.
\begin{enumerate}
\item $\Phi$ is said to satisfy the $\Delta_2$-condition, denoted by $\Phi \in \Delta_2$, if there exists a positive number $\Delta_2(\Phi)$ such that $\Phi(x,2t) \leqslant \Delta_2(\Phi) \, \Phi(x,t)$ for all $x \in \Omega$ and $t \geqslant 0$.
\item $\Phi$ is said to satisfy the $\nabla_2$-condition, denoted by $\Phi \in \nabla_2$, if there exists a positive number $\nabla_2(\Phi)>1$ such that $\Phi(x,\nabla_2(\Phi) \, t) \geqslant 2\nabla_2(\Phi) \, \Phi(x,t)$ for all $x \in \Omega$ and $t \geq 0$.
\item We write $\Phi \in \Delta_2 \cap \nabla_2$ if $\Phi \in \Delta_2$ and $\Phi \in \nabla_2$.
\end{enumerate}
\end{defn}

For a given Young function $\Phi$, we define the complementary function $\Phi^*$ of $\Phi$ by, for each $x \in \Omega$ and $t\geqslant 0$,
\begin{equation*}
\Phi^*(x,t) = \sup \{ st-\Phi(x,s) : s \geq 0 \}.
\end{equation*}
Then $\Phi^*$ satisfies all the conditions to be a Young function. One can see that $(\Phi^*)^*=\Phi$ and that $\Phi \in \nabla_2$ if and only if $\Phi^* \in \Delta_2$ with $2\nabla_2(\Phi)=\Delta_2(\Phi^*)$.

For an Young function $\Phi$, the Musielak-Orlicz class $K^{\Phi}(\Omega;\R^N)$, $N\geqslant 1$, consists of all measurable functions $v : \Omega \rightarrow \R^N$ satisfying
\begin{equation*}
\I_{\Omega} \Phi(x,|v(x)|) \, dx < +\infty.
\end{equation*}
The Musielak-Orlicz space $L^{\Phi}(\Omega;\R^N)$ is the vector space generated by $K^{\Phi}(\Omega;\R^N)$.
If $\Phi \in \Delta_2$, then $K^{\Phi}(\Omega;\R^N)=L^{\Phi}(\Omega;\R^N)$ and this space is a Banach space under the Luxemburg norm
\begin{equation*}
\norm{v}_{L^{\Phi}(\Omega;\R^N)} = \inf \left\lbrace \sigma > 0 : \I_{\Omega} \Phi \left( x,\frac{|v(x)|}{\sigma} \right) \, dx \leq 1 \right\rbrace.
\end{equation*}

The Musielak-Orlicz-Sobolev space $W^{1,\Phi}(\Omega;\R^N)$ is the function space of all measurable functions $v \in L^{\Phi}(\Omega;\R^N)$ such that its distributional gradient vector $Dv$ belongs to $L^{\Phi}(\Omega;\R^{Nn})$.
For $v \in W^{1,\Phi}(\Omega;\R^N)$, we define its norm to be
\begin{equation*}
\norm{v}_{W^{1,\Phi}(\Omega;\R^N)} = \norm{v}_{L^{\Phi}(\Omega;\R^N)} + \norm{Dv}_{L^{\Phi}(\Omega;\R^{Nn})}.
\end{equation*}
The space $W^{1,\Phi}_0(\Omega;\R^N)$ is defined as the closure of $C^{\infty}_0(\Omega;\R^N)$ in $W^{1,\Phi}(\Omega;\R^N)$.
For $N=1$, we simply write $L^{\Phi}(\Omega):=L^{\Phi}(\Omega;\R)$ and $W^{1,\Phi}(\Omega):=W^{1,\Phi}(\Omega;\R)$.
For a detailed discussion of the Musielak-Orlicz spaces and the associated Sobolev spaces, we refer the reader to \cite{AF1,BS, Chleb1, Di1,HH, HHK, Mu, RR1, Si} and references therein.

We end up this preliminary section with presenting some standard technical lemmas which will be applied later, see for instance \cite{Gi1,GT, LU1}.

\begin{lem}
    \label{lem:t0}
    Let $h:[\rho_1,\rho_2]\rightarrow \R$ be a non-negative and bounded function, and $\theta\in (0,1)$, $A_{0}\geqslant 0$, $\gamma_0\geqslant 0$. Assume that 
    \begin{align*}
        h(t) \leqslant \theta h(s) + \frac{A_{0}}{(s-t)^{\gamma_0}}
    \end{align*}
    holds for $\rho_1 \leqslant t < s \leqslant \rho_2$. Then there is a constant $c\equiv c(\theta,\gamma_0)$ satisfying the following inequality:
    \begin{align*}
        h(\rho_0) \leqslant \frac{cA_0}{(\rho_1-\rho_0)^{\gamma_0}} .
    \end{align*}
\end{lem}

\begin{lem}	
	\label{lem:t1}
	Let $\{Y_{i}\}_{i=0}^{\infty}$ be a sequence of nonnegative numbers satisfying the following recursive inequalities
	\begin{align*}
		Y_{i+1} \leqslant Cb^{i}Y_{i}^{1+\tau_0}
	\end{align*}
	with some fixed positive constant $C$, $b>1$ and $\tau_0>0$ for every $i = 0,1,2,\ldots$. If 
	\begin{align*}
		Y_0 \leqslant C^{-\frac{1}{\tau_0}}b^{-\frac{1}{\tau_0^2}},
	\end{align*}
	then $Y_{i}\rightarrow 0$ as $i\rightarrow \infty$.
\end{lem}

\begin{lem}
	\label{lem:t2}
	Let $v\in W^{1,1}(B_{\rho})$ for some ball $B_{\rho}\subset\R^n$. Then there exists $c\equiv c(n)$ such that 
	\begin{align*}
		(l-k)|B_{\rho}\cap \{v>l\}|^{1-\frac{1}{n}}
		\leqslant \frac{c|B_{\rho}|}{|B_{\rho}\setminus\{v>k\}|}\I_{B_{\rho}\cap \{k<v\leqslant l\}} |Dv|\,dx
	\end{align*}
	holds, whenever $l$ and $k$ are real numbers with $l>k$.
\end{lem}


\section{Absence of Lavrentiev phenomenon}
\label{sec:3}
Here we deal with the absence of Lavrantiev phenomenon under the assumptions introduced in $\eqref{ma:1}_{2}$, $\eqref{ma:2}_{2}$ and $\eqref{ma:3}_{2}$. The following theorem widely covers the results of \cite[Theorem 4.1]{CM1}, \cite[Proposition 3.6]{CM2}, \cite[Theorem 3.1]{BO3}, \cite[Theorem 4]{BCM3}, \cite[Theorem 4.1]{BBO1} and \cite[Theorem 3.1]{BBO2}.
\begin{thm}
	\label{lp:thm1}
Let $\P$ be the functional defined in \eqref{ifunct} with the coefficient functions $a(\cdot)\in C^{\omega_a}(\O)$ and $b(\cdot)\in C^{\omega_{b}}(\O)$ for the functions $\omega_{a}, \omega_{b}$ being concave and vanishing at 0.
\begin{enumerate}
		\item[1.] If the condition $\eqref{ma:1}_{2}$ is satisfied, then for any function $v\in W^{1,\Psi}_{\loc}(\O)$ and ball $B_{R}\equiv B_{R}(x_0)\Subset \tilde{B}\Subset\O$ with $\P(v,\tilde{B})<\infty$, there exists a sequence of functions $\{v_{k}\}_{k=1}^{\infty}\subset W^{1,\infty}(B_{R})$ such that 
		\begin{align}
			\label{lp:1}
			v_k\rightarrow v\quad\text{in}\quad W^{1,G}(B_{R})
			\quad\text{and}\quad
			\P(v_k,B_{R})\rightarrow \P(v,B_{R}).
		\end{align}
		\item[2.] If the condition $\eqref{ma:2}_{2}$ is satisfied, then for any function $v\in W^{1,\Psi}_{\loc}(\O)\cap L^{\infty}_{\loc}(\O)$ and ball $B_{R}\equiv B_{R}(x_0)\Subset \tilde{B}\Subset\O$ with $\P(v,\tilde{B})<\infty$, there exists a sequence of functions $\{v_{k}\}_{k=1}^{\infty}\subset W^{1,\infty}(B_{R})$ such that 
		\begin{align}
			\label{lp:2}
			v_k\rightarrow v\quad\text{in}\quad W^{1,G}(B_{R}),
			\quad
			\P(v_k,B_{R})\rightarrow \P(v,B_{R})
			\quad\text{and}\quad
			\limsup_{k\to \infty}\norm{v_k}_{L^{\infty}(B_{R})} \leqslant 
			\norm{v}_{L^{\infty}(B_{R})}.
		\end{align}
		\item[3.] Let $v\in W^{1,\Psi}(\O)\in C^{0,\gamma}(\O)$ with some $\gamma\in (0,1)$ be a local $Q$-minimizer of the functional $\P$ under the assumption $\eqref{ma:3}_{2}$. Then, for every ball $B_{R}\Subset\O$, there exists a sequence of functions $\{v_{k}\}_{k=1}^{\infty}\subset W^{1,\infty}(B_{R})$ such that
	\begin{align}
		\label{lp:3}
		v_{k}\rightarrow v\quad\text{in}\quad W^{1,G}(B_{R})
		\quad\text{and}\quad
		\P(v_k,B_{R})\rightarrow \P(v,B_{R}).
	\end{align}
\end{enumerate}		
\end{thm}
\begin{proof}
	Essentially, the proof for the first two parts is similar to the one of \cite[Theorem 3.1]{BO3}. Since our assumptions are weaker than the assumptions considered there, we provide the detailed proof in any case. First we fix $\varepsilon_{0}\in (0,1)$ such that $B_{R}\Subset B_{R+\varepsilon_{0}}\Subset \tilde{B}\Subset \O$. 
	Let $\rho\in C_{0}^{\infty}(B_1)$ be a non-negative standard mollifier with $\I_{\R^n}\rho\,dx = 1$. Then we set 
	$\rho_{\varepsilon}(x):= \frac{1}{\varepsilon^n}\rho\left( \frac{x}{\varepsilon} \right)$ for $x\in B_{\varepsilon}$ with $0 < \varepsilon < \varepsilon_{0}$. Clearly $\rho_{\varepsilon}\in C_{0}^{\infty}(B_{\varepsilon})$, $\I_{\R^n} \rho_{\varepsilon}\,dx = 1$, $0\leqslant \rho_{\varepsilon} \leqslant c(n)\varepsilon^{-n}$ and $|D\rho_{\varepsilon}| \leqslant c(n)\varepsilon^{-(n+1)}$. For every $0<\varepsilon < \varepsilon_{0}/2$, we consider the following functions:
	\begin{align}
		\label{lp:4}
		v_{\varepsilon}(x):= (v\ast \rho_{\varepsilon} )(x),\quad a_{\varepsilon}(x):= \inf\limits_{y\in B_{2\varepsilon}(x)}a(y),\quad b_{\varepsilon}(x):= \inf\limits_{y\in B_{2\varepsilon}(x)}b(y)
	\end{align}
	and 
	\begin{align}
		\label{lp:npsi}
		\Psi_{\varepsilon}(x,t):= G(t) + a_{\varepsilon}(x)H_{a}(t) + b_{\varepsilon}(x)H_{b}(t)
	\end{align}
	for every $x\in B_{R}$ and $t\geqslant 0$.
\begin{enumerate}
	\item[1.] By Jensen's inequality, for a fixed $x\in B_{R}$, we have 
	\begin{align*}
		G(|Dv_{\varepsilon}(x)|) = G\left(|(Dv\ast \rho_{\varepsilon})(x)|\right) \leqslant \I_{\R^n} G(|Dv(x-y)|)\rho_{\varepsilon}(y)\,dy \leqslant c\varepsilon^{-n}.
	\end{align*}	
	It follows from $\eqref{ma:1}_{2}$ and the last display that 
	\begin{align}
		\label{lp:5}
		\begin{split}
			H_{a}(|Dv_{\varepsilon}(x)|) &= \frac{\left(H_{a}\circ G^{-1}\right)\left( G(|Dv_{\varepsilon}(x)|) \right)}{G(|Dv_{\varepsilon}(x)|)} G(|Dv_{\varepsilon}(x)|)
			\\&
			\leqslant 
			\lambda_{1}\left( 1+ \left[\omega_{a}\left(\left[G
			(|Dv_{\varepsilon}(x)|)\right]^{-\frac{1}{n}}\right)\right]^{-1} \right)G(|Dv_{\varepsilon}(x)|)
			\\&
			\leqslant
			c\left(1+[\omega_{a}(\varepsilon)]^{-1}\right)G(|Dv_{\varepsilon}(x)|) \leqslant c\left(1+[\omega_{a}(\varepsilon)]^{-1}\right) \Psi_{\varepsilon}(x,|Dv_{\varepsilon}(x)|).
		\end{split}
	\end{align}
	Similarly as above, we have 
	\begin{align}
		\label{lp:6}
		H_{b}(|Dv_{\varepsilon}(x)|)
		\leqslant
		c \left(1+[\omega_{b}(\varepsilon)]^{-1}\right) \Psi_{\varepsilon}(x,|Dv_{\varepsilon}(x)|).
	\end{align}
	\item[2.] Since $v$ is locally bounded in $\O$, we have
\begin{align*}
		|Dv_{\varepsilon}(x)|= |(v\ast D\rho_{\varepsilon})(x)| \leqslant \I_{\R^n}|v(x-y)||D\rho_{\varepsilon}(y)|\,dy \leqslant c(n)\norm{v}_{L^{\infty}(\tilde{B})}\varepsilon^{-1}.
\end{align*}
Then the assumption $\eqref{ma:2}_{2}$ and the last display imply
	\begin{align}
		\label{lp:7}
		\begin{split}
			H_{a}(|Dv_{\varepsilon}(x)|) 
			&= \frac{H_{a}(|Dv_{\varepsilon}(x)|)}{G(|Dv_{\varepsilon}(x)|)} G(|Dv_{\varepsilon}(x)|)
			\leqslant \lambda_{2}\left( 1+ 
			\left[\omega_{a}\left(|Dv_{\varepsilon}(x)|^{-1}\right)\right]^{-1} \right)G(|Dv_{\varepsilon}(x)|)
			\\&
			\leqslant
			c\left(1+[\omega_{a}(\varepsilon)]^{-1}\right)G(|Dv_{\varepsilon}(x)|) \leqslant c\left(1+[\omega_{a}(\varepsilon)]^{-1}\right) \Psi_{\varepsilon}(x,|Dv_{\varepsilon}(x)|)
		\end{split}
	\end{align}
	with some constant $c\equiv c\left(n,\lambda_{2},\norm{v}_{L^{\infty}(\tilde{B})} \right)$ for every $x\in B_{R}$. Arguing in the same way, for every $x\in B_{R}$, we have
	\begin{align}
		\label{lp:8}
		H_{b}(|Dv_{\varepsilon}(x)|) 
		\leqslant
		 c\left(1+[\omega_{b}(\varepsilon)]^{-1}\right) \Psi_{\varepsilon}(x,|Dv_{\varepsilon}(x)|).
	\end{align}
\end{enumerate}	
Using the continuity of the coefficient functions $a(\cdot)$ and $b(\cdot)$ and recalling the very definition of $\Psi_{\varepsilon}$ in \eqref{lp:5}, for every $x\in B_{R}$, we have
	\begin{align}
	\label{lp:9}
	\begin{split}
		\Psi(x,|Dv_{\varepsilon}(x)|) 
		&\leqslant \Psi_{\varepsilon}(x,|Dv_{\varepsilon}(x)|)+ |a(x)-a_{\varepsilon}(x)|H_{a}(|Df_{\varepsilon}(x)|)
		+|b(x)-b_{\varepsilon}(x)|H_{b}(|Df_{\varepsilon}(x)|) 
		\\&
		\leqslant \Psi_{\varepsilon}(x,|Dv_{\varepsilon}(x)|)+ 4[a]_{\omega_{a}}\omega_{a}(\varepsilon)H_{a}(|Dv_{\varepsilon}(x)|) + 4[b]_{\omega_{b}}\omega_{b}(\varepsilon)H_{b}(|Dv_{\varepsilon}(x)|).
	\end{split}
	\end{align}
	Therefore, taking into account \eqref{lp:5}-\eqref{lp:6} when the first case comes into play, and \eqref{lp:7}-\eqref{lp:8} when the second case is considered, in any case, it follows from \eqref{lp:9} that
	\begin{align}
		\label{lp:10}
		\begin{split}
			\Psi(x,|Dv_{\varepsilon}(x)|) &\leqslant
			c\Psi_{\varepsilon}(x,|Dv_{\varepsilon}(x)|) + c\omega_{a}(\varepsilon)(1+[\omega_{a}(\varepsilon)]^{-1}) \Psi_{\varepsilon}(x,|Dv_{\varepsilon}(x)|)
			\\&
			\quad
			c\omega_{b}(\varepsilon)(1+[\omega_{b}(\varepsilon)]^{-1}) \Psi_{\varepsilon}(x,|Dv_{\varepsilon}(x)|)
		 \leqslant c\Psi_{\varepsilon}(x,|Dv_{\varepsilon}(x)|)
		\end{split}
	\end{align}	
	for some constant $c$ being independent of $\varepsilon$. Therefore, by Jensen's inequality, we get 
	\begin{align}
		\label{lp:11}
		\begin{split}
		\Psi_{\varepsilon}(x,|Dv_{\varepsilon}(x)|) &\leqslant \I_{B_{\varepsilon}(x)} \Psi_{\varepsilon}(x,|Dv(y)|)\rho_{\varepsilon}(x-y)\,dy
		\leqslant
		\I_{B_{\varepsilon}(x)} \Psi(y,|Dv(y)|)\rho_{\varepsilon}(x-y)\,dy
		\\&
		=[\Psi(\cdot,|Dv(\cdot)|)\ast\rho_{\varepsilon}](x)=: [\Psi(\cdot,|Dv(\cdot)|]_{\varepsilon}(x).
		\end{split}	
	\end{align}
	
	Hence, in any case, using \eqref{lp:10}-\eqref{lp:11}, we conclude that 
	\begin{align}
		\label{lp:12}
		\Psi(x,|Dv_{\varepsilon}(x)|) \leqslant c[\Psi(\cdot,|Dv(\cdot)|)]_{\varepsilon}(x)
	\end{align}
	holds every $x\in B_{R}$ with a constant $c$ independent of $\varepsilon$.
	Since $[\Psi(\cdot,|Dv(\cdot)|)]_{\varepsilon}\rightarrow \Psi(\cdot,|Dv(\cdot)|)$ strongly in $L^{1}(B_{R})$, we are able to apply the general Lebesgue's dominated convergence theorem of \cite[Theorem 19]{RF1} to obtain a sequence of functions $\{v_{k}\} := \{ v_{\varepsilon_{k}}\}\subset C_{0}^{\infty}(\tilde{B})$ satisfying \eqref{lp:1} for the first case and $\eqref{lp:2}_{1,2}$ for the second case with some suitable choice of  $\varepsilon_{k}\rightarrow 0$. Clearly, the assertion $\eqref{lp:2}_{3}$ comes from the very definition of mollification of $v$ defined in \eqref{lp:4}. 
\begin{enumerate}
	\item[3.] Now we turn our attention to proving the last part of the theorem. Applying a Caccioppoli type inequality of Lemma \ref{lem:cacct} under the assumption $\eqref{ma:3}_{2}$ below, we see that 
	\begin{align}
		\label{lp:13}
		\FI_{B_{\varepsilon}(x)}\Psi_{\varepsilon}(x,|Dv(z)|)\,dz
		\leqslant
		c \FI_{B_{2\varepsilon}(x)}\Psi_{\varepsilon}\left(x, \left|\frac{v(z)-(v)_{B_{2\varepsilon}(x)}}{\varepsilon} \right| \right)\,dz
	\end{align}
	for a constant $c$ independent of $\varepsilon$. Therefore, by the very definition of the convolution, the fact that $\Psi_{\varepsilon}(x,\cdot)$ is convex for any fixed $x\in B_{R}$ and \eqref{lp:13}, we have 
	\begin{align}
		\label{lp:14}
		\begin{split}
		|Dv_{\varepsilon}(x)| &\leqslant
		c\left(\Psi_{\varepsilon}(x,\cdot)\right)^{-1}_{t}\circ \Psi_{\varepsilon}\left(x,\FI_{B_{\varepsilon}(x)}|Dv(z)|\,dz \right)
		\\&
		\leqslant
		c\left(\Psi_{\varepsilon}(x,\cdot)\right)^{-1}_{t}\left( \FI_{B_{\varepsilon}(x)}\Psi_{\varepsilon}(x,|Dv(z)|)\,dz\right)
		\\&
		\leqslant
		c\left(\Psi_{\varepsilon}(x,\cdot)\right)^{-1}_{t}\left( \FI_{B_{2\varepsilon}(x)}\Psi_{\varepsilon}\left(x, \left|\frac{v(z)-(v)_{B_{2\varepsilon}(x)}}{\varepsilon} \right| \right)\,dz\right)
		\leqslant 
		c\varepsilon^{\gamma-1}
		\end{split}
	\end{align}
	with some constant $c$ independent of $\varepsilon$, whenever $x\in B_{R}$ and $\varepsilon\in (0,\varepsilon_{0}/4)$, where we have also used the assumption $v\in C^{0,\gamma}(\O)$ for some $\gamma\in (0,1)$ and Lemma \ref{lem:nf1} together with Remark \ref{rmk:nf2}. Recalling the definition of $\Psi_{\varepsilon}$ in \eqref{lp:npsi}, using the modulus of continuity of functions $a(\cdot)$, $b(\cdot)$ and the assumption $\eqref{ma:3}_{2}$, for every $x\in B_{R}$, we estimate
	\begin{align*}
		\begin{split}
		\Psi(x,|Dv_{\varepsilon}(x)|) &\leqslant
		\Psi_{\varepsilon}(x, |Dv_{\varepsilon}(x)|) + |a_{\varepsilon}(x)-a(x)|H_{a}(|Dv_{\varepsilon}(x)|) + |b_{\varepsilon}(x)-b(x)|H_{b}(|Dv_{\varepsilon}(x)|)
		\\&
		\leqslant
		\Psi_{\varepsilon}(x, |Dv_{\varepsilon}(x)|) + c \omega_{a}(\varepsilon) \left( 1+ \left[ \omega_{a}\left( |Dv_{\varepsilon}(x)|^{-\frac{1}{1-\gamma}} \right) \right]^{-1} \right)G(|Dv_{\varepsilon}(x)|)
		\\&
		\quad
		+ c \omega_{b}(\varepsilon) \left( 1+ \left[ \omega_{b}\left( |Dv_{\varepsilon}(x)|^{-\frac{1}{1-\gamma}} \right) \right]^{-1} \right)G(|Dv_{\varepsilon}(x)|)
		\\&
		\leqslant
		c\Psi_{\varepsilon}(x,|Dv_{\varepsilon}(x)|)
		\end{split}
	\end{align*}
	for a constant $c$ independent of $\varepsilon$, where we have also used \eqref{lp:14}. Then arguing in the same way as in \eqref{lp:11}-\eqref{lp:12}, we find a sequence of functions $\{v_{k}\}_{k=1}^{\infty}\subset W^{1,\infty}(B_{R})$ satisfying \eqref{lp:3}. The proof is now finished.
\end{enumerate}
\end{proof}

\section{Sobolev-Poincar\'e type inequalities}
\label{sec:4}
In the present section we provide a Sobolev-Poincar\'e type inequality for functions $v\in W^{1,\Psi}(B_{R})$ with some ball $B_{R}\subset\O$, which is one of key points for further investigations.  For this, first we give a Sobolev-Poincar\'e type inequality for functions of $W^{1,\Phi}(B_{R})$ with $\Phi\in \mathcal{N}$ and a ball $B_{R}\subset\R^n$.

\begin{lem}
    \label{lem:os}
    Let $\Phi\in \mathcal{N}$ with an index $s(\Phi)\geqslant 1$. For any $d_{0} \in [1,\frac{n}{n-1})$, there exists $\theta\equiv \theta(n,s(\Phi),d_{0})\in (0,1)$ such that 
    \begin{align}
        \label{os:os1}
         \left(\FI_{B_{R}}\left[\Phi\left(\left|\frac{v-(v)_{B_{R}}}{R} \right| \right)\right]^{d_{0}}\,dx\right)^{\frac{1}{d_{0}}}
        \leqslant
        c\left(\FI_{B_{R}}\left[\Phi(|Dv|)\right]^{\theta}\,dx\right)^{\frac{1}{\theta}}
    \end{align}
    holds for some constant $c\equiv c(n,s(\Phi),d_{0})$, whenever $v\in W^{1,\Phi}(B_{R})$ and $B_{R}\subset \R^{n}$ is a ball. Moreover, the above estimate still holds with $v-(v)_{B_{R}}$ replaced by $v$ if $v\in W^{1,\Phi}_{0}(B_{R})$.
\end{lem}

\begin{proof}
    First note by Lemma $\ref{lem:nf2}_{4}$ that there exists $\theta \equiv \theta(n,s(\Phi),d_{0})\in \left(\frac{(n-1)d_{0}}{n},1  \right)$ such that $\Phi^{\theta}\in\mathcal{N}$ with an index $s(\Phi^{\theta})$ depending on $n, s(\Phi),d_{0}$. Therefore, the following classical formula 
    \begin{align}
        \label{os:os2}
        |v(x)-(v)_{B_{R}}| \leqslant c(n) \I_{B_{R}} \frac{|Dv(y)|}{|x-y|^{n-1}}\,dy
    \end{align}
    holds for a.e $x\in B_{R}$, see for instance \cite[Lemma 7.14]{GT}. Letting $\displaystyle E:= \FI_{B_{R}}\Phi^{\theta}(|Dv|)\,dx$, one can assume that $E>0$, otherwise $v$ is constant on $B_{R}$ and the inequality \eqref{os:os1} is trivial. Using \eqref{os:os2}, the fact that $\Phi$ is increasing and Lemma \ref{lem:nf1}, we have
    \begin{align*}
        I:= \FI_{B_{R}}\left[\Phi\left(\left|\frac{v-(v)_{B_{R}}}{R} \right| \right)\right]^{d_{0}}\,dx
        \leqslant
        c\FI_{B_{R}}\left[\Phi\left( \I_{B_{R}} \frac{|Dv(y)|}{R|x-y|^{n-1}}\,dy \right)\right]^{d_{0}}\,dx
    \end{align*}
    with $c\equiv c(n,s(\Phi),d_{0})$. Since $\I_{B_{R}}\frac{1}{R|x-y|^{n-1}}\,dy < c(n)$, where this constant is independent of $x\in B_{R}$ and a ball $B_{R}$, we apply Jensen's inequality to the convex function $\Phi^{\theta}$ with respect to the measure $R^{-1}|x-y|^{-(n-1)}\,dy$ to obtain 
    \begin{align}
        \label{os:os4}
        \begin{split}
        \displaystyle
            I 
            \leqslant
            c\FI_{B_{R}}\left( \I_{B_{R}} \frac{\left[\Phi(|Dv(y)|)\right]^{\theta}}{R|x-y|^{n-1}}\,dy \right)^{\frac{d_{0}}{\theta}}\,dx
            &=
            cR^{\frac{(n-1)d_{0}}{\theta}} E^{\frac{d_{0}}{\theta}}
            \FI_{B_{R}}\left( \FI_{B_{R}} \frac{\left[\Phi(|Dv(y)|)\right]^{\theta}}{|x-y|^{n-1}}E^{-1}\,dy \right)^{\frac{d_{0}}{\theta}}\,dx
            \\&
            \leqslant
            cR^{\frac{(n-1)d_{0}}{\theta}} E^{\frac{d_{0}}{\theta}}
            \FI_{B_{R}} \FI_{B_{R}} \frac{\left[\Phi(|Dv(y)|)\right]^{\theta}}{|x-y|^{\frac{(n-1)d_{0}}{\theta}}}E^{-1}\,dy\,dx,
        \end{split}
    \end{align}
where in the last estimate we have applied again Jensen's inequality to the convex function $t\mapsto t^{\frac{d_{0}}{\theta}}$ with respect to the probability measure $E^{-1}\Phi^{\theta}(|Dv(y)|)\,dy$. We observe that  
\begin{align*}
    \displaystyle
    \FI_{B_{R}} \frac{1}{|x-y|^{\frac{(n-1)d_{0}}{\theta}}}\,dx
    \leqslant
    \frac{1}{|B_{R}|} \I_{B_{2R}(y)} \frac{1}{|x-y|^{\frac{(n-1)d_{0}}{\theta}}}\,dx 
    \leqslant
    c(n,s(\Phi),d) R^{-\frac{(n-1)d_{0}}{\theta}},
\end{align*}
which is possible since $\frac{(n-1)d_{0}}{\theta} < n$. Inserting the last estimate into \eqref{os:os4}, the inequality \eqref{os:os1} follows. Finally, if we replace $v-(v)_{B_{R}}$ by $v$ if $v\in W^{1,\Phi}_{0}(B_{R})$, then the estimate \eqref{os:os1} still holds true since the following classical formula 
\begin{align*}
        |v(x)| \leqslant c(n) \I_{B_{R}} \frac{|Dv(y)|}{|x-y|^{n-1}}\,dy
    \end{align*}
    is valid for a.e $x\in B_{R}$, whenever $v\in W^{1,1}_{0}(B_{R})$, see for instance \cite[Lemma 7.14]{GT}.
\end{proof}

\begin{thm}
	\label{thm:sp}
	  Let $v\in W^{1,\Psi}(B_{R})$ for a ball $B_{R}\subset\O$ with $R\leqslant 1$ under $a(\cdot)\in C^{\omega_{a}}(\O)$ and $b(\cdot)\in C^{\omega_{b}}(\O)$. Then, for any $d\in \left[1,\frac{n^2}{n^2-1} \right)$, there exist constants $\theta\equiv \theta(n,s(G),s(H_{a}),s(H_{b}),d)\in (0,1)$ and $c\equiv c(n,s(G),s(H_{a}),s(H_{b}), \omega_a(1), \omega_b(1),d)$ such that the following Sobolev-Poincar\'e-type inequality holds:
	\begin{align}
		\label{sp:1}
		\left[\FI_{B_{R}}\left[\Psi\left(x,\left|\frac{v-(v)_{B_{R}}}{R}\right|\right)\right]^{d}\,dx\right]^{\frac{1}{d}} \leqslant
		c\lambda_{sp}\left[ \FI_{B_{R}} [\Psi(x,|Dv|)]^{\theta}\,dx \right]^{\frac{1}{\theta}},
	\end{align}
	 where
	 
\begin{subequations}
	 \begin{empheq}[left={\lambda_{sp}=\empheqlbrace}]{align}
        \qquad & 1+([a]_{\omega_{a}}+ [b]_{\omega_{b}})\left(\lambda_{1} +  \lambda_{1}\left( \I_{B_{R}}G(|Dv|)\,dx \right)^{\frac{1}{n}}\right)   &\text{if }& v\in W^{1,\Psi}(B_{R}) \text{ with } \eqref{ma:1}_{2} \label{sp:2_1}.\\
        & 1+ ([a]_{\omega_{a}}+ [b]_{\omega_{b}})\left(\lambda_{2} + \lambda_{2} \norm{v}_{L^{\infty}(B_{R})}\right) &\text{if }& v\in  L^{\infty}(B_{R}) \text{ with } \eqref{ma:2}_{2} \label{sp:2_2}.\\
       &1+ ([a]_{\omega_{a}}+ [b]_{\omega_{b}})\left(\lambda_{3} + \lambda_{3}\left[ R^{-\gamma}\osc\limits_{B_{R}} v \right]^{\frac{1}{1-\gamma}}\right)
         &\text{if }& v\in C^{0,\gamma}(B_{R}) \text{ with } \eqref{ma:3}_{2}.\label{sp:2_3}
      \end{empheq}
	\end{subequations}
Moreover, the above estimate \eqref{sp:1} is still valid with $v-(v)_{B_{R}}$ replaced by $v$ depending on which one of \eqref{sp:2_1}-\eqref{sp:2_3} comes into play if $v\in W^{1,\Psi}_{0}(B_{R})$.
\end{thm}

\begin{proof}
The above theorem widely covers the results of \cite[Theorem 4.2]{BBO1}, \cite[Theorem 32.]{BBO2} and also the results of \cite[Theorem 1.6]{CM1}, which is a special case when $G(t)=t^{p}$, $H(t)=t^{q}$, $\omega_{a}(t)=t^{\alpha}$ and $\omega_{b}(\cdot)\equiv 0$ for some constants $1<p\leqslant q$ and $\alpha\in (0,1]$. Then using the continuity of the coefficient functions $a(\cdot)$ and $b(\cdot)$, we find
\begin{align}
		\label{sp:3}
		\begin{split}
		I:= \left(\FI_{B_{R}}\left[\Psi\left(x,\left|\frac{v-(v)_{B_{R}}}{R}\right|\right)\right]^{d}\,dx \right)^{\frac{1}{d}}
		&\leqslant
		 18[a]_{\omega_{a}}\omega_{a}(R)\left(\FI_{B_{R}} \left[H_{a}\left(\left|\frac{v-(v)_{B_{R}}}{R}\right|\right)\right]^{d}\,dx \right)^{\frac{1}{d}}
		 \\&
		 \qquad
		 + 18[b]_{\omega_{b}}\omega_{b}(R)\left(\FI_{B_{R}}\left[H_{b}\left(\left|\frac{v-(v)_{B_{R}}}{R}\right|\right)\right]^{d}\,dx \right)^{\frac{1}{d}}
		\\&
		\qquad
		+ 
		9\left(\FI_{B_{R}}\left[\Psi_{B_{R}}^{-}\left(\left|\frac{v-(v)_{B_{R}}}{R}\right|\right)\right]^{d}\,dx\right)^{\frac{1}{d}}
		\\&
		 =: 18[a]_{\omega_{a}}I_{1} + 18[b]_{\omega_{b}}I_{2} + 9I_{3},
		\end{split}
	\end{align}
where we have used the following elementary inequality
	\begin{align*}
		(t_1+t_2+t_3)^{d} \leqslant 3^{d}\left( t_1^{d} + t_2^{d} + t_3^{d} \right)\quad (\forall t_1,t_2,t_3\geqslant 0).
	\end{align*}
We now estimate the terms $I_{i}$ with $i\in \{1,2,3\}$ in \eqref{sp:3} depending on which one of $\eqref{ma:1}_{2}$, $\eqref{ma:2}_2$ and $\eqref{ma:3}_{2}$ is under consideration. In turn, using $\eqref{ma:1}_{2}$ and \eqref{concave2}, we see 

\begin{align}
		\label{sp:4}
		\begin{split}
	I_{1} &= \displaystyle  \omega_{a}(R) \left(\FI_{B_{R}} \left[\frac{\left(H_{a}\circ G^{-1}\right)\left(G\left(\left|\frac{v-(v)_{B_{R}}}{R}\right|\right)\right)}{G\left(\left|\frac{v-(v)_{B_{R}}}{R}\right|\right)} G\left(\left|\frac{v-(v)_{B_{R}}}{R}\right|\right)\right]^{d}\,dx\right)^{\frac{1}{d}}
	\\&
	\leqslant
	\lambda_{1}\omega_{a}(R)
	\left(\FI_{B_{R}}\left[ \left(1+ \left[ \omega_{a}\left(  \left[ G\left(\left|\frac{v-(v)_{B_{R}}}{R}\right|\right) \right]^{-\frac{1}{n}} \right) \right]^{-1} \right)G\left(\left|\frac{v-(v)_{B_{R}}}{R}\right|\right)\right]^{d}\,dx\right)^{\frac{1}{d}}
	\\&
	\leqslant
	\lambda_{1}\omega_{a}(R)
	\left(\FI_{B_{R}}\left[ \left(1+ \left[ \frac{1}{\omega_{a}(R)} + \frac{R}{\omega_{a}(R)}\left( G\left(\left|\frac{v-(v)_{B_{R}}}{R}\right|\right) \right)^{\frac{1}{n}} \right] \right)G\left(\left|\frac{v-(v)_{B_{R}}}{R}\right|\right)\right]^{d}\,dx\right)^{\frac{1}{d}}
	\\&
	\leqslant
	9\lambda_{1}(1+\omega_{a}(1)) \left(\FI_{B_{R}}\left[G\left(\left|\frac{v-(v)_{B_{R}}}{R}\right|\right)\right]^{d}\,dx\right)^{\frac{1}{d}} 
	\\&
	\quad
	+ 9\lambda_{1} R \left( \FI_{B_{R}}\left[ G\left(\left|\frac{v-(v)_{B_{R}}}{R}\right|\right) \right]^{\left(1+\frac{1}{n}\right)d}\,dx\right)^{\frac{1}{d}},
	\end{split}
	\end{align}
where we have used also that $\omega_{a}(\cdot)$ is non decreasing and $R\leqslant 1$. In the same way, we have 
\begin{align}
	\label{sp:5}
	\begin{split}
	I_{2} &\leqslant 9\lambda_{1}(1+\omega_{b}(1)) \left(\FI_{B_{R}}\left[G\left(\left|\frac{v-(v)_{B_{R}}}{R}\right|\right)\right]^{d}\,dx\right)^{\frac{1}{d}} 
	\\&
	\quad
	+ 9\lambda_{1} R \left( \FI_{B_{R}}\left[ G\left(\left|\frac{v-(v)_{B_{R}}}{R}\right|\right) \right]^{\left(1+\frac{1}{n}\right)d}\,dx\right)^{\frac{1}{d}}.
	\end{split}
\end{align}
Adding the estimates coming from the last two displays and applying Lemma \ref{lem:os} with $\Phi\equiv G$ for $d_{0}\equiv d$ and $d_{0}\equiv \left(1+\frac{1}{n}\right)d < \frac{n}{n-1}$, there exists $\theta_{1}\equiv \theta_{1}(n,s(G),d)\in (0,1)$ such that
\begin{align}
	\label{sp:6}
	\begin{split}
		I_1 + I_2 &\leqslant 
		c\lambda_{1}\left(\FI_{B_{R}}\left[G\left(\left|\frac{v-(v)_{B_{R}}}{R}\right|\right)\right]^{d}\,dx\right)^{\frac{1}{d}}  + c\lambda_{1} R \left( \FI_{B_{R}}\left[ G\left(\left|\frac{v-(v)_{B_{R}}}{R}\right|\right) \right]^{\left(1+\frac{1}{n}\right)d}\,dx\right)^{\frac{1}{d}}
		\\&
		\leqslant
		c\lambda_{1}\left[\FI_{B_{R}}[G(|Dv|)]^{\theta_1}\,dx\right]^{\frac{1}{\theta_1}} + c\lambda_{1} R \left[\FI_{B_{R}}[G(|Dv|)]^{\theta_1}\,dx\right]^{\left(1+\frac{1}{n} \right)\frac{1}{\theta_1}}
		\\&
		\leqslant
		c\lambda_{1}\left[ 1+ \left( \I_{B_{R}}G(|Dv|)\,dx \right)^{\frac{1}{n}}  \right]\left[\FI_{B_{R}}G^{\theta_1}(|Dv|)\,dx\right]^{\frac{1}{\theta_1}}
	\end{split}
\end{align}
for some constant $c\equiv c(n,s(G),\omega_{a}(1),\omega_{b}(1),d)$, where in the last inequality of the above display we have used H\"older's inequality. Since $\Psi_{B_{R}}^{-}\in \N$ with an index $s(\Psi)=s(G)+s(H_{a})+s(H_{b})$ by Remark \ref{rmk:nf2}, we are able to apply Lemma \ref{lem:os} with $\Phi\equiv \Psi_{B_{R}}^{-}$ for $d_{0}\equiv d$. In turn, there exists $\theta_2\equiv \theta_2(n,s(\Psi),d)$ such that 
\begin{align}
	\label{sp:7}
	I_3 \leqslant c \left[\FI_{B_{R}}\left[\Psi_{B_{R}}^{-}(|Dv|)\right]^{\theta_2}\,dx\right]^{\frac{1}{\theta_{2}}}
\end{align}
with some constant $c\equiv c(n,s(\Psi),d)$. Inserting the estimates obtained in \eqref{sp:6}-\eqref{sp:7} into \eqref{sp:3}, recalling the very definition of $\Psi_{B_{R}}^{-}$ in \eqref{ispsi} and setting $\theta := \max\{\theta_1,\theta_2\}$, we arrive at \eqref{sp:2_1}. Now we turn our attention to proving \eqref{sp:2_2}. For this, we estimate the terms $I_{i}$ for $i\in \{1,2,3\}$ for $v\in L^{\infty}(B_{R})$ under the assumption $\eqref{ma:2}_{2}$. In turn, using \eqref{concave2}  and the assumption $\eqref{ma:2}_{2}$, we see
\begin{align}
	\label{sp:8}
	\begin{split}
	I_{1} &= \displaystyle  \omega_{a}(R) \left(\FI_{B_{R}} \left[\frac{H_{a}\left(\left|\frac{v-(v)_{B_{R}}}{R}\right|\right)}{G\left(\left|\frac{v-(v)_{B_{R}}}{R}\right|\right)} G\left(\left|\frac{v-(v)_{B_{R}}}{R}\right|\right)\right]^{d}\,dx\right)^{\frac{1}{d}}
	\\&
	\leqslant
	\lambda_{2}\omega_{a}(R)\left(
	\FI_{B_{R}} \left[\left(1+ \left[ \omega_{a}\left( \left(\left|\frac{v-(v)_{B_{R}}}{R}\right|\right)^{-1} \right)  \right]^{-1} \right)G\left(\left|\frac{v-(v)_{B_{R}}}{R}\right|\right)\right]^{d}\,dx\right)^{\frac{1}{d}}
	\\&
	\leqslant
	\lambda_{2}\omega_{a}(R)
	\left(\FI_{B_{R}}\left[ \left(1+ \left[ \frac{1}{\omega_{a}(R)} + \frac{\left|v-(v)_{B_{R}}\right|}{\omega_{a}(R)} \right] \right)G\left(\left|\frac{v-(v)_{B_{R}}}{R}\right|\right)\right]^{d}\,dx\right)^{\frac{1}{d}}
	\\&
	\leqslant
	2\lambda_{2}\left(1+\omega_{a}(1) + \norm{v}_{L^{\infty}(B_{R})} \right)
	\left(\FI_{B_{R}}\left[G\left(\left|\frac{v-(v)_{B_{R}}}{R}\right|\right)\right]^{d}\,dx\right)^{\frac{1}{d}}.
	\end{split}
\end{align}
In a similar way, one can see 
\begin{align}
	\label{sp:9}
	I_2 \leqslant 2\lambda_{2}\left(1+\omega_{b}(1) + \norm{v}_{L^{\infty}(B_{R})} \right)
	\left(\FI_{B_{R}}\left[G\left(\left|\frac{v-(v)_{B_{R}}}{R}\right|\right)\right]^{d}\,dx\right)^{\frac{1}{d}}.
\end{align}
Adding the estimates in \eqref{sp:8}-\eqref{sp:9} and applying Lemma \ref{lem:os} with $\Phi\equiv G$ for $d_{0}\equiv d$, there exists an exponent $\theta_1\equiv \theta_1(n,s(G),d)\in (0,1)$ such that 
\begin{align}
	\label{sp:10}
	I_1 + I_2 \leqslant
	c \lambda_{2}\left( 1+\norm{v}_{L^{\infty}(B_{R})}\right)
	\left[\FI_{B_{R}}[G(|Dv|)]^{\theta_1}\,dx\right]^{\frac{1}{\theta_1}}
\end{align}
for some constant $c\equiv c(n,s(G),\omega_{a}(1),\omega_{b}(1),d)$. This estimate together with \eqref{sp:7} and the very definition of $\Psi_{B_{R}}^{-}$ in \eqref{ispsi}, we find \eqref{sp:2_2}. It remains to prove \eqref{sp:2_3}. Essentially, it can proved in a similar manner we have shown in \eqref{sp:8}-\eqref{sp:9}. So using the assumption $\eqref{ma:3}_{2}$ and again \eqref{concave2}, we see 
\begin{align}
	\label{sp:11}
	\begin{split}
	I_{1} &= \displaystyle  \omega_{a}(R) \left(\FI_{B_{R}} \left[\frac{H_{a}\left(\left|\frac{v-(v)_{B_{R}}}{R}\right|\right)}{G\left(\left|\frac{v-(v)_{B_{R}}}{R}\right|\right)} G\left(\left|\frac{v-(v)_{B_{R}}}{R}\right|\right)\right]^{d}\,dx\right)^{\frac{1}{d}}
	\\&
	\leqslant
	 \lambda_{3}\omega_{a}(R)
	\left(\FI_{B_{R}} \left[\left(1+ \left[ \omega_{a}\left( \left(\left|\frac{v-(v)_{B_{R}}}{R}\right|\right)^{-\frac{1}{1-\gamma}} \right)  \right]^{-1} \right)G\left(\left|\frac{v-(v)_{B_{R}}}{R}\right|\right)\right]^{d}\,dx\right)^{\frac{1}{d}}
	\\&
	\leqslant
	\lambda_{3}\omega_{a}(R)
	\left(\FI_{B_{R}} \left[\left(1+ \left[ \frac{1}{\omega_{a}(R)} + \frac{R^{\frac{-\gamma}{1-\gamma}}}{\omega_{a}(R)}\left|v-(v)_{B_{R}}\right|^{\frac{1}{1-\gamma}} \right] \right)G\left(\left|\frac{v-(v)_{B_{R}}}{R}\right|\right)\right]^{d}\,dx\right)^{\frac{1}{d}}
	\\&
	\leqslant
	2 \lambda_{3}\left(1+\omega_{a}(1) + \left[ R^{-\gamma}\osc\limits_{B_{R}} v \right]^{\frac{1}{1-\gamma}} \right)
	\left(\FI_{B_{R}}\left[G\left(\left|\frac{v-(v)_{B_{R}}}{R}\right|\right)\right]^{d}\,dx\right)^{\frac{1}{d}}.
	\end{split}
\end{align}
By arguing in the same way, we see 
\begin{align}
	\label{sp:12}
	I_1 + I_2 \leqslant
	c \lambda_{3} \left(1+ \left[ R^{-\gamma}\osc\limits_{B_{R}} v \right]^{\frac{1}{1-\gamma}} \right)
	\left(\FI_{B_{R}}\left[G\left(\left|\frac{v-(v)_{B_{R}}}{R}\right|\right)\right]^{d}\,dx\right)^{\frac{1}{d}}.
\end{align}
for some constant $c\equiv c(s(G), \omega_a(1),\omega_b(1))$. Finally, this estimate together with \eqref{sp:7} leads to \eqref{sp:2_3}. The proof is complete.
\end{proof}

\begin{rmk}
	\label{rmk:sp} We here remark that choosing $d\equiv 1$ in a Sobolev-Poincar\'e type inequality of Theorem \ref{thm:sp}, there exist an exponent  $\theta\equiv \theta(n,s(G),s(H_a),s(H_{b}))$  such that 
	\begin{align}
		\label{rmk:sp:1}
		\FI_{B_{R}}\Psi\left(x,\left|\frac{v-(v)_{B_{R}}}{R}\right|\right)\,dx \leqslant
		c\lambda_{sp}\left[ \FI_{B_{R}} [\Psi(x,|Dv|)]^{\theta}\,dx \right]^{\frac{1}{\theta}},
	\end{align}
	holds for some constant $c\equiv c(n,s(G),s(H_{a}),s(H_{b}),\omega_a(1),\omega_b(1))$	 where
	 
\begin{subequations}
	 \begin{empheq}[left={\lambda_{sp}=\empheqlbrace}]{align}
        \qquad & 1+ ([a]_{\omega_{a}}+ [b]_{\omega_{b}})\left(\lambda_{1} +  \lambda_{1}\left( \I_{B_{R}}G(|Dv|)\,dx \right)^{\frac{1}{n}}\right)  &\text{if }& v\in W^{1,\Psi}(B_{R}) \text{ with } \eqref{ma:1}_{2} \label{rmk:sp:2}.\\
        & 1+ ([a]_{\omega_{a}}+ [b]_{\omega_{b}})\left(\lambda_{2} + \lambda_{2} \norm{v}_{L^{\infty}(B_{R})}\right) &\text{if }& v\in  L^{\infty}(B_{R}) \text{ with } \eqref{ma:2}_{2}.\label{rmk:sp:3}\\
       &1+ ([a]_{\omega_{a}}+ [b]_{\omega_{b}})\left(\lambda_{3} + \lambda_{3}\left[ R^{-\gamma}\osc\limits_{B_{R}} v \right]^{\frac{1}{1-\gamma}}\right)
         &\text{if }& v\in C^{0,\gamma}(B_{R}) \text{ with } \eqref{ma:3}_{2}.\label{rmk:sp:4}
      \end{empheq}
	\end{subequations}
\end{rmk}


 \section{Basic regularity results}
 \label{sec:5}
 We start this section by stating the following Caccioppoli inequality which is a fundamental result for the further investigations. In what follows let $Q= L/\nu$ for the convenience in the future but in general it could be any number larger than one.
 \begin{lem}[Caccioppoli Inequality]
    \label{lem:cacc}
    Let $u\in W^{1,\Psi}(\O)$ be a local $Q$-minimizer of the functional $\mathcal{P}$ defined in \eqref{ifunct} with $0\leqslant a(\cdot),b(\cdot)\in L^{\infty}(\O)$. Then there exists a positive constant $c\equiv c(s(G),s(H_{a}),s(H_{b}),Q)$ such that the following Caccioppoli inequality 
    \begin{align}
        \label{cacc:1}
        \I_{B_{\rho}} \Psi(x,|D(u-k)_{\pm}|)\,dx \leqslant c\I_{B_{R}} \Psi\left(x,\frac{(u-k)_{\pm}}{R-\rho} \right)\,dx
    \end{align}
    holds, whenever $B_{\rho}\Subset B_{R}\subset \O$ are concentric balls and $k\in\R$.
 
 \end{lem}
 
 \begin{proof}
 	The proof is elementary as done for \cite[Lemma 4.6]{BO3}. The only difference lies in that we have an additional one phase. But the inequality \eqref{cacc:1} is still valid since $H_{b}\in \N$ with an index $s(H_{b})\geqslant 1$.
 \end{proof}
\begin{rmk}
	\label{rmk:cacc}
	As a direct consequence of Lemma \ref{lem:cacc}, with $u\in W^{1,\Psi}(B_{R})$ being a local $Q$-minimizer of the functional $\P$ defined in \eqref{ifunct} under the assumptions of Lemma \ref{lem:cacc}, there exists a positive constant $c\equiv c(n,s(G),s(H_{a}),s(H_{b}),Q)$ such that 
	\begin{align*}
		\FI_{B_{R/2}}\Psi(x,|Du|)\,dx \leqslant c\FI_{B_{R}}\Psi\left(x, \left|\frac{u-(u)_{B_{R}}}{R} \right| \right)\,dx
	\end{align*}
	holds, whenever $B_{R}\subset\O$ is a ball.	 
\end{rmk}

\subsection{Local boundedness}
\label{subsec:5.1}
Now we focus on local boundedness of a local $Q$-minimizer $u$ of the functional $\P$ defined in \eqref{ifunct} with obtaining precise estimates under the assumption $\eqref{ma:1}_{2}$.
\begin{thm}
	\label{thm:lb}
	Let $u\in W^{1,\Psi}(\O)$ be a local $Q$-minimizer of the functional $\P$ defined in \eqref{ifunct} under the assumption \eqref{ma:1}. Then there exists a constant $c\equiv c(\data)$ such that 
	\begin{align}
		\label{lb:1}
		\norm{\Psi_{B_{R}}^{-}\left(\left|\frac{(u-(u)_{B_{R}})_{\pm}}{R} \right| \right)}_{L^{\infty}(B_{R/2})}
		\leqslant 
		c \FI_{B_{R}}\Psi\left(x,\left|\frac{(u-(u)_{B_{R}})_{\pm}}{R} \right| \right)\,dx
	\end{align}
	and 
	\begin{align}
		\label{lb:1_1}
		\Psi_{B_{R}}^{-}\left(\left|\frac{u(x_1)-u(x_2)}{R} \right| \right)
		\leqslant 
		c \FI_{B_{R}}\Psi\left(x,\left|Du\right| \right)\,dx
		\quad\text{for a.e}\quad x_1,x_2\in B_{R/2},
	\end{align}
 whenever $B_{R}\equiv B_{R}(x_0)\subset\O$ is a ball with $R\leqslant 1$. In particular, $u\in L^{\infty}_{\loc}(\O)$.	
\end{thm}

\begin{proof}
	Let us consider the following scaling: 
	\begin{align}
		\label{lb:2}
		\begin{split}
			&\bar{u}(x):= \frac{u(x_0 + Rx)-(u)_{B_{R}}}{R},\quad
			\bar{a}(x):= a(x_0+Rx),\quad 
			\bar{b}(x):= b(x_0+Rx),
			\\&
			\bar{\Psi}(x,t):= G(t) + \bar{a}(x)H(t) + \bar{b}(x)H(t),
			\\&
			\bar{A}(k,s):= B_{s}(0)\cap\{\bar{u}>k\}
			\quad\text{and}\quad
			\bar{B}(k,s):= B_{s}(0)\cap\{\bar{u}<k\}
		\end{split}
	\end{align}
	for every $x\in B_{1}(0)$, $t\geqslant 0$, $s\in (0,1)$ and $k\in\R$. The rest of the proof falls in 3 steps.
	
\textbf{Step 1: Sobolev-Poincar\'e inequality under the scaling.}
	 Before going on further, let us consider a Sobolev-Poincar\'e type inequality under the new scaling introduced in \eqref{lb:2}. So we prove that there exists a positive exponent $\theta\equiv \theta(n,s(G),s(H_{a}),s(H_{b}))\in (0,1)$ such that 
	\begin{align}
		\label{lb:3}
		\I_{B_{1}}\bar{\Psi}(x,|f|)\,dx \leqslant c\bar{k}_{sp} \left(\I_{B_{1}}[\bar{\Psi}(x,|Df|)]^{\theta}\,dx\right)^{\frac{1}{\theta}}
	\end{align}
	for some constant $c\equiv c(n,s(G),s(H_{a}),s(H_{b}),\omega_a(1),\omega_{b}(1))$, whenever $f\in W^{1,\bar{\Psi}}_{0}(B_{1})$, where
	\begin{align}
		\label{lb:4}
		\bar{\kappa}_{sp} = 1+([a]_{\omega_{a}} + [b]_{\omega_{b}})\left(\lambda_1 + \lambda_1 R\left(\I_{B_{1}}G(|Df|) \,dx\right)^{\frac{1}{n}}\right).
	\end{align}
Essentially, the proof of the inequality \eqref{lb:3} comes from a careful revealing of the arguments used in \eqref{sp:4}-\eqref{sp:6}. So using continuity properties of $\bar{a}(\cdot)$ and $\bar{b}(\cdot)$, we see 
\begin{align}
	\label{lb:5}
	\begin{split}
	I:&= \I_{B_{1}}\bar{\Psi}(x,|f|)\,dx \leq 
	2[a]_{\omega_{a}}\omega_a(R)\I_{B_{1}}H_{a}(|f|)\,dx + 2[b]_{\omega_{b}}\omega_{b}(R)\I_{B_{1}}H_{b}(|f|)\,dx + \I_{B_{1}}\bar{\Psi}_{B_1}^{-}(|f|)\,dx
	\\&
	: 2[a]_{\omega_{a}}I_1 + 2[b]_{\omega_{b}}I_2 +I_3,
	\end{split}
\end{align}
where 
\begin{align}
	\label{lb:6}
	\bar{\Psi}_{B_{1}}^{-}(t):= G(t) + \inf\limits_{x\in B_{1}}\bar{a}(x)H_{a}(t) + \inf\limits_{x\in B_{1}}\bar{b}(x)H_{b}(t)
	\quad\text{for every}\quad
	t\geqslant 0.
\end{align}
Now we estimate the terms $I_{i}$ for $i\in \{1,2,3\}$ similarly as in the proof of Theorem \ref{thm:sp}. In turn, using the assumption $\eqref{ma:1}_{2}$ and \eqref{concave2}, we have 
\begin{align}
	\label{lb:7}
	\begin{split}
		I_1 &= \omega_{a}(R)\I_{B_{1}}\frac{H_{a}(|f|)}{G(|f|)}G(|f|)\,dx
		\\&
	\leqslant
	\lambda_{1}\omega_{a}(R)
	\I_{B_{1}} \left(1+ \left[ \omega_{a}\left(  \left[ G\left(\left|f\right|\right)\right]^{-\frac{1}{n}}  \right) \right]^{-1} \right)G\left(\left|f\right|\right)\,dx
	\\&
	\leqslant
	\lambda_{1}\omega_{a}(R)
	\I_{B_{1}} \left(1+ \left[ \frac{1}{\omega_{a}(R)} + \frac{R}{\omega_{a}(R)}\left[ G\left(\left|f\right|\right) \right]^{\frac{1}{n}} \right] \right)G\left(\left|f\right|\right)\,dx
	\\&
	\leqslant
	\lambda_{1}(1+\omega_{a}(1)) \I_{B_{1}}G\left(\left|f\right|\right)\,dx + 2\lambda_{1} R \I_{B_{1}}\left[ G\left(\left|f\right|\right) \right]^{1+\frac{1}{n}}\,dx.
	\end{split}
\end{align}
In a similar manner, we find 
\begin{align}
	\label{lb:8}
	I_{2} \leqslant \lambda_{1}(1+\omega_{b}(1)) \I_{B_{1}}G\left(\left|f\right|\right)\,dx + \lambda_{1} R \I_{B_{1}}\left[ G\left(\left|f\right|\right) \right]^{1+\frac{1}{n}}\,dx.
\end{align}
Inserting the estimates \eqref{lb:7}-\eqref{lb:8} into \eqref{lb:6}, the inequality \eqref{lb:3} follows from the similar arguments used in \eqref{sp:6}-\eqref{sp:7} and Lemma \ref{lem:os}.

\textbf{Step 2. Proof of \eqref{lb:1}.}
	 Since $u-(u)_{B_{R}}$ is a local $Q$-minimizer of the functional $\P$, we use a Caccioppoli inequality of Lemma \ref{lem:cacc} to see that 
	\begin{align}
		\label{lb:9}
		\I_{B_{t}} \bar{\Psi}(x,|D(\bar{u}-k)_{\pm}|)\,dx 
		\leqslant 
		c\I_{B_{s}} \bar{\Psi}\left(x,\frac{(\bar{u}-k)_{\pm}}{s-t} \right)\,dx
	\end{align}
	holds for some constant $c\equiv c(s(G),s(H_{a}),s(H_{b}),Q)$, whenever $0<t<s \leqslant 1$ and $k\in\R$. Let us now consider the concentric balls $B_{\rho}\Subset B_{t} \Subset B_{s} $ with $1/2 \leqslant \rho < s \leqslant 1$ and $t:= (\rho+s)/2$. Let $\eta\in C_{0}^{\infty}(B_{t})$ be a standard cut-off function such that $\chi_{B_{\rho}} \leqslant \eta \leqslant \chi_{B_{t}}$ and $|D\eta| \leqslant \frac{2}{t-\rho} = \frac{4}{s-\rho}$. Now we apply inequality \eqref{lb:3} from Step 1 above in order to have a positive exponent $\theta\equiv \theta(n,s(G),s(H_{a}),s(H_{b}))$ such that
	\begin{align}
		\label{lb:10}
		\I_{\bar{A}(k,\rho)}\bar{\Psi}(x,\bar{u}-k)\,dx
		\leqslant
		\I_{B_{1}}\bar{\Psi}(x,\eta(\bar{u}-k)_{+})\,dx
		\leqslant
		c\bar{k}_{sp}\left(\I_{B_{1}} \left[\bar{\Psi}(x,|D(\eta(\bar{u}-k)_{+})|)\right]^{\theta}\,dx \right)^{\frac{1}{\theta}}
	\end{align}
for some constant $c\equiv c(n,s(G),s(H_{a}),s(H_{b}),\omega_{a}(1),\omega_{b}(1))$, where
	\begin{align}
		\label{lb:11}
		\bar{k}_{sp}= 1+ ([a]_{\omega_{a}} + [b]_{\omega_{b}})\left(\lambda_{1} + \lambda_{1}R\left(\I_{B_{1}}G(|D(\eta(\bar{u}-k)_{+})|) \,dx\right)^{\frac{1}{n}}\right).
	\end{align}
By scaling back and using Lemma \ref{lem:nf1}, for any $k\geqslant 0$, we have 
\begin{align}
	\label{lb:11_1}
	\begin{split}
	\bar{\kappa}_{sp} &\leqslant c\left[1+ R \left(\FI_{B_{R}}G(|Du|) \,dx\right)^{\frac{1}{n}} + \frac{R}{(s-\rho)^{\frac{s(G)+1}{n}}} \left(\FI_{B_{R}}G\left(\left|\frac{u-(u)_{B_{R}}}{R}\right|\right) \,dx\right)^{\frac{1}{n}} \right]	
	\\&
	\leqslant
	\frac{c}{(s-\rho)^{s(G)+1}}\left[ 1 + \left(\I_{B_{R}}G(|Du|) \,dx\right)^{\frac{1}{n}} \right]
	\end{split}
\end{align}
with a constant $c\equiv c(n,\lambda_{1}, [a]_{\omega_{a}} + [b]_{\omega_{b}})$, where we have also used Lemma \ref{lem:os} to $\Phi\equiv G$ for $d_{0}\equiv 1$. 

Then, inserting the last estimate into \eqref{lb:10} and applying H\"older inequality together with \eqref{lb:9} yield that
\begin{align}
    \label{lb:12}
    \begin{split}
    \I_{\bar{A}(k,\rho)} \bar{\Psi}(x,\bar{u}-k)\,dx 
    &\leqslant
 	c\frac{1}{(s-\rho)^{1+s(G)}}|\bar{A}(k,t)|^{\frac{1-\theta}{\theta}}\I_{\bar{A}(k,t)}\left( \bar{\Psi}(x,|D\bar{u}|) + \bar{\Psi}
 	\left(x,\frac{\bar{u}-k}{s-\rho} \right) \right)\,dx
 	\\&
 	\leqslant
 	 c\frac{1}{(s-\rho)^{1+s(G)}}|\bar{A}(k,s)|^{\frac{1-\theta}{\theta}}\I_{\bar{A}(k,s)} \bar{\Psi}
 	\left(x,\frac{\bar{u}-k}{s-\rho} \right)\,dx
 	\end{split}
 \end{align}
 holds with some constant $c\equiv c(\data)$,
 where in the last display we have also used \eqref{growth1}. By the very definition of $\bar{A}$ in \eqref{lb:2}, we observe 
 \begin{align*}
 	|\bar{A}(k,s)| \leqslant \I_{\bar{A}(h,s)} \frac{\bar{\Psi}(x,\bar{u}-h)}{\bar{\Psi}(x,k-h)}\,dx
 	\leqslant
 	\frac{1}{\bar{\Psi}_{B_{1}}^{-}(k-h)} \I_{\bar{A}(h,s)}\bar{\Psi}(x,\bar{u}-h)\,dx
 \end{align*}
 and 
 \begin{align*}
 	\I_{\bar{A}(k,s)}\bar{\Psi}(x,\bar{u}-k)\,dx \leqslant
 	\I_{\bar{A}(h,s)}\bar{\Psi}(x,\bar{u}-h)\,dx
 \end{align*}
 for any $h<k$. Putting the last two inequalities into \eqref{lb:12} and applying Lemma \ref{lem:nf1}, we have the following inequality: 
 \begin{align}
 	\label{lb:13}
 	\I_{\tilde{A}(k,\rho)}\bar{\Psi}(x,\bar{u}-k)\,dx 
 	\leqslant 	
 	\frac{c}{[\bar{\Psi}_{B_{1}}^{-}(k-h)]^{\frac{1-\theta}{\theta}}(s-\rho)^{2(\max\{s(G),s(H_{a}),s(H_{b})\}+1)}}\left( \I_{\bar{A}(h,s)}\bar{\Psi}(x,\bar{u}-h)\,dx \right)^{\frac{1}{\theta}}.
 \end{align}
 Now we set sequences of numbers as follows: 
 \begin{align*}
 	\rho_{i} := \frac{1}{2}\left( 1+\frac{1}{2^{i}} \right),\qquad 
 	k_i := 2l_{0}\left(1-\frac{1}{2^{i+1}} \right)\,\,
 	\text{and}\,\,
 	M_{i}:= \frac{1}{\bar{\Psi}_{B_{1}}^{-}(l_0)}\I_{\bar{A}(k_i,\rho_i)}\bar{\Psi}(x,\bar{u}-k_{i})\,dx 
 \end{align*}
 for any integer $i\geqslant 0$ and some number $l_{0}>0$ to be chosen in a few lines.  Then applying \eqref{lb:13} with the choices $k \equiv k_{i+1}$, $h\equiv k_{i} $, $\rho\equiv  \rho_{i+1}$ and $s\equiv\rho_{i}$, we have, for every $i\geqslant 0$,
 \begin{align*}
 	\begin{split}
 	M_{i+1} 
 	&\leqslant
 	 \frac{c}{\left[\bar{\Psi}_{B_{1}}^{-}\left(\frac{l_0}{2^{i+1}} \right)\right]^{\frac{1-\theta}{\theta}}\left( \frac{1}{4^{i+2}}\right)^{\max\{s(G),s(H_{a}),s(H_{b})\}+1}}\left[\bar{\Psi}_{B_{1}}^{-}(l_0)\right]^{\frac{1-\theta}{\theta}}M_{i}^{\frac{1}{\theta}}
 	 \\&
 	 \leqslant
 	 c_{0} \left[ 4^{(\max\{s(G),s(H_{a}),s(H_{b})\} +1)\frac{1}{\theta}}\right]^{i} M_{i}^{1+\frac{1-\theta}{\theta}}
 	\end{split}
 \end{align*}
 with $c_{0}\equiv c_{0}(\data)$, where in the last inequality of the last display we have used again Lemma \ref{lem:nf1}. Now it's turn to apply a standard iteration of Lemma \ref{lem:t1}, which means that if 
 \begin{align*}
 	\frac{1}{\bar{\Psi}_{B_{1}}^{-}(l_0)}\I_{\bar{A}(l_0,1)} \bar{\Psi}(x,\bar{u}-l_0)\,dx=M_{0} \leqslant c_{0}^{-\frac{\theta}{1-\theta}} 4^{-(\max\{s(G),s(H_{a}),s(H_{b})\} +1)\frac{\theta}{(1-\theta)^2}},
 \end{align*}
 then we obtain 
 \begin{align*}
 	\norm{\bar{u}_{+}}_{L^{\infty}(B_{1/2})} \leqslant 2l_0.
 \end{align*}
 Consequently, choosing $l_0>0$ in such a way that 
 \begin{align*}
 	\bar{\Psi}_{B_{1}}^{-}(l_0) =  c_{0}^{\frac{\theta}{1-\theta}} 4^{(\max\{s(G),s(H_{a}),s(H_{b})\} +1)\frac{\theta}{(1-\theta)^2}} \I_{B_{1}} \bar{\Psi}(x,\bar{u}_{+})\,dx,
\end{align*}  
we have 
\begin{align*}
	\norm{\bar{\Psi}_{B_{1}}^{-}\left( \bar{u}_{+} \right)}_{L^{\infty}(B_{1/2})} \leqslant c \I_{B_{1}} \bar{\Psi}(x,\bar{u}_{+})\,dx,
\end{align*}
 which implies that
 \begin{align*}
 	\norm{\Psi_{B_{R}}^{-}\left( \frac{(u-(u)_{B_{R}})_{+}}{R} \right)}_{L^{\infty}(B_{R/2})} \leqslant c \FI_{B_{R}} \Psi\left(x, \frac{(u-(u)_{B_{R}})_{+}}{R}\right)\,dx
 \end{align*}
 holds with $c\equiv c(\data)$.
 Repeating the same argument for $-u$, which is also a local $Q$-minimizer of the functional $\mathcal{P}$ defined in \eqref{functional}, the last inequality holds with $(u-(u)_{B_{R}})_{+}$ replaced by $(u-(u)_{B_{R}})_{-}$.
 
 \textbf{Step 3. Proof of \eqref{lb:1_1}.} Using \eqref{lb:1} and \eqref{growth1}, for a.e $x_1,x_2\in B_{R/2}$, we have 
\begin{align*}
     \begin{split}
         \Psi_{B_{R}}^{-}\left(\left| \frac{u(x_1)-u(x_2)}{R}\right| \right) 
         &
         \leqslant
     c \Psi_{B_{R}}^{-}\left( \left| \frac{u(x_1)-(u)_{B_{R}}}{R}\right| \right)
     + c \Psi_{B_{R}}^{-}\left( \left| \frac{u(x_2)-(u)_{B_{R}}}{R}\right| \right)
     \\&
     \leqslant
     c \FI_{B_{R}} \Psi\left(x, \left|\frac{u-(u)_{B_{R}}}{R}\right|\right)\,dx 
     \leqslant
     c \FI_{B_{R}} \Psi\left(x, \left|Du\right|\right)\,dx
     \end{split}
 \end{align*}
 for some constant $c\equiv c(\data)$, where in the last inequality of the above display we have used a Sobolev-Poincar\'e type inequality of Theorem \ref{thm:sp}. Clearly, the last display implies $u\in L^{\infty}_{\loc}(\O)$. The proof is complete.
\end{proof}


\subsection{Almost standard Caccioppoli inequality}
\label{subsec:5.2}

Now we present the primary results, the so-called almost standard Caccioppoli type inequality, for proving H\"older continuity of a local $Q$-minimizer of the functional $\P$.

\begin{lem}[Almost standard Caccioppoli inequality]
	\label{lem:cacct}
	Let $u\in W^{1,\Psi}(\O)$ be a local $Q$-minimizer of the functional $\P$ defined in \eqref{ifunct} under one of the assumptions \eqref{ma:1}, \eqref{ma:2} and \eqref{ma:3}. Let $B_{2R}\equiv B_{2R}(x_0)\subset\O$ be a ball with $R\leqslant 1$.  Then there exists a constant $c\equiv c(\data)$ such that
	\begin{align}
		\label{cacct:1}
		\begin{split}
		\I_{B_{R_1}}\Psi_{B_{R}}^{-}\left( |D(u-k)_{\pm}| \right)\,dx 
		&\leqslant
		\I_{B_{R_1}}\Psi\left(x,|D(u-k)_{\pm}| \right)\,dx
		\\&
		\leqslant
		c\left(\frac{R}{R_2-R_1} \right)^{s(\Psi)+1} \I_{B_{R_{2}}}\Psi_{B_{R}}^{-}\left(\frac{(u-k)_{\pm}}{R} \right)\,dx
		\end{split}
	\end{align}
	holds, whenever $B_{R_{1}} \Subset B_{R_{2}} \subset B_{R}(x_0)$ are concentric balls and $k\in \R$.
\end{lem}

\begin{proof}
	First we prove the inequality \eqref{cacct:1} for the values of $k\in \R$ with $\inf\limits_{B_{R}} u \leqslant k \leqslant \sup\limits_{B_{R}} u$, depending on which one of the assumptions \eqref{ma:1}-\eqref{ma:3} is in force. Firstly by the very definition of $\Psi_{B_{R}}^{-}$ in \eqref{ispsi} and Lemma \ref{lem:cacc}, we see 
	\begin{align}
		\label{cacct:2}
		\begin{split}
		I&:=\I_{B_{R_1}}\Psi_{B_{R}}^{-}\left( |D(u-k)_{\pm}| \right)\,dx
		\leqslant
		\I_{B_{R_1}}\Psi\left(x, |D(u-k)_{\pm}| \right)\,dx
		\leqslant 
		c_{*}\I_{B_{R_2}} \Psi\left(x,\frac{(u-k)_{\pm}}{R_2-R_1} \right)\,dx
		\\&
		\leqslant
		c_{*} \omega_{a}(R)\I_{B_{R_2}} H_{a}\left(\frac{(u-k)_{\pm}}{R_2-R_1} \right)\,dx
		+
		c_{*} \omega_{b}(R)\I_{B_{R_2}} H_{b}\left(\frac{(u-k)_{\pm}}{R_2-R_1} \right)\,dx
		\\&
		\quad
		+
		c_{*}\I_{B_{R_{2}}}\Psi_{B_{R}}^{-}\left(\frac{(u-k)_{\pm}}{R_2-R_1} \right)\,dx
		=: c_{*}\left( I_1 + I_2 + I_3 \right)
		\end{split}
	\end{align}
	for some constant $c_{*}\equiv c_{*}(n,s(G),s(H_{a}),s(H_{b}),Q, [a]_{\omega_{a}}, [b]_{\omega_{b}})$.
	Now we shall estimate each term $I_{i}$ for $i\in\{1,2,3\}$ in the above display. Then using Lemma \ref{lem:nf1}, the assumption $\eqref{ma:1}_{2}$, \eqref{concave2} and \eqref{lb:1_1} of Lemma \ref{thm:lb}, we see
	\begin{align}
		\label{cacct:3}
		\begin{split}
		I_1 &= \omega_{a}(R) \I_{B_{R_2}} \frac{H_{a}\left(\frac{(u-k)_{\pm}}{R_2-R_1} \right)}{G\left(\frac{(u-k)_{\pm}}{R_2-R_1}\right)} G\left(\frac{(u-k)_{\pm}}{R_2-R_1}\right)\,dx
		\\&
		\leqslant
		\omega_{a}(R)\left( \frac{R}{R_2-R_1} \right)^{s(H_{a})+1}
		\I_{B_{R_2}} \frac{\left(H_{a}\circ G^{-1}\right)\left(G\left(\frac{(u-k)_{\pm}}{R} \right)\right)}{G\left(\frac{(u-k)_{\pm}}{R}\right)} G\left(\frac{(u-k)_{\pm}}{R_2-R_1}\right)\,dx
		\\&
		\leqslant
		\lambda_{1}\omega_{a}(R)\left( \frac{R}{R_2-R_1} \right)^{s(H_{a})+1}
	\I_{B_{R_{2}}} \left(1+ \left[ \omega_{a}\left(  \left[ G\left(\frac{(u-k)_{\pm}}{R}\right) \right]^{-\frac{1}{n}} \right) \right]^{-1} \right)G\left(\frac{(u-k)_{\pm}}{R_2-R_1}\right)\,dx
	\\&
	\leqslant
	c\omega_{a}(R)\left( \frac{R}{R_2-R_1} \right)^{s(H_{a})+1}
	\I_{B_{R_{2}}} \left[1+  \frac{1}{\omega_{a}(R)} + \frac{R}{\omega_{a}(R)}\left( G\left(\frac{\osc\limits_{B_{R}}u}{R}\right) \right)^{\frac{1}{n}} \right]G\left(\frac{(u-k)_{\pm}}{R_2-R_1}\right)\,dx
	\\&
	\leqslant
	c\left[1 + \left(\I_{B_{2R}}\Psi(x,|Du|)\,dx\right)^{\frac{1}{n}} \right]\left( \frac{R}{R_2-R_1} \right)^{s(H_{a})+1}
	\I_{B_{R_{2}}} G\left(\frac{(u-k)_{\pm}}{R_2-R_1}\right)\,dx
	\\&
	\leqslant
	c \left( \frac{R}{R_2-R_1} \right)^{s(\Psi)+1}
	\I_{B_{R_{2}}} G\left(\frac{(u-k)_{\pm}}{R}\right)\,dx
		\end{split}
	\end{align}
	for some constant $c\equiv c(\data)$. In a totally similar way, it can be shown that 
	\begin{align}
		\label{cacct:4}
		I_{2} \leqslant c \left( \frac{R}{R_2-R_1} \right)^{s(\Psi)+1}
	\I_{B_{R_{2}}} G\left(\frac{(u-k)_{\pm}}{R}\right)\,dx
	\end{align}
	with a constant $c\equiv c(\data)$. Clearly, recalling Remark \ref{rmk:nf2} and using Lemma \ref{lem:nf1}, we have 
	\begin{align}
		\label{cacct:5}
		I_3  \leqslant  \left( \frac{R}{R_2-R_1} \right)^{s(\Psi)+1}
	\I_{B_{R_{2}}} \Psi_{B_{R}}^{-}\left(\frac{(u-k)_{\pm}}{R}\right)\,dx 
	\end{align}
 Inserting the estimates obtained in \eqref{cacct:3}-\eqref{cacct:5} into \eqref{cacct:2} and recalling the very definition of $\Psi_{B_{R}}^{-}$ in \eqref{ispsi}, we arrive at \eqref{cacct:1} under the assumption \eqref{ma:1}. The second part of the proof is to show \eqref{cacct:1} under the assumption \eqref{ma:2}. For this, we again estimate the terms $I_{i}$ with $i\in \{1,2,3\}$ in \eqref{cacct:2}. Applying Lemma \ref{lem:nf1}, the assumption \eqref{ma:2} and \eqref{concave2}, we see 
 \begin{align}
 	\label{cacct:6}
 	\begin{split}
 		I_1 &= \omega_{a}(R) \I_{B_{R_2}} \left(\frac{H_{a}}{G}\right)\left(\frac{(u-k)_{\pm}}{R_2-R_1} \right) G\left(\frac{(u-k)_{\pm}}{R_2-R_1}\right)\,dx
 		\\&
 		\leqslant
 		\omega_{a}(R)\left(\frac{R}{R_2-R_1} \right)^{s(H_{a})+1} \I_{B_{R_2}} \left(\frac{H_{a}}{G}\right)\left(\frac{(u-k)_{\pm}}{R} \right) G\left(\frac{(u-k)_{\pm}}{R_2-R_1}\right)\,dx
 		\\&
 		\leqslant
 		2 \lambda_{2}\omega_{a}(R)\left(\frac{R}{R_2-R_1} \right)^{s(H_{a})+1}
	\I_{B_{R_{2}}} \left(1+ \left[ \omega_{a}\left( \frac{R}{(u-k)_{\pm}} \right)  \right]^{-1} \right)G\left(\frac{(u-k)_{\pm}}{R_2-R_1}\right)\,dx
	\\&
	\leqslant
	c\lambda_{2}\omega_{a}(R)\left(\frac{R}{R_2-R_1} \right)^{s(H_{a})+1}
	\I_{B_{R_{2}}} \left(1+ \left[ \frac{1}{\omega_{a}(R)} + \frac{\norm{u}_{L^{\infty}(B_{R})}}{\omega_{a}(R)} \right] \right)G\left(\frac{(u-k)_{\pm}}{R_2-R_1}\right)\,dx
	\\&
	\leqslant
	c \left(\frac{R}{R_2-R_1} \right)^{s(\Psi)+1}
	\I_{B_{R_2}}G\left(\frac{(u-k)_{\pm}}{R}\right)\,dx
 	\end{split}
\end{align}
for some constant $c\equiv c(\data)$. Arguing similarly, we have 
\begin{align}
	\label{cacct:7}
	I_{2} \leqslant
	c \left(\frac{R}{R_2-R_1} \right)^{s(\Psi)+1}
	\I_{B_{R_{2}}}G\left(\frac{(u-k)_{\pm}}{R}\right)\,dx
\end{align}
 with a constant $c\equiv c(\data)$. Plugging the estimates \eqref{cacct:5}-\eqref{cacct:7} into \eqref{cacct:2}, we conclude with \eqref{cacct:1} under the assumption \eqref{ma:2}. Finally, the remaining part of the proof is to obtain the inequality \eqref{cacct:1} under the assumption \eqref{ma:3}. In fact, we continue to estimate the terms $I_{i}$ with $i\in \{1,2,3\}$ in \eqref{cacct:2}. Therefore, using the assumption \eqref{ma:3} and \eqref{concave2}, we find 
 
  \begin{align}
 	\label{cacct:8}
 	\begin{split}
 		I_1 &= \omega_{a}(R) \I_{B_{R_2}} \left(\frac{H_{a}}{G}\right)\left(\frac{(u-k)_{\pm}}{R_2-R_1} \right) G\left(\frac{(u-k)_{\pm}}{R_2-R_1}\right)\,dx
 		\\&
 		\leqslant
 		\omega_{a}(R)\left(\frac{R}{R_2-R_1} \right)^{s(H_{a})+1} \I_{B_{R_2}} \left(\frac{H_{a}}{G}\right)\left(\frac{(u-k)_{\pm}}{R} \right) G\left(\frac{(u-k)_{\pm}}{R_2-R_1}\right)\,dx
 		\\&
 		\leqslant
 		2 \lambda_{3}\omega_{a}(R)\left(\frac{R}{R_2-R_1} \right)^{s(H_{a})+1}
	\I_{B_{R_{2}}} \left(1+ \left[ \omega_{a}\left( \left[\frac{R}{(u-k)_{\pm}}\right]^{\frac{1}{1-\gamma}} \right)  \right]^{-1} \right)G\left(\frac{(u-k)_{\pm}}{R_2-R_1}\right)\,dx
	\\&
	\leqslant
	c\lambda_{3}\omega_{a}(R)\left(\frac{R}{R_2-R_1} \right)^{s(H_{a})+1}
	\I_{B_{R_{2}}} \left(1+ \frac{1}{\omega_{a}(R)} \right)G\left(\frac{(u-k)_{\pm}}{R_2-R_1}\right)\,dx
	\\&
	\leqslant
	c \left(\frac{R}{R_2-R_1} \right)^{s(\Psi)+1}
	\I_{B_{R_2}}G\left(\frac{(u-k)_{\pm}}{R}\right)\,dx
 	\end{split}
\end{align}
for some constant $c\equiv c(\data)$, where we have used Lemma \ref{lem:nf1} several times. Using the same argument as above, we have 
\begin{align}
	\label{cacct:9}
	I_{2} \leqslant
	c \left(\frac{R}{R_2-R_1} \right)^{s(\Psi)+1}
	\I_{B_{R_{2}}}G\left(\frac{(u-k)_{\pm}}{R}\right)\,dx
\end{align}
 with a constant $c\equiv c(\data)$. Inserting the estimates \eqref{cacct:5}, \eqref{cacct:8}-\eqref{cacct:9} into \eqref{cacct:2}, we arrive at \eqref{cacct:1} under the assumption \eqref{ma:3}. So we have proved the inequality \eqref{cacct:1} for the values of $k\in\R$ such that $\inf\limits_{B_{R}} u \leqslant k \leqslant \sup\limits_{B_{R}} u$. Now we consider the remaining cases. Suppose $k< \inf\limits_{B_{R}} u$. In this case, using \eqref{cacct:1} with $k\equiv \inf\limits_{B_{R}}u$, we have 
 
\begin{align}
		\label{cacct:10}
		\begin{split}
		\I_{B_{R_1}}\Psi_{B_{R}}^{-}\left( |D(u-k)_{+}| \right)\,dx 
		&= 
		\I_{B_{R_1}}\Psi_{B_{R}}^{-}\left( \left|D(u-\inf\limits_{B_{R}}u)_{+}\right| \right)\,dx 
		\leqslant
		\I_{B_{R_1}}\Psi\left(x,\left|D(u-\inf\limits_{B_{R}}u)_{+}\right| \right)\,dx
		\\&
		\leqslant
		c\left(\frac{R}{R_2-R_1} \right)^{s(\Psi)+1} \I_{B_{R_{2}}}\Psi_{B_{R}}^{-}\left(\frac{\left(u-\inf\limits_{B_{R}}u\right)_{+}}{R} \right)\,dx
		\\&
		\leqslant
		c\left(\frac{R}{R_2-R_1} \right)^{s(\Psi)+1} \I_{B_{R_{2}}}\Psi_{B_{R}}^{-}\left(\frac{(u-k)_{+}}{R} \right)\,dx
		\end{split}
	\end{align}
for some constant $c\equiv c(\data)$. Similarly, it can seen that \eqref{cacct:10} is valid for the values of $k> \sup\limits_{B_{R}}u$. Since $-u$ is also the local $Q$-minimizer of the functional $\P$ in \eqref{ifunct}, the inequality \eqref{cacct:1} is valued for all $k\in \R$. The proof is complete.
 \end{proof}

From now on also in the rest of paper, for a fixed ball $B_{R}\subset \O$, we say that

\begin{subequations}
	 \begin{empheq}[left={\empheqlbrace}]{align}
        \quad & G-\text{phase occurs in } B_{R}  &\text{if }& a^{-}(B_{R}) \leqslant 4[a]_{\omega_{a}}\omega_{a}(R)\text{ and } b^{-}(B_{R}) \leqslant 4[b]_{\omega_{b}}\omega_{b}(R).\label{G-phase}\\
        & (G,H_{a})-\text{phase occurs in } B_{R}  &\text{if }& a^{-}(B_{R}) > 4[a]_{\omega_{a}}\omega_{a}(R)\text{ and } b^{-}(B_{R}) \leqslant 4[b]_{\omega_{b}}\omega_{b}(R).\label{G-a phase}\\
       & (G,H_{b})-\text{phase occurs in } B_{R}  &\text{if }& a^{-}(B_{R}) \leqslant 4[a]_{\omega_{a}}\omega_{a}(R)\text{ and } b^{-}(B_{R}) > 4[b]_{\omega_{b}}\omega_{b}(R).\label{G-b phase} \\
       & (G,H_{a},H_{b})-\text{phase occurs in } B_{R}  &\text{if }& a^{-}(B_{R}) > 4[a]_{\omega_{a}}\omega_{a}(R)\text{ and } b^{-}(B_{R}) > 4[b]_{\omega_{b}}\omega_{b}(R).\label{G-ab phase}
      \end{empheq}
	\end{subequations}
	
Then we have the following lemma which will be applied later, see Section \ref{sec:8}.

 \begin{lem}
 	\label{lem:1cacct}
 	Let $u\in W^{1,\Psi}(\O)$ be a local $Q$-minimizer of the functional $\P$ defined in \eqref{ifunct} under one of the assumptions \eqref{ma:1}, \eqref{ma:2} and \eqref{ma:3}. Let $B_{2R}\equiv B_{2R}(x_0)\subset\O$ be a ball with $R\leqslant 1$.  Then there exists a constant $c\equiv c(\data)$ such that
	\begin{align}
		\label{1cacct:1}
		\begin{split}
		\I_{B_{R_1}}\Psi_{B_{R}}^{-}\left( |D(u-k)_{\pm}| \right)\,dx 
		&\leqslant
		\I_{B_{R_1}}\Psi\left(x,|D(u-k)_{\pm}| \right)\,dx
		\\&
		\leqslant
		c\left(\frac{R}{R_2-R_1} \right)^{s(\Psi)+1} \I_{B_{R_{2}}}\Phi\left(\frac{(u-k)_{\pm}}{R} \right)\,dx
		\end{split}
	\end{align}
	holds, whenever $B_{R_{1}} \Subset B_{R_{2}} \subset B_{R}(x_0)$ are concentric balls and $k\in \R$, where 
	\begin{subequations}
	 \begin{empheq}[left={\Phi(t)=\empheqlbrace}]{align}
        \qquad & G(t)  &\text{if }& \eqref{G-phase} \text{ is satisfied in } B_{R},\label{1cacct:2}\\
        & G(t) + a^{-}(B_{R})H_{a}(t)  &\text{if }& \eqref{G-a phase} \text{ is satisfied in } B_{R},\label{1cacct:3}\\
       & G(t) + b^{-}(B_{R})H_{b}(t)  &\text{if }& \eqref{G-b phase} \text{ is satisfied in } B_{R},\label{1cacct:4} \\
       & \Psi_{B_{R}}^{-}(t) &\text{if }& \eqref{G-ab phase} \text{ is satisfied in } B_{R},\label{1cacct:5}
      \end{empheq}
	\end{subequations}
	for every $t\geqslant 0$.
 \end{lem}
\begin{proof}
	First we observe that if $a^{-}(B_{R}) > 4[a]_{\omega_{a}}\omega_{a}(R)$, then using the continuity of the function $a(\cdot)$, we have 
	\begin{align}
		\label{1cacct:6}
		a^{-}(B_{R}) \leqslant a(x) = a(x)- a^{-}(B_{R}) + a^{-}(B_{R}) \leqslant
		2[a]_{\omega_{a}}\omega_{a}(R) + a^{-}(B_{R})
		\leqslant 2 a^{-}(B_{R})
	\end{align}
	for every $x\in B_{R}$. On the other hand, if $a^{-}(B_{R}) \leqslant 4[a]_{\omega_{a}}\omega_{a}(R)$, then using again the continuity of $a(\cdot)$, we see 
	\begin{align*}
		a(x) = a(x)- a^{-}(B_{R}) + a^{-}(B_{R}) \leqslant 6[a]_{\omega_{a}}\omega_{a}(R)
	\end{align*}	  
	for every $x\in B_{R}$. Clearly, analogous estimates to the last two displays are valid for the function $b(\cdot)$ in $B_{R}$. After those observations, we argue similarly as in the proof of Lemma \ref{lem:cacct} depending on which case of \eqref{G-a phase}-\eqref{G-ab phase} occurs in the ball $B_{R}$. 
\end{proof}
 
\subsection{H\"older continuity}
\label{subsec:5.3}

 In this subsection we prove  some local boundedness and H\"older continuity  assertions of a local $Q$-minimizer of the functional $\P$ in \eqref{ifunct} with various constants having the precise dependencies.
 
 \begin{thm}
 	\label{thm:hc}
 	Let $u\in W^{1,\Psi}(\O)$ be a local $Q$-minimizer of the functional $\P$ defined in \eqref{ifunct} under the coefficient functions $a(\cdot)\in C^{\omega_a}(\O)$ and $b(\cdot)\in C^{\omega_{b}}(\O)$ for $\omega_{a}, \omega_{b}$ being non-negative concave functions vanishing at the origin.
 	\begin{enumerate}
 		\item[1.] If the assumption $\eqref{ma:1}$ is satisfied, then for every open subset $\O_{0}\Subset\O$, there exists a H\"older continuity exponent
 		$\gamma\equiv \gamma(\data(\O_{0}))\in (0,1)$ such that
 		\begin{align}
 		\label{hc:1}
 		\norm{u}_{L^{\infty}(\O_{0})} + [u]_{0,\gamma;\O_{0}} \leqslant c(\data(\O_{0}))
 	\end{align}
 	and the oscillation estimate
 	\begin{align}
 		\label{hc:2}
 		\osc\limits_{B_{\rho}} u \leqslant c\left(\frac{\rho}{R} \right)^{\gamma}\osc\limits_{B_{R}} u
 	\end{align}
 	holds for some $c\equiv c(\data(\O_0))$ and all concentric balls
 	$B_{\rho}\Subset B_{R}\Subset \O_{0}\Subset \O$ with $R\leqslant 1$.
 		\item[2.] If the assumption $\eqref{ma:2}$ is satisfied, then there exists a H\"older continuity exponent
 		$\gamma\equiv \gamma(\data)\in (0,1)$ such that  
 		\begin{align}
 		\label{hc:3}
 		 [u]_{0,\gamma;\O_{0}} \leqslant c(\data(\O_0))
 	\end{align}
 	and the oscillation estimate
 	\begin{align}
 		\label{hc:4}
 		\osc\limits_{B_{\rho}} u \leqslant c\left(\frac{\rho}{R} \right)^{\gamma}\osc\limits_{B_{R}} u
 	\end{align}
 	holds for some $c\equiv c(\data)$ and all concentric balls
 	$B_{\rho}\Subset B_{R}\subset \O$ with $R\leqslant 1$.
 	\end{enumerate}
 \end{thm}
 
 \begin{proof}
 Basically, we shall use De Giorgi's methods to prove the local H\"older continuity of $u$ based on arguments employed in \cite{BO3,CM1}. For the convenience of the reader, we give a detailed proof. Note that, for any given ball $B_{R}\Subset \O$, either 
 \begin{align}
 	\label{hc:5}
 	\left|\left\{x\in B_{R/2} : u(x) > \sup\limits_{B_{R}}u-\frac{1}{2}\osc_{B_{R}}u  \right\} \right|
 		\leqslant
 		\frac{1}{2} |B_{R/2}|
 \end{align}
 or 
 \begin{align}
 	\label{hc:6}
 	\left|\left\{x\in B_{R/2} : (-u(x)) > \sup\limits_{B_{R}}(-u)-\frac{1}{2}\osc_{B_{R}}u  \right\} \right|\leqslant \frac{1}{2} |B_{R/2}|
 \end{align}
 holds true. It is enough to deal with only the case of \eqref{hc:5} is valid since $-u$ is a local $Q$-minimizer of the functional $\mathcal{P}$. The proof falls in three steps. In what follows, let $B_{2R}\equiv B_{2R}(x_0)\subset\O_0\Subset\O$ be a fixed ball such that $R\leqslant 1$. Let us also denote by 
 \begin{align}
 	\label{hc:6_1}
 	A(k,\rho):= \{x\in B_{\rho} : u(x)> k\}
 	\quad\text{and}\quad 
 	B(k,\rho):= \{x\in B_{\rho} : u(x)< k\}
 \end{align}
 for every concentric ball $B_{\rho}\subset B_{2R}$ and $k\in\R$.
 
 \textbf{Step 1. } We suppose that \eqref{hc:5} is satisfied. Then in this step we prove that, for any $\varepsilon\in (0,1)$, there exists a natural number $m\equiv m(\data(\O_0),\varepsilon)\geqslant 3$ if \eqref{ma:1} is assumed, and $m\equiv m(\data,\varepsilon)\geqslant 3$ if \eqref{ma:2} is assumed,  such that 
 	\begin{align}
 		\label{hc:7}
 		\left|\left\{x\in B_{R/2} : u(x) > \sup\limits_{B_{R}}u-\frac{1}{2^{m}}\osc_{B_{R}}u  \right\} \right|\leqslant
 		\varepsilon |B_{R/2}|.
 	\end{align}
Let $ m\geqslant 3$ be a natural number to be determined in a few lines. For every $i\in \{1,2,\ldots,m\}$, we set
\begin{align*}
	k_i:= \sup\limits_{B_{R}} u - \frac{1}{2^{i}}\osc\limits_{B_{R}} u,\quad 
	\mathcal{D}_{i}:= A(k_i,R/2)\setminus A(k_{i+1},R/2)
\end{align*}
and 
\begin{align*}
	 w_i(x)
	 := \left\{\begin{array}{lll}
        k_{i+1}-k_{i} & \text{if }& u(x)> k_{i+1},\\
        u(x)-k_{i}  & \text{if }& k_{i} < u(x) \leqslant k_{i+1},\\
        0  & \text{if }& u(x) \leqslant k_i.
        \end{array}\right.
\end{align*}
Clearly $w_i\in W^{1,\Psi}(B_{R/2})$ with $w_i\equiv 0$ in $B_{R/2}\setminus A(k_1, R/2)$ for all $i\in \{1,\ldots, m\}$, and also $|B_{R/2}\setminus A(k_1, R/2)|\geqslant 1/2|B_{R/2}|$. Then applying H\"older's inequality, Sobolev's inequality and Lemma \ref{lem:nf3}, for every $\tau\in (0,1)$, we have 
\begin{align}
	\label{hc:8}
	\begin{split}
		\left| A(k_{i+1}, R/2) \right|\Psi_{B_{R}}^{-}\left(\frac{k_{i+1}-k_{i}}{R} \right)
		&\leqslant
		c\I_{A(k_i,R/2)}\Psi_{B_{R}}^{-}\left(\frac{w_i}{R} \right)\,dx
		\\&
		\leqslant
		\left|A(k_i,R/2) \right|^{\frac{1}{n}} \left( \I_{A(k_i,R/2)}\left[\Psi_{B_{R}}^{-}\left(\frac{w_i}{R} \right)\right]^{\frac{n}{n-1}}\,dx \right)^{\frac{n-1}{n}}
		\\&
		\leqslant
		c R \left( \I_{A(k_i,R/2)}\left[\Psi_{B_{R}}^{-}\left(\frac{w_i}{R} \right)\right]^{\frac{n}{n-1}}\,dx \right)^{\frac{n-1}{n}}
		\\&
		\leqslant
		c\I_{\mathcal{D}_{i}}\partial_{t}\Psi_{B_{R}}^{-}\left(\frac{u-k_{i}}{R} \right)|Du|\,dx
		\\&
		\leqslant
		\tau\I_{\mathcal{D}_{i}}\Psi_{B_{R}}^{-}(|Du|)\,dx + \frac{c}{\tau^{s(\Psi)}}\I_{\mathcal{D}_{i}}\Psi_{B_{R}}^{-}\left(\frac{u-k_i}{R} \right)\,dx.
	\end{split}
\end{align}
Now we use a Caccioppoli type inequality of Lemma \ref{lem:cacct} in order to have 
\begin{align*}
	\begin{split}
	\I_{\mathcal{D}_{i}}\Psi_{B_{R}}^{-}(|Du|)\,dx 
	&\leqslant
	c\I_{A(k_i,R)}\Psi_{B_{R}}^{-}\left(\left|\frac{u-k_i}{R}\right| \right)\,dx
	\leqslant
	c\I_{A(k_i,R)}\Psi_{B_{R}}^{-}\left(\frac{\osc\limits_{B_{R}}u}{2^{i}R} \right)\,dx
	\\&
	\leqslant
	c\Psi_{B_{R}}^{-}\left(\frac{k_{i+1}-k_{i}}{R} \right)|A(k_i,R)|
	\leqslant
	c \Psi_{B_{R}}^{-}\left(\frac{k_{i+1}-k_{i}}{R} \right) R^{n}.
	\end{split}
\end{align*}

One can see that 
\begin{align*}
	\I_{\mathcal{D}_{i}}\Psi_{B_{R}}^{-}\left(\frac{u-k_i}{R} \right)\,dx
	\leqslant
	\I_{\mathcal{D}_{i}}\Psi_{B_{R}}^{-}\left(\frac{k_{i+1}-k_i}{R} \right)\,dx
	=
	\Psi_{B_{R}}^{-}\left(\frac{k_{i+1}-k_i}{R}\right) |\mathcal{D}_{i}|.
\end{align*}
Using the estimates coming from the last two displays in \eqref{hc:8}, for every $i\in \{1,\ldots,m-1\}$, we see
\begin{align*}
	A(k_{m-1},R/2) \leqslant A(k_{i+1},R/2) \leqslant
	c\tau R^{n} + \frac{c}{\tau^{s(\Psi)}}|\mathcal{D}_{i}|.
\end{align*}
Summing for $i\in \{1,\ldots,m-1\}$, it yields that 
\begin{align*}
	|A(k_{m-1},R/2)| \leqslant \left(c\tau + \frac{c}{(m-1)\tau^{s(\Psi)}}\right)R^{n}.
\end{align*}
Now taking small enough $\tau\equiv \tau(\data(\O_0),\varepsilon)$ and large enough $m\equiv m(\data(\O_0),\varepsilon)$, we arrive at \eqref{hc:7} when \eqref{ma:1} is assumed. But in the case that \eqref{ma:2} is assumed, we choose small enough $\tau\equiv \tau(\data,\varepsilon)$ and large enough $m\equiv m(\data,\varepsilon)$ to conclude \eqref{hc:7}.

\textbf{Step 2.}
 In this step, we prove that there exists a small positive $\varepsilon_0\equiv \varepsilon_0(\data(\O_0))\in (0,1/2^{n+1})$ such that
 if 
 \begin{align}
 	\label{hc:11}
 	0< \nu_0 < \frac{1}{2}\osc\limits_{B_{R}}u
 	\quad\text{and}\quad
 	\left| \{ x\in B_{R/2} : u(x) > \sup\limits_{B_{R}}u-\nu_0 \}\right| \leqslant \varepsilon_0 |B_{R/2}|,
\end{align} then we have 
	\begin{align}
		\label{hc:12}
		\sup\limits_{B_{R/4}} u \leqslant \sup\limits_{B_{R}} u - \nu_0/2.
	\end{align}
Now we set the sequences by 
\begin{align*}
	\rho_{i}:= \frac{R}{4}\left( 1+\frac{1}{2^{i}} \right)
	\quad\text{and}\quad
	k_i:= \sup\limits_{B_{R}} u - \left(\frac{1}{2} + \frac{1}{2^{i+1}} \right)\nu_0
	\quad\text{for every } i=0,1,2,\ldots,
\end{align*}
and we define
\begin{align*}
\mathcal{D}_{i+1}:= A(k_i,\rho_{i+1})\setminus A(k_{i+1},\rho_{i+1})
\quad\text{and}\quad
Y_{i}:= \frac{|A(k_i,\rho_i)|}{|B_{R/2}|}.
\end{align*}
Applying Lemma \ref{lem:cacct} together with \eqref{hc:11}, we discover 
\begin{align*}
	\begin{split}
		\I_{A(k_i,\rho_{i+1})}\Psi_{B_{R}}^{-}(|Du|)\,dx
		&\leqslant
		c2^{(i+3)(s(\Psi)+1)}\I_{A(k_i,\rho_i)}\Psi_{B_{R}}^{-}\left(\frac{(u-k_i)_{+}}{R} \right)\,dx
		\\&
		\leqslant
		c2^{i(s(\Psi)+1)}\Psi_{B_{R}}^{-}\left(\frac{\nu_0}{R} \right)|A(k_i,\rho_i)|,
	\end{split}
\end{align*}
where we have also used the very definition of $k_i$ and that $(u-k_i)_{+} \leqslant \nu_0 \leqslant \norm{u}_{L^{\infty}(B_{R})}$. The last display and the convexity of $\Psi_{B_{R}}^{-}$ imply that 
\begin{align*}
	\begin{split}
	\Psi_{B_{R}}^{-}\left(\FI_{\mathcal{D}_{i+1}}|Du|\,dx \right)
	&\leqslant
	\FI_{\mathcal{D}_{i+1}} \Psi_{B_{R}}^{-}(|Du|)\,dx
	\leqslant
	c2^{i(s(\Psi)+1)}\frac{|A(k_i,\rho_i)|}{|D_{i+1}|}\Psi_{B_{R}}^{-}\left(\frac{\nu_0}{R} \right)
	\\&
	\leqslant
	\Psi_{B_{R}}^{-}\left( c2^{i(s(\Psi)+1)}\frac{|A(k_i,\rho_i)|}{|D_{i+1}|} \frac{\nu_0}{R} \right).
	\end{split}
\end{align*}
Therefore, we have 
\begin{align*}
	\FI_{\mathcal{D}_{i+1}}|Du|\,dx
	\leqslant
	c2^{i(s(\Psi)+1)}\frac{|A(k_i,\rho_i)|}{|D_{i+1}|} \frac{\nu_0}{R}.
\end{align*}

On the other hand, applying Lemma \ref{lem:t2} together with $\varepsilon_{0}\in (0,1/2^{n+1})$, we discover 
\begin{align*}
	\begin{split}
		\I_{\mathcal{D}_{i+1}}|Du|\,dx 
		&\geqslant 
		c(k_{i+1}-k_{i})|A(k_{i+1},\rho_{i+1})|^{1-\frac{1}{n}}|B_{\rho_{i+1}}\setminus A(k_i,\rho_{i+1})|\rho_{i+1}^{-n}
		\\&
		\geqslant
		c2^{-i}\nu_0 |A(k_{i+1},\rho_{i+1})|^{1-\frac{1}{n}} \left( |B_{R/4}|-\varepsilon_{0} |B_{R/2}|\right)R^{-n}
		\\&
		\geqslant
		c2^{-i}\nu_0 |A(k_{i+1},\rho_{i+1})|^{1-\frac{1}{n}}
		\\&
		\geqslant
		c2^{-i}\nu_0 R^{n-1} Y_{i+1}^{1-\frac{1}{n}}
	\end{split}
\end{align*}
for some constant $c\equiv c(\data(\O_0))$. Combining last two displays, we conclude
\begin{align*}
	Y_{i+1} \leqslant c_{*} \left(2^{\frac{n(s(\Psi)+2)}{n-1}}\right)^{i} Y_{i}^{1+\frac{1}{n-1}}
\end{align*}
for some constant $c_{*}\equiv c_{*}(\data(\O_0))$. Now we apply Lemma \ref{lem:t1} in order to have $Y_{i}\to 0$ as $i\to \infty$, provided 
\begin{align*}
	Y_0 = \frac{|A(k_0,R/2)|}{|B_{R/2}|} \leqslant \varepsilon_0 \leqslant c_{*}^{-(n-1)} 2^{-n(n-1)(s(\Psi)+2)}.
\end{align*}
Therefore, \eqref{hc:12} is satisfied since 
\begin{align*}
	\left|A\left(\sup\limits_{B_{R}}u-\frac{\nu_0}{2},R/4\right)\right| = 0.
\end{align*}
\textbf{Step 3: Proof of H\"older continuity.} Finally, we are now ready to prove a local H\"older continuity of $u$.  For this, let $m\geqslant 3$ be the natural number satisfying \eqref{hc:7} for the choice $\varepsilon\equiv \varepsilon_0\in (0,1/2^{n+1})$, where $\varepsilon_0$ is determined via \eqref{hc:11}. Then we have 
	\begin{align*}
		\osc\limits_{B_{R/4}} u \leqslant
		\left(1-\frac{1}{2^{m+1}} \right) \osc_{B_{R}}u
	\end{align*}
	with $m\equiv m(\data(\O_0))$, whenever $B_{2R}\subset\O_0$ is a ball with $R\leqslant 1$. Clearly, the above display implies that there exists a positive exponent  $\gamma\equiv \gamma(\data(\O_0))\in (0,1)$
 such that, for any fixed ball $B_{8R_0}\subset \O_0$ with $8R_0\leqslant 1$, the following oscillation 
 \begin{align*}
 	\osc\limits_{B_{R}} u \leqslant
 	c\left(\frac{R}{R_0} \right)^{\gamma} \osc\limits_{B_{R_0}}u
 \end{align*}
 holds with some constant $c\equiv c(\data(\O_0))$ for every $R\in (0,R_0]$. Here we note that in the case that the assumption \eqref{ma:2} is in force, the constants appearing in the above lemma depend only on $\data$, but otherwise are independent of the subset $\O_0$. Finally, we have shown that 
 \begin{align*}
     u \in C_{\loc}^{0,\gamma}(\O_{0})
 \end{align*}
if either the assumption \eqref{ma:1} or \eqref{ma:2} is satisfied.
 Therefore by a standard covering argument, the estimates \eqref{hc:1} and \eqref{hc:2} are satisfied. Clearly, if \eqref{ma:2} is assumed instead of \eqref{ma:1}, $\gamma$ in \eqref{hc:3} depends only on $\data$ since $\norm{u}_{L^{\infty}(\O_0)} \leqslant \norm{u}_{L^{\infty}(\O)}$. The proof is complete.
 \end{proof}
 
 \subsection{The Harnack inequality}
 In this subsection we prove the Harnack inequality for a local $Q$-minimizer $u$ of the functional $\P$ in \eqref{ifunct} under one of the assumptions $\eqref{ma:1}$, $\eqref{ma:2}$ and $\eqref{ma:3}$. The analysis similar to the one in the Step 1 of the proof of Theorem \ref{thm:hc} gives the following lemma.
 
 \begin{lem}
 	\label{lem:1har}
 	Let $u\in W^{1,\Psi}(\O)$ be a non-negative local $Q$-minimizer of the functional $\P$ in \eqref{ifunct} under the coefficient functions $a(\cdot)\in C^{\omega_a}(\O)$ and $b(\cdot)\in C^{\omega_{b}}(\O)$ for $\omega_{a}, \omega_{b}$ being non-negative concave functions vanishing at the origin. Suppose that one of the assumptions \eqref{ma:1}, \eqref{ma:2} and \eqref{ma:3} is satisfied. Let $B_{6R}\subset\O_{0}\Subset\O$ be a ball with $6R\leqslant 1$. Then for any $\tau_{1},\tau_{2} \in (0,1)$, there exists a large number $m$ depending on $\data$ and $\tau_{1},\tau_{2}$ such that for any $0< k \leqslant \norm{u}_{L^{\infty}(B_{3R})}$, if 
 	\begin{align}
 		\label{lem:1har:1}
 		\left|\left\{x\in B_{R} : u(x) \geqslant k \right\} \right| \geqslant
 		\tau_{1} |B_{R}|
 	\end{align}
 	holds, then 
 	\begin{align}
 		\label{lem:lhar:2}
 		\left|\{x\in B_{2R} : u(x) \leqslant 2^{-m}k\} \right| \leqslant \tau_{2}\left|B_{2R}\right|.
 	\end{align}
 \end{lem}
 \begin{proof}
 	Let $m\geqslant 3$ be a large number to be determined later.  We set, for $i=0,1,\ldots,m$,
\begin{align*}
 	k_i := \frac{k}{2^{i}},\quad \mathcal{D}_{i}:= B(k_i,2R)\setminus B(k_{i+1},2R)	
\end{align*} 	
and 
\begin{align*}
	 w_i(x)
	 := \left\{\begin{array}{lll}
        k_{i}-k_{i+1} & \text{if }& u(x)< k_{i+1},\\
        u(x)-k_{i+1}  & \text{if }& k_{i+1} \leqslant u(x) < k_{i},\\
        0  & \text{if }& u(x) \geqslant k_{i}.
        \end{array}\right.
\end{align*} 
We observe that $\Psi_{3R}^{-}(w_i)\in W^{1,1}(B_{2R})$	and $\Psi_{3R}^{-}(w_i) \equiv 0 $ on $B_{2R}\setminus 
B(k_0,2R)$ for every $i\in \{0,1,\ldots,m\}$ and $|B_{2R}\setminus B(k_0,2R)|\geqslant \tau_{1} |B_{R}| $. Then using H\"older's inequality, Sobolev's inequality and Lemma \ref{lem:nf3}, we have that
\begin{align}
	\label{1har:1}
	\begin{split}	
	B(k_{i+1},2R)\Psi_{3R}^{-}\left(\frac{k_{i}-k_{i+1}}{3R} \right)
	&\leqslant
	\I_{B(k_i,2R)}\Psi_{3R}^{-}\left(\frac{w_i}{3R} \right)\,dx	
	\\&
	\leqslant
	|B(k_i,2R)|^{\frac{1}{n}}\left( \I_{B(k_i,2R)}\left[\Psi_{3R}^{-}\left(\frac{w_i}{3R} \right) \right]^{\frac{n}{n-1}}\,dx \right)^{\frac{n-1}{n}}
	\\&
	\leqslant
	c R \left( \I_{B(k_i,2R)}\left[\Psi_{3R}^{-}\left(\frac{w_i}{3R} \right) \right]^{\frac{n}{n-1}}\,dx \right)^{\frac{n-1}{n}}
	\\&
	\leqslant
	c\I_{\mathcal{D}_{i}}\left(\Psi_{3R}^{-} \right)'\left(\frac{u-k_{i+1}}{3R} \right)|Du|\,dx
	\\&
	\leqslant
	\varepsilon \I_{\mathcal{D}_{i}}\Psi_{B_{3R}}^{-}(|Du|)\,dx + \frac{c}{\varepsilon^{s(\Psi)}}\I_{\mathcal{D}_{i}}\Psi_{3R}^{-}\left(\frac{u-k_{i+1}}{3R} \right)\,dx
	\end{split}
\end{align}
 for every $\varepsilon\in (0,1)$ and some constant $c\equiv c(\data,\tau_1)$, where we have used Remark \ref{rmk:nf2} that $\Psi_{B_{3R}}^{-}\in \mathcal{N}$ with an index $s(\Psi)= s(G) + s(H_{a}) + s(H_{b})$. It follows from the almost standard Caccioppoli inequality of Lemma \ref{lem:cacct} that 
 \begin{align}
 	\label{1har:2}
 	\begin{split}
 	\I_{\mathcal{D}_{i}}\Psi_{B_{3R}}^{-}(|Du|)\,dx
 	&\leqslant
 	c\I_{B(k_i,2R)}\Psi_{3R}^{-}\left(\frac{k_{i}-u}{3R} \right)\,dx	
 	\leqslant
 	c\I_{B(k_i,2R)}\Psi_{3R}^{-}\left(\left|\frac{2(k_{i}-k_{i+1})}{3R}\right| \right)\,dx
 	\\&
 	\leqslant
 	c|B(k_i,2R)|\Psi_{3R}^{-}\left(\left|\frac{2(k_{i}-k_{i+1})}{3R}\right| \right) \leqslant
 	cR^{n} \Psi_{3R}^{-}\left(\left|\frac{k_{i}-k_{i+1}}{3R}\right| \right),
 	\end{split}	
 \end{align}
 where we have also used the assumption that $u$ is non-negative. Clearly, by the very definition of $\mathcal{D}_{i}$, one can see that 
 \begin{align}
 	\label{1har:3}
 	\I_{\mathcal{D}_{i}}\Psi_{3R}^{-}\left(\frac{u-k_{i+1}}{3R} \right)\,dx
 	\leqslant
 	\I_{\mathcal{D}_{i}}\Psi_{3R}^{-}\left(\frac{k_{i}-k_{i+1}}{3R} \right)\,dx
 	\leqslant
 	c |\mathcal{D}_{i}| \I_{\mathcal{D}_{i}}\Psi_{3R}^{-}\left(\frac{k_{i}-k_{i+1}}{3R} \right)\,dx.
 \end{align}
Combining the estimates obtained in \eqref{1har:1}-\eqref{1har:3}, we find that
\begin{align*}
	|B(k_{m},2R)| \leqslant |B(k_{i+1},2R)| \leqslant c\varepsilon R^{n} + \frac{c}{\varepsilon^{s(\Psi)}}|\mathcal{D}_{i}|
\end{align*}
holds for some constant $c\equiv c(\data,\tau_1)$, whenever $\varepsilon\in (0,1)$ and $i\in\{0,1,\ldots,m-1\}$. Summing the last inequality above over the index $i$ from $0$ to $m-1$ implies 
\begin{align*}
\begin{split}
	|B(k_{m},2R)| 
	&\leqslant
	 c\varepsilon R^{n}  + \frac{c}{\varepsilon^{s(\Psi)}m}|B(k_0,2R)| 
	\leqslant
	\left( c_{*}  \varepsilon + \frac{c_{*}}{\varepsilon^{s(\Psi)}m} \right)|B_{2R}|
	\end{split}
\end{align*}
for some constant $c_{*}\equiv c_{*}(\data,\tau_{1})$. Now choosing small enough $\varepsilon\equiv (\data,\tau_1,\tau_2)$ and sufficiently large $m\equiv m(\data,\tau_1,\tau_2)$ such that 
\begin{align*}
	c_{*}  \varepsilon + \frac{c_{*}}{\varepsilon^{s(\Psi)}m} \leqslant \tau_{2},
\end{align*}
we arrive at the desired estimate \eqref{lem:lhar:2}.

 \end{proof}
 
 \begin{lem}
 	\label{lem:2har}
 	Under the assumptions of Lemma \ref{lem:1har}, let $u\in W^{1,\Psi}(\O)$ be a non-negative $Q$-minimizer of the functional $\P$ in \eqref{ifunct}. Suppose that one of the assumptions \eqref{ma:1}, \eqref{ma:2} and \eqref{ma:3} is satisfied. Then for any $\tau\in (0,1)$, there exists a small $\delta_1\equiv \delta_1(\data(\O_0))$ such that for any $0<k \leqslant \norm{u}_{L^{\infty}(B_{3R})}$, if 
 	\begin{align}
 		\label{lem:2har:1}
 		\left|\{ x\in B_{R} : u(x)\geqslant k \} \right| \geqslant \tau |B_{R}| 		
 	\end{align}
 	holds, then
 	\begin{align}
 		\label{lem:2har:2}
 		\inf\limits_{B_{R}} u \geqslant \delta_{1}k.	
 	\end{align}
 \end{lem}
 
 \begin{proof}
 	It's enough to prove the lemma for $\tau\in \left(0, 2^{-(n+1)} \right)$. Let us fix $m_0\in \mathbb{N}$, and consider the sequences defined by 
 	\begin{align}
 		\label{2har:1}
 		\rho_{i}:= R\left(1+\frac{1}{2^{i}} \right)
 		\quad\text{and}\quad
 		k_i:= \left(\frac{1}{2} + \frac{1}{2^{i}} \right)2^{-m_0}k
 		\quad \left( i=0,1,2,\ldots \right).
 	\end{align}
 Next we also define 
 \begin{align}
 	\label{2har:2}
 	\mathcal{D}_{i+1}^{-}:= B(k_i,\rho_{i+1})\setminus B(k_{i+1},\rho_{i+1})
 	\quad\text{and}\quad 
 	Y_{i}:= \frac{|B(k_i,\rho_i)|}{|B_{\rho_i}|},
 \end{align}
 where the definition of $B(k_i,\rho_i)$ has been introduced in \eqref{hc:6_1}. By using the assumption that $u$ is non-negative, we observe $(u-k_{i})_{-} \leqslant 2^{-m_0}k$. Then by applying Lemma \ref{lem:cacct}, we see 
 \begin{align*}
 	\begin{split}
 		\I_{B(k_i,\rho_{i+1})}\Psi_{B_{2R}}^{-}(|Du|)\,dx
 		&\leqslant
 		c 2^{(i+3)(s(\Psi)+1)}\I_{B(k_i,\rho_i)}\Psi_{B_{2R}}^{-}\left(\frac{(u-k_i)_{-}}{2R} \right)\,dx
 		\\&
 		\leqslant
 		c2^{(i+3)(s(\Psi)+1)}\Psi_{B_{2R}}^{-}\left( \frac{2^{-m_0}k}{R} \right)|B(k_i,\rho_i)|
 	\end{split}
 \end{align*}
 for some constant $c\equiv c(\data)$. This estimate together with the convexity of $\Psi_{B_{2R}}^{-}$ implies
 \begin{align*}
 	\begin{split}
 	\Psi_{B_{2R}}^{-}\left(\FI_{\mathcal{D}_{i+1}^{-}}|Du|\,dx \right)
 	\leqslant
 	\FI_{\mathcal{D}_{i+1}^{-}}\Psi_{B_{2R}}^{-}(|Du|)\,dx
 	&\leqslant
 	c2^{i(s(\Psi)+1)}\frac{|B(k_i,\rho_i)|}{|\mathcal{D}_{i+1}^{-}|}\Psi_{B_{2R}}^{-}\left(\frac{2^{-m_0}k}{R} \right)
 	\\&
 	\leqslant
 	c\Psi_{B_{2R}}^{-}\left( 2^{i(s(\Psi)+1)}\frac{|B(k_i,\rho_i)|}{|\mathcal{D}_{i+1}^{-}|}\frac{2^{-m_0}k}{R}   \right)
 	\end{split}
 \end{align*}
 for some constant $c\equiv c(\data(\O_0))$. Therefore, using the fact that the function $\Psi_{B_{2R}}^{-}$ is increasing and Lemma \ref{lem:nf1}, we have 
 \begin{align*}
 	\FI_{\mathcal{D}_{i+1}^{-}}|Du|\,dx \leqslant c  2^{i(s(\Psi)+1)}\frac{|B(k_i,\rho_i)|}{|\mathcal{D}_{i+1}^{-}|}\frac{2^{-m_0}k}{R}.
 \end{align*}
 Now applying Lemma \ref{lem:t2} together with the fact that $\tau\in \left(0,2^{-(n+1)}\right)$, we see 
 \begin{align*}
 	\begin{split}
 	\I_{\mathcal{D}_{i+1}^{-}} |Du|\,dx 
 	&\geqslant
 	c(k_i-k_{i+1})|B(k_{i+1},\rho_{i+1})|^{1-\frac{1}{n}} \left|B_{\rho_{i+1}}\setminus B(k_i,\rho_{i+1}) \right|\rho_{i+1}^{n}
 	\\&
 	\geqslant
 	c2^{-i}2^{-m_0}k |B(k_{i+1},\rho_{i+1})|^{1-\frac{1}{n}}\left( |B_{2R}|-\tau |B_{R}| \right)R^{-n}
 	\\&
 	\geqslant
 	c2^{-i}2^{-m_0}k |B(k_{i+1},\rho_{i+1})|^{1-\frac{1}{n}}
 	\\&
 	\leqslant
 	c2^{-i}2^{-m_0}k R^{n-1} Y_{i+1}^{1-\frac{1}{n}}.
 	\end{split}
 \end{align*}
 The combination of the last two displays yields
 \begin{align*}
 	Y_{i+1}^{1-\frac{1}{n}} \leqslant c 2^{i(s(\Psi)+1)}R^{-n}|B(k_{i},\rho_{i})|
 	\leqslant
 	c 2^{i(s(\Psi)+1)}Y_{i}
 \end{align*}
 and then we conclude 
 \begin{align*}
 	Y_{i+1} \leqslant  c_{*} 2^{i\frac{(s(\Psi)+1)n}{n-1}} Y_{i}^{1 + \frac{1}{n-1}}
 \end{align*}
 for some constant $c\equiv c(\data(\O_0))$. Now applying Lemma \ref{lem:1har}, we find a large natural number $m_0\equiv m_0(\data(\O_0))$ such that 
 \begin{align*}
 	\left|\{ x\in B_{2R} : u(x) \leqslant 2^{-m_0}k \} \right| \leqslant c_{*}^{-(n-1)} 2^{-n(n-1)(s(\Psi)+1)}.
 \end{align*}
With keeping the above choice of $m_0$, we observe that 
\begin{align*}
	Y_{0} = \frac{|B(k_0,2R)|}{|B_{2R}|} = \frac{|x\in B_{2R} : u(x) \leqslant 2^{-m_0}k|}{|B_{2R}|}
	\leqslant
	c_{*}^{-(n-1)} 2^{-n(n-1)(s(\Psi)+1)}.
\end{align*}
 Now we are at stage in applying Lemma \ref{lem:t1} to obtain that $Y_{i}\rightarrow 0$ as $i\rightarrow \infty$, which is equivalent to 
 \begin{align*}
 	|B(2^{-(m_0+1)}k,R)| = 0.
 \end{align*}
 The last display implies the validity of \eqref{lem:2har:2} with the choice of $\delta_1\equiv 2^{-(m_0+1)}$.
 \end{proof}
 
 From Lemma \ref{lem:2har} and the covering arguments in \cite[Section 7]{KS1}, we obtain the following weak Harnack inequality for a local $Q$-minimizers of the functional $\P$ defined in \eqref{ifunct}. We also refer to \cite{BCM1,BO3,HKKMP} for the proof. 

\begin{thm}[The weak Harnack inequality]
	\label{thm:whar}
	Let $W^{1,\Psi}(\O)$ be a local non-negative $Q$-minimizer of the functional $\P$ defined in \eqref{ifunct} with the coefficient functions $a(\cdot)\in C^{\omega_a}(\O)$ and $b(\cdot)\in C^{\omega_{b}}(\O)$ satisfying $\eqref{ma:1}_{2}$ for functions $\omega_{a}, \omega_{b}$ being concave which vanish at 0. Suppose one of the assumptions \eqref{ma:1}, \eqref{ma:2} and \eqref{ma:3} is satisfied. Let $B_{9R}\equiv B_{9R}(x_0)\subset\O_0\Subset\O$ be a ball with $9R\leqslant 1$. Then there exist $q_{-}>0$ and a constant $c$ depending on $\data(\O_0)$ such that 
	\begin{align}
		\label{thm:whar:1}
		\inf\limits_{x\in B_{R}} u(x) \geqslant \frac{1}{c} \left( \FI_{B_{2R}} u^{q_{-}}\,dx \right)^{\frac{1}{q_{-}}}.
	\end{align}
\end{thm}

To conclude the result of Theorem \ref{thm:har} below, we need to obtain a local sup-estimates for local quasiminizers of $\P$.

\begin{lem}
	\label{lem:3har}
	Under the assumptions of Lemma \ref{lem:1har}, let $u\in W^{1,\Psi}(\O)$ be a non-negative local $Q$-minimizer of the functional $\P$ in \eqref{ifunct} with the coefficient functions $a(\cdot)\in C^{\omega_a}(\O)$ and $b(\cdot)\in C^{\omega_{b}}(\O)$ satisfying $\eqref{ma:1}_{2}$ for functions $\omega_{a}, \omega_{b}$ being concave which vanish at 0. Suppose that one of the assumptions \eqref{ma:1}, \eqref{ma:2} and \eqref{ma:3} is satisfied. Let $B_{9R}\equiv B_{9R}(x_0)\subset\O_0\Subset\O$ be a ball with $9R\leqslant 1$. Then for any $q_{+}>0$, the local estimate holds
\begin{align}
	\label{lem:3har:1}
	\sup\limits_{B_{R}}u \leqslant c\left( \FI_{B_{2R}} |u|^{q_{+}}\,dx \right)^{\frac{1}{q_{+}}}
\end{align}
	 for some constant $c\equiv c(\data(\O_0))$.
\end{lem}
\begin{proof}
The proof consists of two steps. For the convenience, let us consider the scaled functions
\begin{align}
	\label{3har:1}
	\bar{u}(x):= \frac{u(x_0+Rx)}{R}
	\quad\text{for every}\quad
	x\in B_{4}.
\end{align}
Then the almost standard Caccioppoli inequality \eqref{cacct:1} of Lemma \ref{lem:cacct} can be written in the view of $\bar{u}$ as follows:
\begin{align}
	\label{3har:2}
	\I_{B_{r_1}}\Psi_{B_{2R}}^{-}(|D(\bar{u}-k)_{\pm}|)\,dx
	\leqslant
	\frac{c}{(r_2-r_1)^{s(\Psi)+1}}\I_{B_{r_2}}\Psi_{B_{2R}}^{-}((\bar{u}-k)_{\pm})\,dx
\end{align}
with some constant $c\equiv c(\data)$, whenever $B_{r_1}\Subset B_{r_2}\subset B_{2}(0)$ are concentric balls and $k\in\R$. Next for $1\leqslant t \leqslant s \leqslant 2$, we set sequences by 
	\begin{align}
		\label{3har:3}
		\rho_{i}:= \left( t + \frac{s-t}{2^{i}} \right)
		\quad\text{and}\quad
		k_{i}:= 2l_{0}\left(1-\frac{1}{2^{i+1}} \right)
	\end{align}
for some constant $d_0>0$ to be determined later. We also define
\begin{align}
	\label{3har:4}
	\bar{\rho}_{i}:= \frac{\rho_i + \rho_{i+1}}{2}
	\quad\text{and}\quad
	Y_{i}:= \frac{1}{\Psi_{B_{2R}}^{-}(l_0)}\I_{\bar{A}(k_i,\rho_i)} \Psi_{B_{2R}}^{-}((u-k_i)_{+})\,dx, 
\end{align}
where
\begin{align}
	\label{3har:5}
	\bar{A}(k,\rho):= \{x\in B_{\rho} : \bar{u}>k\}.
\end{align} 

Let $\eta_{i} \in C^{\infty}_{0}(B_{\bar{\rho}_{i}})$ be a cut-off function such that $0\leqslant \eta_{i} \leqslant 1$, $\eta_{i}\equiv 1$ on $B_{\rho_{i+1}}$ and $|D\eta_{i}| \leqslant \frac{c(n)2^{i}}{(s-t)}$. 

Then using H\"older's inequality, Sobolev's inequality and Lemma \ref{lem:nf3}, we have

\begin{align*}
	\begin{split}
	\Psi_{B_{2R}}^{-}(l_0)Y_{i+1}
	&\leqslant
	\I_{B_{\bar{\rho}_{i}}}\Psi_{B_{2R}}^{-}\left( (\bar{u}-k_{i+1})\eta_{i} \right)\,dx
	\\&
	\leqslant
	|\bar{A}(k_{i+1},\rho_{i})|^{\frac{1}{n}} \left(\I_{B_{\bar{\rho}_{i}}}\left[\Psi_{B_{2R}}^{-}\left( (\bar{u}-k_{i+1})_{+}\eta_{i} \right)\right]^{\frac{n}{n-1}}\,dx\right)^{\frac{n-1}{n}}
		\\&
		\leqslant
		c|\bar{A}(k_{i+1},\rho_{i})|^{\frac{1}{n}} \I_{B_{\bar{\rho}_i}} \left(\Psi_{B_{2R}}^{-}\right)'\left( (\bar{u}-k_{i+1})_{+}\eta_{i} \right)\left[ |D(\bar{u}-k_{i+1})_{+}|\eta_{i} + (\bar{u}-k_{i+1})_{+}|D\eta_{i}| \right]\,dx
		\\&
		\leqslant
		c|\bar{A}(k_{i+1},\rho_{i})|^{\frac{1}{n}} \I_{B_{\bar{\rho}_i}} \left(\Psi_{B_{2R}}^{-}\right)'\left( (\bar{u}-k_{i+1})_{+} \right)|D(\bar{u}-k_{i+1})_{+}|\,dx
		\\&
		\qquad
		+	
		c|\bar{A}(k_{i+1},\rho_{i})|^{\frac{1}{n}}\frac{2^{i}}{s-t}
		\I_{B_{\bar{\rho}_i}} \left(\Psi_{B_{2R}}^{-}\right)'\left( (\bar{u}-k_{i+1})_{+} \right)(\bar{u}-k_{i+1})_{+}\,dx
		\\&
		\leqslant
		c |\bar{A}(k_{i+1},\rho_{i})|^{\frac{1}{n}}\left( \I_{B_{\bar{\rho}_i}}\Psi_{B_{2R}}^{-}(|D(\bar{u}-k_{i+1})_{+}|)\,dx  + \frac{2^{i}}{s-t} \I_{B_{\bar{\rho}_i}}\Psi_{B_{2R}}^{-}((\bar{u}-k_{i+1})_{+})\,dx  \right)
		\\&
		\leqslant
		c |\bar{A}(k_{i+1},\rho_{i})|^{\frac{1}{n}}
		\left(\frac{2^{i}}{s-t} \right)^{s(\Psi)+1}
		\I_{B_{\rho_{i}}}\Psi_{B_{2R}}^{-}((\bar{u}-k_{i+1})_{+})\,dx
	\end{split}
\end{align*}
for some constant $c\equiv c(\data)$, where in the last inequality of the above display we also have used \eqref{3har:2} and \eqref{3har:5}. Now applying Lemma \ref{lem:nf1}, we see that
\begin{align*}
	\begin{split}
	|\bar{A}(k_{i+1},\rho_{i})| 
	&\leqslant
	\frac{1}{\Psi_{B_{2R}}^{-}\left(k_{i+1}-k_{i}\right)}
	\I_{\bar{A}(k_{i+1},\rho_{i})}\Psi_{B_{2R}}^{-}\left( \bar{u}-k_{i} \right)\,dx
	\\&
	\leqslant
	\frac{1}{\Psi_{B_{2R}}^{-}\left( l_0/2^{i+1} \right)}
	\I_{\bar{A}(k_{i+1},\rho_{i})}\Psi_{B_{2R}}^{-}\left( \bar{u}-k_{i} \right)\,dx
	\\&
	\leqslant
	\frac{\Psi_{B_{2R}}^{-}\left( l_0\right)}{\Psi_{B_{2R}}^{-}\left( l_0/2^{i+1} \right)} Y_{i}
	\leqslant
	2^{(i+1)(s(\Psi)+1)}Y_{i} \leqslant
	c\left(\frac{2^{i}}{s-t} \right)^{s(\Psi)+1}Y_{i}.
	\end{split}
\end{align*}
and 
\begin{align*}
	\begin{split}
	\I_{B_{\rho_{i}}}\Psi_{B_{2R}}^{-}\left( (\bar{u}-k_{i+1})_{+} \right)\,dx
	&= \I_{\bar{A}(k_{i+1},\rho_i)}\Psi_{B_{2R}}^{-}\left( \bar{u}- k_{i+1} \right)\,dx
	\\&
	\leqslant
	\I_{\bar{A}(k_{i},\rho_i)}\Psi_{B_{2R}}^{-}\left( \bar{u}- k_{i} \right)\,dx = 
	\Psi_{B_{2R}}^{-}\left( l_0 \right)Y_{i}.
	\end{split}
\end{align*}
Combining the last three displays, we conclude with the following recursive inequality:
\begin{align*}
	Y_{i+1} \leqslant c_0 \frac{2^{i\left(1+\frac{1}{n} \right)(s(\Psi)+1)}}{(s-t)^{\left(1+\frac{1}{n} \right)(s(\Psi)+1)}}Y_{i}^{1+\frac{1}{n}}
\end{align*}
for some constant $c_0\equiv c_0(\data)$. Now we are at the stage to apply Lemma \ref{lem:t1}. In turn, we have $Y_{i}\rightarrow 0$ as $i\rightarrow \infty$, provided 
\begin{align*}
	Y_0 = \frac{1}{\Psi_{B_{2R}}^{-}\left( l_0 \right)} \I_{\bar{A}(l_0,s)}\Psi_{B_{2R}}^{-}\left(  \bar{u}-l_0 \right)\,dx
	\leqslant
	\left[\frac{c_0}{(s-t)^{\left(1+\frac{1}{n} \right)(s(\Psi)+1)}} \right]^{-n}
	2^{-n\left(1+n \right)(s(\Psi)+1)}.
\end{align*}
The inequality in the last display is satisfied if we choose $l_0>0$ in the following way 
\begin{align*}
	\Psi_{B_{2R}}^{-}\left( l_0 \right) = \frac{c_0^{n}2^{n\left(1+n \right)(s(\Psi)+1)}}{(s-t)^{\left(1+n \right)(s(\Psi)+1)}}\I_{B_{s}}\Psi_{B_{2R}}^{-}\left( (\bar{u})_{+} \right)\,dx.
\end{align*}
Therefore, we obtain $\bar{u} \leqslant 2l_0$ in $B_{t}$. This estimate together with the last display yields
\begin{align}
	\label{3har:12}
	\Psi_{B_{2R}}^{-}\left( \sup\limits_{B_{t}} (\bar{u})_{+} \right)
	\leqslant
	\frac{c}{(s-t)^{\left(1+n \right)(s(\Psi)+1)}}\FI_{B_{s}} \Psi_{B_{2R}}^{-}\left( (\bar{u})_{+} \right)\,dx.
\end{align}

Recalling $\Psi_{B_{2R}}^{-}\in \mathcal{N}$ with an index $s(\Psi)$ and applying Lemma \ref{lem:nf2_1} for $\Psi_{B_{2R}}^{-}$, one can see that $t\mapsto \Psi_{B_{2R}}^{-}\left(t^{\frac{1}{s(\Psi)+1}} \right)$ is a concave function. Using this one together with Jensen's inequality in \eqref{3har:12}, we see 
\begin{align*}
	\begin{split}
	\Psi_{B_{2R}}^{-}\left( \sup\limits_{B_{t}} (\bar{u})_{+} \right)
	&\leqslant
	\frac{c}{(s-t)^{\left(1+n \right)(s(\Psi)+1)}}\FI_{B_{s}} \Psi_{B_{2R}}^{-}\left( (\bar{u})_{+} \right)\,dx
	\\&
	= 
	\frac{c}{(s-t)^{\left(1+n \right)(s(\Psi)+1)}}\FI_{B_{s}} \Psi_{B_{2R}}^{-}\left( \left[(\bar{u})_{+}^{s(\Psi)+1}\right]^{\frac{1}{s(\Psi)+1}} \right)\,dx
	\\&
	\leqslant
	\frac{c}{(s-t)^{\left(1+n \right)(s(\Psi)+1)}}\Psi_{B_{2R}}^{-}\left(\left[\FI_{B_{s}}  (\bar{u})_{+}^{s(\Psi)+1} \,dx \right]^{\frac{1}{s(\Psi)+1}} \right)
	\\&
	\leqslant
	\Psi_{B_{2R}}^{-}\left( \frac{c}{(s-t)^{\left(1+n \right)(s(\Psi)+1)}} \left[\FI_{B_{s}}  (\bar{u})_{+}^{s(\Psi)+1} \,dx \right]^{\frac{1}{s(\Psi)+1}} \right).
	\end{split}
\end{align*}
Since $\Psi_{B_{2R}}^{-}$ is the increasing function, the last display implies
\begin{align*}
	\sup\limits_{B_{t}} (\bar{u})_{+} 
	\leqslant
	 \frac{c}{(s-t)^{\left(1+n \right)(s(\Psi)+1)}} \left[\FI_{B_{s}}  (\bar{u})_{+}^{s(\Psi)+1} \,dx \right]^{\frac{1}{s(\Psi)+1}}.
\end{align*}
Since $-u$ is a local $Q$-minimizer of the functional $\P$, we find 
\begin{align*}
	\sup\limits_{B_{t}} |\bar{u}| 
	\leqslant
	 \frac{c}{(s-t)^{\left(1+n \right)(s(\Psi)+1)}} \left[\FI_{B_{s}}  |\bar{u}|^{s(\Psi)+1} \,dx \right]^{\frac{1}{s(\Psi)+1}}.
\end{align*}
Therefore, for $0<q_{+} < s(\Psi)+1$, we discover from Young's inequality that 
\begin{align*}
	\begin{split}
	\sup\limits_{B_{t}} |\bar{u}| 
	&\leqslant
	 \frac{c}{(s-t)^{\left(1+n \right)(s(\Psi)+1)}}\left( \sup\limits_{B_{s}} |\bar{u}| \right)^{1-\frac{q_{+}}{s(\Psi)+1}} \left[\FI_{B_{s}}  |\bar{u}|^{q_{+}} \,dx \right]^{\frac{1}{s(\Psi)+1}}
	 \\& 
	 \leqslant
	 \frac{1}{2} \sup\limits_{B_{s}} |\bar{u}|
	 + \frac{c}{(s-t)^{\frac{\left(1+n \right)(s(\Psi)+1)^{2}}{q_{+}}}} \left[\FI_{B_{2}}  |\bar{u}|^{q_{+}} \,dx \right]^{\frac{1}{q_{+}}}
	 \end{split}
\end{align*}
holds for every $1\leqslant t < s \leqslant 2$. Then we apply Lemma \ref{lem:t0} for $h(t)=\sup\limits_{B_{t}} |\bar{u}|$ in order to have 
\begin{align}
	\label{3har:17}
	\sup\limits_{B_{1}} |\bar{u}| \leqslant c\left[\FI_{B_{2}}  |\bar{u}|^{q_{+}} \,dx \right]^{\frac{1}{q_{+}}}
\end{align}
for $c\equiv c(\data,q_{+})$. On the other hand, for $q_{+}\geqslant s(\Psi)+1$, the inequality \eqref{3har:17} is still valid by using H\"older's inequality. Scaling back as we introduced in \eqref{3har:1}, we arrive at the desired estimate \eqref{lem:3har:1}. 
\end{proof}

Finally, the main result of the this section is the following:

\begin{thm}
	\label{thm:har}
	Let $u\in W^{1,\Psi}(\O)$ be a non-negative local $Q$-minimizer $u$ of the functional $\P$ defined in \eqref{ifunct} under the coefficient functions $a(\cdot)\in C^{\omega_a}(\O)$ and $b(\cdot)\in C^{\omega_{b}}(\O)$ for $\omega_{a}, \omega_{b}$ being non-negative concave functions vanishing at the origin. Suppose that one of the assumptions \eqref{ma:1}, \eqref{ma:2} and \eqref{ma:3} is satisfied. For every ball $B_{R}$ with $B_{9R}\subset\O_0$ with $\O_{0}\Subset \O$ being an open subset, there exists a positive constant $c\equiv c(\data(\O_0))$ such that 
	\begin{align}
		\label{thm:har:1}
		\sup\limits_{B_{R}} u \leqslant c \inf\limits_{B_{R}}u
\end{align}	 
holds.
\end{thm}
\begin{proof}
	The proof is essentially based on the results we have obtained so far. In fact, applying Theorem \ref{thm:whar} and Lemma \ref{lem:3har} with $q_{-} = q_{+}$, we obtain \eqref{thm:har:1}. 
\end{proof}

\begin{rmk}
	\label{rmk:thm:har}
	The results of the above theorem refine the results of \cite[Theorem 1.3]{HHT} without any extra term in \eqref{thm:har:1} under our multi-phase settings when the assumptions \eqref{ma:1} and \eqref{ma:2} come into play, and see also \cite{HHL1,HKKMP}.
\end{rmk}


 \subsection{Higher integrability results}
 \label{subsec:5.4}
    
 Next, we provide a higher integrability result for a local minimizer of the functional $\mathcal{F}$ defined in \eqref{functional}.
 
 \begin{thm}[Higher Integrability]
 	\label{thm:hi}
 	Let $u\in W^{1,\Psi}(\O)$ be a local minimizer of the functional $\mathcal{F}$ defined in \eqref{functional} under the assumption \eqref{sa:1}. Assume that one of the assumptions \eqref{ma:1}, \eqref{ma:2} and \eqref{ma:3} is satisfied.
 	Then there exists a higher integrability exponent $\delta\equiv\delta(\data)\in (0,1)$ such that the following reverse type H\"older inequality 
 	\begin{align}
 		\label{hi:1}
 		\left( \FI_{B_{R/2}}[\Psi(x,|Du|)]^{1+\delta}\,dx \right)^{\frac{1}{1+\delta}} \leqslant
 		c\FI_{B_{R}} \Psi(x,|Du|)\,dx
 	\end{align}
 	holds for a constant $c\equiv c(\data)$, where $\data$ is clarified in \eqref{data}, whenever $B_{R}\Subset\O$ is a ball with $R\leqslant 1$. In particular, for any open subset $\O_{0}\Subset \O$, it holds that
 	\begin{align}
 		\label{hi:2}
 		\norm{\Psi(x,|Du|)}_{L^{1+\delta}(\O_0)} \leqslant c(\data(\O_0)).
 	\end{align}
 	 \end{thm}
 	 
 \begin{proof}
 	Let $B_{R}\Subset\O$ be a ball with $R\leqslant 1$ as in the statement. Since $u$ is a local $Q:=L/\nu$-minimizer of the functional $\mathcal{P}$ in \eqref{ifunct}, we are able to apply Lemma \ref{lem:cacc} with the choices $\rho \equiv R/2$, $r \equiv R$ and $k\equiv (u)_{B_{R}}$ in order to get 
 	\begin{align}
 		\label{hi:3}
 		\FI_{B_{R/2}}\Psi(x,|Du|)\,dx \leqslant c\FI_{B_{R}}\Psi\left(x,\left|\frac{u-(u)_{B_{R}}}{R}\right| \right)\,dx
 	\end{align}
 	with some constant $c\equiv c(n,s(G),s(H_{a}), s(H_{b}),\nu,L)$. Then, applying Remark  \ref{rmk:sp} depending on which one of the assumptions \eqref{ma:1}, \eqref{ma:2} and \eqref{ma:3} is assumed, we obtain the following reverse H\"older inequality:
 	\begin{align}
 		\label{hi:4}
 		\FI_{B_{R/2}} \Psi(x,|Du|)\,dx \leqslant c\left( \FI_{B_{R}}[\Psi(x,|Du|)]^{\theta}\,dx \right)^{\frac{1}{\theta}},
 	\end{align}
 	where $c\equiv c(\data)$, and $\theta\in (0,1)$ is the same appearing in Remark \ref{rmk:sp}.
 	At this point \eqref{hi:1} follows using a variant of Gehring's lemma on reverse H\"older inequalities, see for instance	\cite[Theorem 6.6]{Gi1}.
 \end{proof}


\section{Harmonic type approximation}
\label{sec:6}

Here we discuss some important regularity results for the solution to the following Dirichlet boundary value problem: 
\begin{align}
	\label{cp:cp1}
	\begin{cases}
        -\Div A_0(Dh) = 0 \text{ in } B_{R}  \\
        h\in \u + W_{0}^{1,\Psi_0}(B_{R}),
     \end{cases}
\end{align}
where $B_{R}\subset\R^n$ is a given ball with $n\geqslant 2$, $\u\in W^{1,\Psi_0}(B_{R})$ is a given function, and $A_{0}:\R^n\rightarrow \R^n$ is a vector field belonging to $C^{0}(\R^n)\cap C^{1}(\R^n\setminus \{0\})$ and satisfies the following ellipticity and coercivity assumptions:
\begin{align}
	\label{cp:cp2}
	\begin{cases}
		|A_{0}(z)||z| + |D_{z}A_{0}(z)||z|^2 \leqslant L\Psi_0(|z|) \\
		\nu \dfrac{\Psi_0(|z|)}{|z|^2}|\xi|^2 
		\leqslant \inner{D_{z}A_0(z)\xi}{\xi}
	\end{cases}
\end{align}
for fixed constants $0<\nu\leqslant L$, whenever $z\in\R^{n}\setminus\{0\}$ and $\xi\in\R^n$, in which the function $\Psi_{0}$ is given by 
\begin{align}
	\label{cp:cp3}
	\Psi_0(t):= G(t) + a_0H_{a}(t) + b_0 H_{b}(t)
\end{align}
with fixed constants $a_{0},b_{0}\geqslant 0$ for every $t\geqslant 0$. By Lemma $\ref{lem:nf2}_{1}$, we get the following  
\begin{align}
	\label{cp:cp4}
	\frac{1}{s(G)+s(H_{a}) + s(H_{b})} \leqslant \frac{\Psi_0^{''}(t)t}{\Psi_0^{'}(t)} \leqslant s(G)+ s(H_{a}) + s(H_{b})
\end{align}
for every $t>0$, which means that $\Psi_0\in \mathcal{N}$ with an index $s(\Psi_0)=s(G)+s(H_{a}) + s(H_{b})$. Therefore, we note that the following monotonicity property that
\begin{align}
	\label{cp:cp5}
	\begin{split}
	|V_{G}(z_1)-V_{G}(z_2)|^2 + a_0|V_{H_{a}}(z_1)-V_{H_{a}}(z_2)|^2
	+ b_0|V_{H_{b}}(z_1)-V_{H_{b}}(z_2)|^2
	&\approx 
	|V_{\Psi_0}(z_1)-V_{\Psi_0}(z_2)|^2
	\\&
	\leqslant
	c\inner{A_0(z_1)-A_0(z_2)}{z_1-z_2}	
	\end{split}
\end{align}
holds with some constant $c\equiv c(n,s(\Psi_0),\nu)$, whenever $z_1,z_2\in\R^n\setminus\{0\}$, where the map $V_{\Phi}$ for a function $\Phi\in\mathcal{N}$ has been defined in \eqref{defV}.

\begin{thm}
\label{hta:thm_cz}
 Let $h\in W^{1,\Psi_0}(B_{R})$ be the weak solution to \eqref{cp:cp1}  under the assumption \eqref{cp:cp2}. Suppose that there exists a higher integrability exponent $\delta_1>0$ such that 
 \begin{align}
 	\label{cp:cp6}
 	\Psi_0(|D\u|)\in L^{1+\delta_1}(B_{R})\quad
 	\text{ and }\quad
 	\norm{\Psi_0(|D\u|)}_{L^{1}(B_{R})} \leqslant L_0
 \end{align}
 for some constant $L_0\geqslant 0$. Then there exists a positive exponent $\delta_0\leqslant \delta_1$ depending on $n,s(\Psi_0),\nu,L$ and $\delta_1$  such that the following inequality
 \begin{align}
 	\label{cp:cp7}
 	\left( \FI_{B_{R}}[\Psi_0(|Dh|)]^{1+\delta_0}\,dx\right)^{\frac{1}{1+\delta_0}} \leqslant c\left(\FI_{B_{R}}[\Psi_0(|D\u|)]^{1+\delta_0}\,dx\right)^{\frac{1}{1+\delta_0}}
 \end{align}
 holds for some constant $c\equiv c(n,s(\Psi_0),\nu,L,L_0,\delta_1)$. 
\end{thm}
\begin{proof}
	First the standard energy estimate implies that 
	\begin{align}
		\label{cp:cp8}
		\I_{B_{R}}\Psi_0(|Dh|)\,dx \leqslant c\I_{B_{R}}\Psi_0(|D\u|)\,dx \leqslant cL_0
	\end{align}
	holds with some constant $c\equiv c(n,s(\Psi_0),\nu,L)$. For a fixed ball $B_{2\rho}\subset B_{R}$, let $\eta\in C_{0}^1(B_{2\rho})$ be a standard cut-off function satisfying $\chi_{B_{\rho}} \leqslant \eta \leqslant \chi_{B_{2\rho}}$ and $|D\eta| \leqslant 4/\rho$. Let us take the function $\varphi= \eta^{s(\Psi_0)+1}\left( h-(h)_{B_{2\rho}} \right)$ as a test function in the equation \eqref{cp:cp1}. Then using the monotonicity property of $A_0(\cdot)$ and Lemma \ref{lem:nf3} with $\Psi_0$, we have 
	\begin{align}
		\label{cp:cp9}
		\begin{split}
		\I_{B_{2\rho}}&\eta^{s(\Psi_0)+1}\Psi_0(|Dh|)\,dx 
		\leqslant
		c\I_{B_{2\rho}} \eta^{s(\Psi_0)}\Psi_{0}^{'}(|Dh|)\left| \frac{h-(h)_{B_{2\rho}}}{\rho} \right|\,dx			\\&
		\leqslant
		c\I_{B_{2\rho}}\eta^{s(\Psi_0)}\left( (\varepsilon\eta)\Psi_0(|Dh|) + \frac{1}{(\varepsilon\eta)^{s(\Psi_0)}}\Psi_0\left( \left| \frac{h-(h)_{B_{2\rho}}}{\rho} \right| \right) \right)\,dx.
		\end{split}
	\end{align}
Choosing $\varepsilon$ sufficiently small in the last display, we conclude that 
\begin{align}
	\label{cp:cp10}
	\FI_{B_{\rho}}\Psi_0(|Dh|)\,dx \leqslant c\FI_{B_{2\rho}} \Psi_0\left( \left| \frac{h-(h)_{B_{2\rho}}}{\rho} \right| \right)\,dx
\end{align}
for a constant $c\equiv c(n,s(\Psi_0),\nu,L)$. By applying Lemma \ref{lem:os} to $\Phi\equiv \Psi_0$ with $d_{0}\equiv 1$, there exists $\theta_0\equiv \theta_0(n,s(\Psi_0))\in (0,1)$ such that 
\begin{align}
	\label{cp:cp11}
	\FI_{B_{\rho}}\Psi_0(|Dh|)\,dx
	\leqslant
	c\FI_{B_{2\rho}} \Psi_0\left( \left| \frac{h-(h)_{B_{2\rho}}}{\rho} \right| \right)\,dx
	\leqslant
	c\left(\FI_{B_{2\rho}}[\Psi_0(|Dh|)]^{\theta_0}\,dx\right)^{\frac{1}{\theta_0}}
\end{align}
holds for some constant $c\equiv c(n,s(\Psi_0),\nu,L)$, whenever $B_{2\rho}\subset B_{R}$ is a ball. Now we prove a version of the last inequality near the boundary of $B_{R}$. For this, let $B_{2\rho}(y)\subset\R^n$ be a ball such that $y\in B_{R}$ 
and $\frac{1}{10} < \frac{|B_{2\rho}(y)\setminus B_{R}|}{|B_{2\rho}(y)|}$. We take a test function by $\varphi \equiv \eta^{s(\Psi_0)+1}(h-\u)$, where $\eta\in C_{0}^1(B_{2\rho})$ is a standard cut-off function as before so that $\chi_{B_{\rho}} \leqslant \eta \leqslant \chi_{B_{2\rho}}$ and $|D\eta|\leqslant 4/\rho$. This choice of $\varphi$ is admissible since $\Supp \varphi \Subset B_{R}\cap B_{2\rho}(y)$. Arguing similarly as we have done above, we have
\begin{align}
	\label{cp:cp12}
	\begin{split}
		\I_{B_{R}\cap B_{2\rho}(y)}&\eta^{s(\Psi_0)+1}\Psi_0(|Dh|)\,dx 
		\leqslant
		c\I_{B_{R}\cap B_{2\rho}(y)} \eta^{s(\Psi_0)}\Psi_{0}^{'}(|Dh|)\left| \frac{h-\u}{\rho} \right|\,dx			\\&
		+
		c\I_{B_{R}\cap B_{2\rho}(y)} \eta^{s(\Psi_0)}\Psi_{0}^{'}(|Dh|)|D\u|\,dx	
		\\&
		\leqslant
		c\I_{B_{R}\cap B_{2\rho}}\eta^{s(\Psi_0)}\left( (\varepsilon\eta)\Psi_0(|Dh|) + \frac{1}{(\varepsilon\eta)^{s(\Psi_0)}}\Psi_0\left( \left| \frac{h-\u}{\rho} \right| \right) \right)\,dx
		\\&
		\qquad
		+
		c\I_{B_{R}\cap B_{2\rho}}\eta^{s(\Psi_0)}\left( (\varepsilon\eta)\Psi_0(|Dh|) + \frac{1}{(\varepsilon\eta)^{s(\Psi_0)}}\Psi_0\left( |D\u| \right) \right)\,dx.
		\end{split}
\end{align}
Again choosing $\varepsilon$ small enough and reabsorbing the terms, we find that
\begin{align*}
	\begin{split}
		\FI_{B_{R}\cap B_{2\rho}(y)}&\eta^{s(\Psi_0)+1}\Psi_0(|Dh|)\,dx 
		\leqslant
		c\FI_{B_{R}\cap B_{2\rho}(y)}\Psi_0\left(\left|\frac{h-\u}{\rho} \right| \right)\,dx
		+
		c\FI_{B_{R}\cap B_{2\rho}(y)}\Psi_0\left(\left|D\u \right| \right)\,dx
	\end{split}
\end{align*}
for some constant $c\equiv c(n,s(\Psi_0),\nu,L)$. Redefining $h-\u \equiv 0$ on $B_{2\rho}(y)\setminus B_{R}$, we are able to apply Lemma \ref{lem:os} to $\Phi\equiv \Psi_0$ with $d_{0}\equiv 1$. In turn, there exists $\theta_0\equiv \theta_0(n,s(\Psi_0))\in (0,1)$ as appearing in \eqref{cp:cp11}  such that
\begin{align*}
	\begin{split}
		\FI_{B_{R}\cap B_{2\rho}(y)}\Psi_0\left(\left|\frac{h-\u}{\rho} \right| \right)\,dx
		&\leqslant
		\left(\FI_{B_{R}\cap B_{2\rho}(y)}[\Psi_0(|Dh-D\u|)]^{\theta_0}\,dx\right)^{\frac{1}{\theta_0}}
		\\&
		\leqslant
		\left(\FI_{B_{R}\cap B_{2\rho}(y)}[\Psi_0(|Dh|)]^{\theta_0}\,dx\right)^{\frac{1}{\theta_0}}
		+
		c\FI_{B_{R}\cap B_{2\rho}(y)}[\Psi_0(|D\u|)]\,dx
	\end{split}
\end{align*}
for some constant $c\equiv c(n,s(\Psi_0))$, where for the last inequality we have used \eqref{growth1} and H\"{o}lder's inequality. Combining the last two displays and \eqref{cp:cp11}, we have 
\begin{align}
	\label{cp:cp15}
	\begin{split}
		\FI_{B_{\rho}(y)} [V(x)]^{\frac{1}{\theta_0}}\,dx \leqslant
		c\left( \FI_{B_{2\rho}(y)}V(x)\,dx \right)^{\frac{1}{\theta_0}} + c\FI_{B_{2\rho}(y)} U(x)\,dx
	\end{split}
\end{align}
for some $c\equiv c(n,s(\Psi_0),\nu,L)$, where 
\begin{align*}
V(x):= [\Psi_0(|Dh|)]^{\theta_0}\chi_{B_{2\rho}(y)}(x)\quad\text{and}\quad U(x):= \Psi_0(|D\u|)\chi_{B_{2\rho}(y)}(x)
\end{align*}
for every ball $B_{2\rho}(y)\subset\R^n$ satisfying either $B_{2\rho}(y)\subset B_{R}$ or $\frac{1}{10} < \frac{|B_{2\rho}(y)\setminus B_{R}|}{|B_{2\rho}(y)|}$ with $y\in B_{R}$. Applying a variant of Gehring's lemma and a standard covering argument, we arrive at the desired estimate \eqref{cp:cp7}.
\end{proof}

 Before going on further, we recall a classical truncation lemma due to \cite{AFu1}. The statement involves the Hardy-Littlewood maximal operator, defined as
 \begin{align}
     \label{hta:eq4}
    M(f)(x) := \sup_{B_{r}(x)\subset\R^n}\FI_{B_{r}(x)} |f(y)|\,dy,\qquad x\in\R^n,
 \end{align}
 whenever $f\in L_{\loc}^{1}(\R^n)$. 
 \begin{thm}[\cite{AFu1}]
    \label{hta:thm_truncation}
    Let $B_{R}\subset\R^n$ be a ball and $f\in W^{1,1}_{0}(B_{R})$. Then, for every $\lambda > 0$, there exists $f_{\lambda}\in W^{1,\infty}_{0}(B_{R})$ such that 
    \begin{align}
        \label{hta:truncation1}
        \norm{Df_{\lambda}}_{L^{\infty}(B_{R})} \leqslant c\lambda
    \end{align}
    for some constant $c$ depending only on $n$. Moreover, it holds that 
    \begin{align}
        \label{hta:truncation2}
        \{x\in B_{R} : f_{\lambda}(x)\neq f(x)\}\subset  \{x\in B_{R} : M(|Df(x)|) > \lambda \} \cup \text{ negligible set. }
    \end{align}
 \end{thm}
 We notice that in this theorem we may assume that $f$ is defined on $\R^n$ by redefining $f\equiv 0$ on $\R^{n}\setminus B_{R}$. We are now ready to state the main result of this section. 
 
 \begin{lem}[Harmonic type approximation]
     \label{hta: lemma_hta}
     Let $B_{R}\subset \R^n$ be a ball with $R\leqslant 1$, $\sigma\in (0,1)$ and $\u\in W^{1,\Psi_{0}}(B_{2R})$ be a function satisfying 
     \begin{align}
        \label{hta:hta1}
         \FI_{B_{2R}}\Psi_{0}(|D\u|)\,dx \leqslant c_0
     \end{align}
     and 
     \begin{align}
        \label{hta:hta2}
         \FI_{B_{R}} [\Psi_0(|D\u|)]^{1+\delta_1}\,dx \leqslant c_1
     \end{align}
     for some constants $c_0, c_1 \geqslant 1$ and $\delta_1 >0$. Suppose that $\Psi_0(1)\geqslant 1$. We further assume that
     \begin{align}
        \label{hta:hta3}
         \left|\FI_{B_{R}} \inner{A_{0}(D\u)}{D\varphi}\,dx \right| \leqslant \sigma
         \norm{D\varphi}_{L^{\infty}(B_{R})} \text{ holds for } \varphi\in C^{\infty}_{0}(B_{R}).
     \end{align}
     Then there exists $h\in \u + W^{1,\Psi_0}_{0}(B_{R})$ such that
     \begin{align}
        \label{hta:hta4}
         \FI_{B_{R}} \inner{A_0(Dh)}{D\varphi}\,dx = 0 \text{ for all } \varphi\in C^{\infty}_{0}(B_{R}),
     \end{align}
     \begin{align}
        \label{hta:hta5}
        \begin{split}
         \FI_{B_{R}} [\Psi_0(|Dh|)]^{1+\delta_0}\,dx &\leqslant c(n,s(\Psi_0),\nu,L,\delta_1,c_0,c_1)
         \\&
         \text{ for some positive } \delta_0\equiv \delta_0(n,s(\Psi_0),\nu,L,\delta_1),
        \end{split}
     \end{align}
     \begin{align}
        \label{hta:hta6}
         \FI_{B_{R}} \left( |V_{G}(D\u)-V_{G}(Dh)|^2 + a_0 |V_{H_{a}}(D\u)-V_{H_{a}}(Dh)|^2 + b_0 |V_{H_{b}}(D\u)-V_{H_{b}}(Dh)|^2 \right)\,dx \leqslant \bar{c}\sigma^{s_1},
     \end{align}
     and 
     \begin{align}
        \label{hta:hta7}
         \FI_{B_{R}}  \Psi_{0}\left( \left|\frac{\u-h}{R} \right|\right) \,dx \leqslant 
         \bar{c}\sigma^{s_0}
     \end{align}
     for some constants with dependence as  $s_1 \equiv s_1(n,s(\Psi_0),\delta_1,c_0) > 0$, $s_0 \equiv s_0(n,s(\Psi_0),\delta_1,c_0)>0$
     and $\bar{c}\equiv \bar{c}(n,s(\Psi_0),\nu,L,\delta_1,c_0,c_1) \geqslant 1$.
 \end{lem}

 \begin{proof}
       By the standard approximation argument, if \eqref{hta:hta3} holds for all functions $\varphi \in C^{\infty}_{0}(B_{R})$, then it also holds for all functions $\varphi\in W^{1,\infty}_{0}(B_{R})$.
       The proof falls in three steps.
       
\textbf{Step 1: Truncation.} The standard energy estimate and \eqref{hta:hta1} give us
\begin{align}
    \label{hta:hta8}
    \FI_{B_{R}} \Psi_0(|Dh|)\,dx \leqslant \FI_{B_{R}} \Psi_0(|D\u|)\,dx \leqslant
    c(n,s(\Psi_0),\nu,L)c_0.
\end{align}
By applying Theorem \ref{hta:thm_cz}, there exists a positive exponent $\delta_0\equiv \delta_0(n,s(\Psi_0),\nu,L,\delta_1)$ satisfying
\begin{align}
    \label{hta:hta9}
    \FI_{B_{R}} [\Psi_0(|Dh|)]^{1+\delta_0}\,dx \leqslant c\FI_{B_{R}} [\Psi_0(|D\u|)]^{1+\delta_0}\,dx \leqslant
    c(n,s(\Psi_0),\nu,L,\delta_1,c_0,c_1),
\end{align}
which is \eqref{hta:hta5}. We now set $f:= \u-h\in W^{1,\Psi_0}_{0}(B_{R})$ and let $\lambda \geqslant 1$ to be chosen later. We consider $f_{\lambda} \in W^{1,\infty}_{0}(B_{R})$ provided by Theorem \ref{hta:thm_truncation}, which satisfies \eqref{hta:truncation1} and \eqref{hta:truncation2}. By these properties, Chebyshev's inequality and then the maximal function theorem for Orlicz spaces (see for instance \cite[Proposition 1.2]{FS1}), we have

\begin{align}
    \label{hta:hta10}
    \begin{split}
        \frac{|\{ f\neq f_{\lambda} \}|}{|B_{R}|} 
        &
        \leqslant 
        \frac{| B_{R}\cap \{ M(|Df|) > \lambda \} |}{|B_{R}|} \leqslant
        \frac{1}{[\Psi_0(\lambda)]^{1+\delta_0}}\FI_{B_{R}} [\Psi_0(M(|Df|))]^{1+\delta_0}\,dx
        \\&
        \leqslant
        \frac{c}{[\Psi_0(\lambda)]^{1+\delta_0}}\FI_{B_{R}} [\Psi_{0}(|Df|)]^{1+\delta_0}\,dx
        \\&
        \leqslant
        \frac{c}{[\Psi_0(\lambda)]^{1+\delta_0}}\left[\FI_{B_{R}} [\Psi_0(|D\u|)]^{1+\delta_0}\,dx +  \FI_{B_{R}}[\Psi_0(|Dh|)]^{1+\delta_0}\,dx  \right]
        \leqslant
        \frac{c}{[\Psi_0(\lambda)]^{1+\delta_0}}
    \end{split}
\end{align}
with $c\equiv c(n,s(\Psi_0), \nu, L, \delta_1,c_0,c_1)$, where we have used 
\eqref{hta:hta5} and \eqref{hta:hta9}. Now we test the equation \eqref{cp:cp1} against $f_{\lambda}$ to obtain 
\begin{align}
    \label{hta:hta12}
    \begin{split}
        \Gamma_{1} &:= \FI_{B_{R}} \inner{A_{0}(D\u)-A_{0}(Dh)}{Df_{\lambda}}\chi_{\{f=f_{\lambda}\}}\,dx
    \\& =
    \FI_{B_{R}}\inner{A_0(D\u)}{Df_{\lambda}}\,dx- \FI_{B_{R}} \inner{A_{0}(D\u)-A_{0}(Dh)}{Df_{\lambda}}\chi_{\{f\neq f_{\lambda}\}}\,dx =: \Gamma_{2} + \Gamma_{3}. 
    \end{split} 
\end{align}

Next we estimate each term appearing in the last equality. By using \eqref{cp:cp5}, we have 
\begin{align*}
    \Gamma_{1} \geqslant c\FI_{B_{R}} \left[ |V_{G}(D\u)-V_{G}(Dh)|^2 + a_{0} |V_{H_{a}}(D\u)-V_{H_{a}}(Dh)|^2 + b_{0} |V_{H_{b}}(D\u)-V_{H_{b}}(Dh)|^2 \right]\chi_{\{f=f_{\lambda}\}}\,dx
\end{align*}
with $c\equiv c(n,s(\Psi_0))$. Using \eqref{hta:hta3}, and then \eqref{hta:truncation1}, we get 
\begin{align*}
    |\Gamma_{2}| \leqslant \sigma \norm{Df_{\lambda}}_{L^{\infty}(B_{R})} \leqslant c(n)\sigma\lambda.
\end{align*}
For $\Gamma_{3}$, we fix $\varepsilon\in (0,1)$ to be chosen later and we estimate
\begin{align*}
    \begin{split}
        |\Gamma_{3}| & \leqslant \FI_{B_{R}} \left( |A_{0}(Dh)| + |A_0(D\u)|  \right) |Df_{\lambda}| \chi_{\{ f\neq f_{\lambda} \}}\,dx
        \\&
        \stackrel{\eqref{cp:cp2}}{\leqslant}
        L\norm{Df_{\lambda}}_{L^{\infty}(B_{R})}\FI_{B_{R}} \left[ \frac{\Psi_0(|D\u|)}{|D\u|} + \frac{\Psi_0(|Dh|)}{|Dh|} \right]\chi_{\{f\neq f_{\lambda}\}}\,dx
        \\&
        \leqslant
        \varepsilon\FI_{B_{R}}[\Psi_0(|D\u|) + \Psi_0(|Dh|)]\,dx + \frac{c}{\varepsilon^{s(\Psi_0)}}\Psi_{0}(\norm{Df_{\lambda}}_{L^{\infty}(B_{R})})\frac{|\{f\neq f_{\lambda}\}|}{|B_{R}|}
        \\&
        \leqslant
        c\left(\varepsilon + \frac{1}{ [\Psi_0(\lambda)]^{\delta_0}\varepsilon^{s(\Psi_0)}}\right)
    \end{split}
\end{align*}
with some $c\equiv c(n,s(\Psi_0),\nu,L,\delta_1,c_0,c_1)$, where in the last two inequalities we have used Lemma \ref{lem:nf1} together with \eqref{hta:truncation1} and \eqref{hta:hta8}. Merging the estimates for $\Gamma_{1},\Gamma_{2}$ and $\Gamma_{3}$ with \eqref{hta:hta12}, we deduce that 
\begin{align}
    \label{hta:hta13}
    \begin{split}
        &\FI_{B_{R}} \left( |V_{G}(D\u)-V_{G}(Dh)|^{2} + a_{0} |V_{H_{a}}(D\u)-V_{H_{a}}(Dh)|^2 +b_{0} |V_{H_{b}}(D\u)-V_{H_{b}}(Dh)|^2 \right)
        \chi_{\{f = f_{\lambda}\}}\,dx
        \\&
        \leqslant
        c_{*}\left( \sigma\lambda + \varepsilon + \frac{1}{[\Psi_0(\lambda)]^{\delta_0}\varepsilon^{s(\Psi_0)}} \right) 
        =: S(\sigma,\lambda,\varepsilon)
    \end{split}
\end{align}
for some constant $c_{*}\equiv c_{*}(n,s(\Psi_0),\nu,L,\delta_1,c_0,c_1)$, where $\varepsilon\in (0,1)$ is still to be chosen later. Now let us use a short notation for the simplicity
\begin{align}
    \label{hta:hta14}
    Z^2 : = |V_{G}(D\u)-V_{G}(Dh)|^2 +  a_{0} |V_{H_{a}}(D\u)-V_{H_{a}}(Dh)|^2 + b_0 |V_{H_{b}}(D\u)-V_{H_{b}}(Dh)|^2 
\end{align}
and fix $\theta\in (0,1)$, again to be chosen later. H\"older's inequality and \eqref{hta:hta13} imply 
\begin{align}
    \label{hta:hta15}
    \left( \FI_{B_{R}} Z^{2\theta}\chi_{\{f=f_{\lambda}\}}\,dx \right)^{\frac{1}{\theta}}
    \leqslant  S(\sigma,\lambda,\varepsilon).
\end{align}
Again using H\"older's inequality, we get 
\begin{align}
    \label{hta:hta16}
    \begin{split}
   & \left( \FI_{B_{R}} Z^{2\theta}\chi_{\{f\neq f_{\lambda}\}}\,dx \right)^{\frac{1}{\theta}}
    \leqslant
    \left( \frac{|\{f\neq f_{\lambda}\}|}{|B_{R}|} \right)^{\frac{1-\theta}{\theta}}
    \FI_{B_{R}} Z^2\,dx
    \\&
    \stackrel{\eqref{hta:hta10}}{\leqslant}
    c[\Psi_{0}(\lambda)]^{-\frac{(1-\theta)(1+\delta_0)}{\theta}}\FI_{B_{R}} [\Psi_0(|D\u|) + \Psi_0(|Dh|)]\,dx 
    \stackrel{\eqref{hta:hta8}}{\leqslant}
    c[\Psi_{0}(\lambda)]^{-\frac{(1-\theta)(1+\delta_0)}{\theta}}
    \end{split}
\end{align}
for some constant $c\equiv c(n,s(\Psi_0),\nu,L,\delta_1,c_0,c_1,\theta)$. Consequently, \eqref{hta:hta15} and \eqref{hta:hta16} yield that
\begin{align*}
    \left(\FI_{B_{R}} Z^{2\theta}\,dx \right)^{\frac{1}{\theta}}
    \leqslant c\left( S(\sigma,\lambda,\varepsilon) + [\Psi_0(\lambda)]^{-\frac{(1-\theta)(1+\delta_0)}{\theta}} \right)
\end{align*}
holds with again $c\equiv c(n,s(\Psi_0),\nu,L,\delta_1,c_0,c_1,\theta)$. Recalling  $S(\sigma,\lambda,\varepsilon)$ in \eqref{hta:hta13} and using Lemma \ref{lem:nf1}, we find
\begin{align*}
    \left(\FI_{B_{R}} Z^{2\theta}\,dx \right)^{\frac{1}{\theta}}\,dx \leqslant
    c\left( \sigma\lambda + \varepsilon + \lambda^{-\delta_0\left(\frac{1}{s(\Psi_0)} + 1\right)}\varepsilon^{-s(\Psi_0)} + \lambda^{-\left( \frac{1}{s(\Psi_0)} + 1  \right)\frac{(1-\theta)(1+\delta_0)}{\theta}} \right),
\end{align*}
where at this moment we have used the assumption that $\Psi_0(1)\geqslant 1$. Choosing $\lambda = \sigma^{-\frac{1}{2}}$ and $\varepsilon = \sigma^{s}$ with
$ s = \frac{\delta_0}{4s(\Psi_0)}\left(\frac{1}{s(\Psi_0)} + 1 \right)$, we obtain
\begin{align}
    \label{hta:hta17}
    \left(\FI_{B_{R}} \left( |V_{G}(D\u)-V_{G}(Dh)|^{2} + a_{0} |V_{H_{a}}(D\u)-V_{H_{a}}(Dh)|^2 + b_{0} |V_{H_{b}}(D\u)-V_{H_{b}}(Dh)|^2 \right)^{\theta}
        \,dx\right)^{\frac{1}{\theta}} \leqslant c\sigma^{m_0}
\end{align}
with constants $m_0 = \min\{\frac{1}{2}, \frac{\delta_0}{4s(\Psi_0)}\left(\frac{1}{s(\Psi_0)} + 1 \right), \left( \frac{1}{s(\Psi_0)} + 1  \right)\frac{(1-\theta)(1+\delta_0)}{2\theta}\}$ and \\ $c\equiv c(n,s(\Psi_0),\nu,L,\delta_1,c_0,c_1,\theta)$. Recall that $\theta$ is yet to be chosen.

\textbf{Step 2: Proof of \eqref{hta:hta6}.}
 By taking $\theta$ properly, we can deduce \eqref{hta:hta6} from \eqref{hta:hta17}. H\"older's inequality with conjugate exponents $\left(\frac{2(1+\delta_0)}{1+2\delta_0},2(1+\delta_0)\right)$ yields
 \begin{align}
     \label{hta:hta18}
     \FI_{B_{R}} Z^{2}\,dx = \FI_{B_{R}} Z\cdot Z\,dx  \leqslant \left( \FI_{B_{R}} Z^{\frac{2(1+\delta_0)}{1+2\delta_0}} \,dx\right)^{\frac{1+2\delta_0}{2(1+\delta_0)}} \left( \FI_{B_{R}} Z^{2(1+\delta_0)}\,dx \right)^{\frac{1}{2(1+\delta_0)}}.
 \end{align}
We now choose $\theta := \frac{1+\delta_0}{1+2\delta_0}\in (0,1)$ in \eqref{hta:hta17} in order to find that
\begin{align}
	\label{hta:hta18_1}
    \left( \FI_{B_{R}} Z^{\frac{2(1+\delta_0)}{1+2\delta_0}} \,dx\right)^{\frac{1+2\delta_0}{2(1+\delta_0)}} \leqslant c\sigma^{\frac{m_0}{2}}.
\end{align}

On the other hand, recalling \eqref{hta:hta14}
and \eqref{hta:hta9}, we have
\begin{align}
	\label{hta:hta18_2}
    \begin{split}
        &\FI_{B_{R}} Z^{2(1+\delta_0)}\,dx 
        \\&
        = \FI_{B_{R}}\left( |V_{G}(D\u)-V_{G}(Dh)|^2 +  a_{0} |V_{H_{a}}(D\u)-V_{H_{a}}(Dh)|^2 +  b_{0} |V_{H_{b}}(D\u)-V_{H_{b}}(Dh)|^2 \right)^{1+\delta_0}\,dx
        \\&\leqslant
        c\FI_{B_{R}} [\Psi_0(|D\u|)]^{1+\delta_0}\,dx + 
        c\FI_{B_{R}} [\Psi_0(|Dh|)]^{1+\delta_0}\,dx
        \leqslant
        c(n,s(\Psi_0),\nu,L,\delta_1,c_0,c_1).
    \end{split}
\end{align}
We combine  the estimates \eqref{hta:hta18}-\eqref{hta:hta18_2} to discover
\begin{align}
    \label{hta:hta19}
    \FI_{B_{R}}\left(|V_{G}(D\u)-V_{G}(Dh)|^2 +  a_{0} |V_{H_{a}}(D\u)-V_{H_{a}}(Dh)|^2 + b_{0} |V_{H_{b}}(D\u)-V_{H_{b}}(Dh)|^2\right)\,dx \leqslant c\sigma^{s_1},
\end{align}
where
\begin{align*}
    s_1 = \frac{1}{2}\min\left\{\frac{1}{2}, \frac{\delta_0}{4s(\Psi_0)}\left(\frac{1}{s(\Psi_0)} + 1 \right), \left( \frac{1}{s(\Psi_0)} + 1  \right)\frac{\delta_0}{2}\right\}
    \text{ and }
    c\equiv c(n,s(\Psi_0),\nu,L,\delta_0,c_0,c_1).
\end{align*}
\textbf{Step 3: Proof of \eqref{hta:hta7}.}
By applying Lemma \ref{lem:os} to $\Phi\equiv \Psi_0$ with $d_{0}\equiv 1$, we see that there exists $\theta_0\equiv \theta_0(n,s(\Psi_0))\in (0,1)$ such that 
 \begin{align*}
     \begin{split}
         \FI_{B_{R}} & \Psi_{0}\left(\left|\frac{\u-h}{R} \right| \right)\,dx
         \leqslant
         c\left(\FI_{B_{R}} [\Psi_{0}(|D\u-Dh|)]^{\theta_0}\,dx \right)^{\frac{1}{\theta_0}}
        \\&
         \leqslant
         c\left(\FI_{B_{R}} \left(\left[\Psi_{0}(|D\u| + |Dh|)\right]^{\frac{1}{2}}\frac{|D\u-Dh|}{\left(|D\u|+|Dh|\right)}\right)^{\theta_0} \left[\Psi_{0}(|D\u| + |Dh|)\right]^{\frac{\theta_0}{2}}\,dx\right)^{\frac{1}{\theta_0}}
         \\&
         \leqslant
         c\left( \FI_{B_{R}} \Psi_0(|D\u| + |Dh|)\frac{|D\u-Dh|^2}{(|D\u| + |Dh|)^2}\,dx \right)^{\frac{1}{2}}
         \left(\FI_{B_{R}} \left[\Psi_{0}(|D\u| + |Dh|)\right]^{\frac{\theta_0}{2-\theta_0}}\,dx \right)^{\frac{2-\theta_0}{2\theta_0}}
         \\&
         \leqslant
         c\left(\FI_{B_{R}} Z^{2}\,dx\right)^{\frac{1}{2}}\left(\FI_{B_{R}} \Psi_0(|D\u| + |Dh|)\,dx\right)^{\frac{1}{2}}
         \\&
         \leqslant
         c\sigma^{\frac{s_1}{2}}
         =
         c\sigma^{s_0}
     \end{split}
 \end{align*}
for some $c\equiv c(n,s(\Psi_0),\nu,L,\delta_1,c_0,c_1)$, where in the last display we have applied H\"older's inequality with conjugate exponents $\left(\frac{2}{\theta_0}, \frac{2}{2-\theta_0}\right)$, and finally used \eqref{defV2} with \eqref{hta:hta19}. This proves \eqref{hta:hta7}. The proof is complete.
\end{proof}


\section{Comparison estimates}
\label{sec:7}

Throughout this section we fix a ball $B_{2R}\equiv B_{2R}(x_0)\subset\O_{0}\Subset\O$ with $R\leqslant 1$ and some open subset $\O_0\Subset\O$. We consider the functional defined by 
\begin{align}
    \label{0ce:1}
    W^{1,1}(B_{2R})\ni \u\mapsto \mathcal{F}_{B_{2R}}(\u):= \I_{B_{2R}} F(x,(u)_{B_{2R}},D\u)\,dx,
\end{align}
where $u$ is a local minimizer of the functional $\F$ in \eqref{functional}. Now we consider a function $w\in u + W^{1,\Psi}_{0}(B_{R})$ being the solution to the following variational Dirichlet problem:

\begin{align}
    \label{0ce:2}
    \displaystyle
\begin{cases}
    w \mapsto \min\limits_{\u} \mathcal{F}_{B_{2R}}(\u)
    \\
    \u \in u + W^{1,\Psi}_{0}(B_{2R}).
\end{cases}
\end{align}
In the following we shall deal with first comparison estimates in order to remove $u$-dependence in the original functional $\mathcal{F}$ in \eqref{functional}.

\begin{lem}
	\label{lem:0ce}
	Let $w\in W^{1,\Psi}(B_{2R})$ be the solution to the variational problem \eqref{0ce:2} under the assumptions \eqref{sa:1}, \eqref{sa:2} and \eqref{omega2}. Let the coefficient functions $a(\cdot)\in C^{\omega_a}(\O)$ and $b(\cdot)\in C^{\omega_{b}}(\O)$ for $\omega_{a}, \omega_{b}$ being non-negative concave functions vanishing at the origin. Assume that one of the assumptions \eqref{ma:1}, \eqref{ma:2} and \eqref{ma:3} is satisfied. Then there exists a constant $c\equiv c(\data(\O_0))$ such that 
	\begin{align}
		\label{0ce:3}
		\begin{split}
		\FI_{B_{2R}}& \left( |V_{G}(Du)-V_{G}(Dw)|^{2} + a(x)|V_{H_{a}}(Du)-V_{H_{a}}(Dw)|^{2} + b(x)|V_{H_{b}}(Du)-V_{H_{b}}(Dw)|^{2}  \right)\,dx
		\\&
		\leqslant
		c\omega(R^{\gamma})\FI_{B_{2R}}\Psi(x,|Du|)\,dx
		\end{split}
	\end{align}
	holds, where $\gamma\equiv \gamma(\data(\O_0))$ is the H\"older exponent determined via Theorem \ref{thm:hc}. Moreover, the following estimates hold true:
	\begin{align}
		\label{0ce:4}
		\FI_{B_{2R}}\Psi(x,|Dw|)\,dx \leqslant \frac{L}{\nu} \FI_{B_{2R}}\Psi(x,|Du|)\,dx,
	\end{align}
	\begin{align}
		\label{0ce:5}
		\norm{w}_{L^{\infty}(B_{2R})} \leqslant \norm{u}_{L^{\infty}(B_{2R})},
	\end{align}
	\begin{align}
		\label{0ce:6}
		\osc\limits_{B_{2R}}w \leqslant \osc\limits_{B_{2R}}u
	\end{align}
	and 
	\begin{align}
		\label{0ce:7}
		\FI_{B_{2R}}\Psi\left(x,\left|\frac{u-w}{R} \right| \right)\,dx
		\leqslant
		c [\omega(R^{\gamma})]^{\frac{1}{2}}\FI_{B_{2R}}\Psi(x,|Du|)\,dx
	\end{align}
	for some constant $c\equiv c(\data(\O_0))$, where in the case that \eqref{ma:3} is considered, $\gamma$ appearing in \eqref{0ce:3} and \eqref{0ce:7} is the same as in the assumption \eqref{ma:3}. 
\end{lem}

\begin{proof}
	   The proof is very standard and we shall follow the structure of the proof of \cite[Lemma 4]{BCM3}. The Euler-Lagrange equation of the functional $\mathcal{F}_{B_{2R}}$, which is 
    \begin{align}
        \label{0ce:8}
        \FI_{B_{2R}} \inner{D_{z} F(x,(u)_{B_{2R}},Dw)}{D\varphi}\,dx = 0,
    \end{align}
    holds for any function $\varphi \in W^{1,\Psi}_{0}(B_{2R})$ (see for instance \cite[Lemma 5.2]{BBO1}). The minimality and growth condition \eqref{sa:1} imply that 
    \begin{align}
        \label{0ce:9}
        \begin{split}
            \FI_{B_{2R}}\Psi(x,|Dw|)\,dx 
            &\leqslant
            \frac{1}{\nu} \FI_{B_{2R}} F(x,(u)_{B_{2R}},Dw)\,dx
            \\&
            \leqslant
            \frac{1}{\nu} \FI_{B_{2R}} F(x,(u)_{B_{2R}},Du)\,dx
            \leqslant
            \frac{L}{\nu}\FI_{B_{2R}}\Psi(x,|Du|)\,dx,
      \end{split}
    \end{align}
    which proves \eqref{0ce:4}. Therefore, we conclude with
    \begin{align}
        \label{0ce:10}
        \FI_{B_{2R}} \inner{D_{z} F(x,(u)_{B_{2R}},Dw)}{Du-Dw}\,dx = 0.
    \end{align}
Letting $u_{B_{2R}}^{+}:= \sup\limits_{x\in B_{2R}} u(x)$ and $u_{B_{2R}}^{-}:= \inf\limits_{x\in B_{2R}}u(x)$, the minimality of $w$ yields
\begin{align*}
    \mathcal{F}_{B_{2R}}(w) \leqslant \mathcal{F}_{B_{2R}}\left(\min\{w,u_{B_{2R}}^{+}\}\right)\qquad
    \text{and} \qquad \mathcal{F}_{B_{2R}}(w) \leqslant \mathcal{F}_{B_{2R}}(\max\{w,u_{B_{2R}}^{-}\}). 
\end{align*}
Consequently, the last display together with \eqref{sa:1} gives us
\begin{align*}
    \I_{B_{2R}\cap \left\{w\geqslant u_{B_{2R}}^{+}\right\}}\Psi(x,|Dw|)\,dx = 0
    \qquad
    \text{and}
    \qquad
    \I_{B_{2R}\cap \left\{w\leqslant u_{B_{2R}}^{-}\right\}}\Psi(x,|Dw|)\,dx = 0.
\end{align*}
By coarea formula, we get that 
\begin{align}
    \label{0ce:11}
    \inf\limits_{x\in B_{2R}}u \equiv u_{B_{2R}}^{-}\leqslant w(x) \leqslant u_{B_{2R}}^{+}\equiv \sup\limits_{x\in B_{2R}}u(x)\qquad \text{a.e.  } x\in B_{2R}.
\end{align}
This proves \eqref{0ce:5} and \eqref{0ce:6}.
    Using \eqref{defV1_2} and \eqref{0ce:10} together with the minimality of $u$ and $w$, we have that 
    \begin{align}
        \label{0ce:12}
        \begin{split}
            \FI_{B_{2R}}& \left( |V_{G}(Du)-V_{G}(Dw)|^2  + a(x)|V_{H_{a}}(Du)-V_{H_{a}}(Dw)|^2
            + b(x)|V_{H_{b}}(Du)-V_{H_{b}}(Dw)|^2 \right)\,dx
            \\&
            \stackrel{\eqref{0ce:10}}{=}
            \FI_{B_{2R}} \left( |V_{G}(Du)-V_{G}(Dw)|^2  + a(x)|V_{H}(Du)-V_{H}(Dw)|^2 + b(x)|V_{H_{b}}(Du)-V_{H_{b}}(Dw)|^2 \right)\,dx
            \\ &
            \qquad + c_{*}\FI_{B_{2R}} \inner{D_{z} F(x,(u)_{B_{2R}},Dw)}{Du-Dw}\,dx
            \\&
            \leqslant
            c_{*}\FI_{B_{2R}} \left[ F(x,(u)_{B_{2R}},Du)-F(x,(u)_{B_{2R}},Dw) \right]\,dx
            \\&
            =
            c_{*}\FI_{B_{2R}} \left[ F(x,(u)_{B_{2R}},Du)-F(x,u,Du) \right]\,dx
             +
            c_{*}\FI_{B_{2R}} \left[ F(x,u,Du)-F(x,w,Dw) \right]\,dx
            \\&
            \,\, +
            c_{*}\FI_{B_{2R}} \left[F(x,w,Dw)-F(x,(w)_{B_{2R}},Dw) \right]\,dx
             +
            c_{*}\FI_{B_{2R}} \left[F(x,(w)_{B_{2R}},Dw)-F(x,(u)_{B_{2R}},Dw) \right]\,dx
            \\&
            =: c_{*}\sum\limits_{i=1}^{4} I_{i}
        \end{split}
    \end{align}
with $c_{*}\equiv c_{*}(n,s(G),s(H_{a}),s(H_{b}),\nu)$. Now we estimate each term $I_{i}$ for $i\in \{1,2,3,4\}$ in the last display. We have
\begin{align}
    \label{0ce:13}
    \begin{split}
         I_{1} &\stackrel{\eqref{sa:2}}{\leqslant}
         c\FI_{B_{2R}} \omega(|u-(u)_{B_{2R}}|)\Psi(x,|Du|)\,dx
         \\&
         \stackrel{\eqref{hc:1}, \eqref{hc:3}}{\leqslant}
         c\omega(2[u]_{0,\gamma;\O_0}R^{\gamma})\FI_{B_{2R}} \Psi(x,|Du|)\,dx
         \\&
         \stackrel{\eqref{concave1}}{\leqslant}
         c(\data(\O_0))\omega(R^{\gamma})\FI_{B_{2R}}\Psi(x,|Du|)\,dx,
    \end{split}
\end{align}
where in the last display we have also used the fact that $\omega(\cdot)$ is concave. The minimality of $u$ implies
\begin{align}
    \label{0ce:14}
    I_{2} \leqslant 0.
\end{align}
We have therefore
\begin{align}
    \label{0ce:15}
    \begin{split}
        I_{3} 
        &\stackrel{\eqref{sa:2}}{\leqslant}
        c\FI_{B_{2R}}\omega(|w-(w)_{B_{2R}}|)\Psi(x,|Dw|)\,dx
        \\&
        \leqslant
        c\FI_{B_{R}}\omega\left(\osc\limits_{B_{2R}}w\right)\Psi(x,|Dw|)\,dx
        \\&
        \stackrel{\eqref{0ce:6}}{\leqslant}
        c\omega\left( \osc\limits_{B_{2R}} u \right)\FI_{B_{2R}}\Psi(x,|Dw|)\,dx
        \\&
        \stackrel{\eqref{hc:1}, \eqref{hc:3}}{\leqslant}
        c\omega(2[u]_{0,\gamma;\O_0}R^{\gamma})\FI_{B_{2R}}\Psi(x,|Du|)\,dx
        \\&
        \stackrel{\eqref{concave1}}{\leqslant}
        c(\data(\O_{0}))\omega(R^{\gamma})\FI_{B_{2R}}\Psi(x,|Du|)\,dx.
    \end{split}
\end{align}
Observing that 
\begin{align}
    \label{0ce:16}
    |(w)_{B_{2R}}-(u)_{B_{2R}}| \stackrel{\eqref{0ce:11}}{\leqslant} \osc\limits_{B_{2R}} u, 
\end{align}
as in the estimate for $I_1$, we still have 
\begin{align}
    \label{0ce:17}
    I_{4} \leqslant c\omega(R^{\gamma})\FI_{B_{2R}}\Psi(x,|Du|)\,dx.
\end{align}
Inserting all the estimates obtained for $I_{i}$ with $i\in \{1,2,3,4\}$ into \eqref{0ce:12} completes the proof of \eqref{0ce:3}.

 Let us now prove \eqref{0ce:7}. By Theorem \ref{thm:sp} with $d \equiv 1$, there exists $\theta_1\equiv \theta_1(n,s(G),s(H_{a}),s(H_{b}))\in (0,1)$ such that 
\begin{align}
    \label{0ce:18}
    \begin{split}
       J & :=  \FI_{B_{2R}}\Psi\left(x,\left|\frac{u-w}{R} \right| \right)\,dx
        \leqslant
        c\left(\FI_{B_{2R}}[\Psi(x,|Du-Dw|)]^{\theta_1}\,dx \right)^{\frac{1}{\theta_1}}
        \\&
        \leqslant
        c \left(\FI_{B_{2R}} \left([\Psi(x,|Du|+|Dw|)]^{\frac{1}{2}}\frac{|Du-Dw|}{|Du|+|Dw|}\right)^{\theta_1}
        [\Psi(x,|Du|+|Dw|)]^{\frac{\theta_1}{2}}\,dx\right)^{\frac{1}{\theta_1}},
    \end{split}
\end{align}
where in the last inequality of the last display we have used \eqref{defV3} for $\Psi$. Applying H\"older's inequality with conjugate exponents $\left( \frac{2}{\theta_1},\frac{2}{2-\theta_1} \right)$ to the right hand side of the last display and \eqref{defV2}, we get 
\begin{align}
    \label{0ce:19}
    \begin{split}
       J
        &\leqslant
        c\left( \FI_{B_{2R}} \Psi(x,|Du|+|Dw|)\frac{|Du-Dw|^2}{(|Du|+|Dw|)^2}\,dx \right)^{\frac{1}{2}}
        \left(\FI_{B_{2R}} [\Psi(x,|Du|+|Dw|)]^{\frac{\theta_1}{2-\theta_1}}\,dx \right)^{\frac{2-\theta_1}{2\theta_1}}
        \\&
        \leqslant
        c\left( \FI_{B_{2R}} |V_{\Psi}(x,Du)-V_{\Psi}(x,Dw)|^2\,dx \right)^{\frac{1}{2}} \left(\FI_{B_{2R}} \Psi(x,|Du|+|Dw|)\,dx \right)^{\frac{1}{2}}
        \\&
        \leqslant
        c(\data(\O_0)) [\omega(R^{\gamma})]^{\frac{1}{2}} \FI_{B_{2R}} \Psi(x,|Du|)\,dx,
    \end{split}
\end{align}
where in the last inequality of the above display we have used \eqref{0ce:3}, and then \eqref{0ce:4}. Combining the last two displays we arrive at \eqref{0ce:7}.
\end{proof}


Next we consider the functional defined by 
\begin{align}
    \label{1ce:1}
    W^{1,1}(B_{R})\ni \u\mapsto \mathcal{F}_{c}(\u):= \I_{B_{R}} F_{c}(x,D\u)\,dx,
\end{align}
where  the density function is given by
\begin{align}
	\label{1ce:1_0}
	F_{c}(x,z):= F_{G}\left(x_{c}, (u)_{B_{2R}},z \right) + a(x)F_{H_{a}}\left(x_{c}, (u)_{B_{2R}},z \right) + b(x)F_{H_{b}}\left(x_{c}, (u)_{B_{2R}},z \right)
\end{align}
for some fixed point $x_c\in B_{R}$ and for every $x\in\O$ and $z\in\R^n$. Now we consider a function $w_{c}\in w + W^{1,\Psi}_{0}(B_{R})$ being the solution to the following variational Dirichlet problem:

\begin{align}
    \label{1ce:2}
\begin{cases}
    w_{c} \mapsto \min\limits_{v} \mathcal{F}_{c}(v)
    \\
    v \in w + W^{1,\Psi}_{0}(B_{R}),
\end{cases}
\end{align}
where $w\in W^{1,\Psi}(B_{2R})$ is the solution to the variational problem \eqref{0ce:2}.

\begin{lem}
	\label{lem:1ce}
	Let $w_{c}\in W^{1,\Psi}(B_{R})$ be the solution to the variational problem \eqref{1ce:2} under the assumptions \eqref{sa:1}, \eqref{sa:2} and \eqref{omega2}. Let  the coefficient functions $a(\cdot)\in C^{\omega_a}(\O)$ and $b(\cdot)\in C^{\omega_{b}}(\O)$ for $\omega_{a}, \omega_{b}$ being non-negative concave functions vanishing at the origin. Assume that one of the assumptions \eqref{ma:1}, \eqref{ma:2} and \eqref{ma:3} is satisfied. Then there exists a constant $c\equiv c(\data(\O_0))$ such that 
	\begin{align}
		\label{1ce:3}
		\begin{split}
		\FI_{B_{R}}& \left( |V_{G}(Dw)-V_{G}(Dw_{c})|^{2} + a(x)|V_{H_{a}}(Dw)-V_{H_{a}}(Dw_{c})|^{2} + b(x)|V_{H_{b}}(Dw)-V_{H_{b}}(Dw_{c})|^{2}  \right)\,dx
		\\&
		\leqslant
		c\omega(R)\FI_{B_{R}}\Psi(x,|Du|)\,dx.
		\end{split}
	\end{align}
	 Moreover, the following estimates hold true:
	\begin{align}
		\label{1ce:4}
		\FI_{B_{R}}\Psi(x,|Dw_{c}|)\,dx \leqslant \frac{L}{\nu} \FI_{B_{R}}\Psi(x,|Dw|)\,dx,
	\end{align}
	\begin{align}
		\label{1ce:5}
		\norm{w_{c}}_{L^{\infty}(B_{R})} \leqslant \norm{w}_{L^{\infty}(B_{R})},
	\end{align}
	\begin{align}
		\label{1ce:6}
		\osc\limits_{B_{R}}w_{c} \leqslant \osc\limits_{B_{R}}w
	\end{align}
	and 
	\begin{align}
		\label{1ce:7}
		\FI_{B_{R}}\Psi\left(x,\left|\frac{w-w_{c}}{R} \right| \right)\,dx
		\leqslant
		c [\omega(R)]^{\frac{1}{2}}\FI_{B_{R}}\Psi(x,|Dw|)\,dx
	\end{align}
	for some constant $c\equiv c(\data)$. Finally, there exists a higher integrability exponent $\delta_0\equiv \delta_0(\data)$ with $\delta_0\leqslant \delta$ with $\delta$ having been determined via Theorem \ref{thm:hi}, and a constant $c\equiv c(\data)$ such that 
	\begin{align}
		\label{1ce:8}
		\left(\FI_{B_{R/2}}[\Psi(x,|Dw_{c}|)]^{1+\delta_0}\,dx \right)^{\frac{1}{1+\delta_0}}
		\leqslant 
		c \FI_{B_{R}}\Psi(x,|Dw_{c}|)\,dx.
	\end{align}
\end{lem}
\begin{proof}
	Essentially, the proof is similar to the proof of Lemma \ref{lem:0ce}. The estimates \eqref{1ce:4}-\eqref{1ce:6} can be obtained as for \eqref{0ce:4}-\eqref{0ce:6}. We now focus on proving \eqref{1ce:3}. The Euler-Lagrange equation arising from the functional $\F_{c}$ defined in \eqref{1ce:1}
	\begin{align}
		\label{1ce:9}
		\FI_{B_{R}}\inner{D_{z}F_{c}(x,Dw_c)}{D\varphi}\,dx = 0
	\end{align}
	is valid, whenever $\varphi\in W^{1,\Psi}_{0}(B_{R})$. Then using \eqref{sa:2}, we have 
	\begin{align}
		\label{1ce:10}
		\begin{split}
			\FI_{B_{R}}& \left( |V_{G}(Dw)-V_{G}(Dw_{c})|^{2} + a(x)|V_{H_{a}}(Dw)-V_{H_{a}}(Dw_{c})|^{2} + b(x)|V_{H_{b}}(Dw)-V_{H_{b}}(Dw_{c})|^{2}  \right)\,dx
			\\&
			\leqslant
			c\FI_{B_{R}}\inner{D_{z}F_{c}(x,Dw)-D_{z}F_{c}(x,Dw_c)}{Dw-Dw_c}\,dx
			\\&
			\leqslant
			c\FI_{B_{R}}|D_{z}F_{G}(x_c,(u)_{B_{2R}},Dw)-D_{z}F_{G}(x,(u)_{B_{2R}},Dw)||Dw-Dw_c|\,dx
			\\&
			\quad
			+
			c\FI_{B_{R}}a(x)|D_{z}F_{H_{a}}(x_c,(u)_{B_{2R}},Dw)-D_{z}F_{H_{a}}(x,(u)_{B_{2R}},Dw)||Dw-Dw_c|\,dx
			\\&
			\quad
			+
			c\FI_{B_{R}}b(x)|D_{z}F_{H_{b}}(x_c,(u)_{B_{2R}},Dw)-D_{z}F_{H_{b}}(x,(u)_{B_{2R}},Dw)||Dw-Dw_c|\,dx
			\\&
			\leqslant
			c\omega(R)\FI_{B_{R}}\Psi(x,|Dw|)\,dx
		\end{split}
	\end{align}
for some constant $c\equiv c(n,s(G),s(H_{a}),s(H_{b}),\nu,L)$. This proves \eqref{1ce:3}, and \eqref{1ce:7} follows from this estimate together with applying the arguments used in \eqref{0ce:18}-\eqref{0ce:19}.
Since $w_{c}$ is a $L/\nu$-minimizer of the functional $\F_{c}$ defined in \eqref{1ce:2}, we are able to apply Lemma \ref{lem:cacc} with the choices of $\u\equiv w_{c}$, $\rho\equiv R/2$, $r\equiv R$ and $k\equiv (w_{c})_{B_{R}}$. In turn, it gives us that
\begin{align}
    \label{1ce:11}
    \FI_{B_{R/2}}\Psi(x,|Dw_c|)\,dx 
    \leqslant
    c\FI_{B_{R}}\Psi\left(x, \left|\frac{w_{c}-(w_{c})_{B_{R}}}{R}\right| \right)\,dx
\end{align}
holds with $c\equiv c(n,s(G),s(H_{a}),s(H_{b}),L,\nu)$. Then applying Remark \ref{rmk:sp}, there exists a positive exponent $\theta\equiv \theta(n,s(G),s(H_{a}),s(H_{b}))\in (0,1)$ such that
\begin{align}
		\label{1ce:12}
		\FI_{B_{R/2}}\Psi\left(x,|Dw_{c}| \right)\,dx \leqslant
		c\bar{\kappa}_{sp}\left[ \FI_{B_{R}} [\Psi(x,|Dw_{c}|)]^{\theta}\,dx \right]^{\frac{1}{\theta}}
	\end{align}
	holds with some constant $c\equiv c(n,s(G),s(H_{a}),s(H_{b}),L,\nu, \omega_{a}(1),\omega_{b}(1))$, where
	\begin{subequations}
	 \begin{empheq}[left={\bar{\kappa}_{sp}=\empheqlbrace}]{align}
        \qquad & 1+\left([a]_{\omega_{a}} + [b]_{\omega_{b}} \right)\left(\lambda_{1} +  \lambda_{1}\left( \I_{B_{R}}G(|Dw_{c}|)\,dx \right)^{\frac{1}{n}}\right)   &\text{if }& \eqref{ma:1} \text{ is considered, } \label{1ce:13_1}\\
        & 1+ \left([a]_{\omega_{a}} + [b]_{\omega_{b}} \right)\left(\lambda_{2} + \lambda_{2} \norm{w_{c}}_{L^{\infty}(B_{R})}\right)  &\text{if }&  \eqref{ma:2} \text{ is considered, }  \label{1ce:13_2}\\
       &1+ \left([a]_{\omega_{a}} + [b]_{\omega_{b}} \right)\left(\lambda_{3} + \lambda_{3}\left[ R^{-\gamma}\osc\limits_{B_{R}} w_{c} \right]^{\frac{1}{1-\gamma}}\right)
         &\text{if }&  \eqref{ma:3} \text{ is considered. } \label{1ce:13_3}
      \end{empheq}
	\end{subequations}
Furthermore, taking into account \eqref{0ce:4}-\eqref{0ce:6} and \eqref{1ce:4}-\eqref{1ce:6} in the last display, we conclude that 
\begin{align}
    \label{1ce:14}
    \FI_{B_{R/2}}\Psi(x,|Dw_{c}|)\,dx 
    \leqslant
    c\left[\FI_{B_{R}}[\Psi\left(x, |Dw_{c}| \right)]^{\theta}\,dx\right]^{\frac{1}{\theta}}
\end{align}
holds for some constant $\theta\equiv\theta(n,s(G),s(H_{a}),s(H_{b}))\in (0,1)$ and $c\equiv c(\data)$. The estimate \eqref{1ce:8} follows from applying a variant of Gehring's lemma.
\end{proof}


To go further let us introduce the excess functional defined by 
\begin{align}
	\label{excess2}
	E(v,B_{r}):= \left( \Psi_{B_{2r}}^{-} \right)^{-1}\left(\FI_{B_{r}}\Psi_{B_{2r}}^{-}\left(\left|\frac{v-(v)_{B_{r}}}{2r}\right| \right)\,dx\right)
\end{align}
for any function $v\in L^{1}(B_{2r})$ and a ball $B_{2r}\subset\O$, where we note that   $\left( \Psi_{B_{2r}}^{-} \right)^{-1}$ is the inverse function of $\Psi_{B_{2r}}^{-}$. By the convexity of $\Psi_{B_{2r}}^{-}$ together with Lemma \ref{lem:nf1} and Remark \ref{rmk:nf2}, one can see that
\begin{align}
	\label{excess4}
	E(v,B_{r}) \leqslant c\left( \Psi_{B_{2r}}^{-} \right)^{-1}\left(\FI_{B_{r}}\Psi_{B_{2r}}^{-}\left(\left|\frac{v-v_0}{2r}\right| \right)\,dx\right)
\end{align}
for some constant $c\equiv c(s(G)+s(H_{a})+s(H_{b}))$, whenever $v_0\in\R$ is an arbitrary number.

\begin{lem}
	\label{lem:2ce}
	Let $u\in W^{1,\Psi}(\O)$ be a local minimizer of the functional $\F$ defined in \eqref{functional} under the assumptions \eqref{sa:1}, \eqref{sa:2} and \eqref{omega2}. Let $w_{c}\in W^{1,\Psi}(B_{R})$ be the solution to the variational problem \eqref{1ce:2}. 
	If one of the assumptions \eqref{mth:md:1}-\eqref{mth:md:5}
is satisfied, then for every $\varepsilon^{*}\in (0,1)$, there exists a positive radius 
	 \begin{align}
	 	\label{2ce:1}
	 	R^{*}\equiv R^{*}(\data(\O_0),\varepsilon^{*})
	 \end{align}
such that
	\begin{align}
		\label{2ce:2}
		\FI_{B_{\tau R}}\Psi_{B_{R}}^{-}\left(\left|\frac{w_{c}-(w_{c})_{B_{\tau R}}}{\tau R} \right| \right)\,dx
		\leqslant
		c\left(1+\tau^{-(n+s(\Psi)+1)}\varepsilon^{*} \right)\FI_{B_{R/2}}\Psi_{B_{R}}^{-}\left(\left|\frac{w_{c}-(w_{c})_{B_{R/2}}}{R} \right| \right)\,dx
	\end{align}
	for some constant $c\equiv c\left(\data(\O_0)\right)$, whenever $\tau\in (0,1/16)$ and $R\leqslant R^{*}$.
	
\end{lem}
\begin{proof}
	We assume $E(w_{c},B_{R/2})>0$, otherwise \eqref{2ce:2} is trivial. For the sake of simplicity during the proof, we write 
	\begin{align}
		\label{2ce:E_1E_2}
		E(R):= E(w_{c},B_{R/2}) = \left(\Psi_{B_{R}}^{-} \right)^{-1}\left( \FI_{B_{R/2}}\Psi_{B_{R}}^{-}\left(\left| \frac{w_{c}-(w_{c})_{B_{R/2}}}{R}\right| \right)\,dx \right).
	\end{align}
	
	 The proof falls in several delicate steps.
	
	\textbf{Step 1: Initial settings on $w_{c}$.} Applying Lemma \ref{lem:cacct} to $B_{R/2}$ with $k\equiv (w_{c})_{B_{R/2}}$, we have 
	\begin{align}
		\label{2ce:3}
		\FI_{B_{R/4}}\Psi\left(x, |Dw_{c}| \right)\,dx
		\leqslant
		c \FI_{B_{R/2}}\Psi_{B_{R}}^{-}\left( \left|\frac{w_{c}-(w_{c})_{B_{R/2}}}{R}\right| \right)\,dx
	\end{align}
	for some constant $c\equiv c(\data)$. Moreover, it follows from Lemma \ref{lem:1ce} that there exists a higher integrability exponent $\delta_{0}\equiv \delta_{0}(\data)$ such that
	\begin{align}
		\label{2ce:4}
		\left(\FI_{B_{R/8}}\left[\Psi(x,|Dw_{c}|)\right]^{1+\delta_0}\,dx \right)^{\frac{1}{1+\delta_{0}}}
		\leqslant
		c\FI_{B_{R/4}}\Psi(x,|Dw_{c}|)\,dx
	\end{align}
	for a constant $c\equiv c(\data)$.
	
	\textbf{Step 2: Scaling.} We set the scaled functions of $w_{c}(\cdot)$, $a(\cdot)$ and $b(\cdot)$ in the ball $B_{1}$ by 
	\begin{subequations}
	 \begin{empheq}[left={\empheqlbrace}]{align}
          & \bar{w}_{c}(x):= \frac{w_{c}(x_0+Rx)-(w_{c})_{B_{R/2}}}{E(R)R}, \label{2ce:5_1}\\
         &\bar{a}(x):= a(x_0+Rx)\frac{H_{a}(E(R))}{\Psi_{B_{R}}^{-}\left( E(R) \right)}\quad\text{and}\quad 
         \bar{b}(x):= b(x_0+Rx)\frac{H_{b}(E(R))}{\Psi_{B_{R}}^{-}\left( E(R) \right)}\label{2ce:5_2}.
      \end{empheq}
	\end{subequations}
	for every $x\in B_{1}$. Now we define the control function and energy integrand associated to our scaling in \eqref{2ce:5_1}-\eqref{2ce:5_2} as
	\begin{subequations}
	 \begin{empheq}[left={\empheqlbrace}]{align}
          & \bar{\Psi}(x,|z|):= \bar{G}(|z|) + \bar{a}(x) \bar{H}_{a}(|z|) + \bar{b}(x) \bar{H}_{b}(|z|)  \label{2ce:6_1},\\
          & \bar{F}(x,z):= \bar{F}_{G}(z) + \bar{a}(x)\bar{F}_{H_{a}}(z) + \bar{b}(x)\bar{F}_{H_{b}}(z), \label{2ce:6_2} \\
         &\bar{F}_{G}(z):= \frac{F_{G}(x_c,(u)_{B_{2R}}, E(R)z)}{\Psi_{B_{R}}^{-}\left( E(R)\right)},\quad
         \bar{F}_{H_{a}}(z):= \frac{F_{H_{a}}(x_c,(u)_{B_{2R}}, E(R)z)}{H_{a}\left( E(R)\right)},\label{2ce:6_3}\\
        &\bar{F}_{H_{b}}(x,z):= \frac{F_{H_{b}}(x_c,(u)_{B_{2R}}, E(R)z)}{H_{b}\left( E(R)\right)}\quad\text{and}\quad
        \bar{A}(x,z):= D_{z}\bar{F}(x,z)\label{2ce:6_4}
      \end{empheq}
	\end{subequations}
for every $x\in B_{1}$ and $z\in\R^n$, where the point $x_{c}\in B_{R}$ has been fixed in \eqref{1ce:2} and
\begin{align}
	\label{2ce:7}
           \bar{G}(t):= \frac{G(E(R)t)}{\Psi_{B_{R}}^{-}\left( E(R) \right)},\quad
	\bar{H}_{a}(t):= \frac{H_{a}(E(R)t)}{H_{a}\left( E(R) \right)},\quad
	\bar{H}_{b}(t):= \frac{H_{b}(E(R)t)}{H_{b}\left( E(R) \right)}
\end{align}
	for every $t\geqslant 0$. Clearly, one can see that $\bar{G},\bar{H}_{a}, \bar{H}_{b}\in \mathcal{N}$ with indices $s(G),s(H_{a}),s(H_{b})$, respectively, and also that
	\begin{align}
		\label{2ce:8}
		\bar{G}(1) \leqslant 1,\quad
		\bar{H}_{a}(1)=1
		\quad\text{and}\quad
		\bar{H}_{b}(1)=1.
	\end{align}
Then one can check that the function $\bar{w}_{c}$ minimizes the following functional 
\begin{align}
	\label{2ce:9}
	W^{1,\bar{\Psi}}(B_{1})\ni v\mapsto \I_{B_{1}}\bar{F}(x,Dv)\,dx,
\end{align}
where the functions $\bar{\Psi}(\cdot)$ and $\bar{F}(\cdot)$ have been defined in \eqref{2ce:6_1} and \eqref{2ce:6_2}, respectively. The Euler-Lagrange equation associated to the functional in \eqref{2ce:9} becomes 
\begin{align}
	\label{2ce:10}
	\FI_{B_{1}}\inner{\bar{A}(x,D\bar{w}_{c})}{D\varphi}\,dx = \FI_{B_{1}}\inner{ D_{z}\bar{F}\left(x, D\bar{w}_{c}\right)}{D\varphi}\,dx = 0
\end{align}
for every $\varphi\in W^{1,\bar{\Psi}}_{0}(B_{1})$. By the assumptions \eqref{sa:1} and \eqref{sa:2} via elementary computations, we have the following structure condition in the scaled settings:

\begin{subequations}
	 \begin{empheq}[left={\empheqlbrace}]{align}
	 &\nu \bar{\Psi}(x,|z|) \leqslant \bar{F}(x,z) \leqslant L\bar{\Psi}(x,|z|),
	 	\label{2ce:11_1} \\
          & |\bar{A}(x,z)||z| + |D_{z}\bar{A}(x,z)||z|^2 \leqslant L\bar{\Psi}(x,|z|)  \label{2ce:11_2},\\
         & \nu \frac{\bar{\Psi}(x,|z|)}{|z|^2}|\xi|^{2} \leqslant
         \inner{D_{z}\bar{A}(x,z)\xi}{\xi}\label{2ce:11_3}
      \end{empheq}
	\end{subequations}
hold true for every $x\in B_{1}$ and $z\in\R^n\setminus \{0\}$.

\textbf{Step 3: Freezing.} Now we consider frozen functional and vector field associated to $\bar{F}(\cdot)$ and $\bar{A}(\cdot)$ defined in \eqref{2ce:6_2}-\eqref{2ce:6_4}. Let $\bar{x}_{a},\bar{x}_{b}\in \overline{B}_{1}$ be points such that $\bar{a}(\bar{x}_a)= \inf\limits_{x\in B_{1}}\bar{a}(x)$ and $\bar{b}(\bar{x}_b)= \inf\limits_{x\in B_{1}}\bar{b}(x)$. Then we denote by
\begin{align}
	\label{2ce:12_1}
	\bar{F}_{0}(z)&:= \bar{F}_{G}(z) + \bar{a}(\bar{x}_a)\bar{F}_{H_{a}}(z) 
	+
	\bar{b}(\bar{x}_b)\bar{F}_{H_{b}}(z),
\end{align}
\begin{align}
	\label{2ce:12_2}
	\bar{A}_{0}(z):= D_{z}\bar{F}_{0}(z)
\end{align}
and 
\begin{align}
	\label{2ce:13}
	\bar{\Psi}_{0}(t):= \bar{G}(t) + \bar{a}(\bar{x}_{a})\bar{H}_{a}(t) + \bar{b}(\bar{x}_{b})\bar{H}_{b}(t)
\end{align}
for every $x\in B_{1}$, $z\in\R^{n}$ and $t\geqslant 0$. By the very definition in \eqref{2ce:6_1}-\eqref{2ce:6_4}, straightforwardly one can see
\begin{align}
	\label{2ce:13_1}
	\bar{\Psi}_0(1) = 1.
\end{align}
 In our new scaled settings, we now consider the functional
\begin{align}
	\label{2ce:14}
	W^{1,\bar{\Psi}_{0}}\left( B_{1/8} \right)\ni v\mapsto \I_{B_{1/8}} \bar{F}_{0}(Dv)\,dx.
\end{align}
We observe that the newly defined integrand $\bar{F}_{0}(\cdot)$ and vector field $\bar{A}_{0}(\cdot)$ satisfy the growth and ellipticity conditions as 
\begin{subequations}
	 \begin{empheq}[left={\empheqlbrace}]{align}
	 &\nu \bar{\Psi}_{0}(|z|) \leqslant \bar{F}_{0}(z) \leqslant L\bar{\Psi}_{0}(|z|),
	 	\label{2ce:15_1} \\
          & |\bar{A}_{0}(z)||z| + |D_{z}\bar{A}_{0}(z)||z|^2 \leqslant L\bar{\Psi}_{0}(|z|)  \label{2ce:15_2},\\
         & \nu \frac{\bar{\Psi}_{0}(|z|)}{|z|^2}|\xi|^{2} \leqslant
         \inner{D_{z}\bar{A}_{0}(z)\xi}{\xi}\label{2ce:15_3}
      \end{empheq}
	\end{subequations}
for every $z\in\R^{n}\setminus\{0\}$ and $\xi\in\R^{n}$. Therefore, the estimates \eqref{2ce:3} and \eqref{2ce:4} are written in the view of $\bar{w}_{c}$ as
\begin{align}
	\label{2ce:15_4}
	\FI_{B_{1/4}}\bar{\Psi}(x,|D\bar{w}_{c}|)\,dx + \left(\FI_{B_{1/8}}[\bar{\Psi}(x,|D\bar{w}_{c}|)]^{1+\delta_0}\,dx\right)^{\frac{1}{1+\delta_0}} \leqslant c(\data).
\end{align}

\textbf{Step 4: Harmonic type approximation.} Let $\varphi\in W^{1,\infty}_{0}\left( B_{1/8} \right)$. Using \eqref{2ce:10}, we see 
\begin{align}
	\label{2ce:16}
	\begin{split}
	I_{0}&:=\left|\FI_{B_{1/8}}\inner{\bar{A}_{0}(D\bar{w}_{c})}{D\varphi}\,dx \right|
	=
	\left|\FI_{B_{1/8}}\inner{\bar{A}_{0}(D\bar{w}_{c})-\bar{A}(x,D\bar{w}_{c})}{D\varphi}\,dx \right|
	\\&
	\leqslant
	\FI_{B_{1/8}}|\bar{A}_{0}(D\bar{w}_{c})-\bar{A}(x,D\bar{w}_{c})|\,dx \norm{D\varphi}_{L^{\infty}(B_{1/8})} =:  I_{1} \norm{D\varphi}_{L^{\infty}(B_{1/8})}.
	\end{split}
\end{align}

Now using \eqref{sa:2}, we see
\begin{align}
	\label{2ce:17}
	\begin{split}
	I_{1} &\leqslant L\FI_{B_{1/8}} |\bar{a}(x)-\bar{a}(\bar{x}_{a})|\frac{\bar{H}_{a}(|D\bar{w}_{c}|)}{|D\bar{w}_{c}|}\,dx
	+
	L\FI_{B_{1/8}} |\bar{b}(x)-\bar{b}(\bar{x}_{b})|\frac{\bar{H}_{b}(|D\bar{w}_{c}|)}{|D\bar{w}_{c}|}\,dx
	 =:L\left( I_{11} + I_{12}\right).
	\end{split}
\end{align}
Now we estimate the terms appearing in the last display. In turn, using \eqref{growth2}, \eqref{2ce:8} and \eqref{2ce:15_4}, we have
\begin{align}
	\label{2ce:18}
	\begin{split}
	I_{11} &\leqslant 
	c\FI_{B_{1/8}}|\bar{a}(x)-\bar{a}(\bar{x}_{a})|\left( [\bar{H}_{a}(|D\bar{w}_{c}|)]^{\frac{1}{s(H_{a})+1}} + [\bar{H}_{a}(|D\bar{w}_{c}|)]^{\frac{s(H_{a})}{s(H_{a})+1}}\right)\,dx
	\\&
	\leqslant
	c\norm{\bar{a}-\bar{a}(\bar{x}_{a})}_{L^{\infty}(B_{1/8})}^{\frac{s(H_{a})}{s(H_{a})+1}}\left(\FI_{B_{1/8}} \bar{a}(x)\bar{H}_{a}(|D\bar{w}_{c}|)\,dx  \right)^{\frac{1}{s(H_{a})+1}}
	\\&
	\quad
	+
	c\norm{\bar{a}-\bar{a}(\bar{x}_{a})}_{L^{\infty}(B_{1/8})}^{\frac{1}{s(H_{a})+1}}\left(\FI_{B_{1/8}} \bar{a}(x)\bar{H}_{a}(|D\bar{w}_{c}|)\,dx  \right)^{\frac{s(H_{a})}{s(H_{a})+1}}
	\\&
	\leqslant
	c(\data)\left( \norm{\bar{a}-\bar{a}(\bar{x}_{a})}_{L^{\infty}(B_{1/8})}^{\frac{1}{s(H_{a})+1}} + \norm{\bar{a}-\bar{a}(\bar{x}_{a})}_{L^{\infty}(B_{1/8})}^{\frac{s(H_{a})}{s(H_{a})+1}} \right),
	\end{split}
\end{align}
where we have used also H\"older's inequality and the fact that $\bar{a}(\bar{x}_a) \leqslant \bar{a}(x)$ for every $x\in B_{1}$. In a similar way, we have 
\begin{align}
	\label{2ce:19}
	I_{12} \leqslant c(\data)\left( \norm{\bar{b}-\bar{b}(\bar{x}_{b})}_{L^{\infty}(B_{1/8})}^{\frac{1}{s(H_{b})+1}} + \norm{\bar{b}-\bar{b}(\bar{x}_{b})}_{L^{\infty}(B_{1/8})}^{\frac{s(H_{b})}{s(H_{b})+1}} \right).
\end{align}
Inserting those estimates into \eqref{2ce:17} and then \eqref{2ce:16}, we find 
\begin{align}
	\label{2ce:20}
	\begin{split}
	I_{0} &\leqslant c(\data)\left( \norm{\bar{a}-\bar{a}(\bar{x}_{a})}_{L^{\infty}(B_{1/8})}^{\frac{1}{s(H_{a})+1}} + \norm{\bar{a}-\bar{a}(\bar{x}_{a})}_{L^{\infty}(B_{1/8})}^{\frac{s(H_{a})}{s(H_{a})+1}}\right)
	\\&
	\quad +
	c(\data)\left( \norm{\bar{b}-\bar{b}(\bar{x}_{b})}_{L^{\infty}(B_{1/8})}^{\frac{1}{s(H_{b})+1}} + \norm{\bar{b}-\bar{b}(\bar{x}_{b})}_{L^{\infty}(B_{1/8})}^{\frac{s(H_{b})}{s(H_{b})+1}} \right).
	\end{split}
\end{align}

Now we estimate the terms $\norm{\bar{a}-\bar{a}(\bar{x}_{a})}_{L^{\infty}(B_{1/8})}$ and $\norm{\bar{b}-\bar{b}(\bar{x}_{b})}_{L^{\infty}(B_{1/8})}$ depending on which assumption of \eqref{mth:md:1}-\eqref{mth:md:5} comes into play. Recalling the definition of $\bar{a}(\cdot)$, $\bar{b}(\cdot)$ in \eqref{2ce:5_2} and the excess functional in \eqref{2ce:E_1E_2}, we have 
\begin{align}
	\label{2ce:22_1}
	I_{a}:=\norm{\bar{a}-\bar{a}(\bar{x}_{a})}_{L^{\infty}(B_{1/8})}
	\leqslant
	c \omega_{a}(R)\frac{H_{a}(E(R))}{\Psi_{B_{R}}^{-}\left( E(R) \right)}
\end{align}
and
\begin{align}
	\label{2ce:22_2}
	I_{b}:=\norm{\bar{b}-\bar{b}(\bar{x}_{b})}_{L^{\infty}(B_{1/8})}
	\leqslant
	c \omega_{b}(R)\frac{H_{b}(E(R))}{\Psi_{B_{R}}^{-}\left( E(R) \right)}.
\end{align}

\textbf{Case 1: Assumption \eqref{mth:md:1} is in force.} The assumption $\eqref{mth:md:1}_{2}$ implies that for any $\varepsilon\in (0,1)$ there exists $\mu_1 >0$ depending on $\varepsilon$ such that 
\begin{align}
	\label{2ce:21}
	\Lambda\left(t, G^{-1}\left(t^{-n} \right) \right) \leqslant \varepsilon
	\quad\text{for every}\quad t\in (0,\mu_1).
\end{align} 
Then using this one and \eqref{ma:1}, we continue to estimate $I_{a}$ in \eqref{2ce:22_1} as 
\begin{align}
	\label{2ce:23}
	\begin{split}
	I_{a}
	&\leqslant
	c\omega_{a}(R)\frac{\left(H_{a}\circ G^{-1}\right)\left(\Psi_{B_{R}}^{-}(E(R)) \right)}{\Psi_{B_{R}}^{-}(E(R))}
	\\&
	\leqslant
	c\omega_{a}(R)\varepsilon\left(1+\frac{1}{\omega_{a}\left([\Psi_{B_{R}}^{-}(E(R))]^{-\frac{1}{n}}\right)} \right)
	+
	c\omega_{a}(R)\left(1+\frac{1}{\omega_{a}\left(\mu_1\right)} \right)
	\end{split}
\end{align}
with $c\equiv c([a]_{\omega_{a}}, \lambda_{1})$, where we have used the fact that $\left(\Psi_{B_{R}}^{-} \right)^{-1}(t) \leqslant G^{-1}(t)$ for every $t\geqslant 0$.  Using \eqref{concave2} together with the energy estimates \eqref{1ce:4} and \eqref{0ce:4}, we observe that 
\begin{align}
	\label{2ce:24}
	\begin{split}
	\frac{1}{\omega_{a}\left(\left[\Psi_{B_{R}}^{-}(E(R))\right]^{-\frac{1}{n}}\right)}
	&\leqslant
	\frac{c}{\omega_{a}(R)} + \frac{c}{\omega_{a}(R)}\I_{B_{R/2}}\Psi_{B_{R}}^{-}\left(\left|\frac{w_c-(w_{c})_{B_{R/2}}}{R} \right|\right)\,dx
	\\&
	\leqslant
	\frac{c}{\omega_{a}(R)} + \frac{c}{\omega_{a}(R)}\I_{B_{2R}}\Psi\left(x,|Du|\right)\,dx
	\leqslant
	\frac{c(\data)}{\omega_{a}(R)}.
	\end{split}
\end{align}
Combining the last two displays, we conclude 
\begin{align}
	\label{2ce:25}
	I_{a} \leqslant c\left( \varepsilon + \omega_{a}(R)\left( 1+ \frac{1}{\omega_{a}(\mu_1)} \right) \right)
\end{align}
with some constant $c\equiv c(\data)$. In the same manner, we see 
\begin{align}
	\label{2ce:26}
	I_{b} \leqslant 
	c\left( \varepsilon + \omega_{b}(R)\left( 1+ \frac{1}{\omega_{b}(\mu_1)} \right) \right)
\end{align}
for some constant $c\equiv c(\data)$. Therefore, inserting the estimates in the last two displays into \eqref{2ce:20} and recalling \eqref{2ce:16}, we have 
\begin{align}
	\label{2ce:27}
	\left|\FI_{B_{1/8}}\inner{\bar{A}_{0}(D\bar{w}_{c})}{D\varphi}\,dx \right|
	\leqslant
	c(\data)p_{1}(\varepsilon,R)\norm{D\varphi}_{L^{\infty}(B_{1/8})},
\end{align}
where 
\begin{align}
	\label{2ce:28}
	\begin{split}
	p_1(\varepsilon,R)
	&:= \left[ \varepsilon + \omega_{a}(R)\left( 1+ \frac{1}{\omega_{a}(\mu_1)} \right) \right]^{\frac{1}{s(H_{a})+1}} + \left[ \varepsilon + \omega_{a}(R)\left( 1+ \frac{1}{\omega_{a}(\mu_1)} \right) \right]^{\frac{s(H_{a})}{s(H_{a})+1}}
	\\&
	\quad
	+
	\left[ \varepsilon + \omega_{b}(R)\left( 1+ \frac{1}{\omega_{b}(\mu_1)} \right) \right]^{\frac{1}{s(H_{b})+1}} + 
	\left[ \varepsilon + \omega_{b}(R)\left( 1+ \frac{1}{\omega_{b}(\mu_1)} \right) \right]^{\frac{s(H_{b})}{s(H_{b})+1}}.
	\end{split}
\end{align}

\textbf{Case 2: Assumption \eqref{mth:md:2} is in force.} From the assumption $\eqref{mth:md:2}_{2}$ it holds that for every $\varepsilon\in (0,1)$ there exists $\mu_{2}>0$ depending on $\varepsilon$ such that 
\begin{align}
	\label{2ce:29}
	\Lambda\left(t,\frac{1}{t}\right) \leqslant \varepsilon
	\quad\text{for every}\quad t\in (0,\mu_2).
\end{align}
Then by the very definition of $\Psi_{B_{R}}^{-}$ in \eqref{ispsi} together with \eqref{2ce:29} and \eqref{ma:2}, we have 
\begin{align}
	\label{2ce:30}
	\begin{split}
	I_{a}
	&\leqslant
	c\omega_{a}(R)\frac{H_{a}(E(R))}{G(E(R))}
	\\&
	\leqslant
	c\omega_{a}(R)\varepsilon\left(1+\frac{1}{\omega_{a}\left([E(R)]^{-1}\right)} \right)
	+
	c\omega_{a}(R)\left(1+\frac{1}{\omega_{a}\left(\mu_2\right)} \right).
	\end{split}
\end{align}
Again using \eqref{concave1} together with taking into account \eqref{1ce:5} and \eqref{0ce:5}, we see 
\begin{align}
	\label{2ce:31}
	\frac{1}{\omega_{a}\left([E(R)]^{-1}\right)} 
	\leqslant
	\frac{1}{\omega_{a}\left(\frac{R}{2\norm{w_{c}}_{L^{\infty}(B_{R})}} \right)}
	\leqslant
	\frac{c(\data)}{\omega_{a}(R)}.
\end{align}
Combining the last two displays, we find 
\begin{align}
	\label{2ce:32}
	I_{a} \leqslant c\left( \varepsilon + \omega_{a}(R)\left( 1+ \frac{1}{\omega_{a}(\mu_2)} \right) \right)
\end{align}
with some constant $c\equiv c(\data)$. Similarly, it holds that 
\begin{align}
	\label{2ce:33}
	I_{b} \leqslant c\left( \varepsilon + \omega_{b}(R)\left( 1+ \frac{1}{\omega_{a}(\mu_2)} \right) \right). 
	\end{align}
Then, plugging the estimates in the last two displays into \eqref{2ce:20} and recalling \eqref{2ce:16}, we have 
\begin{align}
	\label{2ce:34}
	\left|\FI_{B_{1/8}}\inner{\bar{A}_{0}(D\bar{w}_{c})}{D\varphi}\,dx \right|
	\leqslant
	c(\data)p_{2}(\varepsilon,R)\norm{D\varphi}_{L^{\infty}(B_{1/8})},
\end{align}
where 
\begin{align}
	\label{2ce:35}
	\begin{split}
	p_2(\varepsilon,R)
	&:= \left[ \varepsilon + \omega_{a}(R)\left( 1+ \frac{1}{\omega_{a}(\mu_2)} \right) \right]^{\frac{1}{s(H_{a})+1}} + \left[ \varepsilon + \omega_{a}(R)\left( 1+ \frac{1}{\omega_{a}(\mu_2)} \right) \right]^{\frac{s(H_{a})}{s(H_{a})+1}}
	\\&
	\quad
	+
	\left[ \varepsilon + \omega_{b}(R)\left( 1+ \frac{1}{\omega_{b}(\mu_2)} \right) \right]^{\frac{1}{s(H_{b})+1}} + 
	\left[ \varepsilon + \omega_{b}(R)\left( 1+ \frac{1}{\omega_{b}(\mu_2)} \right) \right]^{\frac{s(H_{b})}{s(H_{b})+1}}.
	\end{split}
\end{align}

\textbf{Case 3: Assumption \eqref{mth:md:3} is in force.} The assumption $\eqref{mth:md:3}_{2}$ implies that for any $\varepsilon\in (0,1)$ there exists $\mu_3>0$ depending on $\varepsilon$ such that 
\begin{align}
	\label{2ce:36}
	\Lambda\left(t^{\frac{1}{1-\gamma}}, \frac{1}{t} \right) \leqslant \varepsilon
	\quad\text{for every}\quad t\in (0,\mu_3).
\end{align}
This one together with recalling \eqref{2ce:22_1} and \eqref{ma:3}, we see 
\begin{align}
	\label{2ce:37}
	\begin{split}
	I_{a}
	&\leqslant
	c\omega(R)\frac{H_{a}(E(R))}{G(E(R))}
	\\&
	\leqslant
	c\omega_{a}(R)\varepsilon\left(1+\frac{1}{\omega_{a}\left([E(R)]^{-\frac{1}{1-\gamma}}\right)} \right)
	+
	c\omega_{a}(R)\left(1+\frac{1}{\omega_{a}\left(\mu_3^{\frac{1}{1-\gamma}}\right)} \right).
	\end{split}
\end{align}
Now using \eqref{1ce:6}, \eqref{0ce:6} and \eqref{ma:3}, we have 
\begin{align}
	\label{2ce:38}
	\frac{1}{\omega_{a}\left([E(R)]^{-\frac{1}{1-\gamma}}\right)}
	\leqslant
	\frac{1}{\omega_{a}\left(\left[\frac{\osc\limits_{B_{2R}}u}{R}\right]^{-\frac{1}{1-\gamma}}
	\right)}
	\leqslant
	\frac{c(\data)}{\omega_{a}(R)}.
\end{align}
Combining the last two displays, we find 
\begin{align}
	\label{2ce:39}
	I_{a} \leqslant c\left( \varepsilon + \omega_{a}(R)\left( 1+ \frac{1}{\omega_{a}\left(\mu_3^{\frac{1}{1-\gamma}}\right)} \right) \right)
\end{align}
for some constant $c\equiv c(\data)$. In the same way, we show 
\begin{align}
	\label{2ce:40}
	I_{b} \leqslant c\left( \varepsilon + \omega_{a}(R)\left( 1+ \frac{1}{\omega_{a}\left(\mu_3^{\frac{1}{1-\gamma}}\right)} \right) \right)
\end{align}
for some constant $c\equiv c(\data)$. Using the estimates \eqref{2ce:39}-\eqref{2ce:40} in \eqref{2ce:20}, we conclude 
\begin{align}
	\label{2ce:41}
	\left|\FI_{B_{1/8}}\inner{\bar{A}_{0}(D\bar{w}_{c})}{D\varphi}\,dx \right|
	\leqslant
	c(\data)p_{3}(\varepsilon,R)\norm{D\varphi}_{L^{\infty}(B_{1/8})},
\end{align}
where 
\begin{align}
	\label{2ce:42}
	\begin{split}
	p_3(\varepsilon,R)
	&:= \left[ \varepsilon + \omega_{a}(R)\left( 1+ \frac{1}{\omega_{a}\left(\mu_3^{\frac{1}{1-\gamma}}\right)} \right) \right]^{\frac{1}{s(H_{a})+1}} + \left[ \varepsilon + \omega_{a}(R)\left( 1+ \frac{1}{\omega_{a}\left(\mu_3^{\frac{1}{1-\gamma}}\right)} \right) \right]^{\frac{s(H_{a})}{s(H_{a})+1}}
	\\&
	\quad
	+
	\left[ \varepsilon + \omega_{b}(R)\left( 1+ \frac{1}{\omega_{b}\left(\mu_3^{\frac{1}{1-\gamma}}\right)} \right) \right]^{\frac{1}{s(H_{b})+1}} + 
	\left[ \varepsilon + \omega_{b}(R)\left( 1+ \frac{1}{\omega_{b}\left(\mu_3^{\frac{1}{1-\gamma}}\right)} \right) \right]^{\frac{s(H_{b})}{s(H_{b})+1}}.
	\end{split}
\end{align}

\textbf{Case 4. Assumption \eqref{mth:md:4} is in force.} We treat this case in a different way rather than the estimate used in \eqref{2ce:21}-\eqref{2ce:28}. In fact, we take an advantage that $w_a(\cdot)$ is a power function. Then recalling $I_{a}$ introduced in \eqref{2ce:22_1}, we see that 
\begin{align}
	\label{2ce:42_1}
	\begin{split}
	I_{a} &\leqslant c R^{\alpha} \frac{\left(H_{a}\circ G^{-1} \right)\left(\Psi_{B_{R}}^{-}(E(R))\right)}{\Psi_{B_{R}}^{-}(E(R))}
	\leqslant
	cR^{\alpha}\left( 1 + \left[\FI_{B_{R/2}}\Psi_{B_{R}}^{-}\left(\left|\frac{w_c-(w_c)_{B_{R/2}}}{R} \right| \right)\,dx \right]^{\frac{\alpha}{n}} \right)
	\\&
	\leqslant
	cR^{\alpha} + c\left(\I_{B_{R/2}}\Psi_{B_{R}}^{-}\left(\left|Dw_c \right| \right)\,dx \right)^{\frac{\alpha}{n}}
	\leqslant
	cR^{\alpha} + c\left(\I_{B_{2R}}\Psi\left(x,\left|Du \right| \right)\,dx \right)^{\frac{\alpha}{n}}
	\\&
	\leqslant
	cR^{\alpha} + cR^{\frac{\alpha\delta}{1+\delta}} \left(\I_{B_{2R}}[\Psi\left(x,\left|Du \right| \right)]^{1+\delta}\,dx \right)^{\frac{\alpha}{n(1+\delta)}}
	\\&
	\leqslant
	c(\data(\O_0))R^{\frac{\alpha\delta}{1+\delta}}
	\end{split}
\end{align}
for a higher integrability exponent $\delta$ coming from Theorem \ref{thm:hi}, where we have used \eqref{1ce:4}, \eqref{0ce:4} together with \eqref{hi:2}. By arguing similarly, we estimate $I_{b}$ in \eqref{2ce:22_2} as  
\begin{align}
	\label{2ce:42_2}
	I_{b} \leqslant c(\data(\O_0))R^{\frac{\beta\delta}{1+\delta}}.
\end{align}
Using estimates from the last two displays in \eqref{2ce:20} and recalling $R\leqslant 1$, we see 
\begin{align}
	\label{2ce:42_3}
	\left|\FI_{B_{1/8}}\inner{\bar{A}_{0}(D\bar{w}_{c})}{D\varphi}\,dx \right|
	\leqslant
	c(\data(\O_0))q_{1}(R)\norm{D\varphi}_{L^{\infty}(B_{1/8})},
\end{align}
where 
\begin{align}
	\label{2ce:42_4}
	q_{1}(R):= R^{\frac{\alpha\delta}{(1+\delta)(1+s(H_{a}))}} + R^{\frac{\beta\delta}{(1+\delta)(1+s(H_{b}))}}.
\end{align}

\textbf{Case 5: Assumption \eqref{mth:md:5} is in force.} Again we estimate $I_{a}$ and $I_{b}$ introduced in \eqref{2ce:22_1}-\eqref{2ce:22_2}. Using the assumption \eqref{ma:2}, \eqref{1ce:6} and \eqref{0ce:6}, we have 
\begin{align}
	\label{2ce:42_5}
	\begin{split}
	I_{a} &\leqslant
	cR^{\alpha}\frac{H_{a}(E(R))}{G(E(R))}
	\\&
	\leqslant
	 cR^{\alpha}\left( 1 + \left[ \left(\Psi_{B_{R}}^{-}\right)^{-1}\left(\FI_{B_{R/2}}\Psi_{B_{R}}^{-}\left(\left|\frac{w_c-(w_c)_{B_{R/2}}}{R} \right| \right)\,dx \right) \right]^{\alpha} \right)
	\\&
	\leqslant
	c\left(R^{\alpha}+ \left[\osc\limits_{B_{2R}}u\right]^{\alpha}\right)
	\leqslant
	c(\data(\O_0))R^{\gamma\alpha},
	\end{split}
\end{align}
where we have also used \eqref{hc:3} and the H\"older continuity exponent $\gamma$ came from Theorem \ref{thm:hc}. Similarly, we see 
\begin{align}
	\label{2ce:42_6}
	I_{b} \leqslant c(\data(\O_0))R^{\gamma\beta}.
\end{align}

Inserting the estimates from the last two displays into \eqref{2ce:20} and recalling $R\leqslant 1$, we see 
\begin{align}
	\label{2ce:42_7}
	\left|\FI_{B_{1/8}}\inner{\bar{A}_{0}(D\bar{w}_{c})}{D\varphi}\,dx \right|
	\leqslant
	c(\data(\O_0))q_{2}(R)\norm{D\varphi}_{L^{\infty}(B_{1/8})},
\end{align}
where 
\begin{align}
	\label{2ce:42_8}
	q_{2}(R):= R^{\frac{\alpha\gamma}{1+s(H_{a})}} + R^{\frac{\beta\gamma}{1+s(H_{b})}}
\end{align}

Collecting the estimates obtained in \eqref{2ce:27}, \eqref{2ce:34},\eqref{2ce:41}, \eqref{2ce:42_3} and \eqref{2ce:42_7}, we conclude with 
    \begin{align}
        \label{2ce:43}
        \left|\FI_{B_{1/8}} \inner{\bar{A}_{0}(D\bar{w}_{c})}{D\varphi}\,dx\right|
        \leqslant
        c_{h}d(\varepsilon,R) \norm{D\varphi}_{L^{\infty}(B_{1/8})}
    \end{align}
     for some constant $c_{h}\equiv c_{h}(\data(\O_0))$, whenever $\varphi\in W^{1,\infty}_{0}(B_{1/8})$,
    where 
    \begin{align}
	\label{2ce:44}
	 d(\varepsilon,R)
	 := \left\{\begin{array}{lr}
        p_1(\varepsilon,R) & \text{if } \eqref{mth:md:1} \text{ is assumed, }\\
        p_2(\varepsilon,R)  & \text{if } \eqref{mth:md:2} \text{ is assumed, } \\
        p_3(\varepsilon,R)  & \text{if } \eqref{mth:md:3} \text{ is assumed, } \\
        q_1(R)  & \text{if } \eqref{mth:md:4} \text{ is assumed, } \\
        q_2(R)  & \text{if } \eqref{mth:md:5} \text{ is assumed, }
        \end{array}\right.
\end{align}
in which $p_1, p_2$, $p_3$, $q_1$ and $q_2$ have been defined in \eqref{2ce:28}, \eqref{2ce:35}, \eqref{2ce:42}, \eqref{2ce:42_4} and \eqref{2ce:42_8}, respectively. By \eqref{2ce:13_1}, \eqref{2ce:15_1}-\eqref{2ce:15_4} and \eqref{2ce:43}, we are able to apply Lemma \ref{hta: lemma_hta} with $A_{0}(z)\equiv \bar{A}_{0}(z)$, $\Psi_{0}(t)\equiv \bar{\Psi}_{0}(t)$ with $a_{0}\equiv \bar{a}(\bar{x}_{a})$ and $b_{0}\equiv \bar{b}(\bar{x}_{b})$, to discover that there exists $\bar{h}\in \bar{w}_{c} + W_{0}^{1,\bar{\Psi}_{0}}(B_{1/8})$ such that 
    \begin{align}
        \label{2ce:45}
        \FI_{B_{1/8}}\inner{\bar{A}_0(D\bar{h})}{D\varphi}\,dx = 0\qquad \text{ for all }
        \qquad \varphi\in W^{1,\infty}_{0}(B_{1/8}),
    \end{align}
    \begin{align}
        \label{2ce:46}
        \FI_{B_{1/4}}\bar{\Psi}_{0}(|D\bar{h}|)\,dx + \FI_{B_{1/8}}[\bar{\Psi}_{0}(|D\bar{h}|)]^{1+\delta_1}\,dx \leqslant c
        \text{ for some } \delta_1 \leqslant \delta_0,
    \end{align}
    \begin{align}
        \label{2ce:47}
        \begin{split}
        \FI_{B_{1/8}}&\left( |V_{\bar{G}}(D\bar{w}_{c})-V_{\bar{G}}(D\bar{h})|^{2} + \bar{a}(\bar{x}_{a})|V_{\bar{H}_{a}}(D\bar{w}_{c})-V_{\bar{H}_{a}}(D\bar{h})|^{2} + \bar{b}(\bar{x}_{b})|V_{\bar{H}_{b}}(D\bar{w}_{c})-V_{\bar{H}_{b}}(D\bar{h})|^{2} \right)\,dx
        \\&
        \leqslant
        c[d(\varepsilon, R)]^{s_1}
        \end{split}
    \end{align}
    and finally
    \begin{align}
        \label{2ce:48}
        \FI_{B_{1/8}} \left(\bar{G}\left( |\bar{w}_{c}-\bar{h}|\right) + \bar{a}(\bar{x}_{a}) \bar{H_{a}}\left( |\bar{w}_{c}-\bar{h}| \right) + \bar{b}(\bar{x}_{b}) \bar{H_{b}}\left( |\bar{w}_{c}-\bar{h}| \right) \right)\,dx \leqslant c_{d}[d(\varepsilon, R)]^{s_{0}}
    \end{align}
    with some constants $c,c_d\equiv c,c_d(\data(\O_0))\geqslant 1$ and $s_0,s_1\equiv s_0,s_1(\data)\in (0,1)$, but they are all independent of $R$. Therefore, for a given $\varepsilon^{*}\in (0,1)$ as in the statement of our lemma, we choose small enough $\varepsilon$ and $R^{*}$ to satisfy
    \begin{align}
        \label{2ce:49}
        c_{d}\left[ d(\varepsilon,R^{*}) \right]^{s_0} \leqslant \varepsilon^{*}.
    \end{align}
    Since the constants $c_d$ and $ s_0$ only depend on $\data(\O_0)$ and $\data$, respectively, the last display gives us the dependence of $R^{*}$ as in \eqref{2ce:1}. Furthermore, by \eqref{2ce:48}, we conclude with 
    \begin{align}
        \label{2ce:50}
        \FI_{B_{1/8}} \left[\bar{G}\left( |\bar{w}_{c}-\bar{h}|\right) + \bar{a}(\bar{x}_{a}) \bar{H}_{a}\left( |\bar{w}_{c}-\bar{h}| \right) + \bar{b}(\bar{x}_{b}) \bar{H}_{b}\left( |\bar{w}_{c}-\bar{h}| \right) \right]\,dx \leqslant \varepsilon^{*}.
    \end{align}
    
    \textbf{Proof of \eqref{2ce:2}.} We observe that by a standard density argument, the relation in \eqref{2ce:45} still holds for every $\varphi\in W^{1,1}_{0}(B_{1/8})$ with $\bar{\Psi}_{0}(|D\varphi|)\in L^{1}(B_{1/8})$. Recalling \eqref{2ce:12_1} and \eqref{2ce:12_2}, we see that $\bar{h}$ is a local minimizer of the functional 
    \begin{align}
        \label{2ce:51}
        W^{1,\bar{\Psi}_{0}}(B_{1/8})\ni \u\mapsto \I_{B_{1/8}} \bar{F}_{0}(D\u)\,dx.
    \end{align}
    Since the conditions \eqref{2ce:15_1}-\eqref{2ce:15_3} are satisfied for the integrand $\bar{F}_{0}(\cdot)$, we are in a position to apply the results from \cite{L1} to obtain the following a priori Lipschitz estimate:
    \begin{align}
        \label{2ce:52}
        \sup\limits_{B_{1/16}} \bar{\Psi}_{0}(|D\bar{h}|) \leqslant
        c\FI_{B_{1/8}} \bar{\Psi}_{0}(|D\bar{h}|)\,dx
    \end{align}
    with some constant $c\equiv c(n,s(G),s(H_{a}),s(H_{b}),\nu,L)$. For any $\tau\in \left(0,1/16\right)$, we have that 
    \begin{align}
        \label{2ce:53}
        \begin{split}
            \FI_{B_{\tau}}\bar{\Psi}_{0}\left(\left|\frac{\bar{w}_{c}-(\bar{w}_{c})_{B_{\tau}}}{\tau}\right|\right)\,dx
            &\leqslant
            \FI_{B_{\tau}}\bar{\Psi}_{0}\left(\left|\frac{\bar{w}_{c}-(\bar{h})_{B_{\tau}}}{\tau}\right|\right)\,dx
            \\&
            \leqslant
            \FI_{B_{\tau}}\bar{\Psi}_{0}\left(\left|\frac{\bar{h}-(\bar{h})_{B_{\tau}}}{\tau}\right|\right)\,dx
            +
            \FI_{B_{\tau}}\bar{\Psi}_{0}\left(\left|\frac{\bar{w}_{c}-\bar{h}}{\tau}\right|\right)\,dx
            \\&
            \stackrel{\eqref{2ce:50}}{\leqslant}
            c \sup\limits_{B_{\tau}} \bar{\Psi}_{0}(|D\bar{h}|) + c\tau^{-(n+s(\Psi)+1)}\varepsilon^{*}
            \\&
            \stackrel{\eqref{2ce:52}}{\leqslant}
            c\FI_{B_{1/8}}\bar{\Psi}_{0}(|D\bar{h}|)\,dx + c\tau^{-(n+s(\Psi)+1)}\varepsilon^{*}
            \\&
            \stackrel{\eqref{2ce:46}}{\leqslant}
            c + c\tau^{-(n+s(\Psi)+1)}\varepsilon^{*}.
        \end{split}
    \end{align}
    By scaling back to $w_{c}$ as introduced in \eqref{2ce:5_1}-\eqref{2ce:5_2}, we obtain the desired estimate \eqref{2ce:2}. The proof is complete.
    \end{proof} 

\begin{lem}
	\label{lem:3ce}
	Under the assumptions and notations of Lemma \ref{lem:2ce}, let $w_{c}\in W^{1,\Psi}(B_{R})$ be the solution to the problem defined in \eqref{1ce:2}. If one of the assumptions \eqref{mth:md:1}-\eqref{mth:md:5} is satisfied, then there exists $h \in w_c + W^{1,\Psi_{B_{R}}^{-}}_{0}(B_{R/8})$ being  a local minimizer of the functional defined by 
	\begin{align}
		\label{3ce:1}
		W^{1,1}(B_{R/8})\ni v\mapsto \F_{0}(v):= \I_{B_{R/8}}F_{0}(Dv)\,dx, 
	\end{align}
	where the integrand function is given by
\begin{align}
	\label{3ce:1_1}
	F_{0}(z):= F_{G}\left(x_{c}, (u)_{B_{2R}},z \right) + a(x_{a})F_{H_{a}}\left(x_{c}, (u)_{B_{2R}},z \right) + b(x_{b})F_{H_{b}}\left(x_{c}, (u)_{B_{2R}},z \right)
\end{align}
for some fixed point $x_c\in B_{R}$ having been fixed in \eqref{1ce:2}  and  $x_{a}, x_{b}\in\overline{B}_{R}$ being points such that $a(x_{a}):= \inf\limits_{x\in B_{R}} a(x)$ and $b(x_{b}):= \inf\limits_{x\in B_{R}} b(x)$, whenever $z\in\R^n$, such that 
\begin{align}
	\label{3ce:2}
	\begin{split}
	\FI_{B_{R/8}}&\left[ |V_{G}(Du)-V_{G}(Dh)|^2 + a(x_a)|V_{H_{a}}(Du)-V_{H_{a}}(Dh)|^2 + b(x_b)|V_{H_{b}}(Du)-V_{H_{b}}(Dh)|^2  \right]\,dx
	\\&
	\leqslant
	c\left( \omega\left(R^{\gamma}\right) + [d(\varepsilon,R)]^{s_1}  \right)\FI_{B_{2R}}\Psi(x,|Du|)\,dx
	\end{split}
\end{align}
for some constant $c\equiv c(\data(\O_0))$, where $s_1$ and $d(\varepsilon,R)$ have been defined in \eqref{2ce:47} and \eqref{2ce:44}, respectively. Moreover, we have the energy estimate
\begin{align}
	\label{3ce:3}
	\FI_{B_{R/8}}\Psi_{B_{R}}^{-}(|Dh|)\,dx \leqslant c\FI_{B_{2R}}\Psi(x,|Du|)\,dx
\end{align}
for some constant $c\equiv c(n,\nu,L)$.
\end{lem}
\begin{proof}
	We need to revisit the proof of Lemma \ref{lem:2ce}, specially Step 3 and Step 4. Under the settings of the proof of Lemma \ref{lem:2ce}, we consider a function $\bar{h}\in \bar{w}_{c} + W^{1,\bar{\Psi}_{0}}_{0}(B_{1/8})$ satisfying \eqref{2ce:45}-\eqref{2ce:48}. Let $h$ be the scaled back function of $\bar{h}$ in $B_{R/8}$ as
	\begin{align}
		\label{3ce:4}
		h(x):= E(w_c,B_{R/2})R\bar{h}\left(\frac{x-x_0}{R} \right)
		\quad\text{for every}\quad
		x\in B_{R/8}(x_0).
	\end{align}
Clearly, $h\in w_{c} + W^{1,\Psi_{B_{R}}^{-}}_{0}(B_{R/8})$ is a local minimizer of the functional $\F_0$ defined in \eqref{3ce:1} which means that 
\begin{align}
	\label{3ce:5}
	\F_{0}(h) = \I_{B_{R/8}}F_0(Dh)\,dx \leqslant \I_{B_{R/8}}F_0(Dh+D\varphi)\,dx
	=
	\F_{0}(h+\varphi)
\end{align}
holds for every $\varphi\in W^{1,\Psi_{B_{R}}^{-}}_{0}(B_{R/8})$. As we have shown in \eqref{0ce:9}, recalling \eqref{1ce:4} and \eqref{0ce:4}, we see 
\begin{align}
	\label{3ce:6}
	\begin{split}
	\FI_{B_{R/8}}\Psi_{B_{R}}^{-}(|Dh|)\,dx 
	&\leqslant
	\frac{L}{\nu}\FI_{B_{R/8}}\Psi_{B_{R}}^{-}(|Dw_{c}|)\,dx 
	\leqslant
	\frac{8^{n}L}{\nu} \FI_{B_{R}}\Psi(x,|Dw_c|)\,dx
	\\&
	\leqslant
	c(n,\nu,L) \FI_{B_{R}}\Psi(x,|Dw|)\,dx
	\leqslant
	c(n,\nu,L)  \FI_{B_{2R}}\Psi(x,|Du|)\,dx,
	\end{split}
\end{align}
which proves \eqref{3ce:3}. We write the inequality \eqref{2ce:47} in view of $G,H_{a}, H_{b}$, $w_{c}$ and $h$ in order to have 
\begin{align}
	\label{3ce:7}
	\begin{split}
		\FI_{B_{R/8}}&\left[ |V_{G}(Dw_c)-V_{G}(Dh)|^2 + a(x_a)|V_{H_{a}}(Dw_c)-V_{H_{a}}(Dh)|^2 + b(x_b)|V_{H_{b}}(Dw_c)-V_{H_{b}}(Dh)|^2  \right]\,dx
		\\&
		\leqslant
		c[d(\varepsilon,R)]^{s_1}\FI_{B_{R/2}}\Psi_{B_{R}}^{-}\left(\left|\frac{w_c-(w_c)_{B_{R/2}}}{R} \right| \right)\,dx
		\\&
		\leqslant
		c[d(\varepsilon,R)]^{s_1}\FI_{B_{R/2}}\Psi_{B_{R}}^{-}\left(\left|Dw_c \right| \right)\,dx
		\leqslant
		c[d(\varepsilon,R)]^{s_1}\FI_{B_{R/2}}\Psi\left(x,\left|Du \right| \right)\,dx
	\end{split}
\end{align}
for some constant $c\equiv c(\data(\O_0))$, where we have applied Sobolev-Poincar\'e inequality and \eqref{3ce:6}. Combining this estimate together with \eqref{0ce:3} and \eqref{1ce:3} via some elementary computations and recalling $R\leqslant 1$, we directly arrive at \eqref{3ce:2}. 
\end{proof}

We finally finish this section with a crucial decay estimate on a local minimizer $u$ of the functional $\F$.

\begin{lem}
    \label{lem:4ce}
    Under the assumptions and notations of Lemma \ref{lem:2ce}, if one of the conditions \eqref{mth:md:1}-\eqref{mth:md:5} is satisfied, then  for every $\varepsilon_{*}\in (0,1)$, there exists a positive radius $R_{*}$ with the dependence as 
    \begin{align}
        \label{4ce:1}
        R_{*}\equiv R_{*}(\data(\O_0),\varepsilon_{*})
    \end{align}
    such that if $R\leqslant R_{*}$, then there exists a constant $c\equiv c(\data(\O_0))$ such that
    \begin{align}
        \label{4ce:2}
        \I_{B_{\tau R}} \Psi_{B_{R}}^{-}\left(\left|\frac{u-(u)_{B_{\tau R}}}{\tau R}\right|\right)\,dx \leqslant
        c\left( \tau^{n} + \tau^{-(s(\Psi)+1)}\varepsilon_{*} \right)\I_{B_{2R}}\Psi(x,|Du|)\,dx
    \end{align}
    holds for every $\tau\in \left(0,1/16 \right)$.
\end{lem}

\begin{proof}
    First we apply Lemma \ref{lem:2ce} with $\varepsilon^{*}\in (0,1)$ to be determined in a few lines, and we can use \eqref{2ce:2} provided
    \begin{align*}
        R\leqslant R^{*}\equiv R^{*}(\data(\O_0),\varepsilon^{*})
    \end{align*}
    is found via \eqref{2ce:1}.
    Therefore, using the convexity of $\Psi_{B_{R}}^{-}$, Lemma \ref{lem:2ce} and a Sobolev-Poincar\'e inequality of Lemma \ref{lem:os} via some elementary manipulations, for every $\tau\in (0,1/32)$, we have that 
    \begin{align}
        \label{4ce:3}
        \begin{split}
            \FI_{B_{\tau R}}& \Psi_{B_{R}}^{-}\left(\left|\frac{u-(u)_{B_{\tau R}}}{\tau R} \right|\right)\,dx
            \leqslant
            c\FI_{B_{\tau R}} \Psi_{B_{R}}^{-}\left(\left|\frac{u-(w_c)_{B_{\tau R}}}{\tau R} \right|\right)\,dx
            \\&
            \leqslant
            c\FI_{B_{\tau R}} \Psi_{B_{R}}^{-}\left(\left|\frac{w_{c}-(w_{c})_{B_{\tau R}}}{\tau R} \right|\right)\,dx
            + c\tau^{-(n+s(\Psi)+1)} \FI_{B_{R}} \Psi_{B_{R}}^{-}\left(\left|\frac{u-w_{c}}{R} \right|\right)\,dx
            \\&
            \leqslant
            c\left( 1+\tau^{-(n+s(\Psi)+1)}\varepsilon^{*} \right)\FI_{B_{R/2}} \Psi_{B_{R}}^{-}\left(\left|\frac{w_{c}-(w_{c})_{B_{R/2}}}{R} \right|\right)\,dx 
            \\&
            \qquad
            + c\tau^{-(n+s(\Psi)+1)} \FI_{B_{R}} \Psi_{B_{R}}^{-}\left(\left|\frac{u-w_{c}}{R} \right|\right)\,dx
            \\&
            \leqslant
            c\left( 1+\tau^{-(n+s(\Psi)+1)}\varepsilon^{*} \right)\FI_{B_{R}} \Psi_{B_{R}}^{-}\left(\left| Dw_c \right|\right)\,dx + c\tau^{-(n+s(\Psi)+1)} \FI_{B_{R}} \Psi_{B_{R}}^{-}\left(\left|\frac{u-w_{c}}{R} \right|\right)\,dx
        \end{split}
    \end{align}
    with some constant $c\equiv c(\data(\O_0))$, where throughout the last display we have repeatedly used \eqref{growth1} and \eqref{excess4}. The last display, \eqref{0ce:7} and \eqref{1ce:7} with some elementary manipulations  yield
    \begin{align*}
        \I_{B_{\tau R}}& \Psi_{B_{R}}^{-}\left(\left|\frac{u-(u)_{B_{\tau R}}}{\tau R} \right|\right)\,dx
        \leqslant
        c\left( \tau^{n} + \tau^{-(s(\Psi)+1)}\varepsilon^{*} + \tau^{-(s(\Psi)+1)}[\omega(R^{\gamma})]^{\frac{1}{2}} \right)\I_{B_{2R}}\Psi(x,|Du|)\,dx
    \end{align*}
     for every $\tau\in\left(0,1/16 \right)$ and some $c\equiv c(\data(\O_0))$. Then we choose $\varepsilon_{*} \equiv \varepsilon^{*}/2$ and $R_{*}\leqslant R^{*}$ in such a way that
    $[\omega(R_{*}^{\gamma})]^{\frac{1}{2}}\leqslant \varepsilon_{*}/2$. This choice gives us the dependence as described in \eqref{4ce:1} and yields \eqref{4ce:2}.
\end{proof}


\section{Proof of Theorem \ref{mth:md}}
\label{sec:8}
Now we are ready to provide the proof of Theorem \ref{mth:md}. In fact, it comes from the combination of Lemma \ref{lem:1cacct} and Lemma \ref{lem:4ce}. 

\textbf{Step 1: Different alternatives.}
Now we consider the different alternatives depending on which phase of \eqref{G-phase}-\eqref{G-ab phase}  occurs in some fixed ball $B_{R}\equiv B_{R}(x_0)\subset \O_0 \Subset \O$ with $R\leqslant R_{*}\equiv R_{*}(\data(\O_0),\varepsilon_*)$, which will be determined via Lemma \ref{lem:4ce} depending on $\varepsilon_{*}\in (0,1)$.

\textbf{Alternative 1.} Let $\tau_{ab}\in \left(0, 1/64 \right)$ to be chosen in a few lines. Assume that $G$-phase occurs in the ball $B_{\tau_{ab} R}$, which means that \eqref{G-phase} happens in $B_{\tau_{ab} R}$. In this situation, we have 
\begin{align}
	\label{nmd:1}
	a^{-}(B_{2\tau_{ab}R}) \leqslant 8[a]_{\omega_{a}}\omega_{a}(\tau_{ab}R)
	\quad\text{and}\quad 
	b^{-}(B_{2\tau_{ab}R}) \leqslant 8[b]_{\omega_{b}}\omega_{b}(\tau_{ab}R).
\end{align}

Then we are able to apply Lemma \ref{lem:1cacct} in the ball $B_{2\tau_{ab} R}$. In turn, this one together with applying Lemma \ref{lem:4ce} implies that 
\begin{align}
	\label{nmd:2}
	\begin{split}
	\I_{B_{\tau_{ab} R}}\Psi(x,|Du|)\,dx 
	&\leqslant
	c\I_{B_{2\tau_{ab} R}} G\left(\left|\frac{u-(u)_{B_{2\tau_{ab}R}}}{2\tau_{ab}R} \right| \right)\,dx
	\\&
	\leqslant
	c\I_{B_{2\tau_{ab} R}} \Psi_{B_{R}}^{-}\left(\left|\frac{u-(u)_{B_{2\tau_{ab}R}}}{2\tau_{ab}R} \right| \right)\,dx
	\\&
	\leqslant
	c\left( \tau_{ab}^{n} + \tau_{ab}^{-(s(\Psi)+1)}\varepsilon_{*}\right) \I_{B_{R}}\Psi(x,|Du|)\,dx
	\end{split}
\end{align}
for $c\equiv c(\data(\O_0))$, provided $R\leqslant R_{*}\left(\data(\O_0),\varepsilon_{*} \right)$. Then, for every $\sigma\in (0,n)$, we write down the last inequality in the following form
\begin{align*}
	\I_{B_{\tau_{ab} R}}\Psi(x,|Du|)\,dx \leqslant
	\tau_{ab}^{n-\sigma}\left( c_{ab}\tau_{ab}^{\sigma} + c_{ab}\tau_{ab}^{\sigma-(n+s(\Psi)+1)}\varepsilon_{*} \right)\I_{B_{R}}\Psi(x,|Du|)\,dx
\end{align*}
for some constant $c_{ab}\equiv c_{ab}(\data(\O_0))$. We select small enough $\tau_{ab}$, $\varepsilon_{*}$ depending on $\data(\O_0)$ and $\sigma$ in such a way that $c_{ab}\tau_{ab}^{\sigma} \leqslant 1/2$ and $c_{ab}\tau_{ab}^{\sigma-(n+s(\Psi)+1)}\varepsilon_{*} \leqslant 1/2$. Then we have 
\begin{align}
	\label{nmd:4}
	\I_{B_{\tau_{ab} R}}\Psi(x,|Du|)\,dx \leqslant
	\tau_{ab}^{n-\sigma} \I_{B_{R}}\Psi(x,|Du|)\,dx
\end{align}
for every $R\leqslant R_{ab}\equiv R_{ab}(\data(\O_0),\sigma)$.

\textbf{Alternative 2.} Let $\tau_{b}\in (0,1/64)$ also to be determined later. This time we assume that $(G,H_{a})$-phase occurs in $B_{R}$ (\eqref{G-a phase} happens in $B_{R}$) and that $b^{-}(B_{\tau_bR})\leqslant 4[b]_{\omega_{b}}\omega_{b}(\tau_{b}R)$. Then we have 
\begin{align}
	\label{nmd:5}
	b^{-}(B_{2\tau_bR})\leqslant 8[b]_{\omega_{b}}\omega_{b}(\tau_{b}R).
\end{align}
Also we can observe that
\begin{align}
	\label{nmd:5_1}
	a^{-}(B_{\tau_{b}R}) \geqslant a^{-}(B_{R}) > 4[a]_{\omega_{a}}\omega_{a}(R) \geqslant
	4[a]_{\omega_{a}}\omega_{a}(\tau_{b}R) 
\end{align}
and 
\begin{align}
	\label{nmd:5_2}
	a^{-}(B_{R}) \leqslant a(x) \leqslant 2[a]_{\omega_{a}}\omega_{a}(R) + a^{-}(B_{R})
	\leqslant
	2 a^{-}(B_{R})\quad (\forall x\in B_{R}).
\end{align}
Applying Lemma \ref{lem:1cacct} and then Lemma \ref{lem:4ce} together with recalling \eqref{nmd:5_2}, we have 
\begin{align}
	\label{nmd:6}
	\begin{split}
	\I_{B_{\tau_b R}}\Psi(x,|Du|)\,dx 
	&\leqslant
	c\I_{B_{2\tau_b R}}\left[ G\left(\left|\frac{u-(u)_{B_{2\tau_bR}}}{2\tau_bR} \right| \right) + a^{-}(B_{2\tau_{b}R})H_{a}\left(\left|\frac{u-(u)_{B_{2\tau_bR}}}{2\tau_bR} \right| \right) \right]\,dx 
	\\&
	\leqslant
	c\I_{B_{2\tau_b R}} \Psi_{B_{R}}^{-}\left(\left|\frac{u-(u)_{B_{2\tau_bR}}}{2\tau_bR} \right| \right)\,dx
	\\&
	\leqslant
	c\left( \tau_b^{n} + \tau_b^{-(s(\Psi)+1)}\varepsilon_{*}\right) \I_{B_{R}}\Psi(x,|Du|)\,dx
	\end{split}
\end{align}
for some constant $c\equiv c(\data(\O_0))$, provided $R\leqslant R_{*}(\data(\O_0),\varepsilon_{*})$. Then, for every $\sigma\in (0,n)$, we write down the last display as
\begin{align*}
	\I_{B_{\tau_b R}}\Psi(x,|Du|)\,dx \leqslant
	\tau_b^{n-\sigma}\left( c_{b}\tau_b^{\sigma} + c_{b}\tau_b^{\sigma-(n+s(\Psi)+1)}\varepsilon_{*} \right)\I_{B_{R}}\Psi(x,|Du|)\,dx
\end{align*}
for some constant $c_b\equiv c_b(\data(\O_0))$. We select small enough $\tau_b$, $\varepsilon_{*}$ depending on $\data(\O_0)$ and $\sigma$ in such a way that $c_b\tau_{b}^{\sigma} \leqslant 1/2$ and $c_{b}\tau_b^{\sigma-(n+s(\Psi)+1)}\varepsilon_{*} \leqslant 1/2$. Then we have 
\begin{align}
	\label{nmd:7}
	\I_{B_{\tau_b R}}\Psi(x,|Du|)\,dx \leqslant
	\tau_b^{n-\sigma} \I_{B_{R}}\Psi(x,|Du|)\,dx
\end{align}
for every $R\leqslant R_{b}\equiv R_{b}(\data(\O_0),\sigma)$.

\textbf{Alternative 3.} 
Let $\tau_{a}\in (0,1/64)$ to be fixed later. Assume that $(G,H_{b})$-phase occurs in $B_{R}$ (\eqref{G-b phase} happens in $B_{R}$) and $a^{-}(B_{\tau_aR})\leqslant 4[a]_{\omega_{a}}\omega_{a}(\tau_{a}R)$. Then we have 
\begin{align}
	\label{nmd:8}
	a^{-}(B_{2\tau_aR})\leqslant 8[a]_{\omega_{a}}\omega_{a}(\tau_{a}R).
\end{align}
Applying Lemma \ref{lem:1cacct} and then Lemma \ref{lem:4ce} together with recalling that 
$b^{-}(B_{R}) \leqslant b(x) \leqslant 2 b^{-}(B_{R})$ holds for every $x\in B_{R}$ if $b^{-}(B_{R})>4[b]_{\omega_{b}}\omega_{b}(R)$ likewise in \eqref{nmd:5_2}, we have 

\begin{align}
	\label{nmd:9}
	\begin{split}
	\I_{B_{\tau_{a} R}}\Psi(x,|Du|)\,dx 
	&\leqslant
	c\I_{B_{2\tau_{a} R}}\left[ G\left(\left|\frac{u-(u)_{B_{2\tau_{a}R}}}{2\tau_{a}R} \right| \right) + b^{-}(B_{2\tau_{a}R})H_{b}\left(\left|\frac{u-(u)_{B_{2\tau_{a}R}}}{2\tau_{a}R} \right| \right) \right]\,dx 
	\\&
	\leqslant
	c\I_{B_{2\tau_a R}} \Psi_{B_{R}}^{-}\left(\left|\frac{u-(u)_{B_{2\tau_{a}R}}}{2\tau_{a}R} \right| \right)\,dx
	\\&
	\leqslant
	c\left( \tau_{a}^{n} + \tau_{a}^{-(s(\Psi)+1)}\varepsilon_{*}\right) \I_{B_{R}}\Psi(x,|Du|)\,dx
	\end{split}
\end{align}
for some constant $c\equiv c(\data(\O_0))$, provided $R\leqslant R_{*}(\data(\O_0),\varepsilon_{*})$. Then, for every $\sigma\in (0,n)$, we write down the last display as
\begin{align*}
	\I_{B_{\tau_{a} R}}\Psi(x,|Du|)\,dx \leqslant
	\tau_{a}^{n-\sigma}\left( c_{a}\tau_{a}^{\sigma} + c_{a}\tau_{a}^{\sigma-(n+s(\Psi)+1)}\varepsilon_{*} \right)\I_{B_{R}}\Psi(x,|Du|)\,dx
\end{align*}
for some constant $c_a\equiv c_a(\data(\O_0))$. We select small enough $\tau_{a}$, $\varepsilon_{*}$ depending on $\data(\O_0)$ and $\sigma$ in such a way that $c_{a}\tau_{a}^{\sigma} \leqslant 1/2$ and $c_{a}\tau_{a}^{\sigma-(n+s(\Psi)+1)}\varepsilon_{*} \leqslant 1/2$. Then we have 
\begin{align}
	\label{nmd:10}
	\I_{B_{\tau_{a} R}}\Psi(x,|Du|)\,dx \leqslant
	\tau_{a}^{n-\sigma} \I_{B_{R}}\Psi(x,|Du|)\,dx
\end{align}
for every $R\leqslant R_{a}\equiv R_{a}(\data(\O_0),\sigma)$.

\textbf{Alternative 4.} Let $\tau_0\in \left(0,1/64 \right)$ to be chosen later. We assume that $(G,H_{a},H_{b})$-phase occurs in $B_{R}$, which means that \eqref{G-ab phase} happens in $B_{R}$. In this situation, from the observation in \eqref{nmd:5_2} we see that $a^{-}(B_{R}) \leqslant a(x) \leqslant 2 a^{-}(B_{R})$ and $b^{-}(B_{R}) \leqslant b(x) \leqslant 2 b^{-}(B_{R})$ for every $x\in B_{R}$. Then again applying Lemma \ref{lem:1cacct} and Lemma \ref{lem:4ce}, we find 

\begin{align}
	\label{nmd:11}
	\begin{split}
	\I_{B_{\tau_{0} R}}\Psi(x,|Du|)\,dx 
	&\leqslant
	c\I_{B_{2\tau_{a} R}}  \Psi_{B_{2\tau_{0}R}}^{-}\left(\left|\frac{u-(u)_{B_{2\tau_{0}R}}}{2\tau_{0}R} \right| \right)\,dx
	\\&
	\leqslant
	c\I_{B_{2\tau_{0} R}} \Psi_{B_{R}}^{-}\left(\left|\frac{u-(u)_{B_{2\tau_{0}R}}}{2\tau_{0}R} \right| \right)\,dx
	\\&
	\leqslant
	c\left( \tau_{0}^{n} + \tau_{0}^{-(s(\Psi)+1)}\varepsilon_{*}\right) \I_{B_{R}}\Psi(x,|Du|)\,dx
	\end{split}
\end{align}
for some constant $c\equiv c(\data(\O_0))$, provided $R\leqslant R_{*}(\data(\O_0),\varepsilon_{*})$. Then, for every $\sigma\in (0,n)$, we write down the last display as
\begin{align*}
	\I_{B_{\tau_{0} R}}\Psi(x,|Du|)\,dx \leqslant
	\tau_{0}^{n-\sigma}\left( c_{0}\tau_{0}^{\sigma} + c_{0}\tau_{0}^{\sigma-(n+s(\Psi)+1)}\varepsilon_{*} \right)\I_{B_{R}}\Psi(x,|Du|)\,dx
\end{align*}
for some constant $c_{0}\equiv c_{0}(\data(\O_0))$. Then we choose $\tau_{0}$, $\varepsilon_{*}$ depending on $\data(\O_0)$ and $\sigma$ in such a way that $c_{0}\tau_{0}^{\sigma} \leqslant 1/2$ and $c_{0}\tau_{0}^{\sigma-(n+s(\Psi)+1)}\varepsilon_{*} \leqslant 1/2$. Then we have 
\begin{align}
	\label{nmd:12}
	\I_{B_{\tau_{0} R}}\Psi(x,|Du|)\,dx \leqslant
	\tau_{0}^{n-\sigma} \I_{B_{R}}\Psi(x,|Du|)\,dx
\end{align}
for every $R\leqslant R_{0}\equiv R_{0}(\data(\O_0),\sigma)$. Next we consider the double nested exit time argument based on the proof of \cite[Theorem 2]{DO1}.

\textbf{Step 2: Double nested exit time and iteration.} Now we shall combine all the alternatives we have discussed with the estimates \eqref{nmd:4}, \eqref{nmd:7}, \eqref{nmd:10} and \eqref{nmd:12}. Take a ball $B_{R}\subset\O_{0} \Subset \O$ such that $R\leqslant R_{m} $, where $R_{m}= \min\{R_{ab}, R_{a}, R_{b}, R_{0}\}$ depends on $\data(\O_0)$ and $\sigma$. We consider $G$-phase in $B_{\tau_{ab}^{k+1}R}$ for every integer $k\geqslant 0$ and define the exit time index 
\begin{align}
	\label{nmd:13}
	t_{ab} = \min\{k\in \mathbb{N} : G-\text{phase in the ball } B_{\tau_{ab}^{k+1}R} \text{ does not occur}\}.
\end{align}
If there does not exist such $t_{ab}$, then for any $0<\rho <\tau_{ab}^{2}R < R \leqslant R_{m}$, there exists an integer $m\geqslant 1$ such that $\tau_{ab}^{m+2}R \leqslant \rho < \tau_{ab}^{m+1}R$. Using iterative \eqref{nmd:4}, we have 
\begin{align}
	\label{nmd:13_1}
	\begin{split}
	\I_{B_{\rho}}\Psi(x,|Du|)\,dx 
	&\leqslant \I_{B_{\tau_{ab}^{m+1}R}}\Psi(x,|Du|)\,dx
	\\&
	\leqslant
	\tau_{ab}^{(m-1)(n-\sigma)} \I_{\tau_{ab}^{2}R}\Psi(x,|Du|)\,dx
	=
	\tau_{ab}^{(m+2)(n-\sigma)} \tau_{ab}^{-3(n-\sigma)}\I_{\tau_{ab}^{2}R}\Psi(x,|Du|)\,dx
	\\&
	\leqslant
	c(\data(\O_0),\sigma)\left(\frac{\rho}{R} \right)^{n-\sigma}\I_{R}\Psi(x,|Du|)\,dx.
	\end{split}
\end{align}
Clearly, the above inequality holds true when $\tau_{ab}^{2}R \leqslant \rho \leqslant R \leqslant R_{m}$. So we consider the case of $t_{ab} < \infty$. For every $k\in \{1,\ldots,t_{ab}\}$, we apply \eqref{nmd:4} repeatedly in order to obtain 
\begin{align}
	\label{nmd:14}
	\I_{B_{\tau_{ab}^{k}R}}\Psi(x,|Du|)\,dx \leqslant
	\tau_{ab}^{k(n-\sigma)}\I_{B_{R}}\Psi(x,|Du|)\,dx.
\end{align}

By the very definition of $\tau_{ab}$ in \eqref{nmd:13}, we have three different scenarios: 
either $(G,H_{a})$-phase occurs in $B_{\tau_{ab}^{t_{ab}+1}R}$, $(G,H_{b})$-phase occurs in $B_{\tau_{ab}^{t_{ab}+1}R}$ or $(G,H_{a},H_{b})$-phase occurs in $B_{\tau_{ab}^{t_{ab}+1}R}$. Clearly, the last condition is stable for shrinking balls. Since the first two conditions can be considered similarly, we shall focus on the occurrence of $(G,H_{a})$-phase in the ball $B_{\tau_{ab}^{t_{ab}+1}R}$.  Let us define a second exit time index

\begin{align}
	\label{nmd:15}
	t_{b}:= \min\{k\in \mathbb{N} : (G,H_{a})-\text{phase in the ball } B_{\tau_{b}^{k+1}\tau_{ab}^{t_{ab}+1}R} \text{ does not occur}\}.
\end{align}

Arguing similarly as in \eqref{nmd:13_1} by using \eqref{nmd:7} if there is no such a finite number $t_{b}\in \mathbb{N}$, we are able to arrive at the inequality \eqref{nmd:23} below. So we only focus on the case of $t_{b}< \infty$. Iterating \eqref{nmd:7} with $B_{R}$ replaced by $B_{\tau_{ab}^{t_{ab}+1}R}$, we have 
\begin{align}
	\label{nmd:16}
	\I_{B_{\tau_{b}^{k}\tau_{ab}^{t_{ab}+1}R}} \Psi(x,|Du|)\,dx
	\leqslant
	\tau_{b}^{k(n-\sigma)} \I_{B_{\tau_{ab}^{t_{ab}+1}R}} \Psi(x,|Du|)\,dx
\end{align}
for every $k\in \{1,\ldots,t_{b}\}$. By again the very definition of $t_{b}$, there is only one chance that $(G,H_a,H_{b})$-phase occurs in the ball $B_{\tau_{b}^{t_{b}+1}\tau_{ab}^{t_{ab}+1}R}$. But as this condition is stable, we can iterate \eqref{nmd:12} for every $k\in \mathbb{N}$ in order to have 
\begin{align}
	\label{nmd:17}
	\I_{B_{\tau_{0}^{k}\tau_{b}^{t_{b}+1}\tau_{ab}^{t_{ab}+1}R}} \Psi(x,|Du|)\,dx
	\leqslant
	\tau_{0}^{k(n-\sigma)}
	\I_{B_{\tau_{b}^{t_{b}+1}\tau_{ab}^{t_{ab}+1}R}} \Psi(x,|Du|)\,dx.
\end{align}
Now we have all the needed estimates \eqref{nmd:14}, \eqref{nmd:16} and \eqref{nmd:17}. For $0<\rho <R \leqslant R_{m}$, we consider the following cases.

\textbf{Case 1: $R > \rho \geqslant \tau_{ab}^{t_{ab}+1}R$.} There exists $m\in \{0,1,\ldots,t_{ab}\}$ such that $\tau_{ab}^{m+1}R \leqslant \rho < \tau_{ab}^{m}R$. Then from \eqref{nmd:14}, we have 
\begin{align}
	\label{nmd:18}
	\begin{split}
	\I_{B_{\rho}}\Psi(x,|Du|)\,dx 
	&\leqslant \I_{B_{t_{ab}^{m}R}}\Psi(x,|Du|)\,dx
	\\&
	\leqslant
	\tau_{ab}^{m(n-\sigma)}\I_{B_{R}}\Psi(x,|Du|)\,dx 
	\\&
	\leqslant
	\tau_{ab}^{(m+1)(n-\sigma)}\tau_{ab}^{\sigma-n}\I_{B_{R}}\Psi(x,|Du|)\,dx 
	\\&
	\leqslant
	c(\data(\O_0),\sigma)\left(\frac{\rho}{R} \right)^{n-\sigma} \I_{B_{R}}\Psi(x,|Du|)\,dx ,
	\end{split}
\end{align}
where the last inequality is valid since $\tau_{ab}$ depends on $\data(\O_0)$ and $\sigma$. 

\textbf{Case 2: $\tau_{ab}^{t_{ab}+1}R > \rho \geqslant \tau_{b}\tau_{ab}^{t_{ab}+1} R$.} In this case, using \eqref{nmd:18}, we see 
\begin{align}
	\label{nmd:19}
	\begin{split}
	\I_{B_{\rho}}\Psi(x,|Du|)\,dx
	&\leqslant
	\I_{B_{\tau_{ab}^{t_{ab}+1}R}}\Psi(x,|Du|)\,dx
	\\&
	=
	\tau_{ab}^{(t_{ab}+1)(n-\sigma)}\I_{B_{R}}\Psi(x,|Du|)\,dx
	\\&
	\leqslant
	\left(\tau_{b}\tau_{ab}^{t_{ab}+1}\right)^{n-\sigma}\tau_{b}^{\sigma-n}\I_{B_{R}}\Psi(x,|Du|)\,dx
	\\&
	\leqslant
	 c(\data(\O_0),\sigma)\left(\frac{\rho}{R} \right)^{n-\sigma}\I_{B_{R}}\Psi(x,|Du|)\,dx,
	\end{split}
\end{align}
where again the last inequality is possible by the dependencies of $\tau_{b}$.

\textbf{Case 3: $\tau_{b}\tau_{ab}^{t_{ab}+1}R > \rho \geqslant \tau_{b}^{t_{b}+1}\tau_{ab}^{t_{ab}+1} R$.} Again there exists a natural number $m\in \{1,\ldots,t_{b}\}$ so that $\tau_{b}^{m}\tau_{ab}^{t_{ab}+1}R > \rho \geqslant \tau_{b}^{m+1}\tau_{ab}^{t_{ab}+1}R$. Therefore, using \eqref{nmd:16} and \eqref{nmd:18}, we have
\begin{align}
	\label{nmd:20}
	\begin{split}
	\I_{B_{\rho}}\Psi(x,|Du|)\,dx
	&\leqslant 
	\I_{B_{\tau_{b}^{m}\tau_{ab}^{t_{ab}+1}R}}\Psi(x,|Du|)\,dx
	\\&
	\leqslant
	\tau_{b}^{m(n-\sigma)} \I_{B_{\tau_{ab}^{t_{ab}+1}R}}\Psi(x,|Du|)\,dx
	\\&
	\leqslant
	\tau_{b}^{(m+1)(n-\sigma)}\tau_{b}^{\sigma-n} \tau_{ab}^{(t_{ab}+1)(n-\sigma)} \I_{B_{R}}\Psi(x,|Du|)\,dx
	\\&
	\leqslant
	c(\data(\O_0),\sigma)\left(\frac{\rho}{R} \right)^{n-\sigma}\I_{B_{R}}\Psi(x,|Du|)\,dx.
	\end{split}
\end{align}

\textbf{Case 4: $\tau_{b}^{t_{b}+1}\tau_{ab}^{t_{ab}+1}R > \rho \geqslant \tau_{b}^{t_{b}+1}\tau_{ab}^{t_{ab}+1}\tau_{0} R$.} Now by \eqref{nmd:20}, we find 
\begin{align}
	\label{nmd:21}
	\begin{split}
		\I_{B_{\rho}}\Psi(x,|Du|)\,dx 
		&\leqslant
		\I_{B_{\tau_{b}^{t_{b}+1}\tau_{ab}^{t_{ab}+1}R}}\Psi(x,|Du|)\,dx
		\\&
		\leqslant
		c\left(\tau_{b}^{t_{b}+1}\tau_{ab}^{t_{ab}+1} \right)^{n-\sigma} \I_{B_{R}}\Psi(x,|Du|)\,dx 
		\\&
		\leqslant
		c\tau_{0}^{\sigma-n}\left(\tau_{0}\tau_{b}^{t_{b}+1}\tau_{ab}^{t_{ab}+1} \right)^{n-\sigma} \I_{B_{R}}\Psi(x,|Du|)\,dx 
		\\&
		\leqslant
		c(\data(\O_0),\sigma)\left(\frac{\rho}{R} \right)^{n-\sigma}\I_{B_{R}}\Psi(x,|Du|)\,dx.
	\end{split}
\end{align}

\textbf{Case 5: $\tau_{b}^{t_{b}+1}\tau_{ab}^{t_{ab}+1}\tau_{0}R > \rho > 0$.} This condition implies that there exists a natural number $m\in \mathbb{N}$ such that 
$\tau_{0}^{m+1}\tau_{b}^{t_{b}+1}\tau_{ab}^{t_{ab}+1}R \leqslant \rho  < \tau_{0}^{m}\tau_{b}^{t_{b}+1}\tau_{ab}^{t_{ab}+1}R$. This time we apply \eqref{nmd:17} and \eqref{nmd:21} in order to have 
\begin{align}
	\label{nmd:22}
	\begin{split}
	\I_{B_{\rho}}\Psi(x,|Du|)\,dx 
	&\leqslant
	\I_{B_{\tau_{0}^{m}\tau_{b}^{t_{b}+1}\tau_{ab}^{t_{ab}+1}R}}\Psi(x,|Du|)\,dx 
	\\&
	\leqslant
	\tau_{0}^{m(n-\sigma)} \I_{B_{\tau_{b}^{t_{b}+1}\tau_{ab}^{t_{ab}+1}R}}\Psi(x,|Du|)\,dx
	\\&
	\leqslant
	 \tau_{0}^{m(n-\sigma)} \left( \tau_{b}^{t_{b}+1}\tau_{ab}^{t_{ab}+1} \right)^{n-\sigma} \I_{B_{R}}\Psi(x,|Du|)\,dx
	 \\&
	 \leqslant
	c \tau_0^{\sigma-n} \left(\frac{\rho}{R} \right)^{n-\sigma}\I_{B_{R}}\Psi(x,|Du|)\,dx
	=
	c(\data(\O_0),\sigma) \left(\frac{\rho}{R} \right)^{n-\sigma}\I_{B_{R}}\Psi(x,|Du|)\,dx.
	\end{split}
\end{align}
As we discussed earlier after \eqref{nmd:14}, we can proceed the same for the occurrence of $(G,H_{a})$-phase in the ball $B_{\tau_{ab}^{\tau_{ab}+1}R}$ instead of the occurrence of $(G,H_{b})$-phase in the ball $B_{\tau_{ab}^{\tau_{ab}+1}R}$. Then we can directly jump to the case that $(G,H_{a},H_{b})$-phase occurs in the ball $B_{\tau_{ab}^{\tau_{ab}+1}R}$, which is trivial by \eqref{nmd:12}. Moreover, if we start with the occurrence of $(G,H_{a},H_{b})$-phase in $B_{R}$, then the procedure will be much easier by \eqref{nmd:12}. Taking into account all the possible cases that we considered above, we can conclude that, for every $\sigma\in (0,1)$, there exists $c\equiv c(\data(\O_0),\sigma)$ such that 
\begin{align}
	\label{nmd:23}
	\I_{B_{\rho}}\Psi(x,|Du|)\,dx 
	\leqslant 
	c\left(\frac{\rho}{R} \right)^{n-\sigma}\I_{B_{R}}\Psi(x,|Du|)\,dx 
\end{align}
holds true, whenever $0<\rho <R \leqslant R_{m}$, where $R_{m}$ is some positive radius depending only on $\data(\O_0)$ and $\sigma$ in the beginning of the proof. In order to complete the proof, we need to consider the remaining cases. If $0<R_{m}\leqslant \rho < R \leqslant 1$, then we have 
\begin{align}
	\label{nmd:24}
	\begin{split}
	\I_{B_{\rho}}\Psi(x,|Du|)\,dx
	&\leqslant
	\left(\frac{\rho}{R} \right)^{n-\sigma} \left(\frac{R}{\rho} \right)^{n-\sigma} \I_{B_{R}}\Psi(x,|Du|)\,dx
	\\&
	\leqslant
	\left(\frac{\rho}{R} \right)^{n-\sigma} \left(\frac{R}{R_{m}} \right)^{n-\sigma} \I_{B_{R}}\Psi(x,|Du|)\,dx
	\\&
	\leqslant
	c(\data(\O_0),\sigma)\left(\frac{\rho}{R} \right)^{n-\sigma}\I_{B_{R}}\Psi(x,|Du|)\,dx,
	\end{split}
\end{align} 
where we have used the dependence of $R_{m}$. Finally, if $0<\rho < R_{m} \leqslant R \leqslant 1$, then by \eqref{nmd:23} and \eqref{nmd:24}, we see 
\begin{align}
	\label{nmd:25}
	\begin{split}
	\I_{B_{\rho}}\Psi(x,|Du|)\,dx
	&\leqslant
	c\left(\frac{\rho}{R_{m}} \right)^{n-\sigma}  \I_{B_{R_{m}}}\Psi(x,|Du|)\,dx
	\\&
	\leqslant
	c\left(\frac{\rho}{R_{m}} \right)^{n-\sigma} \left(\frac{R_{m}}{R} \right)^{n-\sigma} \I_{B_{R}}\Psi(x,|Du|)\,dx
	\\&
	=
	c(\data(\O_0),\sigma)\left(\frac{\rho}{R} \right)^{n-\sigma}\I_{B_{R}}\Psi(x,|Du|)\,dx,
	\end{split}.
\end{align}

All in all, collecting the estimates obtained in \eqref{nmd:23}-\eqref{nmd:25}, we arrive at the validity of the Morrey type inequality \eqref{mth:md:7}. The proof is complete.


Now we consider a crucial outcome of Theorem \ref{mth:md}, which plays a crucial role for proving Theorem \ref{mth:mr} afterwards. 
\begin{lem}
	\label{lem:5ce}
	Under the assumptions of Lemma \ref{lem:2ce}, let $w_{c}\in W^{1,\Psi}(B_{R})$ be the solution to the problem defined in \eqref{1ce:2}. Suppose that  \eqref{mth:md:3} is satisfied for $\omega_{a}(t)=t^{\alpha}$ and $\omega_{b}(t)=t^{\beta}$ with some $\alpha,\beta\in (0,1]$. Then there exists $h \in w_c + W^{1,\Psi_{B_{R}}^{-}}_{0}(B_{R/8})$ being  a local minimizer of the functional $\F_{0}$ defined in \eqref{3ce:1} such that 
\begin{align}
	\label{5ce:1}
	\begin{split}
	\FI_{B_{R/8}}&\left[ |V_{G}(Du)-V_{G}(Dh)|^2 + a(x_a)|V_{H_{a}}(Du)-V_{H_{a}}(Dh)|^2 + b(x_b)|V_{H_{b}}(Du)-V_{H_{b}}(Dh)|^2  \right]\,dx
	\\&
	\leqslant
	c\left( \omega\left(R^{\gamma}\right) + \left[R^{\frac{\alpha}{2(1+s(H_{a})})} + R^{\frac{\beta}{2(1+s(H_{b}))}}\right]^{s_1}  \right)\FI_{B_{2R}}\Psi(x,|Du|)\,dx
	\end{split}
\end{align}
for some constant $c\equiv c(\data(\O_0))$ and $s_1\equiv s_1(\data)$, respectively. Moreover, the energy estimate
\begin{align}
	\label{5ce:2}
	\FI_{B_{R/8}}\Psi_{B_{R}}^{-}(|Dh|)\,dx \leqslant c\FI_{B_{2R}}\Psi(x,|Du|)\,dx
\end{align}
holds for some constant $c\equiv c(n,\nu,L)$.
\end{lem}
\begin{proof}
	Essentially, the above lemma is a special case of Lemma \ref{lem:3ce} since we consider a particular case that $\omega_{a}(t)=t^{\alpha}$ and $\omega_{b}(t)=t^{\beta}$ for some $\alpha,\beta\in (0,1)$. But our purpose here is to obtain an estimate such as \eqref{5ce:1} with a different multiplier containing some power of $R$, which will be used for proving Theorem \ref{mth:mr}. Therefore, we are able to apply Theorem \ref{mth:md}. In turn, for every $\theta\in (0,1)$ and open sunset $\O_0\Subset \O$, there exists a constant $c\equiv c(\data(\O_0),\theta)$ such that
	\begin{align}
		\label{5ce:3}
		[u]_{0,\theta;\O_0} \leqslant c(\data(\O_0),\theta).
	\end{align}
In particular, we choose $\theta:= (\gamma+1)/2$. Now we need to revisit the proof of Lemma \ref{lem:2ce}. Under the settings of the proof of Lemma \ref{lem:2ce}, we turn our attention to estimating the terms $I_{a}$ and $I_{b}$ introduced in \eqref{2ce:19}-\eqref{2ce:20}. Using \eqref{ma:3}, \eqref{1ce:6} and \eqref{0ce:6}, we have 
\begin{align}
	\label{5ce:4}
	\begin{split}
	I_{a} &\leqslant cR^{\alpha}\left( 1 + \left[ \left(\Psi_{B_{R}}^{-} \right)^{-1}\left(\FI_{B_{R/2}}\Psi_{B_{R}}^{-}\left(\left|\frac{w_c-(w_c)_{B_{R/2}}}{R} \right| \right)\,dx \right) \right]^{\frac{\alpha}{1-\gamma}} \right)
	\\&
	\leqslant
	c\left(R^{\alpha}+ R^{-\frac{\alpha\gamma}{1-\gamma}}\left[\osc\limits_{B_{2R}}u\right]^{\frac{\alpha}{1-\gamma}}\right)
	\leqslant
	c(\data(\O_0))R^{\alpha/2},
	\end{split}
\end{align}
where we have used \eqref{5ce:3} with the choice of $\theta:= (1+\gamma)/2$ and $B_{2R}\subset\O_0$ with $R\leqslant 1$. In the same way, we show 
\begin{align}
	\label{5ce:5}
	I_{b} \leqslant c(\data(\O_0))R^{\beta/2}.
\end{align}

Inserting those estimates into \eqref{2ce:20}, we see that 

\begin{align}
	\label{5ce:6}
	\left|\FI_{B_{1/8}}\inner{\bar{A}_{0}(D\bar{w}_{c})}{D\varphi}\,dx \right|
	\leqslant
	c(\data(\O_0))q_{3}(R)\norm{D\varphi}_{L^{\infty}(B_{1/8})},
\end{align}
whenever $\varphi\in W^{1,\infty}_{0}(B_{1/8})$, where 
\begin{align}
	\label{5ce:7}
	q_{3}(R):= R^{\frac{\alpha}{2(1+s(H_{a}))}} + R^{\frac{\beta}{2(1+s(H_{b}))}}.
\end{align}
Note that the vector field  $\bar{A}_{0}$ has been defined in \eqref{2ce:12_2}.
We consider a function $\bar{h}\in \bar{w}_{c} + W^{1,\bar{\Psi}_{0}}_{0}(B_{1/8})$ satisfying \eqref{2ce:45}-\eqref{2ce:48} with the term $d(\varepsilon,R)$ replaced by $q_{3}(R)$. 
 Let $h$ be the scaled back function of $\bar{h}$ in $B_{R/8}$ as
	\begin{align}
		\label{5ce:8}
		h(x):= E(w_c,B_{R/2})R\bar{h}\left(\frac{x-x_0}{R} \right)
		\quad\text{for every}\quad
		x\in B_{R/8}(x_0).
	\end{align}
Clearly, $h\in w_{c} + W^{1,\Psi_{B_{R}}^{-}}_{0}(B_{R/8})$ is a local minimizer of the functional $\F_0$ defined in \eqref{3ce:1}, which means that 
\begin{align}
	\label{5ce:9}
	\F_{0}(h) = \I_{B_{R/8}}F_0(Dh)\,dx \leqslant \I_{B_{R/8}}F_0(Dh+D\varphi)\,dx
	=
	\F_{0}(h+\varphi)
\end{align}
holds for every $\varphi\in W^{1,\Psi_{B_{R}}^{-}}_{0}(B_{R/8})$. As we have shown in \eqref{0ce:9}, we recall \eqref{1ce:4} and \eqref{0ce:4} to see that
\begin{align}
	\label{5ce:10}
	\begin{split}
	\FI_{B_{R/8}}\Psi_{B_{R}}^{-}(|Dh|)\,dx 
	&\leqslant
	\frac{L}{\nu}\FI_{B_{R/8}}\Psi_{B_{R}}^{-}(|Dw_{c}|)\,dx 
	\leqslant
	\frac{8^{n}L}{\nu} \FI_{B_{R}}\Psi(x,|Dw_c|)\,dx
	\\&
	\leqslant
	c(n,\nu,L) \FI_{B_{R}}\Psi(x,|Dw|)\,dx
	\leqslant
	c(n,\nu,L)  \FI_{B_{2R}}\Psi(x,|Du|)\,dx,
	\end{split}
\end{align}
which proves \eqref{3ce:3}. We write the inequality \eqref{2ce:47} in view of $G,H_{a}, H_{b}$, $w_{c}$ and $h$ in order to have 
\begin{align}
	\label{5ce:11}
	\begin{split}
		\FI_{B_{R/8}}&\left[ |V_{G}(Dw_c)-V_{G}(Dh)|^2 + a(x_a)|V_{H_{a}}(Dw_c)-V_{H_{a}}(Dh)|^2 + b(x_b)|V_{H_{b}}(Dw_c)-V_{H_{b}}(Dh)|^2  \right]\,dx
		\\&
		\leqslant
		c[q_{3}(R)]^{s_1}\FI_{B_{R/2}}\Psi_{B_{R}}^{-}\left(\left|\frac{w_c-(w_c)_{B_{R/2}}}{R} \right| \right)\,dx
		\leqslant
		c[q_3(R)]^{s_1}\FI_{B_{R/2}}\Psi_{B_{R}}^{-}\left(\left|Dw_c \right| \right)\,dx
		\\&
		\leqslant
		c[q_{3}(R)]^{s_1}\FI_{B_{R/2}}\Psi\left(x,\left|Du \right| \right)\,dx
	\end{split}
\end{align}
for some constant $c\equiv c(\data(\O_0))$, where we have applied Sobolev-Poincar\'e inequality and \eqref{3ce:6}. Combining this estimate together with \eqref{0ce:3} and \eqref{1ce:3} alongside some elementary computations, we arrive at the desired estimate \eqref{5ce:1}. 
\end{proof}

 
\section{Proof of Theorem \ref{mth:mr}.}
 \label{sec:9}
Finally, we are ready to prove Theorem \ref{mth:mr}. First applying Theorem \ref{mth:md} and a standard covering argument, we find that for every open subset $\O_0\Subset \O$ and any number $k>0$, there exists a constant  $c\equiv c(\data(\O_0),k)$ such that 
\begin{align}
	\label{mr:1}
	\FI_{B_{2R}}\Psi(x,|Du|)\,dx \leqslant cR^{-k},
\end{align}
whenever $B_{2R}\subset\O_0$ is a ball with $R\leqslant 1$. Now we fix an open subset  $\O_0\Subset\O$ and a ball $B_{2R}\equiv B_{2R}(x_0)\subset \O_0$ with $R\leqslant 1$. Then applying Lemma \ref{lem:3ce} and Lemma \ref{lem:5ce}, 

\begin{align}
	\label{mr:2}
	\begin{split}
	\FI_{B_{R/8}} &\left( |V_{G}(Du)-V_{G}(Dh)|^2  + a(x_a)|V_{H_{a}}(Du)-V_{H_{a}}(Dh)|^2 + 
	b(x_b)|V_{H_{b}}(Du)-V_{H_{b}}(Dh)|^2 \right)\,dx
	\\&
	\leqslant
	c\left( R^{\mu\gamma} + [q(R)]^{s_1} \right)\FI_{B_{2R}}\Psi(x,|Du|)\,dx
	\end{split}
\end{align}
for some constant $c\equiv c(\data(\O_0))$ and $s_1\equiv s_1(\data)$, where 

\begin{align}
	\label{mr:3}
	 q(R)
	 := \left\{\begin{array}{lr}
        R^{\frac{\alpha\delta}{(1+\delta)(1+s(H_{a}))}} +  R^{\frac{\beta\delta}{(1+\delta)(1+s(H_{b}))}}  & \text{if } \eqref{mth:mr:1} \text{ is assumed, } \\
        R^{\frac{\alpha\gamma}{1+s(H_{a})}} + R^{\frac{\beta\gamma}{1+s(H_{b})}}  & \text{if } \eqref{mth:mr:2} \text{ is assumed, } \\
        R^{\frac{\alpha}{2(1+s(H_{a}))}} + R^{\frac{\beta}{2(1+s(H_{b}))}}  & \text{if } \eqref{mth:mr:3} \text{ is assumed, }
        \end{array}\right.
\end{align}
in which $\gamma$ is the H\"older continuity exponent determined via Theorem \ref{thm:hc} and $\delta$ is the higher integrability exponent coming from Theorem \ref{thm:hi}. We denote by
\begin{align}
	\label{mr:3_1}
	 d\equiv d(\data(\O_0))
	 := \left\{\begin{array}{lr}
        \min\left\{\mu\gamma, \frac{\alpha\delta s_1}{(1+\delta)(1+s(H_{a}))}, \frac{\beta\delta s_1}{(1+\delta)(1+s(H_{b}))} \right\} & \text{if } \eqref{mth:mr:1} \text{ is assumed, } \\
        \min\left\{\mu\gamma, \frac{\alpha\gamma s_1}{1+s(H_{a})}, \frac{\beta\gamma s_1}{1+s(H_{b})} \right\}  & \text{if } \eqref{mth:mr:2} \text{ is assumed, } \\
        \min\left\{\mu\gamma, \frac{\alpha s_1}{2(1+s(H_{a}))} , \frac{\beta s_1}{2(1+s(H_{b}))} \right\}  & \text{if } \eqref{mth:mr:3} \text{ is assumed, }
        \end{array}\right.
\end{align}
and $x_{a},x_{b}\in \overline{B_{R}}$ are points such that $a(x_a) = \inf\limits_{x\in B_{R}}a(x)$ and $b(x_b)=\inf\limits_{x\in B_{R}}b(x) $. Now choosing $k\equiv d/4$ in \eqref{mr:1}, the inequality \eqref{mr:2} can be written as 
\begin{align}
	\label{mr:4}
	\begin{split}
	\FI_{B_{R/8}} &\left( |V_{G}(Du)-V_{G}(Dh)|^2  + a(x_a)|V_{H_{a}}(Du)-V_{H_{a}}(Dh)|^2 + 
	b(x_b)|V_{H_{b}}(Du)-V_{H_{b}}(Dh)|^2 \right)\,dx
	\\&
	\leqslant
	cR^{3d/4}
	\end{split}
\end{align}
for some constant $c\equiv c(\data(\O_0))$, where we again recall that the function $h$ has been defined via Lemma \ref{lem:3ce} and Lemma \ref{lem:5ce}. Recalling, \eqref{3ce:3} and \eqref{5ce:2}, we have the energy estimate 
\begin{align}
	\label{mr:5}
	\FI_{B_{R/8}}\Psi_{B_{R}}^{-}(|Dh|)\,dx
	\leqslant
	c\FI_{B_{2R}}\Psi(x,|Du|)\,dx
\end{align} 
with a constant $c\equiv c(n,\nu,L)$. Now using repeatedly \eqref{defV2}, we have 

\begin{align}
    \label{mr:6}
     \begin{split}
         \FI_{B_{R/8}}& \Psi_{B_{R}}^{-}(|Du-Dh|)\,dx 
         \leqslant
         c\FI_{B_{R/8}} \left(\left[\Psi_{B_{R}}^{-}(|Du| + |Dh|)\right]^{\frac{1}{2}}\frac{|Du-Dh|}{|Du|+|Dh|}\right) \left[\Psi_{B_{R}}^{-}(|Du| + |Dh|)\right]^{\frac{1}{2}}\,dx
         \\&
         \leqslant
         c\left( \FI_{B_{R/8}} \Psi_{B_{R}}^{-}(|Du| + |Dh|)\frac{|Du-Dh|^2}{(|Du| + |Dh|)^2}\,dx \right)^{\frac{1}{2}}
         \left( \FI_{B_{R/8}} \Psi_{B_{R}}^{-}(|Du| + |Dh|)\,dx \right)^{\frac{1}{2}}
         \\&
         \stackrel{\eqref{defV2},\eqref{mr:4}}{\leqslant}
         cR^{3d/8}\left( \FI_{B_{R/8}} \Psi_{B_{R}}^{-}(|Du| + |Dh|)\,dx \right)^{\frac{1}{2}}
         \stackrel{\eqref{mr:5}, \eqref{mr:1}}{\leqslant}
         cR^{d/4}
     \end{split}
\end{align}
with $c\equiv c(\data(\O_0))$, where $d$ has been introduced in \eqref{mr:3}. Since $h$ is a minimizer of functional $\F_{0}$ defined in \eqref{3ce:1}, and this functional satisfies the growth and ellipticity conditions $\eqref{sa:2}_{1,2}$ with $a(x)\equiv a(x_a)$ and $b(x)\equiv b(x_b)$ , we are able to apply the theory in \cite{L1}, which provides the gradient H\"older regularity with the estimates
\begin{align}
    \label{mr:7}
    \begin{split}
        \FI_{B_{\rho}} \Psi_{B_{R}}^{-}(|Dh-(Dh)_{B_{\rho}}|) \,dx
        &\leqslant
        c\left( \frac{\rho}{R} \right)^{\beta_1} \FI_{B_{R/8}} \Psi_{B_{R}}^{-}(|Dh|)\,dx
        \\&
        \stackrel{\eqref{mr:5}}{\leqslant}
        c\left( \frac{\rho}{R} \right)^{\beta_1} \FI_{B_{2R}}\Psi(x,|Du|)\,dx,
    \end{split}
\end{align}
 whenever $0<\rho \leqslant R/8$, where the constants $c,\beta_1$ depend only on $n,s(G),s(H_{a}), s(H_{b}),\nu,L$, but are independent of the values $a(x_a)$ and $b(x_b)$. Therefore, for every $0 < \rho\leqslant R/8$, we have 
\begin{align}
    \label{mr:8}
    \begin{split}
        \FI_{B_{\rho}}& G(|Du-(Du)_{B_{\rho}}|)\,dx 
        \\&
        \leqslant
        c\FI_{B_{\rho}}G(|Dh-(Dh)_{B_{\rho}}|)\,dx + c \FI_{B_{\rho}}G(|Du-Dh|)\,dx
        \\&
        \stackrel{\eqref{mr:7}}{\leqslant}
        c\left( \frac{\rho}{R} \right)^{\beta_1}\FI_{B_{2R}}\Psi(x,|Du|)\,dx
        + c\left( \frac{R}{\rho} \right)^{n}\FI_{B_{R/8}}G(|Du-Dh|)\,dx
        \\&
        \stackrel{\eqref{mr:1},\eqref{mr:6}}{\leqslant}
        c\left(\frac{\rho}{R}\right)^{\beta_1}R^{-k} + c\left(\frac{R}{\rho}\right)^{n}R^{d/4}
    \end{split}
\end{align}
with $c\equiv c(\data(\O_0),k)$. Notice that $k\in (0,1)$ is still arbitrary and $d$ has been defined in \eqref{mr:3} depending only on $\data(\O_0)$.  Taking $k\equiv d\beta_1/(32n)$ and $\rho \equiv (R/8)^{1+d/(16n)}$ in the last display, after some elementary manipulations, we get 
\begin{align}
    \label{mr:9}
    \FI_{B_{\rho}}G(|Du-(Du)_{B_{\rho}}|)\,dx \leqslant c\rho^{\frac{d\beta_1}{64n}} 
\end{align}
 for every $\rho\in (0,1/8)$, provided $B_{8\rho}\Subset \O_0$. In particular, using Jensen's inequality and Lemma $\ref{lem:nf1}_{1}$, we have 
\begin{align}
    \label{mr:10}
    \FI_{B_{\rho}}|Du-(Du)_{B_{\rho}}|\,dx \leqslant c\rho^{\frac{d\beta_1}{64n}\left(1+\frac{1}{s(G)}\right)} 
\end{align}
 for every $\rho\in (0,1/8)$ with $B_{8\rho}\Subset \O_0$. By the integral characterization of H\"older continuity due to Campanato and Meyers and a standard covering argument alongside \eqref{mr:10}, $Du\in C^{0,\theta}_{\loc}(\O)$ for $\theta\equiv \frac{d\beta_1}{64n}\left(1+\frac{1}{s(G)}\right)$. This proves the local H\"older continuity of $Du$. But the proof is not finished yet, since $\theta$ should be independent of $\O_0$ as in the statement of Theorem \ref{mth:mr}. In order to obtain the full completeness, we apply some standard perturbation methods.
  Indeed, once we have that $Du$ is locally bounded, we shall revisit the proof of Lemma \ref{lem:3ce} and Lemma \ref{lem:5ce}. We also observe that the functional defined in \eqref{3ce:1} satisfies the bounded slope condition (see for instance \cite{BBr1}). Then there exists a constant 
$c\equiv c(n,s(G),s(H_{a}),s(H_{b}),\nu,L, \norm{Du}_{L^{\infty}(B_{R})})$ such that 
\begin{align*}
    \norm{Dh}_{L^{\infty}(B_{R})} \leqslant c.
\end{align*}
Since $Du$ is locally bounded, following the proof of Lemma \ref{lem:0ce}, Lemma \ref{lem:3ce} and Lemma \ref{lem:5ce}, specially the estimate in \eqref{0ce:3} can be modified with $\gamma\equiv 1$. Moreover, the estimates in \eqref{5ce:1} and \eqref{3ce:2} can be upgraded by 
\begin{align}
    \label{mr:11}
    \begin{split}
    \FI_{B_{R/8}}&\left( |V_{G}(Du)-V_{G}(Dh)|^2 + a(x_a)|V_{H_{a}}(Du)-V_{H_{a}}(Dh)|^2 + b(x_b)|V_{H_{b}}(Du)-V_{H_{b}}(Dh)|^2\right)\,dx 
    \\&
    \leqslant cR^{\min\{\mu,\alpha,\beta\}}
    \end{split}
\end{align}
with some constant $c\equiv c(n,s(G),s(H_{a}),s(H_{b}),\nu,L,\norm{a}_{L^{\infty}(\O_0)}, \norm{b}_{L^{\infty}(\O_0)},\norm{Du}_{L^{\infty}(B_{2R})})$. In particular, the last estimate via \eqref{mr:6} implies that 
\begin{align}
    \label{mr:12}
    \FI_{B_{R}}G(|Du-Dh|)\,dx \leqslant cR^{\min\{\mu,\alpha,\beta\}/4}.
\end{align}
Therefore, \eqref{mr:7} implies that
\begin{align}
    \label{mr:13}
    \FI_{B_{\rho}} G(|Dh-(Dh)_{B_{\rho}}|)\,dx \leqslant c\left( \frac{\rho}{R} \right)^{\beta_1},
\end{align}
where $\beta_1$ depends on $n,s(G),s(H_{a}), s(H_{b}),\nu,L$ while the constant $c$ depends only on $n,s(G),s(H_{a}),s(H_{b}),\nu,$ $L$, $ \norm{Du}_{L^{\infty}(\O_0)}$, $\norm{a}_{L^{\infty}(\O_0)}$ and $ \norm{b}_{L^{\infty}(\O_0)}$. Combining the last two estimates similarly as shown in \eqref{mr:8}, we deduce the gradient H\"older continuity with the exponent depending only on $n,s(G),s(H_{a}), s(H_{b}),\nu,L,\alpha, \beta$ and $\mu$, which is the desired dependence as described in the statement. The proof is finally complete.


\section{Orlicz double phase problems}
\label{sec:10}

Let us consider a general class of functionals with double phase growth, which is essentially the case when $b(\cdot)\equiv 0$ in \eqref{psi}. The functionals we shall deal with is of type 
\begin{align}
	\label{dfunctional}
	W^{1,1}(\O)\ni v\mapsto \F_{d}(v,\O):= \I_{\O}F_{d}(x,v,Dv)\,dx,
\end{align}
 where $F_{d} : \Omega\times \R \times \R^n\rightarrow \R$ is a Carathe\'odory function fulfilling the following double-sided growth  
 \begin{align}
 	\label{dsa:1}
 	\nu\Psi_{d}(x, |z|) \leqslant F_{d}(x,y,z) \leqslant 
 	L\Psi_{d}(x,|z|),
 \end{align}
whenever $x\in\O$, $y\in\R$ and $z\in\R^n$, in which here and in the rest of the paper we denote by 
\begin{align}
	\label{dpsi}
	\Psi_d(x,t):= G(t) + a(x)H_{a}(t)\quad (\forall x\in\O,\, t\geqslant 0).
\end{align}
As we introduced before we assume $G,H_{a}\in \mathcal{N}$ with indices $s(G), s(H_{a})\geqslant 1$ and $a(\cdot)\in C^{\omega_{a}}(\O)$ with  $\omega_{a} : [0,\infty) \rightarrow [0,\infty)$ being a concave function such that $\omega_{a}(0)=0$. We shall consider a local minimizer $u$ of the functional $\F_{d}$ in \eqref{dfunctional} under one of the assumptions \eqref{ma:1}, \eqref{ma:2} and \eqref{ma:3} with $\omega_{b}(\cdot)\equiv 0$. Since the double sided growth assumption \eqref{dpsi} is not enough for higher regularity properties of a local minimizer $u$ of the functional $\F_{d}$, we shall assume that  $F_{d}$ is a continuous integrand belonging to the space $C^{2}(\R^n\setminus \{0\})$ with respect to $z$-variable and having the the following structure assumptions:

\begin{align}
 	\label{dsa:2}
 	\begin{cases}
 		|D_{z}F_{d}(x,y,z)||z| + |D^2_{zz}F_{d}(x,y,z)||z|^2 \leqslant L\Psi_{d}(x,|z|), \\
 		\nu\dfrac{\Psi_{a}(x,|z|)}{|z|^2}|\xi|^2 \leqslant \inner{D_{zz}^{2}F_{d}(x,y,z)\xi}{\xi}, \\
 		\begin{aligned}
 		|D_{z}F_{d}(x_1,y,z)-D_{z}F_{d}(x_2,y,z)||z|\leqslant {}&
 		L\omega(|x_1-x_2|)[\Psi_{d}(x_1,|z|) 
 		+\Psi_{d}(x_2,|z|)] 
 		\\&
 		+ L|\Psi_{d}(x_1,|z|)-\Psi_{d}(x_2,|z|)|,
 		\end{aligned} \\
 		|F_{d}(x,y_1,z)-F_{d}(x,y_2,z)| \leqslant L\omega(|y_1-y_2|)\Psi_{d}(x,|z|),
 	\end{cases}
 \end{align}
  whenever  $x,x_1,x_2\in\O$, $y,y_1,y_2\in \R$, $z\in\R^n\setminus\{0\}$, $\xi\in\R^n$, where $0 <\nu\leqslant L$ are fixed constants, and the function $\omega$ is the same as defined in \eqref{omega1} or \eqref{omega2}.
 The structure conditions in \eqref{dsa:2} are satisfied for instance by the model functional 
 \begin{align}
 	\label{dcase:1}
 	W^{1,1}(\O)\ni\u\mapsto \I_{\O}f_{d}(x,\u)\Psi_{d}(x,|D\u|)\,dx,
 \end{align}
 where the continuous function $f_{d}(\cdot)$ satisfies $0<\nu_0 \leqslant f(\cdot,\cdot)\leqslant L_0$ for some constants $\nu_0, L_0$ and fulfills the following inequality
 \begin{align*}
 	|f_{d}(x_1,y_1)-f_{d}(x_2,y_2)|\leqslant L_{0}\omega(|x_1-x_2| +|y_1-y_2|),
\end{align*} 
whenever $x_1,x_2\in\R^n$ and $y_1,y_2\in\R$, in which $\omega$ is the same as defined in \eqref{omega1} or \eqref{omega2}. Another model case is given by 
 \begin{align}
 	\label{dcase:2}
 	W^{1,1}(\O)\ni \u\mapsto \I_{\O}[F_{G}(x,\u,D\u)+ a(x)F_{H_{a}}(x,\u,D\u)]\,dx,
 \end{align}
 where $F_{G}(\cdot)$ and $F_{H_{a}}(\cdot)$ have $G-$growth and $H_{a}-$growth respectively, and satisfy the following suitable structure assumptions that
 
 \begin{align*}
 	\begin{cases}
 		|D_{z}F_{\Phi}(x,y,z)||z| + |D^2_{zz}F_{\Phi}(x,y,z)||z|^2 \leqslant L_0\Phi(|z|), \\
 		\nu_0\dfrac{\Phi(|z|)}{|z|^2}|\xi|^2 \leqslant \inner{D_{zz}^{2}F_{\Phi}(x,y,z)\xi}{\xi}, \\
 		|D_{z}F_{\Phi}(x_1,y,z)-D_{z}F_{\Phi}(x_2,y,z)||z| \leqslant  L_0\omega(|x_1-x_2|)\Phi(|z|), \\
 		|F_{\Phi}(x,y_1,z)-F_{\Phi}(x,y_2,z)| \leqslant L_0\omega(|y_1-y_2|)\Phi(|z|)
 	\end{cases}
 \end{align*}
hold with $\Phi\in\{G,H_{a} \}$ for some positive constants $\nu_0,L_0$, where $\omega$ is as in \eqref{omega1} or \eqref{omega2}. The reason we consider the double phase case independently is that we have discussed the various regularity properties of the functional $\F$ in \eqref{functional} in the sense of multi-phase of the type defined in \eqref{type} together with the structure assumptions \eqref{sa:2}, but this one is a special case of \eqref{dfunctional} together with the structure assumptions \eqref{dsa:2} in the sense of the double phase structures. Now we restate and prove Lemma \ref{lem:nf4} in the double phase settings which will be applied later.

\begin{lem}
    \label{lem:dnf4}
    Let $F_{d}: \O\times\R\times\R^n\rightarrow \R$ be a function defined in \eqref{dfunctional} which satisfies \eqref{dsa:1} and \eqref{dsa:2}. There exist positive constants $c_1,c_2\equiv c_1,c_2(n,s(G),s(H_{a}),\nu)$ such that the following inequalities
     \begin{align}
	    \label{ddefV1_1}
	    \begin{split}
	    &|V_{G}(z_1)-V_{G}(z_2)|^2 + a(x)|V_{H_{a}}(z_1)-V_{H_{a}}(z_2)|^2 
	    \leqslant c_1\inner{D_{z}F_{d}(x,y,z_1)-D_{z}F_{d}(x,y,z_2)}{z_1-z_2},
	    \end{split}
    \end{align}
    \begin{align}
        \label{ddefV1_2}
        \begin{split}
            |V_{G}(z_1)-V_{G}(z_2)|^2 + a(x)|V_{H_{a}}(z_1)-V_{H_{a}}(z_2)|^2 
             + & c_2\inner{D_{z}F_{d}(x,y,z_1)}{z_2-z_1}
             \\&
            \leqslant
            c_2[F_{d}(x,y,z_2)-F_{d}(x,y,z_1)]
        \end{split}
    \end{align}
    and 
    \begin{align}
        \label{ddefV1_3}
        \begin{split}
        |F_{d}(x_1,y,z)-F_{d}(x_2,y,z)| 
        \leqslant
        L\omega(|x_1-x_2|)\left[ \Psi_{d}(x_1,|z|) + \Psi_{d}(x_2,|z|) \right]
         + L|a(x_1)-a(x_2)|H_{a}(|z|)
        \end{split}
    \end{align}
    hold true, whenever $z,z_1,z_2\in \R^n \setminus \{0\}$, $x,x_1,x_2\in\O$ and $y\in\R$.
\end{lem}

\begin{proof}
	The arguments of the proof for \eqref{ddefV1_1} and \eqref{ddefV1_2} are essentially the same as done for Lemma \ref{lem:nf4}. Only difference lies in the one for \eqref{ddefV1_3}. 
	Since $F_{d}(x,y,0)=0$ for every $x\in\O$ and $y\in\R$, we have
    \begin{align*}
    \begin{split}
        &|F_{d}(x_1,y,z)-F_{d}(x_2,y,z)| = |(F_{d}(x_1,y,z)-F_{d}(x_1,y,0))-(F_{d}(x_2,y,z)-F_{d}(x_2,y,0))|
        \\&
        =
        \left|\I_{0}^{1}\inner{D_{z}F_{d}(x_1,y,\theta z)}{z}\,d\theta - \I_{0}^{1}\inner{D_{z}F_{d}(x_2,y,\theta z)}{z}\,d\theta\right|
        \\&
        \leqslant
        \I_{0}^{1}|D_{z}F_{d}(x_1,y,\theta z)-D_{z}F_{d}(x_2,y,\theta z)|\,|z|d\theta
        \\&
        \leqslant
        L\omega(|x_1-x_2|)\left[ \Psi_{d}(x_1,|z|) + \Psi_{d}(x_2,|z|) \right]
         + L|a(x_1)-a(x_2)|H_{a}(|z|),
    \end{split}
    \end{align*}
    where the last inequality of the last display is implied by $\eqref{dsa:2}_{3}$. This proves \eqref{ddefV1_3}.
\end{proof}

 In order to simplify the notations in the present section, we use the set of parameters for a minimizer $u$ of the functional $\F_{d}$ depending on which one of the assumptions \eqref{ma:1}-\eqref{ma:3} under $\omega_{b}(\cdot)\equiv 0$ comes into play as the $\data$ in this section.

\begin{align}
	\label{ddata}
	\data_{d}
	 \equiv \left\{\begin{array}{ll}
        &n,\lambda_{1},s(G),s(H_{a}),\nu,L, \norm{a}_{C^{\omega_{a}}(\O)}, \omega(\cdot), \norm{\Psi(x,|Du|)}_{L^1(\O)}, \norm{u}_{L^{1}(\O)}, \omega_a(1) \\
        & \text{if } \eqref{ma:1} \text{ is considered under } \omega_{b}(\cdot)\equiv 0,\\
        & n,\lambda_{2},s(G),s(H_{a}),\nu,L,  \norm{a}_{C^{\omega_{a}}(\O)},\omega(\cdot), \norm{u}_{L^{\infty}(\O)},\omega_a(1) \\
         & \text{if } \eqref{ma:2} \text{ is considered under } \omega_{b}(\cdot)\equiv 0, \\
       & n,\lambda_{3},s(G),s(H_{a}),\nu,L,  \norm{a}_{C^{\omega_{a}}(\O)},\omega(\cdot), [u]_{0,\gamma},\omega_a(1) \\
       & \text{ if } \eqref{ma:3} \text{ is considered under } \omega_{b}(\cdot)\equiv 0,
        \end{array}\right.
\end{align}
where $\lambda_1$, $\lambda_2$, $\lambda_3$ are the same as defined in \eqref{ma:1}-\eqref{ma:3} and $s(G),s(H_a)$ are indices of the functions $G,H_{a}$ in the sense of Definition \ref{N-func}, respectively.
With $\O_0\Subset\O$ being a fixed open subset, we also denote by $\data_{d}(\O_0)$ the above set of parameters together with $\dist(\O_{0},\partial\O)$:
\begin{align}
    \label{ddata0}
    \data_{d}(\O_0) \equiv \data_{d}, \dist(\O_0,\partial\O).
\end{align}

Now we provide the main results in this section, which correspond to Theorem \ref{mth:mr} and Theorem \ref{mth:md}.
 
 \begin{thm}[Maximal regularity]
 	\label{dmth:mr}
 	Let $u\in W^{1,\Psi_{d}}(\O)$ be a local minimizer of the functional $\F_{d}$ defined in \eqref{dfunctional} under the assumptions \eqref{dsa:1}, \eqref{dsa:2} and \eqref{omega1} with $\omega_{b}(\cdot)\equiv 0$. Suppose that $\omega_{a}(t) = t^{\alpha}$ for some $\alpha\in (0,1]$. If one of the following assumptions
 \begin{subequations}
	 \begin{empheq}[left={\empheqlbrace}]{align}
        &\eqref{ma:1}   \label{dmth:mr:1},\\
        &\eqref{ma:2}  \label{dmth:mr:2}, \\
        &\eqref{ma:3} \quad\text{with}\quad \limsup\limits_{t\to 0^{+}} \Lambda\left(t^{\frac{1}{1-\gamma}},\frac{1}{t}\right)=0 \label{dmth:mr:3}
      \end{empheq}
	\end{subequations}
	is satisfied,  then there exists $\theta\in (0,1)$ depending only on $n,s(G),s(H_{a}),\nu,L,\alpha$ and $\mu$ such that $Du\in C^{0,\theta}_{\loc}(\O)$. 	
 \end{thm}

\begin{thm}[Morrey decay]
	\label{dmth:md}
	Let $u\in W^{1,\Psi_{d}}(\O)$ be a local minimizer of the functional $\F_{d}$ defined in \eqref{dfunctional}, under the assumptions \eqref{dsa:1}, \eqref{dsa:2} and \eqref{omega2}. Assume that $\omega_{b}(\cdot)\equiv 0$ in what follows.
 If one of the following assumptions
	\begin{subequations}
	 \begin{empheq}[left={\empheqlbrace}]{align}
          &\eqref{ma:1}  \quad\text{with}\quad \limsup\limits_{t\to 0^{+}} \Lambda\left(t,G^{-1}(t^{-n})\right)=0 \label{dmth:md:1},\\
         &\eqref{ma:2}\quad\text{with}\quad \limsup\limits_{t\to 0^{+}} \Lambda\left(t,\frac{1}{t}\right)=0 \label{dmth:md:2},\\
        &\eqref{ma:3} \quad\text{with}\quad \limsup\limits_{t\to 0^{+}} \Lambda\left(t^{\frac{1}{1-\gamma}},\frac{1}{t}\right)=0,\label{dmth:md:3} \\
        &\eqref{ma:1}  \quad\text{with}\quad \omega_{a}(t)=t^{\alpha}\quad
         \text{for some}\quad \alpha\in (0,1] \label{dmth:md:4},\\
          &\eqref{ma:2}\quad\text{with}\quad \omega_{a}(t)=t^{\alpha}
          \quad \text{for some}\quad \alpha\in (0,1]  \label{dmth:md:5}
      \end{empheq}
	\end{subequations}
	is satisfied, then
	\begin{align}
	    \label{dmth:md:6}
	    u\in C^{0,\theta}_{\loc}(\O)\quad \text{for every}\quad \theta\in (0,1).
	\end{align}
	Moreover, for every $\sigma\in (0,n)$, there exists a positive constant $c\equiv c(\data_{d}(\O_0),\sigma)$ such that the decay estimate
	\begin{align}
		\label{dmth:md:7}
		\I_{B_{\rho}} \Psi_{d}(x,|Du|)\,dx \leqslant c\left(\frac{\rho}{R} \right)^{n-\sigma}\I_{B_{R}}\Psi_{d}(x,|Du|)\,dx
	\end{align}
holds for every concentric balls  $B_{\rho}\subset B_{R} \subset \O_{0}\Subset\O$ with $R\leqslant 1$.
\end{thm}

The above theorems completely cover the main results of \cite{BCM3}, where the special case  that $G(t)=t^{p}$, $H_{a}(t)= t^{q}$ and $\omega_{a}(t)=t^{\alpha}$ with some constants $q\geqslant p >1$ and $\alpha\in (0,1]$ is considered. Also the results of \cite{BCM2} can be considered for a general class of functionals not only for the model functional in \eqref{func:log}. Let us now briefly overview our arguments employed in proving the above theorems comparing with the ones used in \cite{BCM2,BCM3}. We do not distinguish between the $G$-phase, where an inequality of the type $a(\cdot) \leqslant M\omega_{a}(R)$ is satisfied, and $(G,H_{a})$-phase, where a complementary inequality $a(\cdot)\geqslant M\omega_{a}(R)$ holds in a certain ball $B_{R}$ under consideration for some suitable large constant $M$, which has a drawback to deal with the multi-phase type problems and even double phase type problems that we consider. Instead we consider the function $\left[\Psi_{d}\right]_{B_{R}}^{-}(\cdot)$ defined in \eqref{dispsi} for a ball $B_{R}\subset\O$ under the investigation to obtain various estimates, and the advantage of considering this function is that $\left[\Psi_{d}\right]_{B_{R}}^{-} \in \mathcal{N}$ with an index $s(\Psi_{d})= s(G)+s(H_{a})$ by Remark \ref{rmk:nf2}, which is independent of the considered ball $B_{R}$.  Also the approach introduced in this paper may open a gate to study parabolic double phase equations of type 

\begin{align*}
	u_{t} - \Div\left( G'(|Du|)\frac{Du}{|Du|} + a(x,t)H_{a}'(|Du|)\frac{Du}{|Du|}\right) = 0,
\end{align*}
which would be one of attracting topics for the regularity theory in the future, we refer some recent results on this topic \cite{BS1,Def1}. Essentially, the idea of the proofs of Theorem \ref{dmth:md} and Theorem \ref{dmth:mr} is based on the arguments previously used for proving Theorem \ref{mth:md} and Theorem \ref{mth:mr}, but the functional $\F_{d}$ in \eqref{dfunctional} is much more general than the functional $\F$ in \eqref{functional} for the consideration under the double phase settings. In this regard, we need to take care of some points in more detail depending on the structure assumptions \eqref{dsa:2}, specially Lemma \ref{lem:d2ce} below.  Since $u\in W^{1,\Psi_{d}}(\O)$ is a local $L/\nu$-minimizer of the functional $\P$ defined in \eqref{ifunct} with $b(\cdot)\equiv 0$ if $u$ is a minimizer of the functional $\F_{d}$ in \eqref{dfunctional}, we are able to rewrite the results together with their proofs under the double phase settings  up to the end of Section \ref{sec:6}. Starting by Section \ref{sec:7}, we shall investigate in a different way.

In what follows let $B_{R}\equiv B_{R}(x_0)$ be a ball such that $B_{2R}\subset \O_0 \Subset \O$, where $\O_0$ is some fixed open subset of $\O$. We define a functional given by 
\begin{align}
    \label{d0ce:1}
    W^{1,1}(B_{2R})\ni \u\mapsto \mathcal{F}_{d,B_{2R}}(\u):= \I_{B_{2R}} F_{d}(x,(u)_{B_{2R}},D\u)\,dx
\end{align}
with $u$ being a local minimizer of the functional $\F_{d}$ defined in \eqref{dfunctional}. Now we consider a function $w\in u + W^{1,\Psi_{d}}_{0}(B_{2R})$ being the solution to the following variational Dirichlet problem:

\begin{align}
    \label{d0ce:2}
\begin{cases}
    w \mapsto \min\limits_{\u} \mathcal{F}_{d,B_{2R}}(\u)
    \\
    \u \in u + W^{1,\Psi_{d}}_{0}(B_{2R}).
\end{cases}
\end{align}
As in Lemma \ref{lem:0ce} we shal consider the first comparison estimates in order to remove $u$-dependence in the original functional $\mathcal{F}_{d}$ defined in \eqref{dfunctional}.

\begin{lem}
	\label{lem:d0ce}
	Let $w\in W^{1,\Psi}(B_{2R})$ be the solution to the variational problem \eqref{d0ce:2} under the assumptions \eqref{dsa:1}, \eqref{dsa:2} and \eqref{omega2}. Let the coefficient function $a(\cdot)\in C^{\omega_a}(\O)$ for $\omega_{a}$ being non-negative concave function vanishing at the origin. Assume that one of the assumptions \eqref{ma:1}, \eqref{ma:2} and \eqref{ma:3} under $\omega_{b}(\cdot)\equiv 0$ is satisfied. Then there exists a constant $c\equiv c(\data_{d}(\O_0))$ such that 
	\begin{align}
		\label{d0ce:3}
		\begin{split}
		\FI_{B_{2R}} \left( |V_{G}(Du)-V_{G}(Dw)|^{2} + a(x)|V_{H_{a}}(Du)-V_{H_{a}}(Dw)|^{2} \right)\,dx
		\leqslant
		c\omega(R^{\gamma})\FI_{B_{2R}}\Psi_{d}(x,|Du|)\,dx
		\end{split}
	\end{align}
	holds, where $\gamma\equiv \gamma(\data_{d}(\O_0))$ is the H\"older exponent determined via Theorem \ref{thm:hc} in the double phase settings. Moreover, the following estimates holds true:
	\begin{align}
		\label{d0ce:4}
		\FI_{B_{2R}}\Psi_{d}(x,|Dw|)\,dx \leqslant \frac{L}{\nu} \FI_{B_{2R}}\Psi_{d}(x,|Du|)\,dx,
	\end{align}
	\begin{align}
		\label{d0ce:5}
		\norm{w}_{L^{\infty}(B_{2R})} \leqslant \norm{u}_{L^{\infty}(B_{2R})},
	\end{align}
	\begin{align}
		\label{d0ce:6}
		\osc\limits_{B_{2R}}w \leqslant \osc\limits_{B_{2R}}u
	\end{align}
	and 
	\begin{align}
		\label{d0ce:7}
		\FI_{B_{2R}}\Psi_{d}\left(x,\left|\frac{u-w}{R} \right| \right)\,dx
		\leqslant
		c [\omega(R^{\gamma})]^{\frac{1}{2}}\FI_{B_{2R}}\Psi_{d}(x,|Du|)\,dx
	\end{align}
	for some constant $c\equiv c(\data_{d}(\O_0))$.
	Moreover, there exist a positive higher integrability exponent $\delta_0\equiv \delta_0(\data_{d})$ with $\delta_0 \leqslant \delta$, where $\delta$ has been determined via Theorem \ref{thm:hi} under the double phase settings, and a constant $c\equiv c(\data_{d})$ satisfying the following reverse H\"older inequalities: 
     \begin{align}
        \label{d0ce:7_1}
        \left[\FI_{B_{R/2}}[\Psi_{d}(x,|Dw|)]^{1+\delta_{0}}\,dx\right]^{\frac{1}{1+\delta_0}}
        \leqslant
        c\FI_{B_{R}}\Psi_{d}(x,|Dw|)\,dx.
    \end{align}
Here, in the case that \eqref{ma:3} is considered, $\gamma$ appearing in \eqref{d0ce:3} and \eqref{d0ce:7} is the same as in the assumption \eqref{ma:3}. 
\end{lem}

\begin{proof}
  First of all the meaning of $\data_{d}$ and $\data_{d}(\O_0)$ has been defined in \eqref{ddata} and \eqref{ddata0}, respectively.  The proofs for \eqref{d0ce:4}-\eqref{d0ce:7} can be done by arguing similarly as in the proof Lemma \ref{lem:0ce} together with Lemma \ref{lem:dnf4}.  Since $w$ is a $L/\nu$-minimizer of the functional $\F_{d,B_{2R}}$ defined in \eqref{d0ce:1}, we are able to apply Lemma \ref{lem:cacc} under the double phase settings. In turn, it gives us that
\begin{align}
    \label{d0ce:8}
    \FI_{B_{R/2}}\Psi_{d}(x,|Dw|)\,dx 
    \leqslant
    c\FI_{B_{R}}\Psi_{d}\left(x, \left|\frac{w-(w)_{B_{R}}}{R}\right| \right)\,dx
\end{align}
holds with $c\equiv c(n,s(G),s(H_{a}),L,\nu)$. Then applying Theorem \ref{thm:sp}, there exists $\theta\equiv \theta(n,s(G),s(H_{a}))\in (0,1)$ such that
\begin{align}
		\label{d0ce:9}
		\FI_{B_{R/2}}\Psi_{d}\left(x,|Dw| \right)\,dx \leqslant
		c\bar{\kappa}_{sp}\left[ \FI_{B_{R}} [\Psi_{d}(x,|Dw|)]^{\theta}\,dx \right]^{\frac{1}{\theta}}
	\end{align}
	holds with some constant $c\equiv c(n,s(G),s(H_{a}),L,\nu, \omega_{a}(1))$, where
	\begin{subequations}
	 \begin{empheq}[left={\bar{\kappa}_{sp}=\empheqlbrace}]{align}
        \qquad & 1+\lambda_1[a]_{\omega_{a}} +  \lambda_1[a]_{\omega_{a}}\left( \I_{B_{R}}G(|Dw|)\,dx \right)^{\frac{1}{n}}   &\text{if }& \eqref{ma:1} \text{ is in force with } \omega_{b}(\cdot)\equiv 0, \label{d0ce:10_1}\\
        & 1+ \lambda_2[a]_{\omega_{a}} + \lambda_2[a]_{\omega_{a}} \norm{w}_{L^{\infty}(B_{R})} &\text{if }&  \eqref{ma:2} \text{ is in force with } \omega_{b}(\cdot)\equiv 0, \label{d0ce:10_2}\\
       &1+ \lambda_3[a]_{\omega_{a}} + \lambda_3[a]_{\omega_{a}}\left[ R^{-\gamma}\osc\limits_{B_{R}} w \right]^{\frac{1}{1-\gamma}}
         &\text{if }&  \eqref{ma:3} \text{ is in force with } \omega_{b}(\cdot)\equiv 0.\label{d0ce:10_3}
      \end{empheq}
	\end{subequations}
Furthermore, taking into account \eqref{d0ce:4}-\eqref{d0ce:7} in the last display, we conclude that 
\begin{align}
    \label{d0ce:11}
    \FI_{B_{R/2}}\Psi_{d}(x,|Dw|)\,dx 
    \leqslant
    c\left[\FI_{B_{R}}[\Psi_{d}\left(x, |Dw| \right)]^{\theta}\,dx\right]^{\frac{1}{\theta}}
\end{align}
holds for some constants $\theta\equiv\theta(n,s(G),s(H_{a}))\in (0,1)$ and $c\equiv c(\data_{d})$. The last display follows \eqref{d0ce:7_1} by applying a variant of Gehring's lemma.
\end{proof}

At this stage, we do not need to consider Lemma \ref{lem:1ce} because we shall freeze $x$-variable in the non-linearity at once. For this, let us consider the excess functional given by 
\begin{align}
	\label{dexcess2}
	E_{d}(v, B_{r}):= \left( \left[ \Psi_{d} \right]_{B_{2r}}^{-} \right)^{-1}\left(\FI_{B_{r}}\left[\Psi_{d} \right]_{B_{2r}}^{-}\left(\left|\frac{v-(v)_{B_{r}}}{2r} \right| \right)\,dx \right)
\end{align}
for any function $v\in L^1(B_{2r})$ and ball $B_{2r}\subset\O$, where now and in the rest of this section for every open subset $\mathcal{B}\subset\O$, we shall denote by 
\begin{align}
	\label{dispsi}
	\left[\Psi_{d} \right]_{\mathcal{B}}^{-}(t) := G(t) + \inf\limits_{x\in\mathcal{B}}a(x)H_{a}(t)\quad (\forall t\geqslant 0), 
\end{align}
and $\left( \left[ \Psi_{d} \right]_{\mathcal{B}}^{-} \right)^{-1}$ is the inverse function of $ \left[ \Psi_{d} \right]_{\mathcal{B}}^{-}$. By convexity of the function $\left[\Psi_{d}  \right]_{B_{2r}}^{-}$ and Lemma \ref{lem:nf1}, there is a constant $c\equiv c(s(G)+s(H_{a}))$ such that
\begin{align}
	\label{dexcess4}
	E_{d}(v,B_{r})\leqslant c\left( \left[ \Psi_{d} \right]_{B_{2r}}^{-} \right)^{-1}\left(\FI_{B_{r}}\left[\Psi_{d} \right]_{B_{2r}}^{-}\left(\left|\frac{v-v_0}{2r} \right| \right)\,dx \right)
\end{align}
holds for every $v_0\in\R$. Now we consider the estimates corresponding to the outcome of Lemma \ref{lem:2ce} under our double phase settings.

\begin{lem}
	\label{lem:d2ce}
	Let $u\in W^{1,\Psi_{d}}(\O)$ be a local minimizer of the functional $\F_{d}$ defined in \eqref{dfunctional} under the assumptions \eqref{dsa:1}, \eqref{dsa:2} and \eqref{omega2}. Let $w\in W^{1,\Psi}(B_{2R})$ be the solution to the variational problem \eqref{d0ce:2}. Suppose $\omega_{b}(\cdot)\equiv 0$ in what follows.
	If one of the assumptions \eqref{dmth:md:1}-\eqref{dmth:md:5} is satisfied, then for every $\varepsilon^{*}\in (0,1)$, there exists a positive radius 
	 \begin{align}
	 	\label{d2ce:1}
	 	R^{*}\equiv R^{*}(\data_{d}(\O_0),\varepsilon^{*})
	 \end{align}
such that
	\begin{align}
		\label{d2ce:2}
		\FI_{B_{\tau R}}\left[\Psi_{d}\right]_{B_{R}}^{-}\left(\left|\frac{w-(w)_{B_{\tau R}}}{\tau R} \right| \right)\,dx
		\leqslant
		c\left(1+\tau^{-(n+s(\Psi_{d})+1)}\varepsilon^{*} \right)\FI_{B_{R/2}}\left[\Psi_{d}\right]_{B_{R}}^{-}\left(\left|\frac{w-(w)_{B_{R/2}}}{R} \right| \right)\,dx
	\end{align}
	for some constant $c\equiv c\left(\data_{d}(\O_0)\right)$, whenever $\tau\in (0,1/16)$ and $R\leqslant R^{*}$.
	
\end{lem}

\begin{proof}
Again note that the meaning of $\data_{d}$ and $\data_{d}(\O_0)$ already has been introduced in \eqref{ddata}-\eqref{ddata0}. We can always assume $E_{d}(w,B_{R/2})>0$, otherwise there is nothing to prove in \eqref{d2ce:2}. For the simplicity, we shall write 
	\begin{align}
		\label{d2ce:E_1E_2}
		E_{d}(R):= E_{d}(w,B_{R/2}),
	\end{align}
where the notion $E_{d}$ has been defined in \eqref{dexcess2}. The proof falls in several steps, similarly as we have done in the proof of Lemma \ref{lem:2ce}. For the sake of completeness, we provide the proof in a full detail.
	
	\textbf{Step 1: Initial information on $w$.} Applying Lemma \ref{lem:cacct} under the double phase settings to $B_{R/2}$ with $k\equiv (w)_{B_{R/2}}$, we have 
	\begin{align}
		\label{d2ce:3}
		\FI_{B_{R/4}}\Psi_{d}\left(x, |Dw| \right)\,dx
		\leqslant
		c \FI_{B_{R/2}}\left[\Psi_{d}\right]_{B_{R}}^{-}\left( \left|\frac{w-(w)_{B_{R/2}}}{R}\right| \right)\,dx
	\end{align}
	for some constant $c\equiv c(\data_{d})$. Moreover, it follows from Lemma \ref{lem:d0ce} that there exists a higher integrability exponent $\delta_{0}\equiv \delta_{0}(\data_{d})$ such that
	\begin{align}
		\label{d2ce:4}
		\left(\FI_{B_{R/8}}\left[\Psi_{d}(x,|Dw|)\right]^{1+\delta_0}\,dx \right)^{\frac{1}{1+\delta_{0}}}
		\leqslant
		c\FI_{B_{R/4}}\Psi_{d}(x,|Dw|)\,dx
	\end{align}
	for a constant $c\equiv c(\data_{d})$.
	
	\textbf{Step 2: Scaled functions.} We consider scaled functions of $w(\cdot)$ and $a(\cdot)$ in the ball $B_{1}$ by setting
	\begin{subequations}
	 \begin{empheq}[left={\empheqlbrace}]{align}
          & \bar{w}(x):= \frac{w(x_0+Rx)-(w)_{B_{R/2}}}{E_{d}(R)R}, \label{d2ce:5_1}\\
         &\bar{a}(x):= a(x_0+Rx)\frac{H_{a}(E_{d}(R))}{\left[\Psi_{d}\right]_{B_{R}}^{-}\left( E_{d}(R) \right)}\label{d2ce:5_2}
      \end{empheq}
	\end{subequations}
	for every $x\in B_{1}$. Now we introduce the control function and energy density associated to our scaling introduced above in \eqref{d2ce:5_1}-\eqref{d2ce:5_2} as
	\begin{subequations}
	 \begin{empheq}[left={\empheqlbrace}]{align}
          & \bar{\Psi}_{d}(x,|z|):= \bar{G}(|z|) + \bar{a}(x) \bar{H}_{a}(|z|)  \label{d2ce:6_1},\\
         &\bar{F}_{d}(x,z):= \frac{F_{d}(x_0+Rx,(u)_{B_{2R}}, E_{d}(R)z)}{\left[\Psi_{d}\right]_{B_{R}}^{-}\left( E_{d}(R)\right)} \quad\text{and}\quad
        \bar{A}_{d}(x,z):= D_{z}\bar{F}_{d}(x,z)\label{d2ce:6_2}
      \end{empheq}
	\end{subequations}
for every $x\in B_{1}$ and $z\in\R^n$, where to the end of the proof of this lemma, we always shall understand by
\begin{align}
	\label{d2ce:7}
           \bar{G}(t):= \frac{G(E(R)t)}{\left[\Psi_{d}\right]_{B_{R}}^{-}\left( E_{d}(R) \right)}\quad\text{and}\quad
	\bar{H}_{a}(t):= \frac{H_{a}(E(R)t)}{H_{a}\left( E_{d}(R) \right)}
\end{align}
	for every $t\geqslant 0$. By elementary computations, we can observe that $\bar{G},\bar{H}_{a}\in \mathcal{N}$ with indices $s(G),s(H_{a})$, respectively, and also that
	\begin{align}
		\label{d2ce:8}
		\bar{G}(1) \leqslant 1
		\quad\text{and}\quad
		\bar{H}_{a}(1)=1.
	\end{align}
Clearly, the function $\bar{w}$ minimizes the following functional 
\begin{align}
	\label{d2ce:9}
	W^{1,\bar{\Psi}_{d}}(B_{1})\ni v\mapsto \I_{B_{1}}\bar{F}_{d}(x,Dv)\,dx,
\end{align}
where the functions $\bar{\Psi}_{d}(\cdot)$ and $\bar{F}_{d}(\cdot)$ have been defined in \eqref{d2ce:6_1} and \eqref{d2ce:6_2}, respectively. The Euler-Lagrange equation arising from the functional in \eqref{d2ce:9} can be written as 
\begin{align}
	\label{d2ce:10}
	\FI_{B_{1}}\inner{\bar{A}_{d}(x,D\bar{w})}{D\varphi}\,dx = \FI_{B_{1}}\inner{ D_{z}\bar{F}_{d}\left(x, D\bar{w}\right)}{D\varphi}\,dx = 0
\end{align}
for every $\varphi\in W^{1,\bar{\Psi}_{d}}_{0}(B_{1})$. By the assumptions \eqref{dsa:1} and \eqref{dsa:2} via elementary computations, we have the following structure conditions in the scaled settings:

\begin{subequations}
	 \begin{empheq}[left={\empheqlbrace}]{align}
	 &\nu \bar{\Psi}_{d}(x,|z|) \leqslant \bar{F}_{d}(x,z) \leqslant L\bar{\Psi}_{d}(x,|z|),
	 	\label{d2ce:11_1} \\
          & |\bar{A}_{d}(x,z)||z| + |D_{z}\bar{A}_{d}(x,z)||z|^2 \leqslant L\bar{\Psi}_{d}(x,|z|)  \label{d2ce:11_2},\\
         & \nu \frac{\bar{\Psi}_{d}(x,|z|)}{|z|^2}|\xi|^{2} \leqslant
         \inner{D_{z}\bar{A}_{d}(x,z)\xi}{\xi}
         \label{d2ce:11_3},  \\
         &\left|\bar{A}_{d}(x_1,z)-\bar{A}_{d}(x_2,z) \right||z| \leqslant
         L \omega(R|x_1-x_2|)\left[ \bar{\Psi}_{d}(x_1,|z|) + \bar{\Psi}_{d}(x_2,|z|)
         \right] \label{d2ce:11_4}\\ 
         &\hspace{5cm} + L|\bar{a}(x_1)-\bar{a}(x_2)|\bar{H}_{a}(|z|)
         \label{d2ce:11_5}
      \end{empheq}
	\end{subequations}
for every $x,x_1,x_2\in B_{1}$ and $z\in\R^n\setminus \{0\}$.

\textbf{Step 3: Freezing.} Now we shall consider frozen functional and vector field associated to $\bar{F}_{d}(\cdot)$ and $\bar{A}_{d}(\cdot)$ defined in \eqref{d2ce:6_2}. Let $\bar{x}_{a}\in \overline{B}_{1}$ such that $\bar{a}(\bar{x}_a)= \inf\limits_{x\in B_{1}}\bar{a}(x)$. Then we denote by
\begin{align}
	\label{d2ce:12_1}
	\bar{F}_{0}(z)&:= \bar{F}_{d}(\bar{x}_{a},z),
	\quad
	\bar{A}_{0}(z):= D_{z}\bar{F}_{d}(\bar{x}_{a},z),
\end{align}
and 
\begin{align}
	\label{d2ce:13}
	\bar{\Psi}_{0}(t):= \bar{G}(t) + \bar{a}(\bar{x}_{a})\bar{H}_{a}(t)
\end{align}
for every $x\in B_{1}$, $z\in\R^{n}$ and $t\geqslant 0$. Here we single out that here is a difference between Step 3 of the proof for Lemma \ref{lem:2ce} and our present situation. By the very definition in \eqref{d2ce:6_1} and \eqref{d2ce:7}, one can check
\begin{align}
	\label{d2ce:13_1}
	\bar{\Psi}_0(1) = 1.
\end{align}
 In our newly scaled environment, let us now consider the functional
\begin{align}
	\label{d2ce:14}
	W^{1,\bar{\Psi}_{0}}\left( B_{1/8} \right)\ni v\mapsto \I_{B_{1/8}} \bar{F}_{0}(Dv)\,dx.
\end{align}
We observe that the newly defined integrand $\bar{F}_{0}(\cdot)$ and vector field $\bar{A}_{0}(\cdot)$ satisfy the growth and ellipticity conditions as 
\begin{subequations}
	 \begin{empheq}[left={\empheqlbrace}]{align}
	 &\nu \bar{\Psi}_{0}(|z|) \leqslant \bar{F}_{0}(z) \leqslant L\bar{\Psi}_{0}(|z|),
	 	\label{d2ce:15_1} \\
          & |\bar{A}_{0}(z)||z| + |D_{z}\bar{A}_{0}(z)||z|^2 \leqslant L\bar{\Psi}_{0}(|z|)  \label{d2ce:15_2},\\
         & \nu \frac{\bar{\Psi}_{0}(|z|)}{|z|^2}|\xi|^{2} \leqslant
         \inner{D_{z}\bar{A}_{0}(z)\xi}{\xi}\label{d2ce:15_3}
      \end{empheq}
	\end{subequations}
for every $z\in\R^{n}\setminus\{0\}$ and $\xi\in\R^{n}$. Therefore, the energy and higher integralibility estimates in \eqref{d2ce:3} and \eqref{d2ce:4} can be seen in the view of $\bar{w}$ as
\begin{align}
	\label{d2ce:15_4}
	\FI_{B_{1/4}}\bar{\Psi}_{d}(x,|D\bar{w}|)\,dx + \left(\FI_{B_{1/8}}[\bar{\Psi}_{d}(x,|D\bar{w}|)]^{1+\delta_0}\,dx\right)^{\frac{1}{1+\delta_0}} \leqslant c(\data_{d}).
\end{align}

\textbf{Step 4: Harmonic type approximation.} Let $\varphi\in W^{1,\infty}_{0}\left( B_{1/8} \right)$ be any fixed function. Using \eqref{d2ce:10}, we see 
\begin{align}
	\label{d2ce:16}
	\begin{split}
	I_{0}&:=\left|\FI_{B_{1/8}}\inner{\bar{A}_{0}(D\bar{w})}{D\varphi}\,dx \right|
	=
	\left|\FI_{B_{1/8}}\inner{\bar{A}_{0}(D\bar{w})-\bar{A}_{d}(x,D\bar{w})}{D\varphi}\,dx \right|
	\\&
	\leqslant
	\FI_{B_{1/8}}|\bar{A}_{0}(D\bar{w})-\bar{A}_{d}(x,D\bar{w})|\,dx \norm{D\varphi}_{L^{\infty}(B_{1/8})} =:  I_{1} \norm{D\varphi}_{L^{\infty}(B_{1/8})}.
	\end{split}
\end{align}

Now we estimate $I_{1}$ in the last display using \eqref{d2ce:11_4}-\eqref{d2ce:11_5}. In turn, we have
\begin{align}
	\label{d2ce:17}
	\begin{split}
	I_{1} &\leqslant L\omega(R)\FI_{B_{1/8}}\left( \frac{\bar{\Psi}_{d}(\bar{x}_{a},|D\bar{w}|)}{|D\bar{w}|} + \frac{\bar{\Psi}_{d}(x,|D\bar{w}|)}{|D\bar{w}|} \right)\,dx
	\\&
	\quad
	+
	L\FI_{B_{1/8}} |\bar{a}(x)-\bar{a}(\bar{x}_{a})|\frac{\bar{H}_{a}(|D\bar{w}|)}{|D\bar{w}|}\,dx
	\\&
	\leqslant
	2L \omega(R)\FI_{B_{1/8}} \frac{\bar{\Psi}_{d}(\bar{x}_{a},|D\bar{w}|)}{|D\bar{w}|}\,dx
	 +
	 2L(1+\omega(R))\FI_{B_{1/8}} |\bar{a}(x)-\bar{a}(\bar{x}_{a})|\frac{\bar{H}_{a}(|D\bar{w}|)}{|D\bar{w}|}\,dx
	 \\&
	 =:2L\omega(R) I_{11} + 2L(1+\omega(R))I_{12}.
	\end{split}
\end{align}
Now we estimate the terms appearing in the last display. Recalling \eqref{d2ce:13} and \eqref{d2ce:13_1} together with \eqref{growth2}, we find 

\begin{align}
	\label{d2ce:17_1}
	\begin{split}
	I_{11} \leqslant c \FI_{B_{1/8}}\bar{\Psi}_{0}'(|D\bar{w}|)\,dx 
	&\leqslant
	c\left[ \bar{\Psi}_{0}(1)\right]^{\frac{s(\Psi_d)}{1+s(\Psi_d)}}\FI_{B_{1/8}}\left[ \bar{\Psi}_{0}(|D\bar{w}|)\right]^{\frac{1}{1+s(\Psi_d)}}\,dx
	\\&
	\quad
	+
	c\left[ \bar{\Psi}_{0}(1)\right]^{\frac{1}{1+s(\Psi_d)}}\FI_{B_{1/8}}\left[ \bar{\Psi}_{0}(|D\bar{w}|)\right]^{\frac{s(\Psi_d)}{1+s(\Psi_d)}}\,dx
	\\&
	\leqslant
	c \left(\FI_{B_{1/8}}\bar{\Psi}_{0}(|D\bar{w}|)\,dx\right)^{\frac{1}{1+s(\Psi_d)}}
	+
	c \left(\FI_{B_{1/8}}\bar{\Psi}_{0}(|D\bar{w}|)\,dx\right)^{\frac{s(\Psi_d)}{1+s(\Psi_d)}}
	\\&
	\leqslant
	c(\data_{d}),
	\end{split}
\end{align} 
where we have applied the H\"older's inequality together with \eqref{d2ce:15_4} and the fact that $\bar{\Psi}_0\in \mathcal{N}$ with an index $s(\Psi_d)= s(G)+s(H_{a})$. Next we shall deal with estimating the second term $I_{12}$ in \eqref{d2ce:17}. In turn, using \eqref{growth2} and \eqref{d2ce:8}, we have
\begin{align}
	\label{d2ce:18}
	\begin{split}
	I_{12} &\leqslant 
	c\FI_{B_{1/8}}|\bar{a}(x)-\bar{a}(\bar{x}_{a})|\left( [\bar{H}_{a}(|D\bar{w}|)]^{\frac{1}{s(H_{a})+1}} + [\bar{H}_{a}(|D\bar{w}|)]^{\frac{s(H_{a})}{s(H_{a})+1}}\right)\,dx
	\\&
	\leqslant
	c\norm{\bar{a}-\bar{a}(\bar{x}_{a})}_{L^{\infty}(B_{1/8})}^{\frac{s(H_{a})}{s(H_{a})+1}}\left(\FI_{B_{1/8}} \bar{a}(x)\bar{H}_{a}(|D\bar{w}|)\,dx  \right)^{\frac{1}{s(H_{a})+1}}
	\\&
	\quad
	+
	c\norm{\bar{a}-\bar{a}(\bar{x}_{a})}_{L^{\infty}(B_{1/8})}^{\frac{1}{s(H_{a})+1}}\left(\FI_{B_{1/8}} \bar{a}(x)\bar{H}_{a}(|D\bar{w}|)\,dx  \right)^{\frac{s(H_{a})}{s(H_{a})+1}}
	\\&
	\leqslant
	c(\data_{d})\left( \norm{\bar{a}-\bar{a}(\bar{x}_{a})}_{L^{\infty}(B_{1/8})}^{\frac{1}{s(H_{a})+1}} + \norm{\bar{a}-\bar{a}(\bar{x}_{a})}_{L^{\infty}(B_{1/8})}^{\frac{s(H_{a})}{s(H_{a})+1}} \right),
	\end{split}
\end{align}
where we have used also H\"older's inequality and the fact that $\bar{a}(\bar{x}_a) \leqslant \bar{a}(x)$ for every $x\in B_{1}$. Inserting those estimates coming from the last two displays into \eqref{d2ce:17} and then \eqref{d2ce:16}, we find 
\begin{align}
	\label{d2ce:20}
	\begin{split}
	I_{0} &\leqslant c(\data_{d}(\O_0))\left[ \omega(R) + (1+\omega(R))\left( \norm{\bar{a}-\bar{a}(\bar{x}_{a})}_{L^{\infty}(B_{1/8})}^{\frac{1}{s(H_{a})+1}} + \norm{\bar{a}-\bar{a}(\bar{x}_{a})}_{L^{\infty}(B_{1/8})}^{\frac{s(H_{a})}{s(H_{a})+1}}\right)\right].
	\end{split}
\end{align}

Now we shall estimate the term $\norm{\bar{a}-\bar{a}(\bar{x}_{a})}_{L^{\infty}(B_{1/8})}$ depending on which one of the assumptions \eqref{dmth:md:1}-\eqref{dmth:md:5} comes into play. Recalling the definition of $\bar{a}(\cdot)$ in \eqref{d2ce:5_2} and the excess functional in \eqref{d2ce:E_1E_2}, we have 
\begin{align}
	\label{d2ce:22_1}
	I_{a}:=\norm{\bar{a}-\bar{a}(\bar{x}_{a})}_{L^{\infty}(B_{1/8})}
	\leqslant
	c \omega_{a}(R)\frac{H_{a}(E_{d}(R))}{\left[\Psi_{d}\right]_{B_{R}}^{-}\left( E_{d}(R) \right)}.
\end{align}

\textbf{Case 1: Assumption \eqref{dmth:md:1} is in force.} It follows from the assumption $\eqref{dmth:md:1}_{2}$ that for any $\varepsilon\in (0,1)$ there exists $\mu_1 >0$ depending on $\varepsilon$ such that 
\begin{align}
	\label{d2ce:21}
	\Lambda\left(t, G^{-1}\left(t^{-n} \right) \right) \leqslant \varepsilon
	\quad\text{for every}\quad t\in (0,\mu_1).
\end{align} 
Then using the last display, \eqref{ma:1} and the fact that $\left(\left[\Psi_{d}\right]_{B_{R}}^{-} \right)^{-1}(t) \leqslant G^{-1}(t)$ for every $t\geqslant 0$, $I_{a}$ in \eqref{d2ce:22_1} can be estimated as 
\begin{align}
	\label{d2ce:23}
	\begin{split}
	I_{a}
	&\leqslant
	c\omega_{a}(R)\frac{\left(H_{a}\circ G^{-1}\right)\left(\left[\Psi_{d}\right]_{B_{R}}^{-}(E_{d}(R)) \right)}{\left[\Psi_{d}\right]_{B_{R}}^{-}(E_{d}(R))}
	\\&
	\leqslant
	c\omega_{a}(R)\varepsilon\left(1+\frac{1}{\omega_{a}\left(\left[\left[\Psi_{d}\right]_{B_{R}}^{-}(E_{d}(R))\right]^{-\frac{1}{n}}\right)} \right)
	+
	c\omega_{a}(R)\left(1+\frac{1}{\omega_{a}\left(\mu_1\right)} \right)
	\end{split}
\end{align}
with $c\equiv c([a]_{\omega_{a}}, \lambda_{1})$.  Using \eqref{concave2} and the energy estimate \eqref{d0ce:4}, we see
\begin{align}
	\label{d2ce:24}
	\begin{split}
	\frac{1}{\omega_{a}\left(\left[\left[\Psi_{d}\right]_{B_{R}}^{-}(E_{d}(R))\right]^{-\frac{1}{n}}\right)}
	&\leqslant
	\frac{c}{\omega_{a}(R)} + \frac{c}{\omega_{a}(R)}\I_{B_{R/2}}\left[\Psi_{d}\right]_{B_{R}}^{-}\left(\left|\frac{w-(w)_{B_{R/2}}}{R} \right|\right)\,dx
	\\&
	\leqslant
	\frac{c}{\omega_{a}(R)} + \frac{c}{\omega_{a}(R)}\I_{B_{2R}}\Psi_{d}\left(x,|Du|\right)\,dx
	\leqslant
	\frac{c(\data_{d})}{\omega_{a}(R)}.
	\end{split}
\end{align}
Combining the last two displays, we conclude 
\begin{align}
	\label{d2ce:25}
	I_{a} \leqslant c\left( \varepsilon + \omega_{a}(R)\left( 1+ \frac{1}{\omega_{a}(\mu_1)} \right) \right)
\end{align}
with some constant $c\equiv c(\data_{d})$. Therefore, inserting the estimates in the last two displays into \eqref{d2ce:20} and recalling \eqref{d2ce:16}, we have 
\begin{align}
	\label{d2ce:27}
	\left|\FI_{B_{1/8}}\inner{\bar{A}_{0}(D\bar{w})}{D\varphi}\,dx \right|
	\leqslant
	c(\data_{d})P_{1}(\varepsilon,R)\norm{D\varphi}_{L^{\infty}(B_{1/8})},
\end{align}
where 
\begin{align}
	\label{d2ce:28}
	\begin{split}
	P_1(\varepsilon,R)
	&:= \omega(R) + (1+ \omega(R))\left[ \varepsilon + \omega_{a}(R)\left( 1+ \frac{1}{\omega_{a}(\mu_1)} \right) \right]^{\frac{1}{s(H_{a})+1}}
	\\&
	\quad
	 + (1+ \omega(R))\left[ \varepsilon + \omega_{a}(R)\left( 1+ \frac{1}{\omega_{a}(\mu_1)} \right) \right]^{\frac{s(H_{a})}{s(H_{a})+1}}
	\end{split}
\end{align}

\textbf{Case 2: Assumption \eqref{dmth:md:2} is in force.} From the assumption $\eqref{dmth:md:2}_{2}$ it holds that for every $\varepsilon\in (0,1)$ there exists $\mu_{2}>0$ depending on $\varepsilon$ such that 
\begin{align}
	\label{d2ce:29}
	\Lambda\left(t,\frac{1}{t}\right) \leqslant \varepsilon
	\quad\text{for every}\quad t\in (0,\mu_2).
\end{align}
Then by the very definition of $\left[\Psi_{d}\right]_{B_{R}}^{-}$ in \eqref{dispsi} together with \eqref{d2ce:29} and \eqref{ma:2} under $\omega_{b}\equiv 0$, we have 
\begin{align}
	\label{d2ce:30}
	\begin{split}
	I_{a}
	&\leqslant
	c\omega_{a}(R)\frac{H_{a}(E_{d}(R))}{G(E_{d}(R))}
	\\&
	\leqslant
	c\omega_{a}(R)\varepsilon\left(1+\frac{1}{\omega_{a}\left([E_{d}(R)]^{-1}\right)} \right)
	+
	c\omega_{a}(R)\left(1+\frac{1}{\omega_{a}\left(\mu_2\right)} \right).
	\end{split}
\end{align}
Again using \eqref{concave1} together with taking into account  \eqref{d0ce:5}, we see 
\begin{align}
	\label{d2ce:31}
	\frac{1}{\omega_{a}\left([E_{d}(R)]^{-1}\right)} 
	\leqslant
	\frac{1}{\omega_{a}\left(\frac{R}{2\norm{w}_{L^{\infty}(B_{R})}} \right)}
	\leqslant
	\frac{c(\data_{d})}{\omega_{a}(R)}.
\end{align}
Combining the last two displays, we find 
\begin{align}
	\label{d2ce:32}
	I_{a} \leqslant c\left( \varepsilon + \omega_{a}(R)\left( 1+ \frac{1}{\omega_{a}(\mu_2)} \right) \right)
\end{align}
with some constant $c\equiv c(\data_{d})$.
Then, plugging the estimates in the last two displays into \eqref{d2ce:20} and recalling \eqref{d2ce:16}, we have 
\begin{align}
	\label{d2ce:34}
	\left|\FI_{B_{1/8}}\inner{\bar{A}_{0}(D\bar{w})}{D\varphi}\,dx \right|
	\leqslant
	c(\data_{d})P_{2}(\varepsilon,R)\norm{D\varphi}_{L^{\infty}(B_{1/8})},
\end{align}
where 
\begin{align}
	\label{d2ce:35}
	\begin{split}
	P_2(\varepsilon,R)
	&:= \omega(R) + (1+\omega(R)) \left[ \varepsilon + \omega_{a}(R)\left( 1+ \frac{1}{\omega_{a}(\mu_2)} \right) \right]^{\frac{1}{s(H_{a})+1}} 
	\\&
	+ (1+\omega(R))\left[ \varepsilon + \omega_{a}(R)\left( 1+ \frac{1}{\omega_{a}(\mu_2)} \right) \right]^{\frac{s(H_{a})}{s(H_{a})+1}}.
	\end{split}
\end{align}

\textbf{Case 3: Assumption \eqref{dmth:md:3} is in force.} The assumption $\eqref{dmth:md:3}_{2}$ implies that for any $\varepsilon\in (0,1)$, there exists $\mu_3>0$ depending on $\varepsilon$ such that 
\begin{align}
	\label{d2ce:36}
	\Lambda\left(t^{\frac{1}{1-\gamma}}, \frac{1}{t} \right) \leqslant \varepsilon
	\quad\text{for every}\quad t\in (0,\mu_3).
\end{align}
This one together with using \eqref{d2ce:22_1} and \eqref{ma:3} under $\omega_{b}(\cdot)\equiv 0$ implies
\begin{align}
	\label{d2ce:37}
	\begin{split}
	I_{a}
	&\leqslant
	c\omega_{a}(R)\frac{H_{a}(E_{d}(R))}{G(E_{d}(R))}
	\\&
	\leqslant
	c\omega_{a}(R)\varepsilon\left(1+\frac{1}{\omega_{a}\left([E_{d}(R)]^{-\frac{1}{1-\gamma}}\right)} \right)
	+
	c\omega_{a}(R)\left(1+\frac{1}{\omega_{a}\left(\mu_3^{\frac{1}{1-\gamma}}\right)} \right).
	\end{split}
\end{align}
Now using \eqref{d0ce:6} and \eqref{ma:3}, we have 
\begin{align}
	\label{d2ce:38}
	\frac{1}{\omega_{a}\left([E_{d}(R)]^{-\frac{1}{1-\gamma}}\right)}
	\leqslant
	\frac{1}{\omega_{a}\left(\left[\frac{\osc\limits_{B_{2R}}u}{R}\right]^{-\frac{1}{1-\gamma}}
	\right)}
	\leqslant
	\frac{c(\data_{d})}{\omega_{a}(R)}.
\end{align}
Combining the last two displays, we find 
\begin{align}
	\label{d2ce:39}
	I_{a} \leqslant c\left( \varepsilon + \omega_{a}(R)\left( 1+ \frac{1}{\omega_{a}\left(\mu_3^{\frac{1}{1-\gamma}}\right)} \right) \right)
\end{align}
for some constant $c\equiv c(\data_{d})$. Using the estimate \eqref{d2ce:39} in \eqref{d2ce:20}, we conclude 
\begin{align}
	\label{d2ce:41}
	\left|\FI_{B_{1/8}}\inner{\bar{A}_{0}(D\bar{w})}{D\varphi}\,dx \right|
	\leqslant
	c(\data_{d})P_{3}(\varepsilon,R)\norm{D\varphi}_{L^{\infty}(B_{1/8})},
\end{align}
where 
\begin{align}
	\label{d2ce:42}
	\begin{split}
	P_3(\varepsilon,R)
	&:= \omega(R) + (1+\omega(R))
	\left[ \varepsilon + \omega_{a}(R)\left( 1+ \frac{1}{\omega_{a}\left(\mu_3^{\frac{1}{1-\gamma}}\right)} \right) \right]^{\frac{1}{s(H_{a})+1}} 
	\\&
	+ (1+\omega(R))\left[ \varepsilon + \omega_{a}(R)\left( 1+ \frac{1}{\omega_{a}\left(\mu_3^{\frac{1}{1-\gamma}}\right)} \right) \right]^{\frac{s(H_{a})}{s(H_{a})+1}}.
	\end{split}
\end{align}

\textbf{Case 4. Assumption \eqref{dmth:md:4} is in force.} Now we take the advantage that $w_a(\cdot)$ is the power function. Recalling $I_{a}$ denoted in \eqref{d2ce:22_1}, we see that 
\begin{align}
	\label{d2ce:42_1}
	\begin{split}
	I_{a} &\leqslant c R^{\alpha} \frac{\left(H_{a}\circ G^{-1} \right)\left(\left[\Psi_{d} \right]_{B_{R}}^{-}(E_{d}(R))\right)}{\left[\Psi_{d} \right]_{B_{R}}^{-}(E_{d}(R))}
	\leqslant
	cR^{\alpha}\left( 1 + \left[\FI_{B_{R/2}}\left[\Psi_{d} \right]_{B_{R}}^{-}\left(\left|\frac{w-(w)_{B_{R/2}}}{R} \right| \right)\,dx \right]^{\frac{\alpha}{n}} \right)
	\\&
	\leqslant
	cR^{\alpha} + c\left(\I_{B_{R/2}}\left[\Psi_{d} \right]_{B_{R}}^{-}\left(\left|Dw \right| \right)\,dx \right)^{\frac{\alpha}{n}}
	\leqslant
	cR^{\alpha} + cR^{\frac{\alpha\delta_0}{1+\delta_0}} \left(\I_{B_{R/2}}\left[\left[\Psi_{d} \right]_{B_{R}}^{-}\left(\left|Dw \right| \right)\right]^{1+\delta_0}\,dx \right)^{\frac{\alpha}{n(1+\delta_0)}}
	\\&
	\leqslant
	c(\data_{d}(\O_0))R^{\frac{\alpha\delta_0}{1+\delta_0}},
	\end{split}
\end{align}
where we have used the higher integrability estimates \eqref{hi:2} of Theorem \ref{thm:hi} under the double phase settings. Using estimates from the last display in \eqref{d2ce:20} and recalling $R\leqslant 1$, we see 
\begin{align}
	\label{d2ce:42_3}
	\left|\FI_{B_{1/8}}\inner{\bar{A}_{0}(D\bar{w})}{D\varphi}\,dx \right|
	\leqslant
	c(\data_{d}(\O_0))Q_{1}(R)\norm{D\varphi}_{L^{\infty}(B_{1/8})},
\end{align}
where 
\begin{align}
	\label{d2ce:42_4}
	Q_{1}(R):= \omega(R) + (1+\omega(R)) R^{\frac{\alpha\delta_0}{(1+\delta_0)(1+s(H_{a}))}}.
\end{align}

\textbf{Case 5: Assumption \eqref{dmth:md:5} is in force.}  Using the assumption \eqref{ma:2} and \eqref{d0ce:6}, $I_{a}$ in \eqref{d2ce:22_1} can be estimated as 
\begin{align}
	\label{d2ce:42_5}
	\begin{split}
	I_{a} &\leqslant
	cR^{\alpha}\frac{H_{a}(E_{d}(R))}{G(E_{d}(R))}
	\\&
	\leqslant
	 cR^{\alpha}\left( 1 + \left[ \left(\left[\Psi_{d} \right]_{B_{R}}^{-}\right)^{-1}\left(\FI_{B_{R/2}}\left[\Psi_{d} \right]_{B_{R}}^{-}\left(\left|\frac{w-(w)_{B_{R/2}}}{R} \right| \right)\,dx \right) \right]^{\alpha} \right)
	\\&
	\leqslant
	c\left(R^{\alpha}+ \left[\osc\limits_{B_{2R}}u\right]^{\alpha}\right)
	\leqslant
	c(\data_{d}(\O_0))R^{\gamma\alpha},
	\end{split}
\end{align}
where we have also used \eqref{hc:3} and $\gamma$ is the H\"older continuity exponent coming from Theorem \ref{thm:hc} under the double phase settings. Inserting the estimate from the last display into \eqref{d2ce:20} and recalling $R\leqslant 1$, we see 
\begin{align}
	\label{d2ce:42_7}
	\left|\FI_{B_{1/8}}\inner{\bar{A}_{0}(D\bar{w})}{D\varphi}\,dx \right|
	\leqslant
	c(\data_{d}(\O_0))Q_{2}(R)\norm{D\varphi}_{L^{\infty}(B_{1/8})},
\end{align}
where 
\begin{align}
	\label{d2ce:42_8}
	Q_{2}(R):= \omega(R) + (1+\omega(R)) R^{\frac{\alpha\gamma}{1+s(H_{a})}}.
\end{align}

Collecting the estimates obtained in \eqref{d2ce:27}, \eqref{d2ce:34},\eqref{d2ce:41}, \eqref{d2ce:42_3} and \eqref{d2ce:42_7}, we conclude with 
    \begin{align}
        \label{d2ce:43}
        \left|\FI_{B_{1/8}} \inner{\bar{A}_{0}(D\bar{w})}{D\varphi}\,dx\right|
        \leqslant
        c_{h}D(\varepsilon,R) \norm{D\varphi}_{L^{\infty}(B_{1/8})}
    \end{align}
     for some constant $c_{h}\equiv c_{h}(\data_{d}(\O_0))$ for every $\varphi\in W^{1,\infty}_{0}(B_{1/8})$,
    where 
    \begin{align}
	\label{d2ce:44}
	 D(\varepsilon,R)
	 := \left\{\begin{array}{lr}
        P_1(\varepsilon,R) & \text{if } \eqref{dmth:md:1} \text{ is assumed, }\\
        P_2(\varepsilon,R)  & \text{if } \eqref{dmth:md:2} \text{ is assumed, } \\
        P_3(\varepsilon,R)  & \text{if } \eqref{dmth:md:3} \text{ is assumed, } \\
        Q_1(R)  & \text{if } \eqref{dmth:md:4} \text{ is assumed, } \\
        Q_2(R)  & \text{if } \eqref{dmth:md:5} \text{ is assumed, }
        \end{array}\right.
\end{align}
in which $P_1, P_2$, $P_3$, $Q_1$ and $Q_2$ have been defined in \eqref{d2ce:28}, \eqref{d2ce:35}, \eqref{d2ce:42}, \eqref{d2ce:42_4} and \eqref{d2ce:42_8}, respectively. By \eqref{d2ce:13_1}, \eqref{d2ce:15_1}-\eqref{d2ce:15_3} and \eqref{d2ce:43}, we are able to apply Lemma \ref{hta: lemma_hta} with $A_{0}(z)\equiv \bar{A}_{0}(z)$, $\Psi_{0}(t)\equiv \bar{\Psi}_{0}(t)$ with $a_{0}\equiv \bar{a}(\bar{x}_{a})$ and $b_{0}\equiv 0$. By Lemma \ref{hta: lemma_hta}, there exists $\bar{h}\in \bar{w} + W_{0}^{1,\bar{\Psi}_{0}}(B_{1/8})$ such that 
    \begin{align}
        \label{d2ce:45}
        \FI_{B_{1/8}}\inner{\bar{A}_0(D\bar{h})}{D\varphi}\,dx = 0\qquad \text{ for all }
        \qquad \varphi\in W^{1,\infty}_{0}(B_{1/8}),
    \end{align}
    \begin{align}
        \label{d2ce:46}
        \FI_{B_{1/4}}\bar{\Psi}_{0}(|D\bar{h}|)\,dx + \FI_{B_{1/8}}[\bar{\Psi}_{0}(|D\bar{h}|)]^{1+\delta_1}\,dx \leqslant c
        \text{ for some } \delta_1 \leqslant \delta_0,
    \end{align}
    \begin{align}
        \label{d2ce:47}
        \begin{split}
        \FI_{B_{1/8}}&\left( |V_{\bar{G}}(D\bar{w})-V_{\bar{G}}(D\bar{h})|^{2} + \bar{a}(\bar{x}_{a})|V_{\bar{H}_{a}}(D\bar{w})-V_{\bar{H}_{a}}(D\bar{h})|^{2}  \right)\,dx
        \leqslant
        c[D(\varepsilon, R)]^{s_1}
        \end{split}
    \end{align}
    and finally
    \begin{align}
        \label{d2ce:48}
        \FI_{B_{1/8}} \left(\bar{G}\left( |\bar{w}-\bar{h}|\right) + \bar{a}(\bar{x}_{a}) \bar{H_{a}}\left( |\bar{w}-\bar{h}| \right) \right)\,dx \leqslant c_{d}[D(\varepsilon, R)]^{s_{0}}
    \end{align}
    with some constants $c,c_d\equiv c,c_d(\data_{d}(\O_0))\geqslant 1$ and $s_0,s_1\equiv s_0,s_1(\data_{d})\in (0,1)$, but they are all independent of $R$. The rest of the proof is similar as the argument after \eqref{2ce:49} of Lemma \ref{lem:2ce}.
    \end{proof}

\begin{lem}
	\label{lem:d3ce}
	Under the assumptions of Lemma \ref{lem:d2ce}, let $w\in W^{1,\Psi}(B_{2R})$ be the solution to the problem defined in \eqref{d0ce:2}. If one of the assumptions \eqref{dmth:md:1}-\eqref{dmth:md:5} is satisfied, then there exists $h \in w + W^{1,\left[\Psi_{d}\right]_{B_{R}}^{-}}_{0}(B_{R/8})$ being  a local minimizer of the functional defined by 
	\begin{align}
		\label{d3ce:1}
		W^{1,1}(B_{R/8})\ni v\mapsto \F_{0}(v):= \I_{B_{R/8}}F_{0}(Dv)\,dx, 
	\end{align}
	where the integrand function is given by
\begin{align}
	\label{d3ce:1_1}
	F_{0}(z):= F\left(x_{a}, (u)_{B_{2R}},z\right)
\end{align}
for $x_{a}\in\overline{B}_{R}$ being a point such that $a(x_{a}):= a^{-}(B_{R})$, whenever $z\in\R^n$, such that 
\begin{align}
	\label{d3ce:2}
	\begin{split}
	\FI_{B_{R/8}}&\left[ |V_{G}(Du)-V_{G}(Dh)|^2 + a(x_a)|V_{H_{a}}(Du)-V_{H_{a}}(Dh)|^2 \right]\,dx
	\\&
	\leqslant
	c\left( \omega\left(R^{\gamma}\right) + [D(\varepsilon,R)]^{s_1}  \right)\FI_{B_{2R}}\Psi_{d}(x,|Du|)\,dx
	\end{split}
\end{align}
for some constant $c\equiv c(\data_{d}(\O_0))$, where $s_1$ and $D(\varepsilon,R)$ have been defined in \eqref{d2ce:47} and \eqref{d2ce:44}, respectively. Moreover, we have the energy estimate
\begin{align}
	\label{d3ce:3}
	\FI_{B_{R/8}}\left[\Psi_{d} \right]_{B_{R}}^{-}(|Dh|)\,dx \leqslant c\FI_{B_{2R}}\Psi_{d}(x,|Du|)\,dx
\end{align}
for some constant $c\equiv c(n,\nu,L)$.
\end{lem}
\begin{proof}
	We need to revisit the proof of Lemma \ref{lem:d2ce}, specially Step 3 and Step 4. We consider a function $\bar{h}\in \bar{w} + W^{1,\bar{\Psi}_{0}}_{0}(B_{1/8})$ satisfying \eqref{d2ce:45}-\eqref{d2ce:48}. Let $h$ be the scaled back function of $\bar{h}$ in $B_{R/8}$ as
	\begin{align}
		\label{d3ce:4}
		h(x):= E_{d}(w,B_{R/2})R\bar{h}\left(\frac{x-x_0}{R} \right)
		\quad\text{for every}\quad
		x\in B_{R/8}(x_0).
	\end{align}
Clearly, $h\in w + W^{1,\left[\Psi_{d} \right]_{B_{R}}^{-}}_{0}(B_{R/8})$ is a local minimizer of the functional $\F_0$ defined in \eqref{d3ce:1} which means that 
\begin{align}
	\label{d3ce:5}
	\F_{0}(h) = \I_{B_{R/8}}F_0(Dh)\,dx \leqslant \I_{B_{R/8}}F_0(Dh+D\varphi)\,dx
	\leqslant
	\F_{0}(h+\varphi)
\end{align}
holds for every $\varphi\in W^{1,\left[\Psi_{d} \right]_{B_{R}}^{-}}_{0}(B_{R/8})$. As shown in \eqref{0ce:9}, we recall \eqref{d0ce:4} to discover that
\begin{align}
	\label{d3ce:6}
	\begin{split}
	\FI_{B_{R/8}}\left[\Psi_{d} \right]_{B_{R}}^{-}(|Dh|)\,dx 
	&\leqslant
	\frac{L}{\nu}\FI_{B_{R/8}}\left[\Psi_{d} \right]_{B_{R}}^{-}(|Dw|)\,dx 
	\\&
	\leqslant
	\frac{8^{n}L}{\nu} \FI_{B_{R}}\Psi_{d}(x,|Dw|)\,dx
	\leqslant
	c(n,\nu,L)  \FI_{B_{2R}}\Psi_{d}(x,|Du|)\,dx,
	\end{split}
\end{align}
which proves \eqref{d3ce:3}. We write the inequality \eqref{d2ce:47} in view of $G,H_{a}$, $w$ and $h$ in order to have 
\begin{align}
	\label{d3ce:7}
	\begin{split}
		\FI_{B_{R/8}}&\left[ |V_{G}(Du)-V_{G}(Dh)|^2 + a(x_a)|V_{H_{a}}(Du)-V_{H_{a}}(Dh)|^2   \right]\,dx
		\\&
		\leqslant
		c[D(\varepsilon,R)]^{s_1}\FI_{B_{R/2}}\left[\Psi_{d} \right]_{B_{R}}^{-}\left(\left|\frac{w-(w)_{B_{R/2}}}{R} \right| \right)\,dx
		\leqslant
		c[D(\varepsilon,R)]^{s_1}\FI_{B_{R/2}}\left[\Psi_{d} \right]_{B_{R}}^{-}\left(\left|Dw \right| \right)\,dx
		\\&
		\leqslant
		c[D(\varepsilon,R)]^{s_1}\FI_{B_{R/2}}\Psi_{d}\left(x,\left|Du \right| \right)\,dx
	\end{split}
\end{align}
for some constant $c\equiv c(\data_{d}(\O_0))$, where we have applied the Sobolev-Poincar\'e inequality and \eqref{d3ce:6}. Combining this estimate together with \eqref{d0ce:3} via some elementary computations, we directly reach \eqref{d3ce:2}. 
\end{proof}

We finally finish the present subsection with a crucial decay estimate on $u$.

\begin{lem}
    \label{lem:d4ce}
    Under the assumptions of Lemma \ref{lem:d2ce}, if one of the conditions \eqref{dmth:md:1}-\eqref{dmth:md:5} is satisfied, then  for every $\varepsilon_{*}\in (0,1)$, there exists a positive radius $R_{*}$ with the dependence as 
    \begin{align}
        \label{d4ce:1}
        R_{*}\equiv R_{*}(\data_{d}(\O_0),\varepsilon_{*})
    \end{align}
    such that if $R\leqslant R_{*}$, then there exists a constant $c_{G}\equiv c_{G}(\data_{d}(\O_0))$ such that
    \begin{align}
        \label{d4ce:2}
        \I_{B_{\tau R}} \left[\Psi_{d}\right]_{B_{R}}^{-}\left(\left|\frac{u-(u)_{B_{\tau R}}}{\tau R}\right|\right)\,dx \leqslant
        c_{G}\left( \tau^{n} + \tau^{-(s(\Psi_{d})+1)}\varepsilon_{*} \right)\I_{B_{2R}}\Psi_{d}(x,|Du|)\,dx
    \end{align}
    holds for every $\tau\in \left(0,1/32 \right)$.
\end{lem}

\begin{proof}
    For the proof, we apply Lemma \ref{lem:d2ce} with $\varepsilon^{*}\in (0,1)$ to be determined in a few lines, and we can use \eqref{d2ce:2} provided
    \begin{align*}
        R\leqslant R^{*}\equiv R^{*}(\data_{d}(\O_0),\varepsilon^{*})
    \end{align*}
    is found via \eqref{d2ce:1}.
    For every $\tau\in (0,1/32)$ with some elementary manipulations, we see that 
    \begin{align}
        \label{dgp:gp68}
        \begin{split}
            \FI_{B_{\tau R}}& \left[\Psi_{d} \right]_{B_{R}}^{-}\left(\left|\frac{u-(u)_{B_{\tau R}}}{\tau R} \right|\right)\,dx
            \leqslant
            c\FI_{B_{\tau R}} \left[\Psi_{d} \right]_{B_{R}}^{-}\left(\left|\frac{u-(w)_{B_{\tau R}}}{\tau R} \right|\right)\,dx
            \\&
            \leqslant
            c\FI_{B_{\tau R}} \left[\Psi_{d} \right]_{B_{R}}^{-}\left(\left|\frac{w-(w)_{B_{\tau R}}}{\tau R} \right|\right)\,dx
            + c\tau^{-(n+s(\Psi_{d})+1)} \FI_{B_{R}} \left[\Psi_{d} \right]_{B_{R}}^{-}\left(\left|\frac{u-w}{R} \right|\right)\,dx
            \\&
            \leqslant
            c\left( 1+\tau^{-(n+s(\Psi_{d})+1)}\varepsilon^{*} \right)\FI_{B_{R/2}} \left[\Psi_{d} \right]_{B_{R}}^{-}\left(\left|\frac{w-(w)_{B_{R/2}}}{R} \right|\right)\,dx
            \\&
            \quad
             + c\tau^{-(n+s(\Psi_{d})+1)} \FI_{B_{R}} \left[\Psi_{d} \right]_{B_{R}}^{-}\left(\left|\frac{u-w}{R} \right|\right)\,dx
            \\&
            \leqslant
            c\left( 1+\tau^{-(n+s(\Psi_{d})+1)}\varepsilon^{*} \right)\FI_{B_{R}} \left[\Psi_{d} \right]_{B_{R}}^{-}\left(\left| Dw \right|\right)\,dx + c\tau^{-(n+s(\Psi_{d})+1)} \FI_{B_{R}} \left[\Psi_{d} \right]_{B_{R}}^{-}\left(\left|\frac{u-w}{R} \right|\right)\,dx
        \end{split}
    \end{align}
    with some constant $c\equiv c(\data_{d}(\O_0))$, where throughout the last display we repeatedly used \eqref{growth1} and \eqref{excess4}. The last display and \eqref{d0ce:7} along with some elementary manipulations  yield
    \begin{align*}
        \I_{B_{\tau R}}& \left[\Psi_{d} \right]_{B_{R}}^{-}\left(\left|\frac{u-(u)_{B_{\tau R}}}{\tau R} \right|\right)\,dx
        \leqslant
        c\left( \tau^{n} + \tau^{-(s(\Psi_{d})+1)}\varepsilon^{*} + \tau^{-(s(\Psi_{d})+1)}[\omega(R^{\gamma})]^{\frac{1}{2}} \right)\I_{B_{2R}}\Psi_{d}(x,|Du|)\,dx
    \end{align*}
     for every $\tau\in\left(0,1/16 \right)$ and some $c\equiv c(\data_{d}(\O_0))$. Then we choose $\varepsilon^{*} \equiv \varepsilon_{*}/2$ and $R_{*}\leqslant R^{*}$ in such a way that
    $[\omega(R_{*}^{\gamma})]^{\frac{1}{2}}\leqslant \varepsilon_{*}/2$. This choice gives us the dependence as described in \eqref{d4ce:1} and yields \eqref{d4ce:2}.
\end{proof}

We have now discovered all the necessary tools. They are Lemma \ref{lem:d0ce}, Lemma \ref{lem:d2ce} and Lemma \ref{lem:d4ce} in the double phase settings for proving Theorem \ref{dmth:mr} and  Theorem \ref{dmth:md}. Applying those lemmas with arguing in a similar manner as in the proofs of Theorem \ref{mth:md} and Theorem \ref{mth:mr}, we are able to prove Theorem \ref{dmth:mr} and  Theorem \ref{dmth:md}. For the sake of the completeness, we provide a sketch of the proofs.


\textbf{Proof of Theorem \ref{dmth:md}.} The proof of Theorem \ref{dmth:md} can be done similarly as for the proof of Theorem \ref{mth:md}. We just combine Lemma \ref{lem:1cacct} under the double phase settings and Lemma \ref{lem:d4ce}, as we already have done in \eqref{nmd:1}-\eqref{nmd:25}.


\begin{lem}
	\label{lem:d5ce}
	Under the assumptions and notations of Lemma \ref{lem:d2ce} and Lemma \ref{lem:d3ce}, let $w\in W^{1,\Psi_{d}}(B_{R})$ be the function defined in \eqref{d0ce:2}. Suppose that  \eqref{dmth:md:3} is satisfied for $\omega_{a}(t)=t^{\alpha}$ with some $\alpha \in (0,1]$. Then there exists a function $h \in w + W^{1,\left[\Psi_{d}\right]_{B_{R}}^{-}}_{0}(B_{R/8})$ being  a local minimizer of the functional $\F_{0}$ defined in \eqref{d3ce:1} such that 
\begin{align}
	\label{d5ce:1}
	\begin{split}
	\FI_{B_{R/8}}&\left[ |V_{G}(Du)-V_{G}(Dh)|^2 + a(x_a)|V_{H_{a}}(Du)-V_{H_{a}}(Dh)|^2  \right]\,dx
	\\&
	\leqslant
	c\left( \omega\left(R^{\gamma}\right) + \left[ \omega(R) + (1+\omega(R)) R^{\frac{\alpha}{2(1+s(H_{a})}} \right]^{s_1}  \right)\FI_{B_{2R}}\Psi_{d}(x,|Du|)\,dx
	\end{split}
\end{align}
for some constant $c\equiv c(\data_{d}(\O_0))$ and $s_1\equiv s_1(\data_{d})$, respectively. Moreover, the energy estimate
\begin{align}
	\label{d5ce:2}
	\FI_{B_{R/8}}\left[\Psi_{d}\right]_{B_{R}}^{-}(|Dh|)\,dx \leqslant c\FI_{B_{2R}}\Psi_{d}(x,|Du|)\,dx
\end{align}
holds for some constant $c\equiv c(n,\nu,L)$.
\end{lem}
\begin{proof}
	First we apply Theorem \ref{dmth:md} in order to obtain that, for every $\theta\in (0,1)$ and every open subset $\O_0\Subset \O$, there exists a constant $c\equiv c(\data_{d}(\O_0),\theta)$ such that
	\begin{align}
		\label{d5ce:3}
		[u]_{0,\theta;\O_0} \leqslant c(\data_{d}(\O_0),\theta).
	\end{align}
In particular, we choose $\theta\equiv (\gamma+1)/2$. By revisiting the proof of Lemma \ref{lem:d2ce}, we shall estimate the term $I_{a}$ introduced in \eqref{d2ce:22_1}. Using \eqref{ma:3} and \eqref{d0ce:6}, we have 
\begin{align}
	\label{d5ce:4}
	\begin{split}
	I_{a} &\leqslant cR^{\alpha}\left( 1 + \left[ \left(\left[\Psi_{d}\right]_{B_{R}}^{-} \right)^{-1}\left(\FI_{B_{R/2}}\left[\Psi_{d}\right]_{B_{R}}^{-}\left(\left|\frac{w-(w)_{B_{R/2}}}{R} \right| \right)\,dx \right) \right]^{\frac{\alpha}{1-\gamma}} \right)
	\\&
	\leqslant
	c\left(R^{\alpha}+ R^{-\frac{\alpha\gamma}{1-\gamma}}\left[\osc\limits_{B_{2R}}u\right]^{\frac{\alpha}{1-\gamma}}\right)
	\leqslant
	c(\data_d(\O_0))R^{\alpha/2},
	\end{split}
\end{align}
where we have used \eqref{d5ce:3} with the choice of $\theta\equiv (1+\gamma)/2$ and $B_{2R}\subset\O_0$ with $R\leqslant 1$. Plugging this estimate in \eqref{d2ce:20}, we  find 

\begin{align}
	\label{d5ce:6}
	\left|\FI_{B_{1/8}}\inner{\bar{A}_{0}(D\bar{w})}{D\varphi}\,dx \right|
	\leqslant
	c(\data_{d}(\O_0))Q_{3}(R)\norm{D\varphi}_{L^{\infty}(B_{1/8})},
\end{align}
where 
\begin{align}
	\label{d5ce:7}
	Q_{3}(R):= \omega(R) + (1+\omega(R)) R^{\frac{\alpha}{2(1+s(H_{a})}},
\end{align}
where the vector field  $\bar{A}_{0}$ has been defined in \eqref{d2ce:12_1}.
We consider a function $\bar{h}\in \bar{w} + W^{1,\bar{\Psi}_{0}}_{0}(B_{1/8})$ satisfying \eqref{d2ce:45}-\eqref{d2ce:48} with the term $D(\varepsilon,R)$ replaced by $Q_{3}(R)$ defined above. Let $h$ be the scaled back function of $\bar{h}$ in $B_{R/8}$ as
	\begin{align}
		\label{d5ce:8}
		h(x):= E_{d}(w,B_{R/2})R\bar{h}\left(\frac{x-x_0}{R} \right)
		\quad\text{for every}\quad
		x\in B_{R/8}(x_0).
	\end{align}
Clearly, $h\in w + W^{1,\left[\Psi_{d}\right]_{B_{R}}^{-}}_{0}(B_{R/8})$ is a local minimizer of the functional $\F_0$ defined in \eqref{d3ce:1} which means that 
\begin{align}
	\label{d5ce:9}
	\F_{0}(h) = \I_{B_{R/8}}F_0(Dh)\,dx \leqslant \I_{B_{R/8}}F_0(Dh+D\varphi)\,dx
	\leqslant
	\F_{0}(h+\varphi)
\end{align}
holds for every $\varphi\in W^{1,\left[\Psi_{d}\right]_{B_{R}}^{-}}_{0}(B_{R/8})$. Arguing similarly as in the proof of Lemma \ref{lem:d0ce} together with recalling \eqref{d0ce:4}, we see 
\begin{align}
	\label{d5ce:10}
	\begin{split}
	\FI_{B_{R/8}}\left[\Psi_{d}\right]_{B_{R}}^{-}(|Dh|)\,dx 
	\leqslant
	\frac{L}{\nu}\FI_{B_{R/8}}\left[\Psi_{d}\right]_{B_{R}}^{-}(|Dw|)\,dx 
	&\leqslant
	\frac{8^{n}L}{\nu} \FI_{B_{R}}\Psi_{d}(x,|Dw|)\,dx
	\\&
	\leqslant
	c(n,\nu,L)  \FI_{B_{2R}}\Psi_{d}(x,|Du|)\,dx,
	\end{split}
\end{align}
which proves \eqref{d5ce:2}. We write the inequality \eqref{d2ce:47} in view of $G,H_{a}$, $w$ and $h$ in order to have 
\begin{align}
	\label{d5ce:11}
	\begin{split}
		\FI_{B_{R/8}}&\left[ |V_{G}(Dw)-V_{G}(Dh)|^2 + a(x_a)|V_{H_{a}}(Dw)-V_{H_{a}}(Dh)|^2 \right]\,dx
		\\&
		\leqslant
		c[Q_{3}(R)]^{s_1}\FI_{B_{R/2}}\left[\Psi_{d}\right]_{B_{R}}^{-}\left(\left|\frac{w-(w)_{B_{R/2}}}{R} \right| \right)\,dx
		\leqslant
		c[Q_3(R)]^{s_1}\FI_{B_{R/2}}\left[\Psi_{d}\right]_{B_{R}}^{-}\left(\left|Dw \right| \right)\,dx
		\\&
		\leqslant
		c[Q_{3}(R)]^{s_1}\FI_{B_{2R}}\Psi_{d}\left(x,\left|Du \right| \right)\,dx
	\end{split}
\end{align}
for some constant $c\equiv c(\data_{d}(\O_0))$, where we have applied the Sobolev-Poincar\'e inequality and \eqref{d3ce:6}. Combining this estimate together with \eqref{d0ce:3} via some elementary computations implies \eqref{d5ce:1}. 
\end{proof}

 
 \textbf{Proof of Theorem \ref{dmth:mr}.}
It follows from Theorem \ref{dmth:md} and a standard covering argument that, for every open subset $\O_0\Subset \O$ and any number $k>0$, there exists a constant  $c\equiv c(\data_{d}(\O_0),k)$ such that 
\begin{align}
	\label{dmr:1}
	\FI_{B_{2R}}\Psi_{d}(x,|Du|)\,dx \leqslant cR^{-k}
\end{align}
for every $B_{2R}\subset\O_0$ with $R\leqslant 1$. Now we fix an open subset  $\O_0\Subset\O$ and a ball $B_{2R}\equiv B_{2R}(x_0)\subset \O_0$ with $R\leqslant 1$. Then applying Lemma \ref{lem:d3ce} and Lemma \ref{lem:d5ce}, 

\begin{align}
	\label{dmr:2}
	\begin{split}
	\FI_{B_{R/8}} &\left( |V_{G}(Du)-V_{G}(Dh)|^2  + a(x_a)|V_{H_{a}}(Du)-V_{H_{a}}(Dh)|^2 \right)\,dx
	\\&
	\leqslant
	c\left( R^{\mu\gamma} + [Q(R)]^{s_1} \right)\FI_{B_{2R}}\Psi_{d}(x,|Du|)\,dx
	\end{split}
\end{align}
for some constant $c\equiv c(\data_{d}(\O_0))$ and $s_1\equiv s_1(\data_{d})$, where 

\begin{align}
	\label{dmr:3}
	 Q(R)
	 := \left\{\begin{array}{lr}
        R^{\mu} + (1+R^{\mu})R^{\frac{\alpha\delta_0}{(1+\delta_0)(1+s(H_{a}))}}   & \text{if } \eqref{dmth:mr:1} \text{ is assumed, } \\
       R^{\mu} + (1+R^{\mu})R^{\frac{\alpha\gamma}{1+s(H_{a})}}   & \text{if } \eqref{dmth:mr:2} \text{ is assumed, } \\
        R^{\mu} + (1+R^{\mu})R^{\frac{\alpha}{2(1+s(H_{a}))}} & \text{if } \eqref{dmth:mr:3} \text{ is assumed, }
        \end{array}\right.
\end{align}
in which $\gamma$ is the H\"older continuity exponent determined via Theorem \ref{thm:hc} under the double phase settings and $\delta_0$ is the higher integrability exponent coming from Lemma \ref{lem:d0ce}. Denoting by
\begin{align}
	\label{dmr:3_1}
	 d\equiv d(\data_{d}(\O_0))
	 := \left\{\begin{array}{lr}
        \min\left\{\mu\gamma, s_1\mu, \frac{\alpha\delta_0 s_1}{(1+\delta_0)(1+s(H_{a}))} \right\} & \text{if } \eqref{dmth:mr:1} \text{ is assumed, } \\
        \min\left\{\mu\gamma, s_1\mu, \frac{\alpha\gamma s_1}{1+s(H_{a})} \right\}  & \text{if } \eqref{dmth:mr:2} \text{ is assumed, } \\
        \min\left\{\mu\gamma,  s_1\mu, \frac{\alpha s_1}{2(1+s(H_{a}))} \right\}  & \text{if } \eqref{dmth:mr:3} \text{ is assumed, }
        \end{array}\right.
\end{align}
and choosing $k\equiv d/4$ in \eqref{dmr:1}, the inequality \eqref{dmr:2} can be written as 
\begin{align}
	\label{dmr:4}
	\begin{split}
	\FI_{B_{R/8}} &\left( |V_{G}(Du)-V_{G}(Dh)|^2  + a(x_a)|V_{H_{a}}(Du)-V_{H_{a}}(Dh)|^2  \right)\,dx
	\leqslant
	cR^{3d/4}
	\end{split}
\end{align}
for some constant $c\equiv c(\data_{d}(\O_0))$, where we again recall that the function $h$ has been defined via Lemma \ref{lem:d3ce} and Lemma \ref{lem:d5ce}. Recalling \eqref{d3ce:3} and \eqref{d5ce:2}, we have the energy estimate 
\begin{align}
	\label{dmr:5}
	\FI_{B_{R/8}}\left[\Psi_{d}\right]_{B_{R}}^{-}(|Dh|)\,dx
	\leqslant
	c\FI_{B_{2R}}\Psi_{d}(x,|Du|)\,dx
\end{align} 
with a constant $c\equiv c(n,\nu,L)$. Once we arrive at this stage, the rest of the proof is in the same way as argued in the proof of Theorem \ref{mth:mr}.  The proof is complete.


\section{Regularity results under additional  integrability }
\label{sec:11}
We turn our attention to studying properties of a local $Q$-minimizer of the functional $\P$ defined in \eqref{ifunct} under some additional Lebesgue integrability assumption. We shall consider a local $Q$-minimizer $u$ of the functional $\P$ under the following assumptions:

\begin{align}
	\label{ma:4}
	\begin{cases}
		u\in W^{1,\Psi}(\O)\cap L^{\kappa}(\O)\quad \left(\kappa\geqslant 1 \right)
		\\
		\lambda_{4}(\kappa):= \displaystyle \sup\limits_{t>0}\Lambda\left(t^{\frac{\kappa}{n+\kappa}}, \frac{1}{t} \right) < \infty,
	\end{cases}
\end{align}
where the function $\Lambda : (0,\infty)\times (0,\infty) \rightarrow (0,\infty)$ has been defined in \eqref{Lambda} together with $\omega_{a}, \omega_{b} : [0,\infty)\rightarrow [0,\infty)$ being concave functions vanishing at the origin such that $a(\cdot)\in C^{0,\omega_{a}}(\O)$ and $b(\cdot)\in C^{0,\omega_{b}}(\O)$. To see the meaning of the assumption $\eqref{ma:4}_{2}$, let us consider the standard double phase that $G(t)=t^{p}$, $H_{a}(t)=t^{q}$ and $\omega_{a}(t)=t^{\alpha}$, $\omega_{b}\equiv 0$ for $1<p\leqslant q$ and $\alpha\in (0,1]$. Under these standard double phase settings, the assumption $\eqref{ma:4}_{2}$ is equivalent to the following one:
\begin{align}
	\label{ma:4_1}
	q \leqslant p + \frac{\alpha \kappa}{n+\kappa}.
\end{align}
A local $Q$-minimizer $u\in W^{1,\Psi}(\O)$  implies that $u\in W^{1,p}(\O)$. It is clearly interesting point that $p<n$, otherwise we can prove $u\in L^{\infty}_{\loc}(\O)$ by using Morrey-Embedding properties for $p>n$ and using a higher integrability for $p=n$. Then, for $1<p<n$, applying Sobolev embedding properties, one can see that $u\in L^{\frac{np}{n-p}}_{\loc}(\O)$. Choosing $\kappa \equiv \frac{np}{n-p}$, the condition \eqref{ma:4_1} is equivalent to the  following one 
\begin{align*}
q \leqslant p + \frac{\alpha p}{n},
\end{align*}
which generates the same condition as \eqref{(p,q):con1}, as we have discussed in the introduction part. Now if $\kappa > \frac{np}{n-p}$, then we would have 
\begin{align*}
	q \leqslant p + \frac{\alpha p}{n} < p + \frac{\alpha\kappa}{n+\kappa},
\end{align*}
which tells us the possible range of $q$ is larger than the one in \eqref{(p,q):con1}. Considering a local $Q$-minimizers of the functional $\P$ under the assumption \eqref{ma:4} , we shall show that $u\in L^{\infty}_{\loc}(\O)$. To do this, we start by proving a Sobolev-Poincar\'e inequality under the assumption $\eqref{ma:4}_{2}$.
\begin{thm}
	\label{thm:sp_ai}
	  Let $v\in W^{1,\Psi}(B_{R})\cap L^{\kappa}(B_{R})$ for a ball $B_{R}\subset\O$ with $R\leqslant 1$ under the assumption $\eqref{ma:4}_{2}$. Then, for any $d\in \left[1,\frac{n(n+\kappa)}{n(n+\kappa)-\kappa} \right)$, there exist constants $\theta\equiv \theta(n,s(G),s(H_{a}),s(H_{b}),\kappa, d)\in (0,1)$ and $c\equiv c(n, s(G),s(H_{a}),s(H_{b}), \omega_a(1), \omega_b(1),\kappa,d)$ such that the following Sobolev-Poincar\'e-type inequality holds:
	\begin{align}
		\label{sp_ai:1}
		\left[\FI_{B_{R}}\left[\Psi\left(x,\left|\frac{v-(v)_{B_{R}}}{R}\right|\right)\right]^{d}\,dx\right]^{\frac{1}{d}} \leqslant
		c\lambda_{sp}\left[ \FI_{B_{R}} [\Psi(x,|Dv|)]^{\theta}\,dx \right]^{\frac{1}{\theta}},
	\end{align}
	 where
\begin{align}
	\label{sp_ai:2}
	\lambda_{sp}= 1+([a]_{\omega_{a}}+ [b]_{\omega_{b}})\left(\lambda_{4}(\kappa) +  \lambda_{4}(\kappa)\left( \I_{B_{R}}|v|^{\kappa}\,dx \right)^{\frac{1}{n+\kappa}}\right)
\end{align}

Moreover, the above estimate \eqref{sp_ai:1} is still valid with $v-(v)_{B_{R}}$ replaced by $v$ if $v\in W^{1,\Psi}_{0}(B_{R})\cap L^{\kappa}(B_{R})$.
\end{thm}

\begin{proof}
Note that the above theorem covers \cite[Theorem 3.1]{Ok3}, which is a special case when $G(t)=t^{p}$, $H(t)=t^{q}$, $\omega_{a}(t)=t^{\alpha}$ and $\omega_{b}(\cdot)\equiv 0$ for some constants $1<p\leqslant q$ and $\alpha\in (0,1]$. Also our proof is much more elementary comparing with the approach used there. Using the continuity of the coefficient functions $a(\cdot)$ and $b(\cdot)$ and arguing in the same way as in \eqref{sp:3}, we find
\begin{align}
		\label{sp_ai:3}
		\begin{split}
		I:= \left(\FI_{B_{R}}\left[\Psi\left(x,\left|\frac{v-(v)_{B_{R}}}{R}\right|\right)\right]^{d}\,dx \right)^{\frac{1}{d}}
		&\leqslant
		 6[a]_{\omega_{a}}\omega_{a}(R)\left(\FI_{B_{R}} \left[H_{a}\left(\left|\frac{v-(v)_{B_{R}}}{R}\right|\right)\right]^{d}\,dx \right)^{\frac{1}{d}}
		 \\&
		 \qquad
		 + 6[b]_{\omega_{b}}\omega_{b}(R)\left(\FI_{B_{R}}\left[H_{b}\left(\left|\frac{v-(v)_{B_{R}}}{R}\right|\right)\right]^{d}\,dx \right)^{\frac{1}{d}}
		\\&
		\qquad
		+ 
		3\left(\FI_{B_{R}}\left[\Psi_{B_{R}}^{-}\left(\left|\frac{v-(v)_{B_{R}}}{R}\right|\right)\right]^{d}\,dx\right)^{\frac{1}{d}}
		\\&
		 =: 6[a]_{\omega_{a}}I_{1} + 6[b]_{\omega_{b}}I_{2} + 3I_{3}.
		\end{split}
	\end{align}
We now shall deal with estimating the terms $I_{i}$ with $i\in {1,2,3}$ in \eqref{sp_ai:3} using the additional a priori assumption $u\in L^{\kappa}(B_{R})$ under $\eqref{ma:4}_{2}$.
In turn, using \eqref{concave2}  and the assumption $\eqref{ma:4}_{2}$, we see
\begin{align*}
	\begin{split}
	I_{1} &= \displaystyle  \omega_{a}(R) \left(\FI_{B_{R}} \left[\frac{H_{a}\left(\left|\frac{v-(v)_{B_{R}}}{R}\right|\right)}{G\left(\left|\frac{v-(v)_{B_{R}}}{R}\right|\right)} G\left(\left|\frac{v-(v)_{B_{R}}}{R}\right|\right)\right]^{d}\,dx\right)^{\frac{1}{d}}
	\\&
	\leqslant
	\lambda_{4}(\kappa)\omega_{a}(R)\left(
	\FI_{B_{R}} \left[\left(1+ \left[ \omega_{a}\left( \left(\left|\frac{v-(v)_{B_{R}}}{R}\right|\right)^{-\frac{\kappa}{n+\kappa}} \right)  \right]^{-1} \right)G\left(\left|\frac{v-(v)_{B_{R}}}{R}\right|\right)\right]^{d}\,dx\right)^{\frac{1}{d}}
	\\&
	\leqslant
	\lambda_{4}(\kappa)\omega_{a}(R)
	\left(\FI_{B_{R}}\left[ \left(1+ \left[ \frac{1}{\omega_{a}(R)} + \frac{R}{\omega_{a}(R)} \left|\frac{v-(v)_{B_{R}}}{R}\right|^{\frac{\kappa}{n+\kappa}} \right] \right)G\left(\left|\frac{v-(v)_{B_{R}}}{R}\right|\right)\right]^{d}\,dx\right)^{\frac{1}{d}}
	\\&
	\leqslant
	c_{*}\lambda_{4}(\kappa)
	\left(\FI_{B_{R}}\left[G\left(\left|\frac{v-(v)_{B_{R}}}{R}\right|\right)\right]^{d}\,dx\right)^{\frac{1}{d}}
	\\&
	\qquad
	+
	c_{*}\lambda_{4}(\kappa) R^{\frac{n}{n+\kappa}}\left(\FI_{B_{R}} \left|v-(v)_{B_{R}}\right|^{\frac{d\kappa}{n+\kappa}} \left[G\left(\left|\frac{v-(v)_{B_{R}}}{R}\right|\right)\right]^{d}\,dx\right)^{\frac{1}{d}} 
	\end{split}
\end{align*}
for the constant $c_{*} = 2(1+\omega_{a}(1))$. Using H\"older's inequality with conjugate exponents $\left( \frac{n+\kappa}{d}, \frac{n+\kappa}{n+\kappa-d} \right)$, we have 
\begin{align*}
	\begin{split}
	& R^{\frac{n}{n+\kappa}}\left(\FI_{B_{R}} \left|v-(v)_{B_{R}}\right|^{\frac{d\kappa}{n+\kappa}} \left[G\left(\left|\frac{v-(v)_{B_{R}}}{R}\right|\right)\right]^{d}\,dx\right)^{\frac{1}{d}}
	\\&
	\leqslant
	R^{\frac{n}{n+\kappa}}\left(\FI_{B_{R}} \left|v-(v)_{B_{R}}\right|^{\kappa}\,dx\right)^{\frac{1}{n+\kappa}} \left(\FI_{B_{R}}\left[G\left(\left|\frac{v-(v)_{B_{R}}}{R}\right|\right)\right]^{\frac{(n+\kappa)d}{n+\kappa-d}}\,dx\right)^{\frac{n+\kappa-d}{(n+\kappa)d}}
	\\&
	\leqslant
	c \left(\I_{B_{R}} \left|v\right|^{\kappa}\,dx\right)^{\frac{1}{n+\kappa}} \left(\FI_{B_{R}}\left[G\left(\left|\frac{v-(v)_{B_{R}}}{R}\right|\right)\right]^{\frac{(n+\kappa)d}{n+\kappa-d}}\,dx\right)^{\frac{n+\kappa-d}{(n+\kappa)d}}
	\end{split}
\end{align*}
for some constant $c\equiv c(n)$. 
Combining the last two displays and arguing similarly for $I_{2}$, we discover 
\begin{align*}
	\begin{split}
	I_{1} +I_{2} 
	&\leqslant
	c\lambda_{4}(\kappa)
	\left(\FI_{B_{R}}\left[G\left(\left|\frac{v-(v)_{B_{R}}}{R}\right|\right)\right]^{d}\,dx\right)^{\frac{1}{d}}
	\\&
	\qquad
	+
	c \lambda_{4}(\kappa) \left(\I_{B_{R}} \left|v\right|^{\kappa}\,dx\right)^{\frac{1}{n+\kappa}} \left(\FI_{B_{R}}\left[G\left(\left|\frac{v-(v)_{B_{R}}}{R}\right|\right)\right]^{\frac{(n+\kappa)d}{n+\kappa-d}}\,dx\right)^{\frac{n+\kappa-d}{(n+\kappa)d}}
	\end{split}
\end{align*}
for some constant $c\equiv c(n,\omega_{a}(1),\omega_{b}(1))$. Now we apply Lemma \ref{lem:os} to $\Phi\equiv G$ with $d_{0}\equiv d$ and $d_{0}\equiv \frac{n+\kappa-d}{(n+\kappa)d}$ in order to have an exponent $\theta_1\equiv \theta_1(n,s(G),\kappa,d)\in (0,1)$ such that
\begin{align}
	\label{sp_ai:7}
	I_{1} + I_{2} \leqslant c\left( \lambda_{4}(\kappa) + \lambda_{4}(\kappa)\left(\I_{B_{R}} \left|v\right|^{\kappa}\,dx\right)^{\frac{1}{n+\kappa}} \right)
	\left(\FI_{B_{R}}\left[G\left(\left|Dv\right|\right)\right]^{\theta_{1}}\,dx\right)^{\frac{1}{\theta_1}}
\end{align}
holds for some constant $c\equiv c(n,s(G),\omega_{a}(1), \omega_{b}(1),\kappa,d)$. On the other hand, since $\Psi_{B_{R}}^{-}\in \N$ with an index $s(\Psi)=s(G)+s(H_{a})+s(H_{b})$ by Remark \ref{rmk:nf2}, we are able to apply Lemma \ref{lem:os} with $\Phi\equiv \Psi_{B_{R}}^{-}$ for $d_{0}\equiv d$. In turn, there exists $\theta_2\equiv \theta_2(n,s(\Psi),d)$ such that 
\begin{align}
	\label{sp_ai:8}
	I_3 \leqslant c \left[\FI_{B_{R}}\left[\Psi_{B_{R}}^{-}(|Dv|)\right]^{\theta_2}\,dx\right]^{\frac{1}{\theta_{2}}}
\end{align}
with some constant $c\equiv c(n,s(\Psi),d)$. Taking into account the estimates obtained in \eqref{sp_ai:7}-\eqref{sp_ai:8} into \eqref{sp_ai:3}, recalling the very definition of $\Psi_{B_{R}}^{-}$ in \eqref{ispsi} and setting $\theta := \max\{\theta_1,\theta_2\}$, we arrive at \eqref{sp_ai:1}. The proof is finished.
\end{proof}

\begin{rmk}
	\label{rmk:sp_ai} We here remark that choosing $d\equiv 1$ in a Sobolev-Poincar\'e type inequality of Theorem \ref{thm:sp_ai}, we see that there exists an exponent  $\theta\equiv \theta(n,s(G),s(H_a),s(H_{b}),\kappa)$  such that 
	\begin{align}
		\FI_{B_{R}}\Psi\left(x,\left|\frac{v-(v)_{B_{R}}}{R}\right|\right)\,dx \leqslant
		c\lambda_{sp}\left[ \FI_{B_{R}} [\Psi(x,|Dv|)]^{\theta}\,dx \right]^{\frac{1}{\theta}}
	\end{align}
	holds for some constant $c\equiv c(n,s(G),s(H_{a}),s(H_{b}),\omega_a(1),\omega_b(1),\kappa)$, where $\lambda_{sp}$ is the one same as in \eqref{sp_ai:2}.
\end{rmk}

\begin{rmk}
	\label{rmk:sp_ai:2}
	With $u\in W^{1,\Psi}(\O)$ being a local $Q$-minimizer of the functional $\P$, we here point out that it is also possible to suppose a priori $u\in W^{\Phi}(\O)$ for some Young function $\Phi$. In this case, discovering a relevant assumption like $\eqref{ma:4}_{2}$ would be an interesting point to find how it is connected to Embedding properties in Orlicz-Sobolev spaces \cite{CP1, Cia1, Cia2} likewise we have discussed above in Lebesgue settings. Moreover, proving various regularity results under a new relevant condition may generate a different phenomenon even for a Lavrentiev gap. We can also a priori assume that local $Q$-minimizers belong to certain Campanato, BMO, VMO, or some other spaces. Under all those a priori assumptions, it should be necessary to discover out the relevant optimal conditions under which various regularity results are obtainable.  
\end{rmk}

For a local $Q$-minimizer $u$ of the functional $\P$ under the assumption \eqref{ma:4}, the data of the problem is understood by the following set of parameters: 
\begin{align}
	\label{data_ai}
	\data_{i} \equiv \{n,\lambda_{4}(\kappa),\kappa, s(G),s(H_{a}), s(H_{b}), \omega_a(1), \omega_{b}(1), \norm{u}_{L^{\kappa}(\O)}, Q\}.
\end{align}
As usual, for any open subset $\O_0\Subset \O$, we denote by $\data_{i}(\O_0)$ the set of parameters defined above together with $\dist(\O_0,\partial\O)$. Now we focus on showing  local boundedness estimates of a local $Q$-minimizer $u$ of the functional $\P$ in \eqref{ifunct} under the assumption $\eqref{ma:4}$.

\begin{thm}
	\label{thm:lb_ai}
	Let $u\in W^{1,\Psi}(\O)$ be a local $Q$-minimizer of the functional $\P$ in \eqref{ifunct} under the assumption \eqref{ma:4}. Then there exists a constant $c\equiv c(\data_{i})$ such that 
	\begin{align}
		\label{lb_ai:1}
		\norm{\Psi_{B_{R}}^{-}\left(\left|\frac{(u-(u)_{B_{R}})_{\pm}}{R} \right| \right)}_{L^{\infty}(B_{R/2})}
		\leqslant 
		c \FI_{B_{R}}\Psi\left(x,\left|\frac{(u-(u)_{B_{R}})_{\pm}}{R} \right| \right)\,dx
	\end{align}
	and 
	\begin{align}
		\label{lb_ai:1_1}
		\Psi_{B_{R}}^{-}\left(\left|\frac{u(x_1)-u(x_2)}{R} \right| \right)
		\leqslant 
		c \FI_{B_{R}}\Psi\left(x,\left|Du\right| \right)\,dx
		\quad\text{for a.e}\quad x_1,x_2\in B_{R/2},
	\end{align}
 whenever $B_{R}\equiv B_{R}(x_0)\subset\O$ is a ball with $R\leqslant 1$. In particular, $u\in L^{\infty}_{\loc}(\O)$.	
\end{thm}

\begin{proof}
	The meaning of $\data_{i}$ under the assumption \eqref{ma:4}, already has been introduced in \eqref{data_ai}. As in the proof of Theorem \ref{thm:lb}, we consider the following scaled functions as: 
	\begin{align}
		\label{lb_ai:2}
		\begin{split}
			&\bar{u}(x):= \frac{u(x_0 + Rx)-(u)_{B_{R}}}{R},\quad
			\bar{a}(x):= a(x_0+Rx),\quad 
			\bar{b}(x):= b(x_0+Rx),
			\\&
			\bar{\Psi}(x,t):= G(t) + \bar{a}(x)H(t) + \bar{b}(x)H(t),
			\\&
			\bar{A}(k,s):= B_{s}(0)\cap\{\bar{u}>k\}
			\quad\text{and}\quad
			\bar{B}(k,s):= B_{s}(0)\cap\{\bar{u}<k\}
		\end{split}
	\end{align}
	for every $x\in B_{1}(0)$, $t\geqslant 0$, $s\in (0,1)$ and $k\in\R$. The remaining part of the proof consists of 3 steps as in the proof of Theorem \ref{thm:lb}.
	
\textbf{Step 1: Sobolev-Poincar\'e under the scaling in \eqref{lb_ai:2}.}
	 In this step, we prove that there exists a positive exponent $\theta\equiv \theta(n,s(G),s(H_{a}),s(H_{b}),\kappa)\in (0,1)$ such that 
	\begin{align}
		\label{lb_ai:3}
		\I_{B_{1}}\bar{\Psi}(x,|f|)\,dx \leqslant c\bar{k}_{sp} \left(\I_{B_{1}}[\bar{\Psi}(x,|Df|)]^{\theta}\,dx\right)^{\frac{1}{\theta}}
	\end{align}
	for some constant $c\equiv c(n,s(G),s(H_{a}),s(H_{b}),\omega_a(1),\omega_{b}(1),\kappa)$, whenever $f\in W^{1,\bar{\Psi}}_{0}(B_{1})\cap L^{\kappa}(B_{1})$, where
	\begin{align*}
		\bar{\kappa}_{sp} = 1+([a]_{\omega_{a}} + [b]_{\omega_{b}})\left(\lambda_1 + \lambda_1 R\left(\I_{B_{1}}|f|^{\kappa} \,dx\right)^{\frac{1}{n+\kappa}}\right).
	\end{align*}
Using the continuity properties of $\bar{a}(\cdot)$ and $\bar{b}(\cdot)$, we see 
\begin{align*}
	\begin{split}
	I:&= \I_{B_{1}}\bar{\Psi}(x,|f|)\,dx 
	\\&
	\leqslant 
	2[a]_{\omega_{a}}\omega_a(R)\I_{B_{1}}H_{a}(|f|)\,dx + 2[b]_{\omega_{b}}\omega_{b}(R)\I_{B_{1}}H_{b}(|f|)\,dx + \I_{B_{1}}\bar{\Psi}_{B_1}^{-}(|f|)\,dx
	\\&
	=: 2[a]_{\omega_{a}}I_1 + 2[b]_{\omega_{b}}I_2 +I_3,
	\end{split}
\end{align*}
where 
\begin{align*}
	\bar{\Psi}_{B_{1}}^{-}(t):= G(t) + \inf\limits_{x\in B_{1}}\bar{a}(x)H_{a}(t) + \inf\limits_{x\in B_{1}}\bar{b}(x)H_{b}(t)
	\quad\text{for every}\quad
	t\geqslant 0.
\end{align*}
Now we estimate the terms $I_{i}$ for $i\in \{1,2,3\}$ similarly as in the proof of Theorem \ref{thm:sp_ai}. In turn, using the assumption $\eqref{ma:4}_{2}$ and \eqref{concave2}, we have 
\begin{align*}
	\begin{split}
		I_1 &= \omega_{a}(R)\I_{B_{1}}\frac{H_{a}(|f|)}{G(|f|)}G(|f|)\,dx
		\\&
	\leqslant
	\lambda_{4}(\kappa)\omega_{a}(R)
	\I_{B_{1}} \left(1+ \left[ \omega_{a}\left( \left|f\right|^{-\frac{\kappa}{n+\kappa}}  \right) \right]^{-1} \right)G\left(\left|f\right|\right)\,dx
	\\&
	\leqslant
	\lambda_{4}(\kappa)\omega_{a}(R)
	\I_{B_{1}} \left(1+ \left[ \frac{1}{\omega_{a}(R)} + \frac{R}{\omega_{a}(R)}\left|f\right|^{\frac{\kappa}{n+\kappa}} \right] \right)G\left(\left|f\right|\right)\,dx
	\\&
	\leqslant
	\lambda_{4}(\kappa)(1+\omega_{a}(1)) \I_{B_{1}}G\left(\left|f\right|\right)\,dx + 2\lambda_{4}(\kappa) R \I_{B_{1}} |f|^{\frac{\kappa}{n+\kappa}} G\left(\left|f\right|\right) \,dx.
	\end{split}
\end{align*}
Arguing in the same way, we have
\begin{align*}
	I_{2} \leqslant \lambda_{4}(\kappa)(1+\omega_{b}(1)) \I_{B_{1}}G\left(\left|f\right|\right)\,dx + 2\lambda_{4}(\kappa) R \I_{B_{1}} |f|^{\frac{\kappa}{n+\kappa}} G\left(\left|f\right|\right) \,dx.
\end{align*}
Then the inequality \eqref{lb_ai:3} follows from the arguments used in the proof of Theorem \ref{thm:sp_ai} and Lemma \ref{lem:os}.

\textbf{Step 2. Proof of \eqref{lb_ai:1}.}
	 Since $u-(u)_{B_{R}}$ is a local $Q$-minimizer of the functional $\P$ in \eqref{ifunct}, using a Caccioppoli inequality of Lemma \ref{lem:cacc}, one can see that 
	\begin{align}
		\label{lb_ai:9}
		\I_{B_{t}} \bar{\Psi}(x,|D(\bar{u}-k)_{\pm}|)\,dx 
		\leqslant 
		c\I_{B_{s}} \bar{\Psi}\left(x,\frac{(\bar{u}-k)_{\pm}}{s-t} \right)\,dx
	\end{align}
	holds for some constant $c\equiv c(s(G),s(H_{a}),s(H_{b}),Q)$, whenever $0<t<s \leqslant 1$ and $k\in\R$. Let us now consider the concentric balls $B_{\rho}\Subset B_{t} \Subset B_{s} $ with $1/2 \leqslant \rho < s \leqslant 1$ and $t:= (\rho+s)/2$. Let $\eta\in C_{0}^{\infty}(B_{t})$ be a standard cut-off function such that $\chi_{B_{\rho}} \leqslant \eta \leqslant \chi_{B_{t}}$ and $|D\eta| \leqslant \frac{2}{t-\rho} = \frac{4}{s-\rho}$. Now we apply inequality \eqref{lb_ai:3} from Step 1 above in order to have a positive exponent $\theta\equiv \theta(n,s(G),s(H_{a}),s(H_{b}),\kappa)$ such that
	\begin{align*}
		\I_{\bar{A}(k,\rho)}\bar{\Psi}(x,\bar{u}-k)\,dx
		\leqslant
		\I_{B_{1}}\bar{\Psi}(x,\eta(\bar{u}-k)_{+})\,dx
		\leqslant
		c\bar{k}_{sp}\left(\I_{B_{1}} [\bar{\Psi}(x,|D(\eta(\bar{u}-k)_{+})|)]^{\theta}\,dx \right)^{\frac{1}{\theta}}
	\end{align*}
for some constant $c\equiv c(n,s(G),s(H_{a}),s(H_{b}),\omega_{a}(1),\omega_{b}(1),\kappa)$, where
	\begin{align*}
		\bar{\kappa}_{sp}= 1+ ([a]_{\omega_{a}} + [b]_{\omega_{b}})\left(\lambda_{4}(\kappa) + \lambda_{4}(\kappa)R\left(\I_{B_{1}}\left[\eta(\bar{u}-k)_{+})\right]^{\kappa} \,dx\right)^{\frac{1}{n+\kappa}}\right).
	\end{align*}
By recalling the definition of $\bar{u}$ in \eqref{lb_ai:2}, we have 
\begin{align*}
	\begin{split}
	\bar{\kappa}_{sp} &\leqslant c\left[1+ R \left(\FI_{B_{R}}\left|\frac{u-(u)_{B_{R}}}{R}\right|^{\kappa} \,dx\right)^{\frac{1}{n+\kappa}} \right]	
	\leqslant
	c\left[ 1 + \left(\I_{B_{R}} |u|^{\kappa} \,dx\right)^{\frac{1}{n+\kappa}} \right]
	\end{split}
\end{align*}
with a constant $c\equiv c(n,\lambda_{4}(\kappa), [a]_{\omega_{a}} + [b]_{\omega_{b}})$. Once we arrive at this stage the rest of the proof can be proceed in the same way as in the proof of Theorem \ref{thm:lb}.
\end{proof}

\begin{thm}
 	\label{thm:hc_ai}
 	Let $u\in W^{1,\Psi}(\O)$ be a local $Q$-minimizer of the functional $\P$ defined in \eqref{ifunct} under the coefficient functions $a(\cdot)\in C^{\omega_a}(\O)$ and $b(\cdot)\in C^{\omega_{b}}(\O)$ for $\omega_{a}, \omega_{b}$ being non-negative concave functions vanishing at the origin. If the assumption $\eqref{ma:4}$ is satisfied, then for for every open subset $\O_{0}\Subset\O$, there exists a H\"older continuity exponent
 		$\gamma\equiv \gamma(\data_{i}(\O_{0}))\in (0,1)$ such that
 		\begin{align}
 		\label{hc_ai:1}
 		\norm{u}_{L^{\infty}(\O_{0})} + [u]_{0,\gamma;\O_{0}} \leqslant c(\data_{i}(\O_{0}))
 	\end{align}
 	and the oscillation estimate
 	\begin{align}
 		\label{hc_ai:2}
 		\osc\limits_{B_{\rho}} u \leqslant c\left(\frac{\rho}{R} \right)^{\gamma}\osc\limits_{B_{R}} u
 	\end{align}
 	holds for some $c\equiv c(\data_{i}(\O_0))$ and all concentric balls
 	$B_{\rho}\Subset B_{R}\Subset \O_{0}\Subset \O$ with $R\leqslant 1$.
 \end{thm}
 
\begin{proof}
First let us observe that, for every $t\geqslant 1$, we have 
\begin{align*}
	\begin{split}
	\frac{\omega_{a}(t)}{1+\omega_{a}(t)} \frac{1+ \omega_{a}\left( t^{\frac{\kappa}{n+\kappa}} \right)}{\omega_{a}\left( t^{\frac{\kappa}{n+\kappa}} \right)}
	&\leqslant
	1 + \frac{\omega_{a}(t)}{\omega_{a}\left( t^{\frac{\kappa}{n+\kappa}} \right) + \omega_{a}(t)\omega_{a}\left( t^{\frac{\kappa}{n+\kappa}} \right)}
	\\&
	\leqslant
	1 + \frac{1}{\omega_{a}\left( t^{\frac{\kappa}{n+\kappa}} \right)}
	\leqslant 
	1+\frac{1}{\omega_{a}(1)}.
	\end{split}
\end{align*}
This same inequality holds true also for $\omega_{b}$. Therefore, for every $t\geqslant 1$, we see that 
\begin{align*}
	\begin{split}
	\Lambda\left(t, \frac{1}{t} \right) 
	&\leqslant \lambda_{4}(\kappa)\left( \frac{\omega_{a}(t)}{1+\omega_{a}(t)} \frac{1+ \omega_{a}\left( t^{\frac{\kappa}{n+\kappa}} \right)}{\omega_{a}\left( t^{\frac{\kappa}{n+\kappa}} \right)} + 
	\frac{\omega_{b}(t)}{1+\omega_{b}(t)} \frac{1+ \omega_{b}\left( t^{\frac{\kappa}{n+\kappa}} \right)}{\omega_{b}\left( t^{\frac{\kappa}{n+\kappa}} \right)}\right)
	\\&
	\leqslant
	\lambda_{4}(\kappa)\left( 1 + \frac{1}{\omega_{a}(1)} + \frac{1}{\omega_{b}(1)} \right)=: \lambda_{2},
	\end{split}
\end{align*}
where we have used the assumption $\eqref{ma:4}_{2}$. On the other hand, recalling that the functions $\omega_{a}$ and $\omega_{b}$ are increasing, we have 
\begin{align*}
	\Lambda\left(t,\frac{1}{t} \right) \leqslant 
	\Lambda\left(t^{\frac{\kappa}{n+\kappa}},\frac{1}{t} \right) \leqslant \lambda_{4}(\kappa) \leqslant \lambda_{2}
\end{align*}
for every $t\in (0,1]$. Recalling that $u\in L^{\infty}_{\loc}(\O)$ by Theorem \ref{thm:lb_ai} and taking into account the last two displays, we are able to apply Theorem \ref{thm:hc} in order to have \eqref{hc_ai:1} and \eqref{hc_ai:2}.
\end{proof}

\begin{rmk}
	\label{rmk:hc_ai:1}
	As a consequence of the last two theorems like we have that if $u\in W^{1,\Psi}(\O)$ is a local $Q$-minimizer of the functional $\P$ under the assumption \eqref{ma:4}, then $u\in L^{\infty}_{\loc}(\O)$ and $\eqref{ma:2}_{2}$ is satisfied. Therefore, the results of Theorem \ref{mth:mr}, Theorem \ref{mth:md}, Theorem \ref{dmth:mr} and Theorem \ref{dmth:md} are still available under the assumption \eqref{ma:4}. Furthermore, the results of the present section can be considered under multi-phase settings, as we have pointed out in Remark \ref{rmk:multi}.
\end{rmk}

%


\bibliographystyle{amsplain}

\end{document}